\pdfoutput=1
\documentclass[11pt]{amsart}
\usepackage{lmodern}
\usepackage{ifthen}
\usepackage{amsfonts}
\usepackage{amsxtra}
\usepackage{amssymb}
\usepackage{mathdots}
\usepackage{array}
\usepackage[margin=0.8in]{geometry}
\usepackage{xcolor}
\definecolor{cite}{rgb}{0.30,0.60,1.00}
\definecolor{url}{rgb}{0.00,0.00,0.80}
\definecolor{link}{rgb}{0.40,0.10,0.20}
\usepackage[colorlinks,linkcolor=link,urlcolor=url,citecolor=cite,pagebackref,breaklinks]{hyperref}
\usepackage{bbm}
\usepackage{mathtools}
\usepackage{mathrsfs}
\usepackage{appendix}
\usepackage[all]{xy}
\usepackage[lite,abbrev,msc-links,alphabetic]{amsrefs}

\usepackage{graphicx}
\usepackage{multirow}
\usepackage{pstricks}
\usepackage{pst-pdf}
\usepackage{enumitem}
\usepackage{pifont}
\usepackage{marvosym}
\usepackage{txfonts}

\usepackage[OT2,T1]{fontenc}
\DeclareSymbolFont{cyrletters}{OT2}{wncyr}{m}{n}
\DeclareMathSymbol{\Sha}{\mathalpha}{cyrletters}{"58}

\usepackage{graphicx}

\makeatletter
\providecommand*{\Dashv}{%
  \mathrel{%
    \mathpalette\@Dashv\vDash
  }%
}
\newcommand*{\@Dashv}[2]{%
  \reflectbox{$\m@th#1#2$}%
}
\makeatother

\usepackage{nameref}

\makeatletter
\let\orgdescriptionlabel\descriptionlabel
\renewcommand*{\descriptionlabel}[1]{%
  \let\orglabel\label
  \let\label\@gobble
  \phantomsection
  \edef\@currentlabel{#1\unskip}%
  \let\label\orglabel
  \orgdescriptionlabel{#1}%
}
\makeatother

\numberwithin{equation}{section}

\theoremstyle{plain}
\newtheorem{proposition}{Proposition}[section]
\newtheorem{conjecture}[proposition]{Conjecture}
\newtheorem{corollary}[proposition]{Corollary}
\newtheorem{lem}[proposition]{Lemma}
\newtheorem{theorem}[proposition]{Theorem}

\theoremstyle{definition}
\newtheorem{definition}[proposition]{Definition}
\newtheorem{construction}[proposition]{Construction}
\newtheorem{notation}[proposition]{Notation}
\newtheorem{assumption}[proposition]{Assumption}

\theoremstyle{remark}
\newtheorem{remark}[proposition]{Remark}

\renewcommand{\b}[1]{\mathbf{#1}}
\renewcommand{\c}[1]{\mathcal{#1}}
\renewcommand{\d}[1]{\mathbb{#1}}
\newcommand{\f}[1]{\mathfrak{#1}}
\renewcommand{\r}[1]{\mathrm{#1}}
\newcommand{\s}[1]{\mathscr{#1}}
\renewcommand{\sf}[1]{\mathsf{#1}}
\renewcommand{\(}{\left(}
\renewcommand{\)}{\right)}
\newcommand{\res}{\mathbin{|}}
\newcommand{\ol}[1]{\overline{#1}{}}

\newcommand{\ang}[1]{\langle{#1}\rangle}

\renewcommand{\leq}{\leqslant}
\renewcommand{\geq}{\geqslant}

\newcommand{\bM}{\b M}

\newcommand{\bP}{\b P}
\newcommand{\bQ}{\b Q}

\newcommand{\bV}{\b V}

\newcommand{\bi}{\b i}

\newcommand{\cC}{\c C}
\newcommand{\cD}{\c D}
\newcommand{\cE}{\c E}
\newcommand{\cF}{\c F}

\newcommand{\cI}{\c I}

\newcommand{\cM}{\c M}
\newcommand{\cN}{\c N}
\newcommand{\cO}{\c O}

\newcommand{\cR}{\c R}

\newcommand{\cT}{\c T}

\newcommand{\cX}{\c X}

\newcommand{\dA}{\d A}
\newcommand{\dB}{\d B}
\newcommand{\dC}{\d C}

\newcommand{\dE}{\d E}
\newcommand{\dF}{\d F}

\newcommand{\dK}{\d K}
\newcommand{\dL}{\d L}

\newcommand{\dQ}{\d Q}

\newcommand{\dT}{\d T}

\newcommand{\dZ}{\d Z}

\newcommand{\fB}{\f B}

\newcommand{\fL}{\f L}

\newcommand{\fN}{\f N}

\newcommand{\fP}{\f P}

\newcommand{\fS}{\f S}
\newcommand{\fT}{\f T}
\newcommand{\fU}{\f U}

\newcommand{\fX}{\f X}

\newcommand{\fa}{\f a}

\newcommand{\fc}{\f c}

\newcommand{\fe}{\f e}
\newcommand{\ff}{\f f}

\newcommand{\fk}{\f k}
\newcommand{\fl}{\f l}
\newcommand{\fm}{\f m}
\newcommand{\fn}{\f n}

\newcommand{\fp}{\f p}

\newcommand{\fr}{\f r}
\newcommand{\fs}{\f s}

\newcommand{\rB}{\r B}

\newcommand{\rD}{\r D}
\newcommand{\rE}{\r E}
\newcommand{\rF}{\r F}
\newcommand{\rG}{\r G}
\newcommand{\rH}{\r H}
\newcommand{\rI}{\r I}
\newcommand{\rJ}{\r J}
\newcommand{\rK}{\r K}
\newcommand{\rL}{\r L}
\newcommand{\rM}{\r M}

\newcommand{\rP}{\r P}
\newcommand{\rQ}{\r Q}
\newcommand{\rR}{\r R}

\newcommand{\rU}{\r U}
\newcommand{\rV}{\r V}

\newcommand{\rd}{\r d}

\newcommand{\rk}{\r k}

\newcommand{\rt}{\r t}
\newcommand{\ru}{\r u}

\newcommand{\sC}{\s C}
\newcommand{\sD}{\s D}
\newcommand{\sE}{\s E}
\newcommand{\sF}{\s F}
\newcommand{\sG}{\s G}

\newcommand{\sM}{\s M}
\newcommand{\sN}{\s N}
\newcommand{\sO}{\s O}

\newcommand{\sR}{\s R}
\newcommand{\sS}{\s S}
\newcommand{\sT}{\s T}
\newcommand{\sU}{\s U}
\newcommand{\sV}{\s V}
\newcommand{\sW}{\s W}
\newcommand{\sX}{\s X}
\newcommand{\sY}{\s Y}
\newcommand{\sZ}{\s Z}

\newcommand{\sfB}{\sf B}

\newcommand{\sfG}{\sf G}
\newcommand{\sfH}{\sf H}

\newcommand{\sfK}{\sf K}

\newcommand{\sfR}{\sf R}

\newcommand{\sfT}{\sf T}
\newcommand{\sfU}{\sf U}

\newcommand{\tC}{\mathtt{C}}

\newcommand{\tI}{\mathtt{I}}

\newcommand{\tR}{\mathtt{R}}

\newcommand{\tT}{\mathtt{T}}
\newcommand{\tU}{\mathtt{U}}
\newcommand{\tV}{\mathtt{V}}

\newcommand{\tc}{\mathtt{c}}

\newcommand{\te}{\mathtt{e}}

\newcommand{\tj}{\mathtt{j}}

\newcommand{\blambda}{\boldsymbol{\lambda}}

\newcommand{\bkappa}{\boldsymbol{\kappa}}

\newcommand{\bbt}{\boldsymbol{t}}

\newcommand{\pres}[2]{\prescript{#1}{}{#2}}

\newcommand{\ab}{\r{ab}}

\newcommand{\cosp}{\r{cosp}}

\newcommand{\et}{{\acute{\r{e}}\r{t}}}

\newcommand{\FL}{\r{FL}}

\renewcommand{\graph}{\triangle}

\newcommand{\inc}{\r{inc}}
\newcommand{\Inc}{\r{Inc}}

\newcommand{\loc}{\r{loc}}

\newcommand{\ordi}{\r{ord}}

\newcommand{\ram}{\r{ram}}
\newcommand{\rec}{\r{rec}}
\newcommand{\reg}{\r{reg}}

\newcommand{\sing}{\r{sing}}

\newcommand{\temp}{\r{temp}}

\newcommand{\St}{\r{St}}

\newcommand{\tor}{\r{tor}}
\newcommand{\univ}{\r{univ}}
\newcommand{\unr}{\r{unr}}
\newcommand{\ur}{\r{ur}}

\newcommand{\eff}{\mathrm{eff}}

\newcommand{\smooth}{\r{sm}}

\newcommand{\PP}{\lozenge}
\newcommand{\finite}{\sphericalangle}
\renewcommand{\prec}{\to}

\newcommand{\defin}{\r{def}}
\newcommand{\indef}{\r{indef}}

\DeclareMathOperator{\Ann}{Ann}

\DeclareMathOperator{\Aut}{Aut}

\DeclareMathOperator{\Char}{char}
\DeclareMathOperator{\coker}{coker}

\DeclareMathOperator{\Div}{Div}

\DeclareMathOperator{\End}{End}
\DeclareMathOperator{\Ext}{Ext}
\DeclareMathOperator{\Fil}{Fil}

\DeclareMathOperator{\Gal}{Gal}

\DeclareMathOperator{\GL}{GL}
\DeclareMathOperator{\Gr}{Gr}

\DeclareMathOperator{\Hom}{Hom}

\DeclareMathOperator{\ind}{ind}

\DeclareMathOperator{\Ker}{ker}

\DeclareMathOperator{\length}{length}

\DeclareMathOperator{\Nm}{Nm}

\DeclareMathOperator{\rank}{rank}

\DeclareMathOperator{\Sh}{Sh}

\DeclareMathOperator{\Sp}{Sp}
\DeclareMathOperator{\SU}{SU}

\DeclareMathOperator{\Spec}{Spec}
\DeclareMathOperator{\Spf}{Spf}

\DeclareMathOperator{\supp}{supp}

\DeclareMathOperator{\Tor}{Tor}

\DeclareMathOperator{\val}{val}

\begin{document}

\title{Bessel periods and Selmer groups over ordinary Rankin--Selberg eigenvariety}

\author{Yifeng Liu}
\address{Institute for Advanced Study in Mathematics, Zhejiang University, Hangzhou 310058, China}
\email{liuyf0719@zju.edu.cn}

\date{\today}

\begin{abstract}
  We introduce the notion of ordinary distributions for unitary groups and their Rankin--Selberg products, based on which we (re)define the ordinary eigenvarieties. In both definite and indefinite cases, we construct the Bessel period on the Rankin--Selberg eigenvariety as an ordinary distribution and as an element in the Selmer group of ordinary distributions, respectively. We then propose an Iwasawa type conjecture relating the vanishing divisor of the Bessel period and the characteristic divisor of the Selmer group of the associated Rankin--Selberg Galois module over the eigenvariety, and prove one side of the divisibility under certain conditions.
\end{abstract}

\maketitle

\tableofcontents

\section{Introduction}
\label{ss:1}

In the 1980s, Beilinson, Bloch, and Kato in a series of articles (mainly \cites{Bei87,BK90}) proposed a vast generalization of the famous Birch and Swinnerton-Dyer conjecture to motives over number fields, which is nowadays referred as the Beilinson--Bloch--Kato conjecture. It, in particular, predicts a close relation between the central vanishing order of the $L$-function associated with a motive of symplectic type and the rank of the (Bloch--Kato) Selmer group of such a motive. In \cite{LTXZZ}, the authors, for the first time, provided strong evidence toward this conjecture for a series of motives with ranks (or dimensions) tending to infinity. More precisely, they show that for certain complex-conjugate self-dual automorphic motives of $\GL_n\times\GL_{n+1}$ over a CM field $F$, the nonvanishing of the central $L$-value implies that the rank of the Selmer group is zero when the root number is $1$; and the nonvanishing of a certain class in the Selmer group implies that the rank of the Selmer group is one when the root number is $-1$. Later, in \cite{LTX}, the authors updated these results to the level of Iwasawa algebra. More precisely, they consider twists of the motive by Dirichlet characters of $\tilde{F}/F$ for $\tilde{F}$ varying in a certain anticyclotomic $\dZ_p^d$-extension of $F$, extending the $L$-function and the Selmer group to objects lying over the corresponding Iwasawa algebra; then they formulate the corresponding Iwasawa's main conjecture and prove one side of the divisibility (\emph{upper} bounding the Selmer group) for both root numbers. Later, in \cites{JNS,LS24}, the authors also prove the similar results in the rank $1$ case (that is, for root number $-1$) using a different method.

In this article, we further upgrade the Iwasawa theory in \cite{LTX} to the level of ordinary eigenvariety, by noting that twisting anticyclotomic Dirichlet characters is just a certain direction in the space of ordinary deformations. We introduce the notion of ordinary distributions for unitary groups and their Rankin--Selberg products, based on which we (re)define the ordinary eigenvarieties. In both definite and indefinite cases (parallel to cases of root number $1$ and $-1$), we construct the Bessel period on the Rankin--Selberg eigenvariety as an ordinary distribution and as an element in the Selmer group of ordinary distributions, respectively. We then propose an Iwasawa type conjecture relating the vanishing divisor of the Bessel period and the characteristic divisor of the Selmer group of the associated Rankin--Selberg Galois module over the eigenvariety, and prove one side of the divisibility under certain conditions. Now we describe our results in more details.

Throughout the article, we fix a CM extension $F/F^+$, a finite set $\spadesuit^+$ of nonarchimedean places of $F^+$ containing all of those whose underlying rational primes ramify in $F$, a rational prime $p$ not underlying $\spadesuit^+$ (hence unramified in $F$), a (possibly empty) set $\PP^+$ of $p$-adic places of $F^+$, a power $q$ of $p$ that is a power of the residue cardinality of $F_w$ for every $p$-adic place $w$ of $F$. The ordinary $p$-adic family will be deformed at places in $\PP^+$.

To simply the introduction, we only focus on the \emph{definite} case. Take a positive integer $n\geq 1$. We start from a pair $\xi=(\xi^n,\xi^{n+1})$ of dominant hermitian weights of rank $n$ and $n+1$ (see Definition \ref{de:weight0}) that are trivial at places in $\PP^+$, and a collection $\bV=(\rV_n,\rV_{n+1};\rK_n,\rK_{n+1};\sfB)$ in which
\begin{itemize}
  \item $\rV_n$ is a totally definite hermitian space over $F$ of rank $n$ that is split at every place in $\PP^+$, and $\rV_{n+1}=\rV_n\oplus F\cdot 1$, where $1$ has norm $1$ (put $\sfG_N\coloneqq\rU(\rV_N)$ for $N\in\{n,n+1\}$ and $\sfG\coloneqq\sfG_n\times\sfG_{n+1}$);

  \item $\rK$ is a decomposable neat open compact subgroup of $\sfG(\dA_{F^+}^{\infty,\PP^+})$ that is relatively hyperspecial at places away from $\spadesuit^+$;

  \item $\sfB$ is a Borel subgroup of $\sfG\otimes_{F^+}F^+_{\PP^+}$ that has \emph{trivial} intersection with $\sfH\otimes_{F^+}F^+_{\PP^+}$, where $\sfH$ is the graph of the natural homomorphism $\sfG_n\to\sfG_{n+1}$.
\end{itemize}
See \S\ref{ss:product} for the accurate description of the data. Let $\sfT$ be the Levi quotient of $\sfB$ and $\rI^0$ the maximal open compact subgroup of $\sfT(F^+_{\PP^+})$. We also denote by $\dT$ the spherical Hecke algebra of $\sfG$ away from $p$ and rational primes underlying $\spadesuit^+$.

In \cite{LS}*{\S2.5}, the authors introduced, for every subgroup $\rI<\rI^0$ of finite index (always below), a partially ordered set $\fk_\rI$ of open compact subgroups $\rk$ of $\sfG(F^+_{\PP^+})$ satisfying that the image of $\rk_\sfB\coloneqq\rk\cap\sfB(F^+_{\PP^+})$ in $\sfT(F^+_{\PP^+})$ coincides with $\rI$, with the partial order $\rk\to\rk'$ if
\[
\rk_\sfB\subseteq\rk'_\sfB,\qquad
\rk'\subseteq\rk\cdot\rk'_\sfB
\]
hold simultaneously. There is a natural functor from $\fk_\rI^{\r{op}}$ to the category of $\dT\otimes\dZ_q$-modules sending $\rk$ to $\Gamma(\Sh(\sfG,\rK\rk),\dZ_\xi)$ -- the set of $\dZ_q$-valued ``automorphic forms'' on the definite Shimura set
\[
\Sh(\sfG,\rK\rk)=\sfG(F^+)\backslash\sfG(\dA_{F^+}^\infty)/\rK\rk
\]
of weight $\xi$.

Take a subgroup $\rJ<\rI^0$, not necessarily of finite index. We define the space of \emph{$\rJ$-invariant ordinary distributions} to be
\[
\cD_\rJ(\bV,\xi)\coloneqq\varprojlim_{\rJ<\rI<\rI^0}
\(\Hom_{\dZ_q}\(\varprojlim_{\fk_\rI^{\r{op}}}\Gamma(\Sh(\sfG,\rK\rk),\dZ_\xi),\dZ_q\)\),
\]
which is naturally a topological module over $\dT\otimes\Lambda_{\rI^0/\rJ,\dZ_q}$, where
\[
\Lambda_{\rI^0/\rJ,\dZ_q}\coloneqq\varprojlim_{\rJ<\rI<\rI^0}\dZ_q[\rI^0/\rI]
\]
is the Iwasawa algebra of $\rI^0/\rJ$ with coefficients in $\dZ_q$. Though not so obvious, $\cD_\rJ(\bV,\xi)$ is indeed the continuous Pontryagin dual of the $\dZ_q$-module of all $\rJ$-invariant (nearly) $\sfB$-ordinary $\dZ_q$-valued ``automorphic forms'' of $\sfG$ of away-from-$\PP^+$ level $\rK$ and weight $\xi$ \cite{Hid98}; see Lemma \ref{le:ordinary}. One upshot of this formulation is that one only needs to choose a Borel subgroup of $\sfG$ at places in $\PP^+$, not even any open compact subgroup of $\sfG(F^+_{\PP^+})$.

Now we introduce the residual Galois representation from which we form the ordinary family, which is a pair of hermitian homomorphisms
\begin{align*}
\gamma=(\gamma_n,\gamma_{n+1})\colon\Gamma_{F^+}\to\sG_n(\dF_q)\times\sG_{n+1}(\dF_q)
\end{align*}
(see Definition \ref{de:hermitian}) that is unramified away from $\spadesuit^+$ and $p$-adic places. Let $\cD_\rJ(\bV,\xi)_\gamma$ be the localization of $\cD_\rJ(\bV,\xi)$ at the maximal ideal of $\dT$ that is the kernel of the Hecke character associated with $\gamma$. Let $\cE_\rJ(\bV,\xi)_\gamma$ be the $\Lambda_{\rI^0/\rJ,\dZ_q}$-subring of $\End_{\Lambda_{\rI^0/\rJ,\dZ_q}}\(\cD_\rJ(\bV,\xi)_\gamma\)$ generated by the image of $\dT$, which is reduced, local, and of (pure) dimension $1+\rank_{\dZ_p}\rI^0/\rJ$ (Proposition \ref{pr:eigen_2}). We regard $\Spec\cE_\rJ(\bV,\xi)_\gamma$ as the \emph{integral $\rJ$-invariant ordinary eigenvariety} (for the residual Galois representation $\gamma$).

From now on, we assume that the $\Gamma_F$-representation $\gamma_n^\natural\otimes\gamma_{n+1}^\natural$ (Definition \ref{de:hermitian}) is absolutely irreducible. Then one can associate, via the theory of pseudo-characters, a continuous Galois representation of $\Gamma_F$ on a free $\cE_\rJ(\bV,\xi)_\gamma$-module $\cR_\rJ(\bV,\xi)_\gamma$ of rank $n(n+1)$ which interpolates Rankin--Selberg products of standard automorphic Galois representations (see Remark \ref{re:galois}). For both $\cR_\rJ(\bV,\xi)_\gamma$ and its Cartier dual $\cR_\rJ(\bV,\xi)_\gamma^*(1)$, we may define their Selmer groups $\rH^1_f(F,\cR_\rJ(\bV,\xi)_\gamma)$ and $\rH^1_f(F,\cR_\rJ(\bV,\xi)_\gamma^*(1))$ (see \S\ref{ss:selmer}), which are naturally finitely generated and cofinitely generated modules over $\cE_\rJ(\bV,\xi)_\gamma$, respectively. Then the continuous Pontryagin dual of $\rH^1_f(F,\cR_\rJ(\bV,\xi)_\gamma^*(1))$, denoted by $\cX_\rJ(\bV,\xi)_\gamma$, is a finitely generated $\cE_\rJ(\bV,\xi)_\gamma$-module. Finally, put
\[
\sE_\rJ(\bV,\xi)_\gamma\coloneqq\cE_\rJ(\bV,\xi)_\gamma\otimes_{\dZ_q}\dQ_q,\qquad
\sR_\rJ(\bV,\xi)_\gamma\coloneqq\cR_\rJ(\bV,\xi)_\gamma\otimes_{\dZ_q}\dQ_q,\qquad
\sX_\rJ(\bV,\xi)_\gamma\coloneqq\cX_\rJ(\bV,\xi)_\gamma\otimes_{\dZ_q}\dQ_q,
\]
and then $\rH^1_f(F,\sR_\rJ(\bV,\xi)_\gamma)\coloneqq\rH^1_f(F,\cR_\rJ(\bV,\xi)_\gamma)\otimes_{\dZ_q}\dQ_q$.

To characterize a coherent sheaf $\sT$ over (the reduced scheme) $\Spec\sE_\rJ(\bV,\xi)_\gamma$, a coarse invariant would be its generic rank (as a function on the set of irreducible components of $\Spec\sE_\rJ(\bV,\xi)_\gamma$), which can be regarded as an invariant of ``codimension zero''. We now introduce a finer invariant of ``codimension one'', which we call the \emph{characteristic divisor}, as follows: for an irreducible component $\sE'$ of the normal locus of $\Spec\sE_\rJ(\bV,\xi)_\gamma$ on which $\sT$ is torsion, we define the \emph{characteristic divisor} of $\sT$ along $\sE'$ to be
\[
\Char_{\sE'}(\sT)\coloneqq\sum_{z}\length_{\cO_{\sE',z}}(\sT_z)\cdot z\in\Div(\sE'),
\]
where $z$ runs over all closed points of $\sE'$ of codimension one. This notion generalizes the well-known notion of the characteristic ideal of a torsion module over (the generic fiber of) a regular local $\dZ_q$-ring, like the Iwasawa algebra.

We now propose a generalization of Iwasawa's main conjecture to the realm of eigenvarieties in this particular setting, using the so-called \emph{Bessel period} as an ordinary distribution when the weight $\xi$ is \emph{interlacing} (Definition \ref{de:weight}(1)). Specifically, in \S\ref{ss:bessel_definite}, we define an element
\[
\blambda_\rJ(\bV)\in\cD_\rJ(\bV,\xi)_\gamma^\sfH,
\]
where $\cD_\rJ(\bV,\xi)_\gamma^\sfH$ denotes the subspace of ``$\sfH$-invariant'' distributions (see Definition \ref{de:distinction} for the precise meaning), which is essentially just a compatible family of period sums over the diagonal subgroup $\sfH$. It is not hard to show that the generic rank of $\cD_\rJ(\bV,\xi)_\gamma^\sfH$ over $\cE_\rJ(\bV,\xi)_\gamma$ is at most one.

\begin{conjecture}[Iwasawa's main conjecture, special case of Conjecture \ref{co:iwasawa}]
Let $\xi$, $\bV$, and $\gamma$ be as above, and consider a subgroup $\rJ<\rI^0$. Suppose that
\begin{itemize}
  \item the $\Gamma_F$-representation $\gamma_n^\natural\otimes\gamma_{n+1}^\natural$ is absolutely irreducible;

  \item $\xi$ is interlacing.
\end{itemize}
Let $\sE'$ be an irreducible component of the normal locus of $\Spec\sE_\rJ(\xi,\bV)_\gamma$. If $\blambda_\rJ(\bV)$ is nonzero on $\sE'$ (so that $\sD_\rJ(\xi,\bV)_\gamma^\sfH/\blambda_\rJ(\bV)$ is torsion over $\sE'$), then
\begin{enumerate}[label=(\alph*)]
  \item $\rH^1_f(F,\sR_\rJ(\xi,\bV)_\gamma)$ vanishes over $\sE'$;

  \item $\sX_\rJ(\xi,\bV)_\gamma$ is torsion over $\sE'$;

  \item we have
      \[
      2\Char_{\sE'}\(\sD_\rJ(\xi,\bV)_\gamma^\sfH/\blambda_\rJ(\bV)\)=\Char_{\sE'}\(\sX_\rJ(\xi,\bV)_\gamma\).
      \]
\end{enumerate}
\end{conjecture}

What we can prove in this article is essentially half of the above conjecture, under further assumptions.

\begin{theorem}[special case of Theorem \ref{th:iwasawa}]\label{th:iwasawa0}
Let $\xi$, $\bV$, and $\gamma$ be as above, and consider a subgroup $\rJ<\rI^0$. Suppose that
\begin{enumerate}[label=(\roman*)]
  \item $\gamma$ satisfies all of Assumption \ref{as:galois};

  \item $p>2(n_0+1)$, where $n_0$ is the even member in $\{n,n+1\}$;

  \item $\xi$ is interlacing and Fontaine--Laffaille regular (Definition \ref{de:weight});

  \item $\xi^{n_0}$ is Fontaine--Laffaille regular (Definition \ref{de:weight0}).
\end{enumerate}
Let $\sE'$ be an irreducible component of the tempered locus of $\Spec\sE_\rJ(\xi,\bV)_\gamma$. If $\blambda_\rJ(\bV)$ is nonzero on $\sE'$, then
\begin{enumerate}[label=(\alph*)]
  \item $\rH^1_f(F,\sR_\rJ(\xi,\bV)_\gamma)$ vanishes over $\sE'$;

  \item $\sX_\rJ(\xi,\bV)_\gamma$ is torsion over $\sE'$;

  \item the divisor
      \[
      2\Char_{\sE'}\(\sD_\rJ(\xi,\bV)_\gamma^\sfH/\blambda_\rJ(\bV)\)-\Char_{\sE'}\(\sX_\rJ(\xi,\bV)_\gamma\)
      \]
      of $\sE'$ is effective.
\end{enumerate}
\end{theorem}

Here, the tempered locus (Definition \ref{de:tempered}) is an open dense subset of the normal locus, possibly strict; but their difference does not contain any classical point and in many cases is just empty (Remark \ref{re:tempered}). Note that we do not require places in $\PP^+$ to be split in $F$ in the above theorem.

\begin{remark}
Theorem \ref{th:iwasawa0} (and more generally, Theorem \ref{th:iwasawa}) recover previous results in \cite{LTXZZ} and \cite{LTX}.
\begin{enumerate}
  \item If one takes $\PP^+$ to be the empty set and take $\xi$ to be the trivial weight, then one recovers \cite{LTXZZ}*{Theorem~1.1.1~\&~Theorem~1.1.7} (and their rank one analogues). In other words, if one only takes $\PP^+$ to be the empty set, then Theorem \ref{th:iwasawa0} (and more generally, Theorem \ref{th:iwasawa}) extends \cite{LTXZZ}*{Theorem~1.1.1~\&~Theorem~1.1.7} (and their rank one analogues) to more general interlacing weights $\xi$ rather than the trivial one.

  \item A variant of a special case of Theorem \ref{th:iwasawa0} (and more generally, Theorem \ref{th:iwasawa}), namely, Theorem \ref{th:iwasawa_bis}, recovers \cite{LTX}*{Theorem~1.2.3}. See Remark \ref{re:recover} for more explanation.
\end{enumerate}

\end{remark}

\begin{remark}\label{re:conditional}
In order to simplify the exposition, we will assume \cite{LTXZZ}*{Hypothesis~2.2.5} (for all dominant weights, not necessarily the minimal one) for $N=n,n+1$ throughout this article. In particular, the results (Theorem \ref{th:iwasawa0}, Theorem \ref{th:iwasawa}, Theorem \ref{th:iwasawa_bis}) are still conditional when $F^+=\dQ$ and $n>2$.
\end{remark}

In this article, all rings are commutative and unital; and ring homomorphisms preserve units. For a (topological) ring $L$, a (topological) $L$-ring is a (topological) ring $R$ together with a (continuous) ring homomorphism from $L$ to $R$. However, we use the word \emph{algebra} in the general sense, which is not necessarily commutative or unital. If a base ring is not specified in the tensor operation $\otimes$, then it is $\dZ$.

\subsubsection*{Acknowledgements}

The author would like to thank Dongwen~Liu and Yichao~Tian for useful comments and discussion, and Dongwen~Liu and Binyong~Sun for sharing their early manuscript \cite{LS}. The author is supported by the National Key R\&D Program of China No.~2022YFA1005300.

\section{Some preliminaries}

In this section, we fix some general notation that will be frequently used in the entire article. We also recall the construction of algebraic representations of unitary groups in Construction \ref{co:weight}.

\begin{notation}
Throughout the article, we use the following notation.
\begin{itemize}
  \item Let $\tc\in\Aut(\dC/\dQ)$ be the complex conjugation.

  \item We fix a \emph{subfield} $F\subseteq\dC$ that is a CM number field.

  \item Let $F^+\subseteq F$ be the maximal subfield on which $\tc$ acts by the identity.

  \item Denote by $\Sigma$ and $\Sigma^+$ the sets of places of $F$ and $F^+$, respectively. For every $w\in\Sigma$, denote by $w^\tc\in\Sigma$ its complex conjugation.

  \item Let $\ol{F}$ be \emph{the} Galois closure of $F$ in $\dC$ and denote by $F^\ab$ the maximal abelian extension of $F$ in $\ol{F}$. Put $\Gamma_\cF\coloneqq\Gal(\ol{F}/\cF)$ for every subfield $\cF\subseteq\ol{F}$.

  \item For finitely many places $w_1,w_2,\dots$ of $\dQ$, denote by $\Sigma_{w_1,w_2,\dots}$ and $\Sigma^+_{w_1,w_2,\dots}$ the sets of places of $F$ and $F^+$ above the set $\{w_1,w_2,\dots\}$, respectively.

  \item Denote by $\spadesuit$ and $\PP$ the sets of places of $F$ above $\spadesuit^+$ and $\PP^+$, respectively.

  \item For every set $\fm$ of places of $F^+$, denote by $\ang{\fm}$ the set of places of $F^+$ whose underlying rational places underly $\spadesuit^+\cup\Sigma^+_{\infty,p}\cup\fm$.

  \item For every place $v$ of $F^+$, we
      \begin{itemize}
        \item denote by $\|v\|$ the residual cardinality of $F^+_v$ when it is nonarchimedean;

        \item put $F_v\coloneqq F\otimes_{F^+}F^+_v$;

        \item fix an algebraic closure $\ol{F}^+_v$ of $F^+_v$ containing $\ol{F}$;

        \item put $\Gamma_{F^+_v}\coloneqq\Gal(\ol{F}^+_v/F^+_v)$ as a subgroup of $\Gamma_{F^+}$;

        \item denote by $w(v)$ the unique place of $F$ above $v$ induced by the fixed embedding $F\hookrightarrow\ol{F}^+_v$.
      \end{itemize}

  \item For every nonarchimedean place $w$ of $F$, we
      \begin{itemize}
        \item identify the Galois group $\Gamma_{F_w}$ with $\Gamma_{F^+_v}\cap\Gamma_F$ (resp.\ $\tc(\Gamma_{F^+_v}\cap\Gamma_F)\tc$), where $v$ is the underlying place of $F^+$, if $w=w(v)$ (resp.\ $w\neq w(v)$);

        \item let $\rI_{F_w}\subseteq\Gamma_{F_w}$ be the inertia subgroup;

        \item let $\kappa_w$ be the residue field of $F_w$, and identify its Galois group $\Gamma_{\kappa_w}$ with $\Gamma_{F_w}/\rI_{F_w}$;

        \item denote by $\phi_w\in\Gamma_{F_w}$ a lifting of the \emph{arithmetic} Frobenius element in $\Gamma_{\kappa_w}$.
      \end{itemize}

  \item Denote by $\Upsilon^+$ and $\Upsilon$ the set of all embeddings from $F^+$ and $F$ to $\dQ_q$, respectively. For $\tau\in\Upsilon$ and $w\in\Sigma_p$, we write $\tau\mid w$ if $\tau$ induces $w$. Similarly, for $\sigma\in\Upsilon^+$ and $v\in\Sigma^+_p$, we write $\sigma\mid v$ if $\sigma$ induces $v$.

  \item Denote by $\epsilon_p\colon\Gamma_\dQ\to\dZ_p^\times$ the $p$-adic cyclotomic character, and $\eta_{F/F^+}\colon\Gamma_{F^+}\to\{\pm 1\}$ the quadratic character associated with $F/F^+$.
\end{itemize}
\end{notation}

\begin{notation}\label{no:complete}
Let $G$ be an abelian group.
\begin{itemize}
  \item Put $G[p^\infty]\coloneqq\bigcup_m G[p^m]\subseteq G$.

  \item We write $H\finite G$ for a subgroup $H$ of $G$ such that $G/H$ is finite.

  \item Put $\widehat{G}\coloneqq\varprojlim_{H\finite G}G/H$ as the its profinite completion, regarded as a topological group.

  \item Put
      \[
      \Lambda_{G,\dZ_q}\coloneqq\dZ_q[[G]]=\varprojlim_{H\finite G}\dZ_q[G/H]
      \]
      and $\Lambda_{G,\dQ_q}\coloneqq\Lambda_{G,\dZ_q}\otimes_{\dZ_q}\dQ_q$, as topological rings.

  \item If $G$ is a topological $\dZ_q$-module, we denote by
     \[
     G^*\coloneqq\Hom_{\dZ_q}^{\r{ct}}(G,\dQ_q/\dZ_q)
     \]
     its continuous ($\dZ_q$-linear) Pontryagin dual.
\end{itemize}
\end{notation}

\begin{definition}[Hermitian space]
Let $R$ be an $O_{F^+}[(\spadesuit^+)^{-1}]$-ring. A \emph{hermitian space} over $O_F\otimes_{O_{F^+}}R$ of rank $N$ is a projective $O_F\otimes_{O_{F^+}}R$-module $\rV$ of rank $N$ together with a perfect pairing
\[
(\;,\;)_\rV\colon\rV\times\rV\to O_F\otimes_{O_{F^+}}R
\]
that is $O_F\otimes_{O_{F^+}}R$-linear in the first variable and $(O_F\otimes_{O_{F^+}}R,\tc\otimes 1_R)$-linear in the second variable, and satisfies $(x,y)_\rV=(y,x)_\rV^\tc$ for $x,y\in\rV$. We denote by $\rU(\rV)$ the group of $O_F\otimes_{O_{F^+}}R$-linear isometries of $\rV$, which is a reductive group over $R$. Moreover, we denote by $\rV_\sharp$ the hermitian space $\rV\oplus O_F\otimes_{O_{F^+}}R\cdot 1$, where $1$ has norm $1$.
\end{definition}

\begin{definition}[\cite{LTXZZ}*{Definition~3.2.1}]\label{de:standard_hermitian_space}
Let $\rV$ be a hermitian space over $F$ of rank $N$.
\begin{enumerate}
  \item We say that $\rV$ is \emph{standard definite} if it has signature $(N,0)$ at every place in $\Sigma^+_\infty$.

  \item We say that $\rV$ is \emph{standard indefinite} if it has signature $(N-1,1)$ at the default real place of $F^+$ and $(N,0)$ at other places in $\Sigma^+_\infty$.

  \item We say that $\rV$ is \emph{standard} if it is either standard definite or indefinite.
\end{enumerate}
For a standard hermitian space $\rV$, we have the system of Shimura sets or varieties $\Sh(\sfG,\rK)$, according to whether $\rV$ is definite or indefinite, indexed by neat open compact subgroups $\rK$ of $\sfG(\dA_{F^+})$, where $\sfG\coloneqq\rU(\rV)$ (see \cite{LTXZZ}*{\S2}).\footnote{In \cite{LTXZZ}*{\S2}, it was denoted by $\Sh(\rV,\rK)$; but we slightly change the notation here.} In what follows, we will uniformly call the system $\Sh(\sfG,\rK)$ as \emph{Shimura spaces}.
\end{definition}

\begin{construction}\label{co:weight}
Consider an $O_{F^+}[(\spadesuit^+)^{-1}]$-ring $R$, a homomorphism $\sigma\colon R\to\dL$ to a ring $\dL$, such that $\Sigma\coloneqq\{\tau\colon O_F\otimes_{O_{F^+}}R\to\dL\res \tau\res_R=\sigma\}$ is nonempty (hence a set of two elements). Let $\rV$ be a hermitian space over $O_F\otimes_{O_{F^+}}R$ of rank $N\geq1$, with $G\coloneqq\rU(\rV)$ a reductive group over $R$. For every tuple
\[
\xi=\((\xi_{\tau,0},\dots,\xi_{\tau,N-1})\)_{\tau\in\Sigma}\in(\dZ^N)^\Sigma
\]
satisfying $\xi_{\tau,i}+\xi_{\tau',N-1-i}=0$ for $\tau'\neq\tau$ and $0\leq i\leq N-1$, we will construct a representation $\dL_\xi$ of $G(R)$ with coefficients in $\dL$.

Choose a Borel subgroup $B$ of $G$ over $R$, with the Levi quotient $T$. Choose an element $\tau\in\Sigma$. Then the natural map $G\otimes_{R,\sigma}\dL\to\GL(\rV_\tau)$ is an isomorphism, where $\rV_\tau\coloneqq\rV\otimes_{O_F\otimes_{O_{F^+}}R,\tau}\dL$. Let
\[
0=\rV_\tau^0\subseteq \rV_\tau^1\subseteq\cdots\subseteq \rV_\tau^N=\rV_\tau
\]
be the unique complete filtration of free $\dL$-submodules which $B\otimes_{R,\sigma}\dL$ stabilizes. We obtain an isomorphism
\begin{align}\label{eq:weight}
T\otimes_{R,\sigma}\dL\simeq\prod_{i=0}^{N-1}\Aut_\dL(\rV_\tau^{i+1}/\rV_\tau^i)=\prod_{i=0}^{N-1}\GL_{1,\dL}
\end{align}
of tori over $\dL$. The assignment
\[
t=(t_0,\dots,t_{N-1})\mapsto \prod_{i=0}^{N-1}t_i^{\xi_{\tau,i}}
\]
induces a character
\[
\chi_\xi\colon T(R)\to\dL^\times
\]
that is independent of the choice of $\tau$. We also regard $\chi_\xi$ as a character of $B(R)$ via inflation. Finally, define
\[
\dL_\xi\coloneqq
\left\{f\in R[G]\otimes_{R,\sigma}\dL\left| f(bg)=\chi_\xi(b)\cdot f(g)\text{ for all }b\in B(R),g\in G(R)\right.\right\}
\]
as a representation of $G(R)$ via right translation; its isomorphism class does not depend on the choice of $B$.

When $\dL=\dZ_q,\dQ_q$, we simply write $\dZ_\xi,\dQ_\xi$ for $(\dZ_q)_\xi,(\dQ_q)_\xi$, respectively.
\end{construction}

\begin{definition}[Abstract unitary Hecke algebra]\label{de:abstract_hecke}
For every place $v$ of $F^+$ not in $\spadesuit^+$, we have the \emph{local spherical Hecke algebra}
\[
\dT_{N,v}\coloneqq\dZ[\rU_{N,v}(O_{F^+_v})\backslash\rU_{N,v}(F^+_v)/\rU_{N,v}(O_{F^+_v})],
\]
where $\rU_{N,v}$ denotes the unique up to isomorphism unitary group over $O_{F^+_v}$ (with respect to the \'{e}tale extension $(O_F\otimes_{O_{F^+}}O_{F^+_v})/O_{F^+_v}$) in $N$ variables, which is a ring with the unit element $1_{\rU_{N,v}(O_{F^+_v})}$.

For a finite set $\Box$ of places of $F^+$ containing $\spadesuit^+$, we define the \emph{abstract unitary Hecke algebra away from $\Box$} to be the restricted tensor product
\[
\dT_N^{\Box}\coloneqq{\bigotimes_v}'\dT_{N,v}
\]
over all places $v\in\Sigma^+\setminus\Box$ with respect to unit elements.
\end{definition}

\begin{definition}\label{de:hermitian}
We recall the group scheme $\sG_N$ from \cite{CHT08}*{\S1}. Put
\[
\sG_N\coloneqq(\GL_N\times\GL_1)\rtimes\{1,\fc\}
\]
with $\fc^2=1$ and
\[
\fc(g,\mu)\fc=(\mu\pres{\rt}{g}^{-1},\mu)
\]
for $(g,\mu)\in\GL_N\times\GL_1$.

We say that a homomorphism $\gamma\colon\Gamma_{F^+}\to\sG_N(R)$, in which $R$ is a topological $\dZ_p$-ring, is \emph{hermitian} if it is continuous, satisfies $\gamma^{-1}(\GL_N(R)\times R^\times)=\Gamma_F$ and $\nu\circ\gamma=\eta_{F/F^+}^N\epsilon_p^{1-N}$. For a hermitian homomorphism $\gamma\colon\Gamma_{F^+}\to\sG_N(R)$, we denote by $\gamma^\natural\colon\Gamma_F\to\GL_N(R)$ the composition of $\gamma\res_{\Gamma_F}$ with the projection $\GL_N(R)\times R^\times\to\GL_N(R)$.
\end{definition}

\section{Ordinary distributions and ordinary eigenvariety}
\label{ss:eigenvariety}

In this section, we introduce the notion of ordinary distribution and ordinary eigenvariety for the unitary group of a single hermitian space. The main properties of these objects are summarized in Proposition \ref{pr:eigen_1}. However, in order to reduce the complexity of the discussion in this section where several new constructions appear for the first time, we do not introduce the localized version. As a price to pay, Proposition \ref{pr:eigen_1} is only stated in the definite case (but whose proof works in more general situation).

Let $N\geq 1$ be an integer with $r\coloneqq\lfloor\tfrac{N}{2}\rfloor$.

\begin{definition}\label{de:weight0}
We denote by $\Xi^N$ the set of \emph{hermitian weights} of rank $N$, that is, tuples
\[
\xi^N=\(\xi^N_\tau=(\xi^N_{\tau,0},\dots,\xi^N_{\tau,N-1})\in\dZ^N\)_{\tau\in\Upsilon}
\]
satisfying $\xi^N_{\tau,i}+\xi^N_{\tau\tc,N-1-i}=0$ for every $\tau\in\Upsilon$ and every $0\leq i\leq N-1$. It is naturally an abelian group under the componentwise addition.
\begin{enumerate}
  \item We say that a hermitian weight $\xi^N$ of rank $N$ is
      \begin{itemize}
        \item \emph{neutral}, if $\sum_{i=0}^{N-1}\xi^N_{\tau,i}=0$ for every $\tau\in\Upsilon$;

        \item \emph{dominant}, if $\xi^N_{\tau,0}\leq\cdots\leq\xi^N_{\tau,N-1}$ holds for every $\tau\in\Upsilon$;

        \item \emph{Fontaine--Laffaille regular}, if it is dominant and satisfies
            \[
            \max_{\tau|w}\left\{\xi^N_{\tau,N-1}\right\}-\min_{\tau|w}\left\{\xi^N_{\tau,0}\right\}+N-1\leq p-2
            \]
            for every $w\in\Sigma_p$;

        \item \emph{$\blacklozenge^+$-trivial} for a subset $\blacklozenge^+\subseteq\Sigma^+_p$, if $\xi^N_\tau=0^N$ for all $\tau$ that induces places in $\blacklozenge^+$.
      \end{itemize}

  \item There is an involution on $\Xi^N$ sending $\xi^N$ to $(\xi^N)^\vee$ defined by the formula
      \[
      (\xi^N)^\vee_{\tau,i}=-\xi^N_{\tau,N-1-i},
      \]
      which preserves the subgroups of neutral, dominant, Fontaine--Laffaille regular, $\blacklozenge^+$-trivial elements, respectively.

  \item Finally, for every element $\sigma\in\Upsilon^+$, write
      \[
      \xi^N_\sigma\coloneqq\(\xi^N_\tau\)_{\tau\res_{F^+}=\sigma}.
      \]
\end{enumerate}
\end{definition}

We consider a $\PP^+$-trivial dominant hermitian weight $\xi^N\in\Xi^N$ and a standard hermitian space $\rV_N$ of rank $N$ over $F$ that is split at every place in $\Sigma_p^+$ (and put $\sfG_N\coloneqq\rU(\rV_N)$ as a reductive group over $F^+$), together with
\begin{itemize}
  \item a decomposable neat open compact subgroup $\rK_N$ of $\sfG_N(\dA_{F^+}^{\infty,\PP^+})$ such that $(\rK_N)_v$ is the stabilizer of a self-dual lattice in $\rV_{N,v}$ for $v\in\Sigma^+_p\setminus\PP^+$;

  \item a Borel subgroup $\sfB_N$ of $\sfG_N\otimes_{F^+}F^+_{\PP^+}$, with $\sfT_N$ its Levi quotient;

  \item a finite set $\Box$ of places of $F^+$ containing $\ang{\emptyset}$ such that for every $v\in\Sigma^+\setminus\Box$, $\rK_{N,v}$ is hyperspecial maximal.
\end{itemize}
Denote by $\rI_N^0$ the maximal open compact subgroup of $\sfT_N(F^+_{\PP^+})$.

For $v\in\Sigma^+_p\setminus\PP^+$ and $\sigma\in\Upsilon^+$ inducing $v$, we have a representation $\dZ_{\xi^N_\sigma}$ of $(\rK_N)_v$ hence of $\rK_N\sfG_N(F^+_{\PP^+})$ from Construction \ref{co:weight} (with $R=O_{F^+_v}$ and $\dL=\dZ_q$). For $v\in\PP^+$ and $\sigma\in\Upsilon^+$ inducing $v$, we regard $\dZ_{\xi^N_\sigma}=\dZ_q$ as the trivial representation of $\rK_N\sfG_N(F^+_{\PP^+})$. Put
\[
\dZ_{\xi^N}\coloneqq\bigotimes_{v\in\Sigma^+_p}
\(\bigotimes_{\sigma\in\Upsilon^+,\sigma|v}\dZ_{\xi^N_\sigma}\)
\]
as a representation of $\rK_N\sfG_N(F^+_{\PP^+})$ with coefficients in $\dZ_q$ (here the tensor product is taken over $\dZ_q$). Since $\xi^N$ is dominant, $\dZ_{\xi^N}$ is a nontrivial finite free $\dZ_q$-module.

Let $\rk$ be an open compact subgroup of $\sfG_N(F^+_{\PP^+})$. The representation $\dZ_{\xi^N}$ induces a local system by the same notation on the Shimura space $\Sh(\sfG_N,\rK_N\rk)$. Put
\begin{align*}
\rH(\Sh(\sfG_N,\rK_N\rk),\dZ_{\xi^N}/p^m)\coloneqq
\begin{dcases}
\Gamma(\Sh(\sfG_N,\rK_N\rk),\dZ_{\xi^N}/p^m), &\text{when $\rV_N$ is standard definite,}\\
\rH^{N-1}_{\et}(\Sh(\sfG_N,\rK_N\rk)_{\ol{F}},\dZ_{\xi^N}/p^m(N-1-r)), &\text{when $\rV_N$ is standard indefinite,}
\end{dcases}
\end{align*}
for $1\leq m\leq \infty$ ($\dZ_{\xi^N}/p^\infty$ is understood as $\dZ_{\xi^N}$); it is naturally a $\dZ_q[(\rK_N)_{\spadesuit^+}\backslash\sfG_N(F^+_{\spadesuit^+})/(\rK_N)_{\spadesuit^+}]
\otimes\dT_N^\Box$-module, and admits a
$\dZ_q[(\rK_N)_{\spadesuit^+}\backslash\sfG_N(F^+_{\spadesuit^+})/(\rK_N)_{\spadesuit^+}]
\otimes\dT_N^\Box$-linear continuous action of $\Gamma_F$ (which is trivial if $\rV_N$ is definite).

\begin{construction}\label{co:liusun}
We recall a key construction from \cite{LS}*{\S2.5}. For every subgroup $\rI\finite\rI_N^0$, let $\fk_\rI$ be the set of open compact subgroups $\rk$ of $\sfG_N(F^+_{\PP^+})$ satisfying that the image of $\rk\cap\sfB_N(F^+_{\PP^+})$ in $\sfT_N(F^+_{\PP^+})$ coincides with $\rI$. There is a partial order on $\bigcup_\rI\fk_\rI$: For two elements $\rk,\rk'\in\bigcup_\rI\fk_\rI$, we impose $\rk\prec\rk'$ if
\[
\rk\cap\sfB_N(F^+_{\PP^+})\subseteq\rk'\cap\sfB_N(F^+_{\PP^+}),\qquad
\rk'\subseteq\rk\cdot\(\rk'\cap\sfB_N(F^+_{\PP^+})\)
\]
hold simultaneously.

For $\rk\prec\rk'$ in $\bigcup_\rI\fk_\rI$, we have the transfer map
\[
\rho_{\rk',\rk}\colon\rH(\Sh(\sfG_N,\rK_N\rk'),\dZ_{\xi^N}/p^m)\to\rH(\Sh(\sfG_N,\rK_N\rk),\dZ_{\xi^N}/p^m)
\]
as the composition of the pullback map $\rho_{\rk',\rk'\cap\rk}$ and the pushforward map $\rho_{\rk'\cap\rk,\rk}$. By \cite{LS}*{Lemma~2.14~\&~Lemma~3.3}, for $\rk_1\prec\rk_2\prec\rk_3$, we have $\rho_{\rk_3,\rk_1}=\rho_{\rk_2,\rk_1}\circ\rho_{\rk_3,\rk_2}$. In other words, we have a functor from $\(\bigcup_\rI\fk_\rI\)^{\r{op}}$ to the category of modules over $\dZ_q[(\rK_N)_{\spadesuit^+}\backslash\sfG_N(F^+_{\spadesuit^+})/(\rK_N)_{\spadesuit^+}]
\otimes\dT_N^\Box$.
\end{construction}

Take a subgroup $\rI\finite\rI_N^0$. Construction \ref{co:liusun} allows us to put
\[
\rH(\Sh(\sfG_N,\rK_N\rI),\dZ_{\xi^N}/p^m)\coloneqq
\varprojlim_{\fk_\rI^{\r{op}}}\rH(\Sh(\sfG_N,\rK_N\rk),\dZ_{\xi^N}/p^m).
\]
For every element $t\in\sfT_N(F^+_{\PP^+})$, there is a unique endomorphism $[t]$ of $\rH(\Sh(\sfG_N,\rK_N\rI),\dZ_{\xi^N}/p^m)$ such that the diagram
\[
\xymatrix{
\rH(\Sh(\sfG_N,\rK_N\rI),\dZ_{\xi^N}/p^m) \ar[r]^-{[t]} \ar[d] &
\rH(\Sh(\sfG_N,\rK_N\rI),\dZ_{\xi^N}/p^m) \ar[d] \\
\rH(\Sh(\sfG_N,\rK_N\rk),\dZ_{\xi^N}/p^m) \ar[r] &
\rH(\Sh(\sfG_N,\rK_N(t\rk t^{-1})),\dZ_{\xi^N}/p^m)
}
\]
commutes for every $\rk\in\fk_\rI$, where the bottom map is the pullback map along the right translation morphism $\Sh(\sfG_N,\rK_N(t\rk t^{-1}))\to\Sh(\sfG_N,\rK_N\rk)$ by $t$ between the Shimura spaces. It follows that $\rH(\Sh(\sfG_N,\rK_N\rI),\dZ_{\xi^N}/p^m)$ is a module over $\dZ[(\rK_N)_{\spadesuit^+}\backslash\sfG_N(F^+_{\spadesuit^+})/(\rK_N)_{\spadesuit^+}]
\otimes\dT_N^\Box\otimes\Lambda_{\sfT_N(F^+_{\PP^+})/\rI,\dZ_q}$. We endow
\[
\Hom_{\dZ_q}\(\rH(\Sh(\sfG_N,\rK_N\rI),\dZ_{\xi^N}/p^m),\dZ_q/p^m\)
\]
the $\dZ[(\rK_N)_{\spadesuit^+}\backslash\sfG_N(F^+_{\spadesuit^+})/(\rK_N)_{\spadesuit^+}]
\otimes\dT_N^\Box\otimes\Lambda_{\sfT_N(F^+_{\PP^+})/\rI,\dZ_q}$-module structure via the \emph{dual} action (in particular, the $\Lambda_{\sfT_N(F^+_{\PP^+})/\rI,\dZ_q}$-module structure is the usual one twisted by the inverse of $\sfT_N(F^+_{\PP^+})/\rI$).

We now relate $\rH(\Sh(\sfG_N,\rK_N\rI),\dZ_{\xi^N}/p^m)$ to the classical notion of (nearly) ordinary cohomology. For an open compact subgroup $\rk$ of $\sfG_N(F^+_{\PP^+})$, put
\[
\fB_\rk\coloneqq\left\{\left.b\in\sfB_N(F^+_{\PP^+})\right| \rk\to b\rk b^{-1}\right\},
\]
which is a submonoid of $\sfB_N(F^+_{\PP^+})$ by Lemma \ref{le:monoid} below. For $b\in\fB_\rk$ and $1\leq m\leq\infty$, we define $\tU_b$ to be the composite map
\[
\rH(\Sh(\sfG_N,\rK_N\rk),\dZ_{\xi^N}/p^m)
\to\rH(\Sh(\sfG_N,\rK_N(b\rk b^{-1})),\dZ_{\xi^N}/p^m)
\xrightarrow{\rho_{b\rk b^{-1},\rk}}\rH(\Sh(\sfG_N,\rK_N\rk),\dZ_{\xi^N}/p^m)
\]
in which the first map is the pullback map along the right translation morphism by $b$ between the Shimura spaces. By \cite{LS}*{Lemma~2.14~\&~Proposition~3.14}, the assignment $b\mapsto\tU_b$ defines a homomorphism
\[
\tU_\bullet\colon\fB_\rk\to\End_{\dZ_q}\(\rH(\Sh(\sfG_N,\rK_N\rk),\dZ_{\xi^N}/p^m)\)
\]
of monoids.

\begin{lem}\label{le:monoid}
For an open compact subgroup $\rk$ of $\sfG_N(F^+_{\PP^+})$, the subset $\fB_\rk\subseteq\sfB_N(F^+_{\PP^+})$ is a submonoid.
\end{lem}

\begin{proof}
For every subgroup $\rk'$ of $\sfG_N(F^+_{\PP^+})$, put $\rk'_\sfB\coloneqq\rk'\cap\sfB_N(F^+_{\PP^+})$. Clearly, the identity belongs to $\fB_\rk$. It remains to show that $\fB_\rk$ is closed under multiplication. Take $b,b'\in\fB_\rk$. As $\rk\to b\rk b^{-1}$, we have
\[
\rk_\sfB\subseteq (b\rk b^{-1})_\sfB,\qquad b\rk b^{-1}\subseteq \rk\cdot (b\rk b^{-1})_\sfB.
\]
On the other hand, as $\rk\to b'\rk b^{\prime-1}$, we have $b\rk b^{-1}\to (bb')\rk (bb')^{-1}$, so that
\[
(b\rk b^{-1})_\sfB\subseteq((bb')\rk (bb')^{-1})_\sfB,\qquad (bb')\rk (bb')^{-1}\subseteq (b\rk b^{-1})\cdot((bb')\rk (bb')^{-1})_\sfB.
\]
Together, we have
\[
\rk_\sfB\subseteq (b\rk b^{-1})_\sfB\subseteq((bb')\rk (bb')^{-1})_\sfB,
\]
and
\[
(bb')\rk (bb')^{-1}\subseteq (b\rk b^{-1})\cdot((bb')\rk (bb')^{-1})_\sfB
\subseteq \rk\cdot (b\rk b^{-1})_\sfB\cdot((bb')\rk (bb')^{-1})_\sfB
=\rk\cdot((bb')\rk (bb')^{-1})_\sfB.
\]
It follows that $\rk\to(bb')\rk (bb')^{-1}$, that is, $bb'\in\fU_\rk$. The lemma is proved.
\end{proof}

Denote by $\fk_\rI^\dag$ the subset of $\fk_\rI$ of elements $\rk$ such that the image of $\fB_\rk$ in $\sfT_N(F^+_{\PP^+})$ generates the entire group. For every $\rk\in\fk_\rI^\dag$ and every $1\leq m\leq\infty$, we define
\[
\rH^\ordi(\Sh(\sfG_N,\rK_N\rk),\dZ_{\xi^N}/p^m)\subseteq\rH(\Sh(\sfG_N,\rK_N\rk),\dZ_{\xi^N}/p^m)
\]
to be the common image of $\tU_b$ for every $b\in\fB_\rk$.

\begin{lem}\label{le:ordinary}
Let $\rk$ be an element of $\fk_\rI^\dag$.
\begin{enumerate}
  \item The subset $\{b\rk b^{-1}\res b\in\fB_\rk\}$ is cofinal in $\fk_\rI$.

  \item For every $1\leq m\leq \infty$, the natural map
     \[
     \rH(\Sh(\sfG_N,\rK_N\rI),\dZ_{\xi^N}/p^m)\to\rH(\Sh(\sfG_N,\rK_N\rk),\dZ_{\xi^N}/p^m)
     \]
     is injective with image $\rH^\ordi(\Sh(\sfG_N,\rK_N\rk),\dZ_{\xi^N}/p^m)$.

  \item For every other element $\rk'\in\fk_\rI^\dag$ satisfying $\rk\to\rk'$ and every $1\leq m\leq \infty$, the map $\rho_{\rk',\rk}$ in Construction \ref{co:liusun} induces an isomorphism
      \[
      \rH^\ordi(\Sh(\sfG_N,\rK_N\rk'),\dZ_{\xi^N}/p^m)\xrightarrow\sim\rH^\ordi(\Sh(\sfG_N,\rK_N\rk),\dZ_{\xi^N}/p^m).
      \]
\end{enumerate}
\end{lem}

\begin{proof}
For (1), choose a maximal unipotent subgroup $\sfU_N^-$ of $\sfG_N\otimes_{F^+}F^+_{\PP^+}$ that is opposite to $\sfB_N$, that is, the natural map $\sfU_N^-\to(\sfG_N\otimes_{F^+}F^+_{\PP^+})/\sfB_N$ is injective with Zariski dense image. Choose an open compact subgroup $\ru^-$ of $\sfU_N^-(F^+_{\PP^+})$ such that $\rk$ is contained in $\ru^-\cdot\rk_\sfB$ (as a subset of $\sfG_N(F^+_{\PP^+})$). Now for another element $\rk'\in\fk_\rI$, choose an open compact subgroup $\ru^{-\prime}$ of $\sfU_N(F^+_{\PP^+})$ that is contained in $\rk'$. Then $\ru^{-\prime}\cdot\rk'_\sfB$ (as a subset of $\sfG_N(F^+_{\PP^+})$) is contained in $\rk'$. Since $\fB_\rk$ generates $\sfT_N(F^+_{\PP^+})$, we may find an element $b\in\fB_\rk$ that normalizes $\sfU_N^-$ satisfying $\rk'_\sfB\subseteq b\rk_\sfB b^{-1}$ and $b\ru^- b^{-1}\subseteq\ru^{-\prime}$. It follows that $\rk'\to b\rk b^{-1}$. Part (1) is proved.

Part (2) immediately follows from part (1); and part (3) immediately follows from part (2).
\end{proof}

\begin{remark}\label{re:supersingular}
Take an element $\rk\in\fk_\rI^\dag$. Put
\begin{align*}
\rH^{\r{ss}}(\Sh(\sfG_N,\rK_N\rk),\dZ_{\xi^N}/p^m)&\coloneqq\sum_{b\in\fB_\rk}
\Ker\(\tU_b\colon\rH(\Sh(\sfG_N,\rK_N\rk),\dZ_{\xi^N}/p^m)\to\rH(\Sh(\sfG_N,\rK_N\rk),\dZ_{\xi^N}/p^m)\) \\
&\subseteq\rH(\Sh(\sfG_N,\rK_N\rk),\dZ_{\xi^N}/p^m)
\end{align*}
for $1\leq m<\infty$ and
\begin{align*}
\rH^{\r{ss}}(\Sh(\sfG_N,\rK_N\rk),\dZ_{\xi^N})&\coloneqq\varprojlim_{m}\rH^{\r{ss}}(\Sh(\sfG_N,\rK_N\rk),\dZ_{\xi^N}/p^m) \\
&\subseteq\varprojlim_{m}\rH(\Sh(\sfG_N,\rK_N\rk),\dZ_{\xi^N}/p^m)=\rH(\Sh(\sfG_N,\rK_N\rk),\dZ_{\xi^N}).
\end{align*}
The the natural map
\[
\rH^\ordi(\Sh(\sfG_N,\rK_N\rk),\dZ_{\xi^N}/p^m)\oplus\rH^{\r{ss}}(\Sh(\sfG_N,\rK_N\rk),\dZ_{\xi^N}/p^m)
\to\rH(\Sh(\sfG_N,\rK_N\rk),\dZ_{\xi^N}/p^m)
\]
is an isomorphism for every $1\leq m\leq\infty$.
\end{remark}

For subgroups $\rI\finite\rI'\finite\rI_N^0$, we have a unique map
\[
\rho_{\rI',\rI}\colon\rH(\Sh(\sfG_N,\rK_N\rI'),\dZ_{\xi^N}/p^m)\to\rH(\Sh(\sfG_N,\rK_N\rI),\dZ_{\xi^N}/p^m)
\]
such that the diagram
\[
\xymatrix{
\rH(\Sh(\sfG_N,\rK_N\rI'),\dZ_{\xi^N}/p^m) \ar[r]^-{\rho_{\rI',\rI}} \ar[d] &
\rH(\Sh(\sfG_N,\rK_N\rI),\dZ_{\xi^N}/p^m) \ar[d] \\
\rH(\Sh(\sfG_N,\rK_N\rk'),\dZ_{\xi^N}/p^m) \ar[r]^-{\rho_{\rk',\rk}} &
\rH(\Sh(\sfG_N,\rK_N\rk),\dZ_{\xi^N}/p^m)
}
\]
commutes for every $\rk\in\fk_\rI$ and $\rk'\in\fk_{\rI'}$ with $\rk\prec\rk'$. For subgroups $\rI_1\finite\rI_2\finite\rI_3\finite\rI_N^0$, we have $\rho_{\rI_3,\rI_1}=\rho_{\rI_2,\rI_1}\circ\rho_{\rI_3,\rI_2}$.

\begin{definition}\label{de:distribution}
Let $\rJ<\rI_N^0$ be a subgroup.
\begin{enumerate}
  \item We define the space of \emph{$\rJ$-invariant ordinary distributions} to be
      \[
      \cD_\rJ(\xi^N,\rV_N,\rK_N)\coloneqq\varprojlim_{\rJ<\rI\finite\rI_N^0}
      \Hom_{\dZ_q}\(\rH(\Sh(\sfG_N,\rK_N\rI),\dZ_{\xi^N}),\dZ_q\),
      \]
      which is a module over $\dZ[(\rK_N)_{\spadesuit^+}\backslash\sfG_N(F^+_{\spadesuit^+})/(\rK_N)_{\spadesuit^+}]
      \otimes\dT_N^\Box\otimes\Lambda_{\sfT_N(F^+_{\PP^+})/\rJ,\dZ_q}$ (and also naturally a topological module over $\Lambda_{\sfT_N(F^+_{\PP^+})/\rJ,\dZ_q}$), admitting a continuous action by $\Gamma_F$ (which is trivial when $\rV_N$ is definite). Here, the inverse limit is taken with respect to the dual of $\rho_{\rI',\rI}$.

  \item We denote by $\cE_\rJ(\xi^N,\rV_N,\rK_N)$ the $\Lambda_{\sfT_N(F^+_{\PP^+})/\rJ,\dZ_q}$-subalgebra of
      \[
      \End_{\Lambda_{\sfT_N(F^+_{\PP^+})/\rJ,\dZ_q}}\(\cD_\rJ(\xi^N,\rV_N,\rK_N)\)
      \]
      generated by the image of $\dT_N^\Box$, which is a ring. Finally, put
      \begin{align*}
      \sD_\rJ(\xi^N,\rV_N,\rK_N)&\coloneqq\cD_\rJ(\xi^N,\rV_N,\rK_N)\otimes_{\dZ_q}\dQ_q,\\
      \sE_\rJ(\xi^N,\rV_N,\rK_N)&\coloneqq\cE_\rJ(\xi^N,\rV_N,\rK_N)\otimes_{\dZ_q}\dQ_q,
      \end{align*}
      as a module and a ring over $\dZ[(\rK_N)_{\spadesuit^+}\backslash\sfG_N(F^+_{\spadesuit^+})/(\rK_N)_{\spadesuit^+}]
      \otimes\dT_N^\Box\otimes\Lambda_{\sfT_N(F^+_{\PP^+})/\rJ,\dQ_q}$, respectively.
\end{enumerate}
We suppress $\rJ$ in all the subscripts when it is the trivial subgroup.
\end{definition}

\begin{definition}\label{de:classical}
We define a \emph{classical point} of $\Spec\Lambda_{\rI_N^0/\rJ,\dQ_q}$ to be a closed point $x$ (with the residue field $\dQ_x$) such that the corresponding character $\chi_x\colon\rI_N^0\to\dZ_x^\times$ (where $\dZ_x$ is the ring of integers of $\dQ_x$) satisfying the following property: There exists a (necessarily unique) $\Sigma_p^+\setminus\PP^+$-trivial dominant hermitian weight $\zeta^N_x\in\Xi^N$ such that
\[
\chi_x^\smooth\coloneqq\chi_x\cdot\chi_{(\zeta^N_x)^\vee}\res_{\rI_N^0}
\]
is a smooth (that is, locally constant) character of $\rI_N^0$, where
\[
\chi_{(\zeta^N_x)^\vee}\coloneqq\bigotimes_{v\in\PP^+}\prod_{\sigma\in\Upsilon^+,\sigma|v}\chi_{(\zeta^N_x)^\vee_v}
\]
in which $\chi_{(\zeta^N_x)^\vee_v}\colon\sfT_N(F^+_v)\to\dQ_q^\times$ is the character introduced in Construction \ref{co:weight} (with $R=F^+_v$ and $\dL=\dQ_q$).

We define a \emph{classical point} of $\Spec\sE_\rJ(\xi^N,\rV_N,\rK_N)$ to be a closed point whose image in $\Spec\Lambda_{\rI_N^0/\rJ,\dQ_q}$ is classical.
\end{definition}

Now take a classical point $x$ of $\Spec\sE_\rJ(\xi^N,\rV_N,\rK_N)$, which gives rise to a character $\chi_x\colon\rI_N^0\to\dZ_x^\times$ as in Definition \ref{de:classical} and also a character
\[
\phi_x\colon\dT_N^\Box\to\dQ_x.
\]
We have the representation
\[
\dQ_{\xi^N+\zeta^N_x}\coloneqq\bigotimes_{v\in\Sigma^+_p}
\(\bigotimes_{\sigma\in\Upsilon^+,\sigma|v}\dQ_{(\xi^N+\zeta^N_x)_\sigma}\)
\]
of $\sfG_N(F^+_p)$ with coefficients in $\dQ_q$ from Construction \ref{co:weight}, which is also regarded as a local system on the Shimura space $\Sh(\sfG_N,\rK_N\rk)$ for every open compact subgroup $\rk$ of $\sfG_N(F^+_{\PP^+})$. Similar to $\rH(\Sh(\sfG_N,\rK_N\rk),\dZ_{\xi^N})$, put
\begin{align*}
\rH(\Sh(\sfG_N,\rK_N\rk),\dQ_{\xi^N+\zeta^N_x})\coloneqq
\begin{dcases}
\Gamma(\Sh(\sfG_N,\rK_N\rk),\dQ_{\xi^N+\zeta^N_x}), &\text{when $\rV_N$ is standard definite,}\\
\rH^{N-1}_{\et}(\Sh(\sfG_N,\rK_N\rk)_{\ol{F}},\dQ_{\xi^N+\zeta^N_x}(N-1-r)), &\text{when $\rV_N$ is standard indefinite,}
\end{dcases}
\end{align*}
which is naturally a $\dQ_q[(\rK_N)_{\spadesuit^+}\backslash\sfG_N(F^+_{\spadesuit^+})/(\rK_N)_{\spadesuit^+}]
\otimes\dT_N^\Box$-module.

For every $\rI\finite\rI_N^0$ and every $\rk\in\fk_\rI^\dag$, we denote by
\begin{align*}
\rH^\ordi(\Sh(\sfG_N,\rK_N\rk),\dQ_{\xi^N+\zeta^N_x})\subseteq
\rH(\Sh(\sfG_N,\rK_N\rk),\dQ_{\xi^N+\zeta^N_x})
\end{align*}
the maximal subspace on which the operator $\chi_{-(\zeta^N_x)^\vee}(b)^{-1}\cdot\tU_b$ is volume-preserving for every element $b\in\fB_\rk$, which is a $\dQ_q[(\rK_N)_{\spadesuit^+}\backslash\sfG_N(F^+_{\spadesuit^+})/(\rK_N)_{\spadesuit^+}]
\otimes\dT_N^\Box$-submodule. Then this module does not depend on the choice of $\rk\in\fk_\rI^\dag$ in the same sense as in Lemma \ref{le:ordinary}(3), which we denote by $\rH^\ordi(\Sh(\sfG_N,\rK_N\rI),\dQ_{\xi^N+\zeta^N_x})$. Put
\[
\rD_\rJ(\xi^N,\rV_N,\rK_N)_x\coloneqq\varprojlim_{\rJ<\rI\finite\rI_N^0}
\Hom\(\rH^\ordi(\Sh(\sfG_N,\rK_N\rI),\dQ_{\xi^N+\zeta^N_x}),\dQ_x\)_{\chi_x^\smooth,\phi_x},
\]
where the transfer maps are the dual maps of (those similarly defined as) $\rho_{\rI',\rI}$.

\begin{proposition}\label{pr:eigen_1}
Let $\rJ<\rI_N^0$ be a subgroup. Suppose that $\rV_N$ is definite.
\begin{enumerate}
  \item For every intermediate group $\rJ<\rJ'<\rI_N^0$, the natural map
      \[
      \cD_\rJ(\xi^N,\rV_N,\rK_N)\otimes_{\Lambda_{\rI_N^0/\rJ,\dZ_q}}\Lambda_{\rI_N^0/\rJ',\dZ_q}
      \to\cD_{\rJ'}(\xi^N,\rV_N,\rK_N)
      \]
      is an isomorphism.

  \item The $\Lambda_{\rI_N^0/\rJ,\dZ_q}$-module $\cD_\rJ(\xi^N,\rV_N,\rK_N)$ is locally finite free.

  \item The natural map
      \[
      \cE_\rJ(\xi^N,\rV_N,\rK_N)\to\varprojlim_{\rJ<\rI\finite\rI_N^0}\cE_\rI(\xi^N,\rV_N,\rK_N)
      \]
      is an isomorphism.

  \item The ring $\cE_\rJ(\xi^N,\rV_N,\rK_N)$ is reduced and of pure dimension $1+\rank_{\dZ_p}\rI_N^0/\rJ$.

  \item The natural morphism $\Spec\cE_\rJ(\xi^N,\rV_N,\rK_N)\to\Spec\Lambda_{\rI_N^0/\rJ,\dZ_q}$ is finite and generically flat.

  \item For every classical point $x$ of $\Spec\sE_\rJ(\xi^N,\rV_N,\rK_N)$, we have an isomorphism
      \[
      \sD_\rJ(\xi^N,\rV_N,\rK_N)\res_x\simeq\rD_\rJ(\xi^N,\rV_N,\rK_N)_x
      \]
      of $\dQ_x[(\rK_N)_{\spadesuit^+}\backslash\sfG_N(F^+_{\spadesuit^+})/(\rK_N)_{\spadesuit^+}]$-modules; and moreover, the ring $\cO_{\sE_\rJ(\xi^N,\rV_N,\rK_N),x}$ is regular if and only if it is flat over $\Lambda_{\rI_N^0/\rJ,\dQ_q}$.

  \item The $\sE_\rJ(\xi^N,\rV_N,\rK_N)$-module $\sD_\rJ(\xi^N,\rV_N,\rK_N)$ is finitely generated, and locally torsion free over the normal locus of $\Spec\sE_\rJ(\xi^N,\rV_N,\rK_N)$.
\end{enumerate}
\end{proposition}

The proof of the proposition requires the analogue of $\cD(\xi^N,\rV_N,\rK_N)$ for general dominant weights, a key step known as weight independence in Hida's theory (see also \cite{Ger19}*{\S2}). We start from some preparation.

For every place $v\in\PP^+$, we fix a self-dual $O_{F_v}$-lattice $\Lambda_{N,v}$ in $\rV_{N,v}$ such that $\sfB_{N,v}$ induces a Borel subgroup of $\rU(\Lambda_{N,v})$. Using this lattice, we regard $\sfG_{N,v}$, $\sfB_{N,v}$ and $\sfT_{N,v}$ as reductive groups defined over $O_{F^+_v}$. We denote by $\sfU_{N,v}\subseteq\sfB_{N,v}$ the unipotent radical and fix a section $\sfT_{N,v}\to\sfB_{N,v}$ (over $O_{F^+_v}$). Let $\sfB^-_{N,v}\subseteq\sfG_{N,v}$ be the Borel subgroup satisfying $\sfB^-_{N,v}\cap\sfB_{N,v}=\sfT_{N,v}$, with $\sfU^-_{N,v}$ its unipotent radical. For integers $0\leq c\leq l$, we denote by $\rI_{N,v}^c$ the subgroup of $\sfT_{N,v}(O_{F^+_v})$ of elements whose reduction modulo $p^c$ is the identity, and by $\rk_v^{c,l}$ the subgroup of $\sfG_{N,v}(O_{F^+_v})$ of elements whose reduction modulo $p^l$ belong to $\rI_{N,v}^c\sfU^-_{N,v}(O_{F^+_v})$. Put
\[
\rI_N^c\coloneqq\prod_{v\in\PP^+}\rI_{N,v}^c,\quad
\rk^{c,l}\coloneqq\prod_{v\in\PP^+}\rk_v^{c,l},\quad
\rK_N^{c,l}\coloneqq\rK_N\rk^{c,l}
\]
so that $\rI_N^0$ is the one we denoted before. It is clear that $\rk^{c,l}$ belongs to $\fk_{\rI_N^c}$.

It is straightforward to check that $\rk^{c,l}$ belongs to $\fk_{\rI_N^c}^\dag$ if $l\geq 2$;\footnote{It is even true for $l=1$ if all places in $\PP^+$ split in $F$.} and that $\sfT_N(F^+_{\PP^+})\cap\fB_{\rk^{c,l}}$ is independent of $c,l\geq 2$, which we denote by $\rI_N^\infty$, a monoid containing $\rI_N^0$. Fix a finite subset $\Delta$ of $\rI_N^\infty$ that reduces to a basis of the free abelian group $\sfT_N(F^+_{\PP^+})/\rI_N^0$.

Let $\zeta^N\in\Xi^N$ be a $\Sigma_p^+\setminus\PP^+$-trivial dominant hermitian weight. Denote by $[\zeta^N,-(\zeta^N)^\vee]$ the finite set of $\Sigma_p^+\setminus\PP^+$-trivial hermitian weights $\zeta$ such that both $\zeta^N-\zeta$ and $\zeta+(\zeta^N)^\vee$ are neutral and dominant (Definition \ref{de:weight0}). For every $\zeta\in[\zeta^N,-(\zeta^N)^\vee]$, we have a character $\chi_{\zeta}\colon\sfT_N(F^+_{\PP^+})\to\dQ_q^\times$ defined similarly as in Definition \ref{de:classical}. We have
\begin{enumerate}
  \item for every $t\in\rI_N^\infty$,
      \[
      \frac{\chi_\zeta(t)}{\chi_{-(\zeta^N)^\vee}(t)}\in\dZ_q;
      \]

  \item for every $\zeta\in[\zeta^N,-(\zeta^N)^\vee]\setminus\{-(\zeta^N)^\vee\}$,
      \[
      \prod_{t\in\Delta}\frac{\chi_\zeta(t)}{\chi_{-(\zeta^N)^\vee}(t)}\in p\dZ_q.
      \]
\end{enumerate}
Consider the representation
\[
\dZ_{\zeta^N}\coloneqq\bigotimes_{v\in\PP^+}\(\bigotimes_{\sigma\in\Upsilon^+,\sigma|v}\dZ_{\zeta^N_\sigma}\)
\]
of $\rk^{0,0}=\prod_{v\in\PP^+}\rU(\Lambda_{N,v})$ with coefficients in $\dZ_q$; it has a unique decomposition
\begin{align}\label{eq:weight1}
\dZ_{\zeta^N}=\bigoplus_{\zeta\in[\zeta^N,-(\zeta^N)^\vee]}\dL_{\zeta},
\end{align}
where $\dL_\zeta$ is a free $\dZ_q$-module of rank one on which $\rI_N^0$ acts by $\chi_{\zeta}$.

We now define a homomorphism
\[
\tV_\bullet\colon\rI_N^\infty\to
\End_{\dZ_q[(\rK_N)_{\spadesuit^+}\backslash\sfG_N(F^+_{\spadesuit^+})/(\rK_N)_{\spadesuit^+}]
\otimes\dT_N^\Box}\(\rH(\Sh(\sfG_N,\rK_N^{c,l}),\dZ_{\xi^N+\zeta^N}/p^m)\)
\]
of monoids for $1\leq m\leq \infty$, where
\[
\rH(\Sh(\sfG_N,\rK_N^{c,l}),\dZ_{\xi^N+\zeta^N}/p^m)\coloneqq\Gamma(\Sh(\sfG_N,\rK_N^{c,l}),\dZ_{\xi^N+\zeta^N}/p^m).
\]
Consider first $m<\infty$. In this case, put
\[
\rH(\Sh(\sfG_N,\rK_N),\dZ_{\xi^N})\coloneqq\varinjlim_\rk\Gamma(\Sh(\sfG_N,\rK_N\rk),\dZ_{\xi^N}),
\]
where the colimit is taken over the poset of open compact subgroups of $\sfG_N(F^+_{\PP^+})$ with respect to inclusion, on which $\sfG_N(F^+_{\PP^+})$ acts by pullback maps along right translations; and we have
\[
\rH(\Sh(\sfG_N,\rK_N^{c,l}),\dZ_{\xi^N+\zeta^N}/p^m)=
\(\rH(\Sh(\sfG_N,\rK_N),\dZ_{\xi^N})\otimes_{\dZ_q}\dZ_{\zeta^N}/p^m\)^{\rk^{c,l}}
\]
with respect to the diagonal action of $\rk^{c,l}$. Take an element $t\in\rI_N^\infty$. Denote by $\ang{t}$ the action on $\dZ_{\zeta^N}/p^m$ whose restriction to $\dL_\zeta$ under the decomposition \eqref{eq:weight1} is given by the multiplication by $\frac{\chi_\zeta(t)}{\chi_{-(\zeta^N)^\vee}(t)}\in\dZ_q$ (by property (1)). Write $\rk^{c,l}t\rk^{c,l}=\coprod_{k\in K_t} k t\rk^{c,l}$ with a (finite) subset $K_t\subseteq\rk^{c,l}$. It is easy to see that the endomorphism
\[
\sum_{k\in K_t} \tR_{kt}\otimes(k\circ\ang{t})
\]
of $\rH(\Sh(\sfG_N,\rK_N),\dZ_{\xi^N})\otimes_{\dZ_q}\dZ_{\zeta^N}/p^m$, where $\tR_{kt}$ denotes the pullback map along the right translation by $kt$, preserves the $\rk^{c,l}$-invariant submodule; and its induced action on $\rH(\Sh(\sfG_N,\rK_N^{c,l}),\dZ_{\xi^N+\zeta^N}/p^m)$ does not depend on the choice of $K_t$, which we denote by $\tV_t$. Passing to the limit, we also have the action $\tV_t$ on $\rH(\Sh(\sfG_N,\rK_N^{c,l}),\dZ_{\xi^N+\zeta^N})$ (the case for $m=\infty$) and hence on $\rH(\Sh(\sfG_N,\rK_N^{c,l}),\dQ_{\xi^N+\zeta^N})$. Then
\begin{enumerate}\setcounter{enumi}{2}
  \item when $l\geq 2$, the action of $\tV_t$ on $\rH(\Sh(\sfG_N,\rK_N^{c,l}),\dZ_{\xi^N}/p^m)$ coincides with $\tU_t$ for $t\in\rI_N^\infty$;

  \item when $l\geq 2$, the action of $\tV_t$ on $\rH(\Sh(\sfG_N,\rK_N^{c,l}),\dQ_{\xi^N+\zeta^N})$ coincides with $\chi_{-(\zeta^N)^\vee}(t)^{-1}\cdot\tU_t$ for $t\in\rI_N^\infty$.
\end{enumerate}
For $l\geq 2$, we have the operator
\begin{align*}
\te\coloneqq\lim_{M\to\infty}\(\prod_{t\in\Delta}\tV_t\)^{M!}
\end{align*}
on $\rH(\Sh(\sfG_N,\rK_N^{c,l}),\dZ_{\xi^N+\zeta^N}/p^m)$, which is an idempotent; and put
\[
\rH^\ordi(\Sh(\sfG_N,\rK_N^{c,l}),\dZ_{\xi^N+\zeta^N}/p^m)\coloneqq
\te\rH(\Sh(\sfG_N,\rK_N^{c,l}),\dZ_{\xi^N+\zeta^N}/p^m).
\]
When $\zeta^N=0$, it coincides with the one we have defined previously by (3).

Fix a map $\delta\colon\dZ_{\zeta^N}\to\dZ_q$ with kernel $\bigoplus_{\zeta\in[\zeta^N,-(\zeta^N)^\vee]\setminus\{-(\zeta^N)^\vee\}}\dL_{\zeta}$ in the decomposition \eqref{eq:weight1}. Then $\delta$ is invariant under the action of $\prod_{v\in\PP^+}\sfU_{N,v}^-(O_{F^+_v})$ on the source, so that its base change
\[
1\otimes\delta\colon\rH(\Sh(\sfG_N,\rK_N),\dZ_{\xi^N})\otimes_{\dZ_q}\dZ_{\zeta^N}/p^m\to
\rH(\Sh(\sfG_N,\rK_N),\dZ_{\xi^N})\otimes_{\dZ_q}\dZ_q/p^m
\]
induces a map
\[
\mu\colon\rH(\Sh(\sfG_N,\rK_N^{c,l}),\dZ_{\xi^N+\zeta^N}/p^m)\to\rH(\Sh(\sfG_N,\rK_N^{c,l}),\dZ_{\xi^N}/p^m)
\]
as long as $m\leq c\leq l$.

\begin{lem}\label{le:weight}
Suppose that $l\geq 2$. The map $\mu$ restricts to an isomorphism
\[
\rH^\ordi(\Sh(\sfG_N,\rK_N^{c,l}),\dZ_{\xi^N+\zeta^N}/p^m)\to\rH^\ordi(\Sh(\sfG_N,\rK_N^{c,l}),\dZ_{\xi^N}/p^m)
\]
of $\dZ_q[(\rK_N)_{\spadesuit^+}\backslash\sfG_N(F^+_{\spadesuit^+})/(\rK_N)_{\spadesuit^+}]
\otimes\dT_N^\Box$-modules, which satisfies
\[
\mu(t.x)=\chi_{-(\zeta^N)^\vee}(t)\cdot t.\mu(x)
\]
for every $x\in\rH(\Sh(\sfG_N,\rK_N^{c,l}),\dZ_{\xi^N+\zeta^N}/p^m)$ and $t\in\rI_N^0$.
\end{lem}

\begin{proof}
It suffices to show that $\mu$ restricts to a bijection between $\rH^\ordi$ as the remaining claims follow from the definition. We first construct a map
\[
\nu\colon\rH(\Sh(\sfG_N,\rK_N^{c,l}),\dZ_{\xi^N}/p^m)\to\rH(\Sh(\sfG_N,\rK_N^{c,l}),\dZ_{\xi^N+\zeta^N}/p^m)
\]
in the opposite direction. Write $\bbt\coloneqq\prod_{t\in\Delta}t$. Let $\epsilon\colon\dZ_q\to\dZ_{\zeta^N}$ be the unique section of $\delta$ with image $\dL_{-(\zeta^N)^\vee}$. Then we have $h\circ\epsilon=\epsilon$ for every $h\in\rk^{c,l}\cap\bbt^m\rk^{c,l}\bbt^{-m}$.

Choose a set $K$ of representatives for the (finite) coset $\rk^{c,l}/(\rk^{c,l}\cap\bbt^m\rk^{c,l}\bbt^{-m})$. We claim that the map
\[
\sum_{k\in K}\tR_{k\bbt^m}\otimes(k\circ\epsilon)\colon
\rH(\Sh(\sfG_N,\rK_N),\dZ_{\xi^N})\otimes_{\dZ_q}\dZ_q/p^m\to
\rH(\Sh(\sfG_N,\rK_N),\dZ_{\xi^N})\otimes_{\dZ_q}\dZ_{\zeta^N}/p^m
\]
preserves $\rk^{c,l}$-invariants. Indeed, take an element $h\in\rk^{c,l}$. For every $k\in K$, write $h\cdot k=k'\cdot h_k$ for unique elements $k'\in K$ and $h_k\in\rk^{c,l}\cap\bbt^m\rk^{c,l}\bbt^{-m}$. Then after restricting to the submodule
\[
\(\rH(\Sh(\sfG_N,\rK_N),\dZ_{\xi^N})\otimes_{\dZ_q}\dZ_q/p^m\)^{\rk^{c,l}}
=\rH(\Sh(\sfG_N,\rK_N),\dZ_{\xi^N})^{\rk^{c,l}}\otimes_{\dZ_q}\dZ_q/p^m,
\]
we have
\begin{align*}
h\circ\sum_{k\in K}\tR_{k\bbt^m}\otimes(k\circ\epsilon)
&=\sum_{k\in K}\tR_{hk\bbt^m}\otimes(hk\circ\epsilon)
=\sum_{k\in K}\tR_{k'h_k\bbt^m}\otimes(k'h_k\circ\epsilon) \\
&=\sum_{k\in K}\tR_{k'\bbt^m(\bbt^{-m}h_k\bbt^m)}\otimes(k'\circ\epsilon)
=\sum_{k\in K}\tR_{k'\bbt^m}\otimes(k'\circ\epsilon)
=\sum_{k\in K}\tR_{k\bbt^m}\otimes(k\circ\epsilon).
\end{align*}
Thus, $\sum_{k\in K}\tR_{k\bbt^m}\otimes(k\circ\epsilon)$ induces a map
\[
\nu\colon\rH(\Sh(\sfG_N,\rK_N^{c,l}),\dZ_{\xi^N}/p^m)\to\rH(\Sh(\sfG_N,\rK_N^{c,l}),\dZ_{\xi^N+\zeta^N}/p^m),
\]
whose independence of the choice of $K$ can be easily seen.

We now compute $\nu\circ\mu$ and $\mu\circ\nu$. By definition, $\nu\circ\mu$ is the restriction of the map
\[
\sum_{k\in K}\tR_{k\bbt^m}\otimes(k\circ\epsilon\circ\delta)
\]
to the $\rk^{c,l}$-invariants. As $\epsilon\circ\delta=\ang{\bbt^m}$ by the property (2) above, we conclude that $\nu\circ\mu$ coincides with the operator $\tV_{\bbt^m}$ on $\rH(\Sh(\sfG_N,\rK_N^{c,l}),\dZ_{\xi^N+\zeta^N}/p^m)$. By definition, $\mu\circ\nu$ is the restriction of the map
\[
\sum_{k\in K}\tR_{k\bbt^m}\otimes(\delta\circ k\circ\epsilon)
\]
to the $\rk^{c,l}$-invariants. As $\delta\circ k\circ\epsilon$ is the identity map for $k\in\rk^{c,l}$, we conclude that $\mu\circ\nu$ coincides with the operator $\tV_{\bbt^m}$ on $\rH(\Sh(\sfG_N,\rK_N^{c,l}),\dZ_{\xi^N}/p^m)$.

Now the lemma follows since $\rH^\ordi(\Sh(\sfG_N,\rK_N^{c,l}),\dZ_{\xi^N}/p^m)$ (resp.\ $\rH^\ordi(\Sh(\sfG_N,\rK_N^{c,l}),\dZ_{\xi^N+\zeta^N}/p^m)$) is just the image of $\tV_{\bbt^m}$; and $\tV_{\bbt^m}$ induces an automorphism of $\rH^\ordi(\Sh(\sfG_N,\rK_N^{c,l}),\dZ_{\xi^N}/p^m)$ (resp.\ $\rH^\ordi(\Sh(\sfG_N,\rK_N^{c,l}),\dZ_{\xi^N+\zeta^N}/p^m)$).
\end{proof}

\begin{proof}[Proof of Proposition \ref{pr:eigen_1}]
We prove the proposition in the order (1,3,6,2,5,4,7).

For (1), note that for subgroups $\rI\finite\rI'\finite\rI_N^0$, the natural map
\[
\cD_\rI(\xi^N,\rV_N,\rK_N)\otimes_{\Lambda_{\rI_N^0/\rI,\dZ_q}}\Lambda_{\rI_N^0/\rI',\dZ_q}
\to\cD_{\rI'}(\xi^N,\rV_N,\rK_N)
\]
is an isomorphism; in particular, the natural map $\cD_\rI(\xi^N,\rV_N,\rK_N)\to\cD_{\rI'}(\xi^N,\rV_N,\rK_N)$ is surjective. It follows that the natural map
\begin{align*}
\cD_\rJ(\xi^N,\rV_N,\rK_N)\otimes_{\Lambda_{\rI_N^0/\rJ,\dZ_q}}\Lambda_{\rI_N^0/\rJ',\dZ_q}
&=\(\varprojlim_{\rJ<\rI\finite\rI_N^0}\cD_\rI(\xi^N,\rV_N,\rK_N)\)
\otimes_{\Lambda_{\rI_N^0/\rJ,\dZ_q}}\Lambda_{\rI_N^0/\rJ',\dZ_q}  \\
&\to\varprojlim_{\rJ<\rI\finite\rI_N^0}
\(\cD_\rI(\xi^N,\rV_N,\rK_N)\otimes_{\Lambda_{\rI_N^0/\rJ,\dZ_q}}\Lambda_{\rI_N^0/\rJ',\dZ_q}\)
\end{align*}
is an isomorphism. Then (1) follows as
\[
\varprojlim_{\rJ<\rI\finite\rI_N^0}
\(\cD_\rI(\xi^N,\rV_N,\rK_N)\otimes_{\Lambda_{\rI_N^0/\rJ,\dZ_q}}\Lambda_{\rI_N^0/\rJ',\dZ_q}\)
=\varprojlim_{\rJ'<\rI'\finite\rI_N^0}\cD_{\rI'}(\xi^N,\rV_N,\rK_N)
=\cD_{\rJ'}(\xi^N,\rV_N,\rK_N).
\]

For (3), the map in question is clearly injective by definition. Thus, it suffices to show that the reduction
\begin{align}\label{eq:eigen1}
\cE_\rJ(\xi^N,\rV_N,\rK_N)_{\dF_p}
\to\varprojlim_{\rJ<\rI\finite\rI_N^0}\cE_\rI(\xi^N,\rV_N,\rK_N)_{\dF_p}
\end{align}
of that map is surjective. Note that both sides of \eqref{eq:eigen1} are closed subspaces of the topological $\dF_q$-vector space $V\coloneqq\End_{\Lambda_{\rI_N^0/\rJ,\dZ_q}}\(\cD_\rJ(\xi^N,\rV_N,\rK_N)_{\dF_p}\)$. For every intermediate subgroup $\rJ<\rI\finite\rI_N^0$, denote by $V_\rI$ the kernel of the reduction map $V\to\End_{\Lambda_{\rI_N^0/\rI,\dZ_q}}\(\cD_\rI(\xi^N,\rV_N,\rK_N)_{\dF_p}\)$. Then $\{V_\rI\}_{\rJ<\rI\finite\rI_N^0}$ forms an open basis of the topological space $V$. By (1), the natural map $\cE_\rJ(\xi^N,\rV_N,\rK_N)_{\dF_p}\to\cE_\rI(\xi^N,\rV_N,\rK_N)_{\dF_p}$ is surjective for every $\rJ<\rI\finite\rI_N^0$, which implies that
\[
\varprojlim_{\rJ<\rI\finite\rI_N^0}\cE_\rI(\xi^N,\rV_N,\rK_N)_{\dF_p}\subseteq
\bigcap_{\rJ<\rI\finite\rI_N^0}\(\cE_\rJ(\xi^N,\rV_N,\rK_N)_{\dF_p}\oplus V_\rI\)
=\cE_\rJ(\xi^N,\rV_N,\rK_N)_{\dF_p}.
\]
Thus, \eqref{eq:eigen1} follows.

For (6), we first construct the isomorphism
\begin{align}\label{eq:weight3}
\sD_\rJ(\xi^N,\rV_N,\rK_N)\res_x\simeq\rD_\rJ(\xi^N,\rV_N,\rK_N)_x.
\end{align}
It suffices to consider the case where $\rJ=1$ by (1). Take a classical point $y$ of $\Spec\Lambda_{\rI_N^0}$. By Lemma \ref{le:ordinary}(2) and Lemma \ref{le:weight}, we have a natural isomorphism
\[
\rH^\ordi(\Sh(\sfG_N,\rK_N^{c,c}),\dZ_{\xi^N+\zeta^N_y}/p^m)[(\chi_y^\smooth)^{-1}]\simeq
\rH(\Sh(\sfG_N,\rK_N\rI_N^c),\dZ_{\xi^N}/p^m)[\chi_y^{-1}]
\]
of $\dZ_q[(\rK_N)_{\spadesuit^+}\backslash\sfG_N(F^+_{\spadesuit^+})/(\rK_N)_{\spadesuit^+}]
\otimes\dT_N^\Box$-modules, functorial in $c$ and $m$ as long as $2\leq m\leq c$. Applying the functor
\[
\Hom_{\dZ_q}\(\varprojlim_m\varinjlim_c-,\dQ_q\),
\]
the right-hand side becomes
\begin{align*}
&\quad\Hom_{\dZ_q}\(\varprojlim_m\varinjlim_c\rH(\Sh(\sfG_N,\rK_N\rI_N^c),\dZ_{\xi^N}/p^m)[\chi_y^{-1}],\dQ_q\) \\
&=\Hom_{\dZ_q}\(\varprojlim_m\varinjlim_c\rH(\Sh(\sfG_N,\rK_N\rI_N^c),\dZ_{\xi^N}/p^m),\dQ_q\)_{\chi_y} \\
&=\(\Hom_{\dZ_q}\(\varprojlim_m\varinjlim_c\rH(\Sh(\sfG_N,\rK_N\rI_N^c),\dZ_{\xi^N}/p^m),\dZ_q\)\otimes_{\dZ_q}\dQ_q\)_{\chi_y} \\
&=\(\varprojlim_m\Hom_{\dZ_q}\(\varinjlim_c\rH(\Sh(\sfG_N,\rK_N\rI_N^c),\dZ_{\xi^N}/p^m),\dZ/p^m\)\otimes_{\dZ_q}\dQ_q\)_{\chi_y} \\
&=\(\varprojlim_m\varprojlim_c\Hom_{\dZ_q}\(\rH(\Sh(\sfG_N,\rK_N\rI_N^c),\dZ_{\xi^N}/p^m),\dZ/p^m\)\otimes_{\dZ_q}\dQ_q\)_{\chi_y} \\
&=\(\varprojlim_c\varprojlim_m\Hom_{\dZ_q}\(\rH(\Sh(\sfG_N,\rK_N\rI_N^c),\dZ_{\xi^N}/p^m),\dZ/p^m\)\otimes_{\dZ_q}\dQ_q\)_{\chi_y} \\
&=\(\varprojlim_c\Hom_{\dZ_q}\(\rH(\Sh(\sfG_N,\rK_N\rI_N^c),\dZ_{\xi^N}),\dZ_q\)\otimes_{\dZ_q}\dQ_q\)_{\chi_y} \\
&=\sD(\xi^N,\rV_N,\rK_N)_{\chi_y};
\end{align*}
and the left-hand side becomes
\begin{align*}
&\quad\Hom_{\dZ_q}
\(\varprojlim_m\varinjlim_c\rH^\ordi(\Sh(\sfG_N,\rK_N^{c,c}),\dZ_{\xi^N+\zeta^N_y}/p^m)[(\chi_y^\smooth)^{-1}],\dQ_q\) \\
&=\Hom_{\dZ_q}\(\varinjlim_c\rH^\ordi(\Sh(\sfG_N,\rK_N^{c,c}),\dZ_{\xi^N+\zeta^N_y}),\dQ_q\)_{\chi_y^\smooth} \\
&=\varprojlim_{c}\Hom\(\rH^\ordi(\Sh(\sfG_N,\rK_N\rI_N^c),\dQ_{\xi^N+\zeta^N_y}),\dQ_q\)_{\chi_y^\smooth}
\end{align*}
using property (4). Thus, we obtain an isomorphism
\begin{align}\label{eq:weight2}
\sD(\xi^N,\rV_N,\rK_N)_{\chi_y}\simeq
\Hom\(\rH^\ordi(\Sh(\sfG_N,\rK_N\rI_N^c),\dQ_{\xi^N+\zeta^N_y}),\dQ_q\)_{\chi_y^\smooth}
\end{align}
of $\dQ_q[(\rK_N)_{\spadesuit^+}\backslash\sfG_N(F^+_{\spadesuit^+})/(\rK_N)_{\spadesuit^+}]
\otimes\dT_N^\Box$-modules for every $c\geq 0$ such that $\chi_y^\smooth$ is trivial on $\rI_N^c$. Then the isomorphism \eqref{eq:weight3} is obtained from \eqref{eq:weight2} by taking $y$ to be the image of $x$ and then taking $\phi_x$-quotients on both sides.

For the remaining statement, let $y$ be the image of $x$ in $\Spec\Lambda_{\rI_N^0/\rJ,\dQ_q}$, which is a classical point. If $\cO_{\sE_\rJ(\xi^N,\rV_N,\rK_N),x}$ is regular, then it is flat over $\Lambda_{\rI_N^0/\rJ,\dQ_q}$ by the miracle flatness criterion over regular local rings \cite{Mat89}*{Theorem~23.1}. Conversely, if $\cO_{\sE_\rJ(\xi^N,\rV_N,\rK_N),x}$ is flat over $\Lambda_{\rI_N^0/\rJ,\dQ_q}$, then \eqref{eq:weight3} implies that $\cO_{\sE_\rJ(\xi^N,\rV_N,\rK_N),x}\otimes_{\Lambda_{\rI_N^0/\rJ,\dQ_q}}\dQ_y$ is a field, where $\dQ_y$ denotes the residue field of $y$. Since $(\Lambda_{\rI_N^0/\rJ,\dQ_q})_y$ is regular and has the same dimension as $\cO_{\sE_\rJ(\xi^N,\rV_N,\rK_N),x}$, it follows that $\cO_{\sE_\rJ(\xi^N,\rV_N,\rK_N),x}$ is regular as well. Part (6) has been proved.

For (2), it suffices to consider the case where $\rJ=1$ by (1). Choose an arbitrary maximal ideal $\fm$ of $\Lambda_{\rI_N^0}$ and consider the $\Lambda_{\rI_N^0,\fm}$-module $\cD(\xi^N,\rV_N,\rK_N)_\fm$, which is equivalent to choose a character $\chi_\fm\colon\rI_N^0/\rI_N^1\to\dZ_q^\times$. Let $r_\fm$ be the $\dF_q$-dimension of $\cD(\xi^N,\rV_N,\rK_N)\otimes_{\Lambda_{\rI_N^0}}\Lambda_{\rI_N^0}/\fm$. By Nakayama's lemma, it suffices to show that the generic rank of the $\Lambda_{\rI_N^0,\fm}$-module $\cD(\xi^N,\rV_N,\rK_N)_\fm$ is $r_\fm$. Assume the converse. Then there exists a $\Sigma_p^+\setminus\PP^+$-trivial dominant hermitian weight $\zeta^N\in\Xi^N$ such that $\chi_\fm\cdot\chi_{-(\zeta^N)^\vee}$ corresponds to a closed point $y$ of $\Spec(\Lambda_{\rI_N^0,\fm}\otimes_{\dZ_q}\dQ_q)$ satisfying that $\dim_{\dQ_q}\sD(\xi^N,\rV_N,\rK_N)_{\chi_y}<r_\fm$. By \eqref{eq:weight2}, we have an isomorphism
\[
\sD(\xi^N,\rV_N,\rK_N)_{\chi_y}\simeq
\Hom\(\rH^\ordi(\Sh(\sfG_N,\rK_N\rI_N^1),\dQ_{\xi^N+\zeta^N_y}),\dQ_q\)_{\chi_\fm}
\]
of $\dQ_q$-vector spaces. Since $\rH^\ordi(\Sh(\sfG_N,\rK_N\rI_N^1),\dZ_{\xi^N+\zeta^N_y})$ is a free $\dZ_q$-module and $p\nmid|\rI_N^0/\rI_N^1|$, we have
\begin{align*}
\dim_{\dQ_q}\sD(\xi^N,\rV_N,\rK_N)_{\chi_y}
&=\rank_{\dZ_q}\rH^\ordi(\Sh(\sfG_N,\rK_N\rI_N^1),\dZ_{\xi^N+\zeta^N_y})[\chi_\fm] \\
&=\dim_{\dF_q}\rH^\ordi(\Sh(\sfG_N,\rK_N\rI_N^1),\dZ_{\xi^N+\zeta^N_y}/p)[\chi_\fm],
\end{align*}
which by Lemma \ref{le:weight} equals
\[
\dim_{\dF_q}\rH^\ordi(\Sh(\sfG_N,\rK_N\rI_N^1),\dZ_{\xi^N}/p)[\chi_\fm]
=\dim_{\dF_q}\cD(\xi^N,\rV_N,\rK_N)\otimes_{\Lambda_{\rI_N^0}}\Lambda_{\rI_N^0}/\fm=r_\fm.
\]
This is a contradiction hence (2) is proved.

For (5), it follows directly from (2).

For (4), we know by (5) that $\cE_\rJ(\xi^N,\rV_N,\rK_N)$ is of pure dimension $1+\rank_{\dZ_p}\rI_N^0/\rJ$ (which is the dimension of $\Lambda_{\rI_N^0/\rJ}$). For being reduced, by (3), we may assume that $\rJ=\rI\finite\rI_N^0$. Then $\cE_\rI(\xi^N,\rV_N,\rK_N)$ is a finite free $\dZ_q$-module. Since the action of $\dT_N^\Box$ on $\sD_\rI(\xi^N,\rV_N,\rK_N)=\Hom_{\dQ_q}\(\rH(\Sh(\sfG_N,\rK_N\rI),\dQ_{\xi^N}),\dZ_q\)$ is semisimple (via global Arthur packets), $\sE_\rI(\xi^N,\rV_N,\rK_N)$ is a finite product of fields. It follows that $\cE_\rI(\xi^N,\rV_N,\rK_N)$ is reduced. Thus, (4) is proved.

For (7), we know by (2) that $\sD_\rJ(\xi^N,\rV_N,\rK_N)$ is finitely generated over $\sE_\rJ(\xi^N,\rV_N,\rK_N)$ and is locally torsion free over (the regular ring) $\Lambda_{\rI_N^0/\rJ,\dQ_q}$; it follows that $\sD_\rJ(\xi^N,\rV_N,\rK_N)$ is also locally torsion free over the normal locus of $\Spec\sE_\rJ(\xi^N,\rV_N,\rK_N)$ since $\sE_\rJ(\xi^N,\rV_N,\rK_N)$ is a finite locally torsion free $\Lambda_{\rI_N^0/\rJ,\dQ_q}$-module.

\end{proof}

\section{Rankin--Selberg product}
\label{ss:product}

In this section, we present the Rankin--Selberg version of the material in the previous section, with the additional discussion on localization. We also introduce the standard Galois representation on the integral eigenvariety.

From now on, we fix an integer $n\geq 1$ and a pair $\xi=(\xi^n,\xi^{n+1})\in\Xi^n\times\Xi^{n+1}$ of $\PP^+$-trivial dominant hermitian weights. Denote by $n_0$ and $n_1$ the unique even and odd numbers in $\{n,n+1\}$, respectively. Write $n_0=2r_0$ and $n_1=2r_1+1$, so that $n=r_0+r_1$.

\begin{definition}\label{de:weight}
We say that a pair $\zeta=(\zeta^n,\zeta^{n+1})\in\Xi^n\times\Xi^{n+1}$ of dominant hermitian weights is
\begin{itemize}
  \item \emph{interlacing}, if
      \[
      \zeta^{n+1}_{\tau,0}\leq\zeta^n_{\tau\tc,0}\leq\zeta^{n+1}_{\tau,1}\leq\zeta^n_{\tau\tc,1}
      \leq\cdots\leq\zeta^{n+1}_{\tau,n-1}\leq\zeta^n_{\tau\tc,n-1}\leq \zeta^{n+1}_{\tau,n}
      \]
      holds for every $\tau\in\Upsilon$;

  \item \emph{Fontaine--Laffaille regular}, if
      \[
      \max_{\tau|w}\left\{\zeta^{n}_{\tau,n-1}+\zeta^{n+1}_{\tau,n}\right\}
      -\min_{\tau|w}\left\{\zeta^n_{\tau,0}+\zeta^{n+1}_{\tau,0}\right\}+2n-1\leq p-2
      \]
      holds for every $w\in\Sigma_p$.
\end{itemize}
\end{definition}

\begin{definition}\label{de:datum}
A \emph{definite/indefinite datum} is a collection $\bV=(\fm;\rV_n,\rV_{n+1};\Lambda_n,\Lambda_{n+1};\rK_n,\rK_{n+1};\sfB)$ in which
\begin{itemize}
  \item $\fm$ is a (possibly empty) finite set of nonarchimedean places of $F^+$ not in $\langle\emptyset\rangle$ that are inert in $F$;

  \item $\rV_n$ is a standard definite/indefinite hermitian space over $F$ (Definition \ref{de:standard_hermitian_space}) of rank $n$ that is split at every place in $\PP^+$, and $\rV_{n+1}=(\rV_n)_\sharp$;

  \item $\Lambda_n$ is a $\prod_{v\not\in\Sigma^+_\infty\cup\spadesuit^+\cup\PP^+}O_{F_v}$-lattice in $\rV_n\otimes_F\dA_F^{\Sigma_\infty\cup\spadesuit\cup\PP}$ satisfying $\Lambda_n\subseteq(\Lambda_n)^\vee$ and that $(\Lambda_{n,v})^\vee/\Lambda_{n,v}$ has length one (resp.\ zero) when $v\in\fm$ (resp.\ $v\not\in\fm$), and $\Lambda_{n+1}=(\Lambda_n)_\sharp$;

  \item $(\rK_n,\rK_{n+1})$ is a pair of decomposable neat open compact subgroups of $\sfG_n(\dA_{F^+}^{\infty,\PP^+})$ and $\sfG_{n+1}(\dA_{F^+}^{\infty,\PP^+})$, respectively, of the form
      \begin{align*}
      \rK_N=\prod_{v\in\spadesuit^+}(\rK_N)_v\times
      \prod_{v\not\in\Sigma^+_\infty\cup\spadesuit^+\cup\PP^+}\rU(\Lambda_N)(O_{F^+_v})
      \end{align*}
      for $N\in\{n,n+1\}$;

  \item $\sfB$ is a Borel subgroup of $\sfG\otimes_{F^+}F^+_{\PP^+}$ that has \emph{trivial} intersection with $\sfH\otimes_{F^+}F^+_{\PP^+}$; here we put $\sfG\coloneqq\sfG_n\times\sfG_{n+1}$, and denote by $\sfH$ the graph of the natural homomorphism $\sfG_n\to\sfG_{n+1}$, which is a closed subgroup of $\sfG$.
\end{itemize}
\end{definition}

Choose a definite/indefinite datum as in the above definition.
\begin{itemize}
  \item Put $\dT^{\ang{\fm}}\coloneqq\dT_n^{\ang{\fm}}\otimes\dT_{n+1}^{\ang{\fm}}$.

  \item Put $\rK\coloneqq\rK_n\times\rK_{n+1}$. For every open compact subgroup $\rk$ of $\sfG(F^+_{\PP^+})$, we have the map of Shimura spaces
      \[
      \graph\colon\Sh(\sfH,\rK_\sfH\rk_\sfH)\to\Sh(\sfG,\rK\rk),
      \]
      in which
      \[
      \rK_\sfH\coloneqq\rK\cap\sfH(\dA_{F^+}^{\infty,\PP^+}),\qquad
      \rk_\sfH\coloneqq\rk\cap\sfH(F^+_{\PP^+}),
      \]
      together with the representation $\dZ_\xi=\dZ_{\xi^n}\boxtimes\dZ_{\xi^{n+1}}$ of $\rK\sfG(F^+_{\PP^+})$ with coefficients in $\dZ_q$, which gives rise to a $\dZ_q$-local system (with the same notation) on $\Sh(\sfG,\rK\rk)$.

  \item Denote by $\sfT=\sfT_n\times\sfT_{n+1}$ the Levi quotient of $\sfB=\sfB_n\times\sfB_{n+1}$. Put $\rI^0\coloneqq\rI_n^0\times\rI_{n+1}^0$.

  \item For every subgroup $\rI\finite\rI^0$, we have the set $\fk_\rI$ of open compact subgroups $\rk$ of $\sfG(F^+_{\PP^+})$ satisfying that the image of $\rk_\sfB\coloneqq\rk\cap\sfB(F^+_{\PP^+})$ in $\sfT(F^+_{\PP^+})$ coincides with $\rI$. There is a natural poset structure on $\bigcup_\rI\fk_\rI$ analogous to the one in Construction \ref{co:liusun}.
\end{itemize}

For $1\leq m\leq \infty$ and an open compact subgroup $\rk$ of $\sfG(F^+_{\PP^+})$, put
\begin{align*}
\rH(\Sh(\sfG,\rK\rk),\dZ_\xi/p^m)\coloneqq
\begin{dcases}
\Gamma(\Sh(\sfG,\rK\rk),\dZ_\xi/p^m), &\text{when $\bV$ is definite,}\\
\rH^{2n-1}_{\et}(\Sh(\sfG,\rK\rk)_{\ol{F}},\dZ_\xi/p^m(n-1)), &\text{when $\bV$ is indefinite,}
\end{dcases}
\end{align*}
which is naturally a $\dZ_q[\rK_{\spadesuit^+}\backslash\sfG(F^+_{\spadesuit^+})/\rK_{\spadesuit^+}]
\otimes\dT^{\ang{\fm}}$-module. For every subgroup $\rI\finite\rI^0$, we define
\[
\rH(\Sh(\sfG,\rK\rI),\dZ_{\xi^N}/p^m)\coloneqq
\varprojlim_{\fk_\rI^{\r{op}}}\rH(\Sh(\sfG,\rK\rk),\dZ_{\xi^N}/p^m)
\]
as a module over $\dZ[\rK_{\spadesuit^+}\backslash\sfG(F^+_{\spadesuit^+})/\rK_{\spadesuit^+}]
\otimes\dT^{\ang{\fm}}\otimes\Lambda_{\sfT(F^+_{\PP^+})/\rI,\dZ_q}$. We endow
\[
\Hom_{\dZ_q}\(\rH(\Sh(\sfG,\rK\rI),\dZ_{\xi^N}/p^m),\dZ_q/p^m\)
\]
the $\dZ[\rK_{\spadesuit^+}\backslash\sfG(F^+_{\spadesuit^+})/\rK_{\spadesuit^+}]
\otimes\dT^{\ang{\fm}}\otimes\Lambda_{\sfT(F^+_{\PP^+})/\rI,\dZ_q}$-module structure via the \emph{dual} action (in particular, the $\Lambda_{\sfT(F^+_{\PP^+})/\rI,\dZ_q}$-module structure is the usual one twisted by the inverse of $\sfT(F^+_{\PP^+})/\rI$).

\begin{remark}\label{re:ordinary}
For every open compact subgroup $\rk$ of $\sfG(F^+_{\PP^+})$, we have similarly the submonoid $\fB_\rk\subseteq\sfB(F^+_{\PP^+})$ and the homomorphism
\[
\tU_\bullet\colon\fB_\rk\to\End_{\dZ_q}\(\rH(\Sh(\sfG,\rK\rk),\dZ_\xi/p^m)\)
\]
of monoids for every $1\leq m\leq\infty$, as in the previous section.

For a subgroup $\rI\finite\rI^0$, denote by $\fk_\rI^\dag$ the subset of $\fk_\rI$ of elements $\rk$ such that the image of $\fB_\rk$ in $\sfT(F^+_{\PP^+})$ generates the entire group. Then for every $\rk\in\fk_\rI^\dag$ and every $1\leq m\leq\infty$, we have a similar decomposition
\begin{align*}
\rH(\Sh(\sfG,\rK\rk),\dZ_\xi/p^m)=
\rH^\ordi(\Sh(\sfG,\rK\rk),\dZ_\xi/p^m)\oplus
\rH^{\r{ss}}(\Sh(\sfG,\rK\rk),\dZ_\xi/p^m),
\end{align*}
such that the natural map
\[
\rH(\Sh(\sfG,\rK\rI),\dZ_\xi/p^m)\to\rH(\Sh(\sfG,\rK\rk),\dZ_\xi/p^m)
\]
is injective with image $\rH^\ordi(\Sh(\sfG,\rK\rk),\dZ_\xi/p^m)$.
\end{remark}

Similar to Definition \ref{de:distribution}, for every subgroup $\rJ<\rI^0$, put
\[
\cD_\rJ(\xi,\bV)\coloneqq\varprojlim_{\rJ<\rI\finite\rI^0}
\Hom_{\dZ_q}\(\rH(\Sh(\sfG,\rK\rI),\dZ_{\xi^N}),\dZ_q\)
\]
with respect to the dual of $\rho_{\rI',\rI}$, as a module over $\dZ[\rK_{\spadesuit^+}\backslash\sfG(F^+_{\spadesuit^+})/\rK_{\spadesuit^+}]\otimes
\dT^{\ang{\fm}}\otimes\Lambda_{\sfT(F^+_{\PP^+})/\rJ,\dZ_q}$ (and a topological module over $\Lambda_{\sfT(F^+_{\PP^+})/\rJ,\dZ_q}$), admitting an action by $\Gamma_F$ when $\bV$ is indefinite. Let $\cE_\rJ(\xi,\bV)$ be the $\Lambda_{\sfT(F^+_{\PP^+})/\rJ,\dZ_q}$-subalgebra of
\[
\End_{\Lambda_{\sfT(F^+_{\PP^+})/\rJ,\dZ_q}}\(\cD_\rJ(\xi,\bV)\)
\]
generated by the image of $\dT^{\ang{\fm}}$. Put
\begin{align*}
\sD_\rJ(\xi,\bV)\coloneqq\cD_\rJ(\xi,\bV)\otimes_{\dZ_q}\dQ_q,\qquad
\sE_\rJ(\xi,\bV)\coloneqq\cE_\rJ(\xi,\bV)\otimes_{\dZ_q}\dQ_q,
\end{align*}
as a module and a ring over $\dZ[\rK_{\spadesuit^+}\backslash\sfG(F^+_{\spadesuit^+})/\rK_{\spadesuit^+}]
\otimes\dT^{\ang{\fm}}\otimes\Lambda_{\sfT(F^+_{\PP^+})/\rJ,\dQ_q}$, respectively. We suppress $\rJ$ in all the subscripts when it is the trivial subgroup.

The following three definitions are exclusive for the Rankin--Selberg case, which will be frequently used later.

\begin{definition}
For every subgroup $\rI\finite\rI^0$, define $\fk_\rI^\sfH$ to be the subset of $\fk_\rI$ of elements $\rk$ satisfying $\rk=\rk_\sfH\cdot\rk_\sfB$, which is cofinal.
\end{definition}

\begin{definition}\label{de:distinction}
A \emph{$\xi$-distinction} is a $(\rK_\sfH)_{\Sigma^+_p\setminus\PP^+}$-invariant surjective homomorphism $\dZ_\xi\to\dZ_q$. A $\xi$-distinction exists if and only if $\xi$ is interlacing (Definition \ref{de:weight}); and it is unique up to a scalar in $\dZ_q^\times$ if exists.
\end{definition}

\begin{definition}
For a $\dZ[\rK_{\spadesuit^+}\backslash\sfG(F^+_{\spadesuit^+})/\rK_{\spadesuit^+}]$-module $\cM$, we denote by $\cM^\sfH$ the maximal submodule on which $\dZ[(\sfK_\sfH)_{\spadesuit^+}\backslash\sfH(F^+_{\spadesuit^+})/(\sfK_\sfH)_{\spadesuit^+}]$ acts by the degree map.
\end{definition}

From now on, we fix a pair of hermitian homomorphisms (Definition \ref{de:hermitian})
\begin{align}\label{eq:gamma}
\gamma=(\gamma_n,\gamma_{n+1})\colon\Gamma_{F^+}\to\sG_n(\dF_q)\times\sG_{n+1}(\dF_q)
\end{align}
that is unramified away from $\spadesuit^+\cup\Sigma^+_p$. Associated with $\gamma$, we have the Satake homomorphisms
\[
\phi_{\gamma_N}\colon\dT^{\ang{\emptyset}}_N\to\dF_q
\]
for $N=n,n+1$ (which is, for example, constructed above \cite{LTXZZ1}*{Theorem~3.38}). Put
\[
\phi_\gamma\coloneqq\phi_{\gamma_n}\otimes\phi_{\gamma_{n+1}}\colon\dT^{\ang{\emptyset}}
\coloneqq\dT^{\ang{\emptyset}}_n\otimes\dT^{\ang{\emptyset}}_{n+1}\to\dF_q.
\]

\begin{assumption}\label{as:galois}
We impose the following assumptions which will be used later.
\begin{description}
  \item[(G1)] The $\Gamma_F$-representation $\gamma_n^\natural\otimes\gamma_{n+1}^\natural$ is absolutely irreducible.

  \item[(G2)] Both $\phi_{\gamma_n}$ and $\phi_{\gamma_{n+1}}$ are cohomologically generic in the sense of \cite{LTXZZ}*{Definition~D.1.1}.

  \item[(G3)] Either one of the following two assumptions holds:
            \begin{itemize}
              \item $(\gamma_n^\natural\otimes\gamma_{n+1}^\natural)(\Gamma_F)$ contains a nontrivial scalar element;

              \item $p\geq\max\{2n(n+1)-1,5\}$; and we have
                  \[
                  \Hom_{\dF_q[\Gamma_F]}
                  \(\End(\gamma_0^\natural\otimes\gamma_1^\natural),\gamma_0^\natural\otimes\gamma_1^\natural\)=0.
                  \]
            \end{itemize}

  \item[(G4)] The image of the restriction of the homomorphism
            \[
            (\gamma_n,\gamma_{n+1},\epsilon_p)\colon\Gamma_{F^+}\to
            \sG_n(\dF_q)\times\sG_{n+1}(\dF_q)\times\dF_q^\times
            \]
            to $\Gal(\ol{F}/F^+_{\r{rflx}})$ (see \cite{LTXZZ}*{Definition~3.3.2} for $F^+_{\r{rflx}}$) contains an element $(g_n,g_{n+1},s)$ satisfying
            \begin{itemize}
              \item $\prod_{i=1}^{n_0}(1-(-s)^i)$ is invertible in $\dF_q$;

              \item for $N=n,n+1$, $g_N$ belongs to $(\GL_N(\dF_q)\times\dF_q^\times,\fc)$ with order coprime to $p$;

              \item if we denote by $h_N$ the $\GL_N(\dF_q)$-component of $g_N^2$ for $N=n,n+1$, then the eigenvalues of each of $h_n$, $h_{n+1}$, and $h_n\otimes h_{n+1}$ contain $1$ exactly once; $h_{n_0}$ does not have eigenvalue $-1$; and $h_{n_1}$ does not have eigenvalue $-s$.
            \end{itemize}

  \item[(G5)] For $N=n,n+1$ and every $v$ in $\spadesuit^+$, every lifting of $(\gamma_N)_v$ is minimally ramified \cite{LTXZZ}*{Definition~3.4.8}.

  \item[(G6)] The restriction $\gamma_{n_0}^\natural\res_{\Gal(\ol{F}/F(\zeta_p))}$ is absolutely irreducible.

  \item[(G7)] For every $w\in\PP$, $(\gamma_{n_0}^\natural)_w$ is ordinary and crystalline with regular Fontaine--Laffaille weights $\{0,1,\dots,n_0-1\}$. In particular, this implicitly requires $p>n_0$ if $\PP\neq\emptyset$, and implies that there is a (necessarily) unique increasing $\Gamma_F$-stable filtration
      \[
      0=\Fil^{-1}(\gamma_{n_0}^\natural)_w\subseteq\Fil^0(\gamma_{n_0}^\natural)_w\subseteq
      \cdots\subseteq\Fil^{n_0-1}(\gamma_{n_0}^\natural)_w=(\gamma_{n_0}^\natural)_w
      \]
      such that $\Gr^i(\gamma_{n_0}^\natural)_w\coloneqq\Fil^i(\gamma_{n_0}^\natural)_w/\Fil^{i-1}(\gamma_{n_0}^\natural)_w$ is one-dimensional on which $\rI_{F_w}$ acts by $\epsilon_p^{-i}$ for $0\leq i\leq n_0-1$.

  \item[(G8)] Under (G7), the induced (unramified) characters of $\Gamma_{F_w}$ on the set
      \[
      \left\{\left.(\Gr^i(\gamma_{n_0}^\natural)_w)(i)\right| 0\leq i\leq n_0-1\right\}
      \]
      are distinct.
\end{description}
\end{assumption}

\begin{remark}\label{re:generic}
Assumption \ref{as:galois}(G2) implies that for $N=n,n+1$, every standard indefinite hermitian space $\rV_N$ of rank $N$ over $F$ with $\sfG_N\coloneqq\rU(\rV_N)$, every neat open compact subgroup $\rL$ of $\sfG_N(\dA_{F^+}^\infty)$ that is hyperspecial away from a finite set $\Box$ of places of $F^+$ containing $\ang{\emptyset}$, and every local system $\dL$ given by a finitely generated $\dZ_q$-module on which $\rL$ acts, we have
\[
\rH^i_{\et}(\Sh(\sfG_N,\rL)_{\ol{F}},\dL)_{\ker\phi_{\gamma_N}\cap\dT_N^\Box}=0,\qquad i\neq N;
\]
and that $\rH^N_{\et}(\Sh(\sfG_N,\rL)_{\ol{F}},\dL)_{\ker\phi_{\gamma_N}\cap\dT_N^\Box}$ is a finite free $\dZ_q$-module when $\dL$ is.
\end{remark}

For every open compact subgroup $\rk$ of $\sfG(F^+_{\PP^+})$, put
\[
\rH(\Sh(\sfG,\rK\rk),\dZ_\xi/p^m)_\gamma\coloneqq\rH(\Sh(\sfG,\rK\rk),\dZ_\xi/p^m)_{\ker\phi_\gamma\cap\dT^{\ang{\fm}}},
\]
which is naturally a direct summand of $\rH(\Sh(\sfG,\rK\rk),\dZ_\xi/p^m)$, and similarly for $\rH(\Sh(\sfG,\rK\rI),\dZ_\xi/p^m)_\gamma$. For every subgroup $\rJ<\rI^0$, put
\[
\cD_\rJ(\xi,\bV)_\gamma\coloneqq\varprojlim_{\rJ<\rI\finite\rI^0}
\Hom_{\dZ_q}\(\rH(\Sh(\sfG,\rK\rI),\dZ_\xi)_\gamma,\dZ_q\)
\]
with respect to the dual of $\rho_{\rI',\rI}$, which is naturally a $\dZ[\rK_{\spadesuit^+}\backslash\sfG(F^+_{\spadesuit^+})/\rK_{\spadesuit^+}]
\otimes\dT^{\ang{\fm}}\otimes\Lambda_{\sfT(F^+_{\PP^+})/\rJ,\dZ_q}$-linear direct summand of $\cD_\rJ(\xi,\bV)$. We define similarly the $\dZ[\rK_{\spadesuit^+}\backslash\sfG(F^+_{\spadesuit^+})/\rK_{\spadesuit^+}]
\otimes\dT^{\ang{\fm}}\otimes\Lambda_{\sfT(F^+_{\PP^+})/\rJ,\dZ_q}$-ring $\cE_\rJ(\xi,\bV)_\gamma$ as the $\Lambda_{\sfT(F^+_{\PP^+})/\rJ,\dZ_q}$-subalgebra of
\[
\End_{\Lambda_{\sfT(F^+_{\PP^+})/\rJ,\dZ_q}}\(\cD_\rJ(\xi,\bV)_\gamma\)
\]
generated by the image of $\dT^{\ang{\fm}}$. Finally, put
\begin{align*}
\sD_\rJ(\xi,\bV)_\gamma&\coloneqq\cD_\rJ(\xi,\bV)_\gamma\otimes_{\dZ_q}\dQ_q,\\
\sE_\rJ(\xi,\bV)_\gamma&\coloneqq\cE_\rJ(\xi,\bV)_\gamma\otimes_{\dZ_q}\dQ_q,
\end{align*}
as a module and a ring over $\dZ[\rK_{\spadesuit^+}\backslash\sfG(F^+_{\spadesuit^+})/\rK_{\spadesuit^+}]
\otimes\dT^{\ang{\fm}}\otimes\Lambda_{\sfT(F^+_{\PP^+})/\rJ,\dQ_q}$, respectively. We suppress $\rJ$ in all the subscripts when it is the trivial subgroup.

\begin{definition}\label{de:classical2}
We define a \emph{classical point} of $\Spec\Lambda_{\rI^0/\rJ,\dQ_q}$ to be a closed point whose natural image in $\Spec\Lambda_{\rI_N^0,\dQ_q}$ is classical (Definition \ref{de:classical}) for both $N=n,n+1$. We define a \emph{classical point} of $\Spec\sE_\rJ(\xi,\bV)_\gamma$ to be a closed point whose image in $\Spec\Lambda_{\sfT(F^+_{\PP^+})/\rJ,\dQ_q}$ is. A classical point $x$ of $\Spec\sE_\rJ(\xi,\bV)_\gamma$ (with the residue field $\dQ_x$) gives rise to a pair of $\Sigma_p^+\setminus\PP^+$-trivial dominant hermitian weights $\zeta_x=(\zeta^n_x,\zeta^{n+1}_x)\in\Xi^n\times\Xi^{n+1}$ and characters
\[
\chi_x\colon\rI^0/\rJ\to\dZ_x^\times,\qquad
\chi_x^\smooth\colon\rI^0/\rJ\to\dZ_x^\times,\qquad
\phi_x\colon\dT^{\ang{\fm}}\to\dQ_x.
\]

We say that a classical point $x$ of $\Spec\sE_\rJ(\xi,\bV)_\gamma$ is \emph{interlacing} if $\zeta_x$ is interlacing in the sense of Definition \ref{de:weight}.
\end{definition}

Similar to the previous section, for a classical point $x$ of $\Spec\sE_\rJ(\xi,\bV)_\gamma$, we have the $\dQ_q[\rK_{\spadesuit^+}\backslash\sfG(F^+_{\spadesuit^+})/\rK_{\spadesuit^+}]
\otimes\dT^{\ang{\fm}}$-module $\rH^\ordi(\Sh(\sfG,\rK\rI),\dQ_{\xi+\zeta_x})$ for every $\rI\finite\rI^0$; and we put
\[
\rD_\rJ(\xi,\bV)_x\coloneqq\varprojlim_{\rJ<\rI\finite\rI^0}
\Hom\(\rH^\ordi(\Sh(\sfG,\rK\rI),\dQ_{\xi+\zeta_x}),\dQ_x\)_{\chi_x^\smooth,\phi_x}.
\]

\begin{proposition}\label{pr:eigen_2}
Let $\rJ<\rI^0$ be a subgroup. Assume Assumption \ref{as:galois}(G2) when $\bV$ is indefinite.
\begin{enumerate}
  \item For every intermediate group $\rJ<\rJ'<\rI^0$, the natural map
      \[
      \cD_\rJ(\xi,\bV)_\gamma\otimes_{\Lambda_{\rI^0/\rJ,\dZ_q}}\Lambda_{\rI^0/\rJ',\dZ_q}
      \to\cD_{\rJ'}(\xi,\bV)_\gamma
      \]
      is an isomorphism.

  \item The $\Lambda_{\rI^0/\rJ,\dZ_q}$-module $\cD_\rJ(\xi,\bV)_\gamma$ is locally finite free.

  \item The natural map
      \[
      \cE_\rJ(\xi,\bV)_\gamma\to\varprojlim_{\rJ<\rI\finite\rI^0}\cE_\rI(\xi,\bV)_\gamma
      \]
      is an isomorphism.

  \item The ring $\cE_\rJ(\xi,\bV)_\gamma$ is reduced, local, and of pure dimension $1+\rank_{\dZ_p}\rI^0/\rJ$.

  \item The natural morphism $\Spec\cE_\rJ(\xi,\bV)_\gamma\to\Spec\Lambda_{\rI^0/\rJ,\dZ_q}$ is finite and generically flat.

  \item For every classical point $x$ of $\Spec\sE_\rJ(\xi,\bV)_\gamma$, we have a ($\Gamma_F$-equivariant when $\bV$ is indefinite) isomorphism
      \[
      \sD_\rJ(\xi,\bV)_\gamma\res_x\simeq\rD_\rJ(\xi,\bV)_x
      \]
      of $\dQ_x[\rK_{\spadesuit^+}\backslash\sfG(F^+_{\spadesuit^+})/\rK_{\spadesuit^+}]$-modules; and moreover, the ring $\cO_{\sE_\rJ(\xi,\bV)_\gamma,x}$ is regular if and only if it is flat over $\Lambda_{\rI^0/\rJ,\dQ_q}$.

  \item The $\sE_\rJ(\xi,\bV)_\gamma$-module $\sD_\rJ(\xi,\bV)_\gamma$ is finitely generated, and locally torsion free over the normal locus of $\Spec\sE_\rJ(\xi,\bV)_\gamma$.
\end{enumerate}
\end{proposition}

\begin{proof}
Except the locality property in (4), the proofs of the rest are exactly same as those of Proposition \ref{pr:eigen_1}, using Assumption \ref{as:galois}(G2) (via Remark \ref{re:generic}) when $\bV$ is indefinite. For the locality, we may assume that $\rI^0/\rJ$ is finite by (3); then there is only one maximal ideal of $\cE_\rJ(\xi,\bV)_\gamma$, namely, the one generated by the image of $\Ker\phi_\gamma$ under the natural homomorphism $\dT^{\ang{\fm}}\to\cE_\rJ(\xi,\bV)_\gamma$.
\end{proof}

\begin{remark}\label{re:flat}
Write $\sY_\rJ(\xi,\bV)_\gamma$ for the locus of $\Spec\sE_\rJ(\xi,\bV)_\gamma$ that is not flat over $\Spec\Lambda_{\rI^0/\rJ,\dQ_q}$, which is a Zariski closed subset. Then Proposition \ref{pr:eigen_2}(2,5) implies that $\sY_\rJ(\xi,\bV)_\gamma$ is of codimension at least two. In particular, by Proposition \ref{pr:eigen_2}(6), classical points are nowhere dense in the complement of the normal locus of $\Spec\sE_\rJ(\xi,\bV)_\gamma$.

In fact, We conjecture that under Assumption \ref{as:galois}(G1), and when $\bV$ is indefinite (G2), $\sE_\rJ(\xi,\bV)_\gamma$ is flat over $\Lambda_{\rI^0/\rJ,\dQ_q}$ at all classical points. In particular, by Proposition \ref{pr:eigen_2}, the normal locus of $\Spec\sE_\rJ(\xi,\bV)_\gamma$ contains all classical points.
\end{remark}

For $N=n,n+1$, we have similarly the $\dZ[(\rK_N)_{\spadesuit^+}\backslash\sfG_N(F^+_{\spadesuit^+})/(\rK_N)_{\spadesuit^+}]
\otimes\dT_N^{\ang{\fm}}\otimes\Lambda_{\sfT_N(F^+_{\PP^+}),\dZ_q}$-module of localized ordinary distributions $\cD(\xi^N,\rV_N,\rK_N)_{\gamma_N}$ (together with $\cE(\xi^N,\rV_N,\rK_N)_{\gamma_N}$). Projection maps induce continuous homomorphisms $\Lambda_{\sfT(F^+_{\PP^+}),\dZ_q}\to\Lambda_{\sfT_N(F^+_{\PP^+}),\dZ_q}$ for $N=n,n+1$. By construction, the exterior cup product induces an isomorphism
\[
\cD(\xi^n,\rV_n,\rK_n)_{\gamma_n}\otimes_{\Lambda_{\sfT(F^+_{\PP^+}),\dZ_q}}\cD(\xi^{n+1},\rV_{n+1},\rK_{n+1})_{\gamma_{n+1}}
\simeq\cD(\xi,\bV)_\gamma
\]
of $\dZ[\rK_{\spadesuit^+}\backslash\sfG(F^+_{\spadesuit^+})/\rK_{\spadesuit^+}]
\otimes\dT^{\ang{\fm}}\otimes\Lambda_{\sfT(F^+_{\PP^+}),\dZ_q}$-modules, which further induces an isomorphism
\[
\cE(\xi^n,\rV_n,\rK_n)_{\gamma_n}\otimes_{\Lambda_{\sfT(F^+_{\PP^+}),\dZ_q}}\cE(\xi^{n+1},\rV_{n+1},\rK_{n+1})_{\gamma_{n+1}}
\simeq\cE(\xi,\bV)_\gamma
\]
of $\dZ[\rK_{\spadesuit^+}\backslash\sfG(F^+_{\spadesuit^+})/\rK_{\spadesuit^+}]
\otimes\dT^{\ang{\fm}}\otimes\Lambda_{\sfT(F^+_{\PP^+}),\dZ_q}$-rings.

\begin{remark}\label{re:galois}
Assume Assumption \ref{as:galois}(G1), and when $\bV$ is indefinite (G2). Take a rank $N\in\{n,n+1\}$. By the theory of pseudo-characters in \cite{BC09}, we have a unique up to isomorphism free $\cE(\xi^N,\rV_N,\rK_N)_{\gamma_N}$-module $\cR(\xi^N,\rV_N,\rK_N)_{\gamma_N}$ of rank $N$ with a continuous action of $\Gamma_F$ lifting $(\gamma_N^\natural)^\tc$ such that for every closed point $x$ of $\sE(\xi^N,\rV_N,\rK_N)_{\gamma_N}$, the $\dQ_x[\Gamma_F]$-module $\cR(\xi^N,\rV_N,\rK_N)_{\gamma_N}\res_x$ is the (residually irreducible) Galois representation associated with $\Pi_x^\vee$ in the sense of \cite{LTXZZ1}*{Proposition~2.9(1)}, where $\Pi_x$ denotes the cuspidal automorphic representation of $\GL_N(\dA_F)$ whose Satake parameter is given by the base change of $\phi_x$.

By taking limits over $\rI\finite\rI_N^0$, we obtain for every $w\in\PP$ a unique $\Gamma_F$-stable complete free filtration
\[
0=\Fil_w^{-1}\cR(\xi^N,\rV_N,\rK_N)_{\gamma_N}\subseteq\Fil_w^0\cR(\xi^N,\rV_N,\rK_N)_{\gamma_N}
\subseteq\cdots\subseteq\Fil^{N-1}_w\cR(\xi^N,\rV_N,\rK_N)_{\gamma_N}=\cR(\xi^N,\rV_N,\rK_N)_{\gamma_N}
\]
of $\cE(\xi^N,\rV_N,\rK_N)_{\gamma_N}$-modules such that for $0\leq i\leq N-1$,
\[
\Gr^i_w\cR(\xi^N,\rV_N,\rK_N)_{\gamma_N}\coloneqq
\frac{\Fil_w^i\cR(\xi^N,\rV_N,\rK_N)_{\gamma_N}}{\Fil_w^{i-1}\cR(\xi^N,\rV_N,\rK_N)_{\gamma_N}}
\]
is a free $\cE(\xi^N,\rV_N,\rK_N)_{\gamma_N}$-module of rank $1$ on which $\Gamma_{F_w}$ acts by the character $\psi_{N,w,i}^{-1}\cdot\epsilon_p^{-i}$. Here, for every $w\in\PP$ and every $0\leq i\leq N-1$, $\psi_{N,w,i}$ is the character
\[
\Gamma_{F_w}\xrightarrow{\rec_w}\widehat{F_w^\times}\xrightarrow{\inc_i}\widehat{\sfT_N(F_w)}
\hookrightarrow\widehat{\sfT_N(F_\PP)}\xrightarrow{\Nm_{F/F^+}}\widehat{\sfT_N(F^+_{\PP^+})}
\]
in which $\inc_i$ is the inclusion of the factor in \eqref{eq:weight} indexed by $i$ (with $\sigma$ the natural embedding $F^+_v\to F_w$).
\end{remark}

Put
\[
\cR_\rJ(\xi,\bV)_\gamma\coloneqq
\(\(\cR(\xi^n,\rV_n,\rK_n)_{\gamma_n}\boxtimes\cR(\xi^{n+1},\rV_{n+1},\rK_{n+1})_{\gamma_{n+1}}\)
\otimes_{\cE(\xi,\bV)_\gamma}\cE_\rJ(\xi,\bV)_\gamma\)(n),
\]
equipped with an increasing free filtration
\[
\Fil_w^i\cR_\rJ(\xi,\bV)_\gamma=
\(\sum_{j+j'=i+n}\Fil^j_w\cR(\xi^n,\rV_n,\rK_n)_{\gamma_n}\boxtimes\Fil^{j'}_w\cR(\xi^{n+1},\rV_{n+1},\rK_{n+1})_{\gamma_{n+1}}\)
\otimes_{\cE(\xi,\bV)_\gamma}\cE_\rJ(\xi,\bV)_\gamma
\]
of $\cE(\xi^N,\rV_N,\rK_N)_{\gamma_N}$-modules for every $w\in\PP$. It follows that
\[
\Gr^i_w\cR_\rJ(\xi,\bV)_\gamma\coloneqq
\frac{\Fil^i_w\cR_\rJ(\xi,\bV)_\gamma}{\Fil^{i-1}_w\cR_\rJ(\xi,\bV)_\gamma}
\]
is a free $\cE_\rJ(\xi,\bV)_\gamma$-module of rank $\max\{n+1+i,0\}$ (resp.\ $\max\{n-i,0\}$) for $i<0$ (resp.\ $i\geq 0$).

Finally, we put
\[
\sR_\rJ(\xi,\bV)_\gamma\coloneqq\cR_\rJ(\xi,\bV)_\gamma\otimes_{\dZ_q}\dQ_q,
\]
equipped with the induced filtration for every $w\in\PP$.

\begin{remark}\label{re:pure}
For every classical point $x$ of $\Spec\sE_\rJ(\xi,\bV)_\gamma$, the $\dQ_x[\Gamma_F]$-module $\sR_\rJ(\xi,\bV)_\gamma\res_x$ is pure of weight $-1$ at every prime $w$ of $F$ by \cite{Car14}*{Theorem~1.1} (for $w\mid p$) and \cite{LTXZZ}*{Lemma~8.1.6} (for $w\nmid p$).
\end{remark}

Put
\[
\cD_\rJ^\flat(\xi,\bV)_\gamma\coloneqq
\begin{dcases}
\cD_\rJ(\xi,\bV)_\gamma, &\text{when $\bV$ is definite,}\\
\Hom_{\cE_\rJ(\xi,\bV)_\gamma[\Gamma_F]}
\(\cR_\rJ(\xi,\bV)_\gamma,\cD_\rJ(\xi,\bV)_\gamma\), &\text{when $\bV$ is indefinite,}
\end{dcases}
\]
as a $\dZ[\rK_{\spadesuit^+}\backslash\sfG(F^+_{\spadesuit^+})/\rK_{\spadesuit^+}]
\otimes\cE_\rJ(\xi,\bV)_\gamma$-module; and put
\[
\sD_\rJ^\flat(\xi,\bV)_\gamma\coloneqq\cD_\rJ^\flat(\xi,\bV)_\gamma\otimes_{\dZ_q}\dQ_q.
\]

\begin{lem}\label{le:flat}
Suppose that $\bV$ is indefinite and assume Assumption \ref{as:galois}(G1,G2). The tautological map
\[
\cD_\rJ^\flat(\xi,\bV)_\gamma\otimes_{\cE_\rJ(\xi,\bV)_\gamma}\cR_\rJ(\xi,\bV)_\gamma
\to\cD_\rJ(\xi,\bV)_\gamma
\]
is an isomorphism of $\dZ[\rK_{\spadesuit^+}\backslash\sfG(F^+_{\spadesuit^+})/\rK_{\spadesuit^+}]
\otimes\cE_\rJ(\xi,\bV)_\gamma[\Gamma_F]$-modules.
\end{lem}

\begin{proof}
By Proposition \ref{pr:eigen_2}(3) and the fact that $\cR_\rJ(\xi,\bV)_\gamma$ is a finite free $\cE_\rJ(\xi,\bV)_\gamma$-module, we have
\[
\cD_\rJ^\flat(\xi,\bV)_\gamma\otimes_{\cE_\rJ(\xi,\bV)_\gamma}\cR_\rJ(\xi,\bV)_\gamma
\to\cD_\rJ(\xi,\bV)_\gamma
=\varprojlim_{\rJ<\rI\finite\rI^0}
\(\cD_\rI^\flat(\xi,\bV)_\gamma\otimes_{\cE_\rJ(\xi,\bV)_\gamma}\cR_\rI(\xi,\bV)_\gamma
\to\cD_\rI(\xi,\bV)_\gamma\).
\]
Thus, we may assume that $\rJ$ is a subgroup of $\rI^0$ of finite index. Then the isomorphism follows by (the same argument for) \cite{Sch18}*{Theorem~5.6}, using (G1) and also Remark \ref{re:conditional}.
\end{proof}

In the situation of Lemma \ref{le:flat}, we have an induced filtration
\[
\Fil^\bullet_w\cD_\rJ(\xi,\bV)_\gamma\coloneqq
\cD_\rJ^\flat(\xi,\bV)_\gamma\otimes_{\cE_\rJ(\xi,\bV)_\gamma}\Fil^\bullet_w\cR_\rJ(\xi,\bV)_\gamma
\]
for every $w\in\PP$.

\begin{lem}\label{le:ggp}
Assume Assumption \ref{as:galois}(G1), and when $\bV$ is indefinite (G2). For every classical point $x$ of $\Spec\sE_\rJ(\xi,\bV)_\gamma$, $(\sD_\rJ^\flat(\xi,\bV)_\gamma\res_x)^\sfH$ is a $\dQ_x$-vectors space of dimension at most one.
\end{lem}

\begin{proof}
Take a classical point $x$ of $\Spec\sE_\rJ(\xi,\bV)_\gamma$ with residue field $\dQ_x$ and fix an isomorphism $\ol{\dQ_x}\simeq\dC$. Proposition \ref{pr:eigen_2}(6) implies that for $N=n,n+1$, the induced character $\phi_{x,N}\colon\dT^{\ang{\fm}}_N\to\dC$ comes from a hermitian cuspidal automorphic representation $\Pi_{x,N}$ of $\GL_N(\dA_F)$, unique up to isomorphism by the strong multiplicity one. Moreover, by the same proposition, we have a natural embedding
\[
(\sD_\rJ^\flat(\xi,\bV)_\gamma\res_x)^\sfH\otimes_{\dQ_x}\dC\hookrightarrow
\prod_{v\in\spadesuit^+}\(\bigoplus_{\fS_{n,v}\times\fS_{n+1,v}}(\pi_{n,v}\boxtimes\pi_{n+1,v})^{\sfH(F^+_v)}\)
\times\prod_{v\in\PP^+}\(\bigoplus_{\fS_{n,v}\times\fS_{n+1,v}}(\pi_{n,v}^\ordi\boxtimes\pi_{n+1,v}^\ordi)\),
\]
where for $N=n,n+1$, $\fS_{N,v}$ is the set isomorphism classes of tempered irreducible admissible representations $\pi_{N,v}$ of $\sfG_N(F^+_v)$ whose base change is $\Pi_{x,N,v}$; and $\pi_{N,v}^\ordi$ denotes the space of ordinary vectors of $\pi_{N,v}$ with respect to the Borel subgroup $\sfB_{N,v}$, which has dimension at most one.

By the solution of the local Gan--Gross--Prasad conjecture \cite{BP16}, we have
\[
\dim_\dC\bigoplus_{\fS_{n,v}\times\fS_{n+1,v}}(\pi_{n,v}\boxtimes\pi_{n+1,v})^{\sfH(F^+_v)}\leq 1
\]
for every $v\in\spadesuit^+$. Thus, it suffices to show that for every $v\in\PP^+$, $\fS_{N,v}$ is a singleton for $N=n,n+1$. Take an element $v\in\PP^+$. When $v$ splits in $F$, the statement is trivial. We assume $v$ inert in $F$ and suppress it in the following argument (so that $F/F^+$ is a quadratic extension of finite unramified extensions of $\dQ_p$). We prove the case for $N$ even and leave the other similar case to the readers.

Since $\Pi_N$ is tempered and ordinary, it can be embedded into a principal series hence has the form
\[
\Pi_N=\chi_1\St_{N_1}\boxplus\cdots\boxplus\chi_s\St_{N_s},
\]
in which $\chi_i$ is a unitary smooth character of $F^\times$ and $\St_{N_i}$ is the Steinberg representation of $\GL_{N_i}(F)$ (so that $N=N_1+\cdots+N_s$). Since $\Pi_N$ is ordinary, we have
\begin{align}\label{eq:ggp}
&\left\{\left.\val(\chi_i(p))+\tfrac{N_i-1}{2},\val(\chi_i(p))+\tfrac{N_i-3}{2},\dots,\val(\chi_i(p))
+\tfrac{3-N_i}{2},\val(\chi_i(p))+\tfrac{1-N_i}{2}\right| 1\leq i\leq s \right\} \\
&=\left\{\tfrac{N-1}{2},\tfrac{N-3}{2},\dots,\tfrac{3-N}{2},\tfrac{1-N}{2}\right\}, \notag
\end{align}
where $\val$ denotes the $p$-adic valuation on $\ol{\dQ_x}$ normalized so that $\val(p)=[F:\dQ_p]^{-1}$. In particular, the factors in $\Pi_N$ are mutually distinct. Put
\[
I\coloneqq\left\{1\leq i\leq s\left|\;\chi_i\res_{(F^+)^\times}=\eta_{F/F^+}^{N_i}\right.\right\},
\]
where $\eta_{F/F^+}\colon(F^+)^\times\to\dC^\times$ denotes the quadratic character associated with $F/F^+$ via the local class field theory. Then the cardinality of the Vogan packet of $\Pi_N$ equals $2^{|I|}$ (see, for example, \cite{GGP12}*{\S10}). It follows from \eqref{eq:ggp} that $|I|\leq 1$. When $|I|=0$, the Vogan packet of $\Pi_N$ has only one element, which must be a representation of $\sfG_N(F^+_v)$ hence belong to $\fS_v$. When $|I|=1$, the Vogan packet of $\Pi_N$ has two elements, each contributing to each pure inner form of $\rV_N$, so that $\fS_v$ is again a singleton.

The lemma is proved.
\end{proof}

\section{Selmer groups over complete group rings}
\label{ss:selmer}

In this section, we define and study the Selmer groups of Galois representations over eigenvarieties.

We first recall the definition of the Bloch--Kato Selmer groups $\rH^1_f(F,T)$ and $\rH^1_f(F,T^*(1))$ \cite{BK90} for a finitely generated $\dZ_q$-module $T$ with a continuous action by $\Gamma_F$ and its Cartier dual $T^*(1)=\Hom_{\dZ_q}(T,\dQ_q/\dZ_q)(1)$. For every place $w$ of $F$, we have the local (Bloch--Kato) Selmer groups/structures
\[
\rH^1_f(F_w,T)\coloneqq
\begin{dcases}
0, & w\in\Sigma_\infty, \\
\Ker\(\rH^1(F_w,T)\to\rH^1(\rI_{F_w},T\otimes_{\dZ_p}\dQ_p)\), & w\in\Sigma\setminus(\Sigma_\infty\cup\Sigma_p), \\
\Ker\(\rH^1(F_w,T)\to\rH^1(F_w,T\otimes_{\dZ_p}\dB_{\r{cris}})\), & w\in\Sigma_p,
\end{dcases}
\]
and $\rH^1_f(F_w,T^*(1))$ that is the annihilator of $\rH^1_f(F_w,T)$ under the local Tate pairing
\[
\rH^1(F_w,T)\times\rH^1(F_w,T^*(1))\to\dQ_q/\dZ_q.
\]
The global (Bloch--Kato) Selmer group $\rH^1_f(F,T)$ (resp.\ $\rH^1_f(F,T^*(1))$) is defined to be the $\dZ_q$-submodule of $\rH^1(F,T)$ (resp.\ $\rH^1(F,T^*(1))$) of classes whose localization at every place $w$ of $F$ belong to $\rH^1_f(F_w,M)$ (resp.\ $\rH^1_f(F_w,T^*(1))$).

Take a subgroup $\rJ<\rI^0$. We now define the Selmer groups over the Iwasawa algebra $\Lambda_{\rI^0/\rJ,\dZ_q}$.

\begin{definition}\label{de:selmer}
Let $\cT$ be a finitely generated $\Lambda_{\rI^0/\rJ,\dZ_q}$-module equipped a continuous $\Lambda_{\rI^0/\rJ,\dZ_q}$-linear action by $\Gamma_F$, so that its continuous Pontryagin dual $\cT^*$ is a discrete $\Lambda_{\rI^0/\rJ,\dZ_q}$-module with respect to the dual action.
\begin{enumerate}
  \item For every place $w$ of $F$, define the local Selmer groups to be
      \begin{align*}
      \rH^1_f(F_w,\cT)&\coloneqq
      \varprojlim_{\rJ<\rI\finite\rI^0}\rH^1_f\(F_w,\cT\otimes_{\Lambda_{\rI^0/\rJ,\dZ_q}}\dZ_q[\rI^0/\rI]\), \\
      \rH^1_f(F_w,\cT^*(1))&\coloneqq
      \varinjlim_{\rJ<\rI\finite\rI^0}\rH^1_f\(F,(\cT\otimes_{\Lambda_{\rI^0/\rJ,\dZ_q}}\dZ_q[\rI^0/\rI])^*(1)\),
      \end{align*}
      which are naturally $\Lambda_{\rI^0/\rJ,\dZ_q}$-modules. Note that $\rH^1_f(F_w,\cT)$ and $\rH^1_f(F_w,\cT^*(1))$ are mutual annihilators under the local Tate pairing $\rH^1(F_w,\cT)\times\rH^1(F_w,\cT^*(1))\to\dQ_q/\dZ_q$.

  \item The global (Bloch--Kato) Selmer group $\rH^1_f(F,\cT)$ (resp.\ $\rH^1_f(F,\cT^*(1))$) is defined to be the $\Lambda_{\rI^0/\rJ,\dZ_q}$-submodule of $\rH^1(F,\cT)$ (resp.\ $\rH^1(F,\cT^*(1))$) of classes whose localization at every place $w$ of $F$ belong to $\rH^1_f(F_w,M)$ (resp.\ $\rH^1_f(F_w,\cT^*(1))$).

  \item Finally, for $\sT\coloneqq\cT\otimes_{\dZ_q}\dQ_q$, we define $\rH^1_f(F,\sT)\coloneqq\rH^1_f(F,\cT)\otimes_{\dZ_q}\dQ_q$ as a $\Lambda_{\rI^0/\rJ,\dQ_q}$-submodule of $\rH^1(F,\sT)$.
\end{enumerate}
\end{definition}

Take a datum $\bV$ from Definition \ref{de:datum}. We introduce the $\sE_\rJ(\xi,\bV)_\gamma$-module
\[
\cX_\rJ(\xi,\bV)_\gamma\coloneqq\rH^1_f(F,\cR_\rJ(\xi,\bV)_\gamma^*(1))^*
\]
(Definition \ref{de:selmer}) as the continuous Pontryagin dual of $\rH^1_f(F,\cR_\rJ(\xi,\bV)_\gamma^*(1))$. Put
\[
\sX_\rJ(\xi,\bV)_\gamma\coloneqq\cX_\rJ(\xi,\bV)_\gamma\otimes_{\dZ_q}\dQ_q.
\]

\begin{lem}\label{le:selmer1}
Assume Assumption \ref{as:galois}(G1), and when $\bV$ is indefinite (G2). Then
\begin{enumerate}
  \item both $\rH^1_f(F,\cR_\rJ(\xi,\bV)_\gamma)$ and $\cX_\rJ(\xi,\bV)_\gamma$ are finitely generated modules over $\cE_\rJ(\xi,\bV)_\gamma$;

  \item when $\bV$ is indefinite, $\rH^1_f(F,\cD_\rJ(\xi,\bV)_\gamma)$ is a finitely generated module over $\cE_\rJ(\xi,\bV)_\gamma$;

  \item $\rH^1_f(F,\sR_\rJ(\xi,\bV)_\gamma)$ is locally torsion free over the normal locus of $\Spec\sE_\rJ(\xi,\bV)_\gamma$.
\end{enumerate}
\end{lem}

\begin{proof}
Put $\Gamma\coloneqq\Gal(F^\ur/F)$ where $F^\ur/F^+$ is the maximal extension of $F^+$ (contained in $\ol{F}$) that is unramified outside $\ang{\fm}$.

For (1), it suffices to show that both $\rH^1(\Gamma,\cR_\rJ(\xi,\bV)_\gamma)$ and $\rH^1(\Gamma,\cR_\rJ(\xi,\bV)_\gamma^*(1))^*$ are finitely generated modules over $\Lambda_{\rI^0/\rJ,\dZ_q}$. Choose a finite decreasing filtration $\rJ=\rJ_d<\rJ_{d-1}<\dots<\rJ_1<\rJ_0=\rI^0$ of subgroups such that $\rJ_i/\rJ_{i+1}$ is topologically cyclic for $0\leq i\leq d-1$. For $0\leq i\leq d$, put $\Lambda_i\coloneqq\Lambda_{\rI^0/\rJ_i,\dZ_q}$. We show by induction on $i$ the following statement:
\begin{itemize}
  \item[($\ast$)] For every finitely generated $\Lambda_i$-module $\cT$ equipped with a continuous $\Lambda_i$-linear action of $\Gamma$ and every $j\in\dZ$, both $\rH^j(\Gamma,\cT)$ and $\rH^j(\Gamma,\cT^*(1))^*$ are finitely generated $\Lambda_i$-modules.
\end{itemize}
When $i=0$, ($\ast$) is well-known. Assume the statement for $i$ and choose a topological generator $t$ of $\rJ_i/\rJ_{i+1}$. Put $\cT[(t-1)^\infty]\coloneqq\bigcup_{m\geq 1}\cT[(t-1)^m]$. By the short exact sequence $0\to\cT[(t-1)^\infty]\to\cT\to\cT/\cT[(t-1)^\infty]\to 0$, it allows us to assume $\cT$ either $t$-unipotent or $(t-1)$-torsion free in ($\ast$) since $\Lambda_{i+1}$ is noetherian. When $\cT$ is $t$-unipotent, we may reduce the case to $\Lambda_i$-modules, for which ($\ast$) is known by the induction hypothesis. When $\cT$ is $(t-1)$-torsion free, we use the short exact sequence $0\to\cT\xrightarrow{t-1}\cT\to\cT/(t-1)\cT\to 0$ to obtain an injection $\rH^j(\Gamma,\cT)/(t-1)\hookrightarrow\rH^{j+1}(\Gamma,\cT/(t-1)\cT)$ and a surjection $\rH^j(\Gamma,(\cT/(t-1)\cT)^*(1))\twoheadrightarrow\rH^j(\Gamma,\cT^*(1))[t-1]$. By the induction hypotheses and Nakayama's lemma, ($\ast$) holds for $i+1$.

Part (2) has already followed from ($\ast$).

For (3), it suffices to show that $\rH^1(\Gamma,\sR_\rJ(\xi,\bV)_\gamma)$ is locally torsion free over the normal locus of $\Spec\sE_\rJ(\xi,\bV)_\gamma$. By abuse of notation, we may just assume that $\sE_\rJ(\xi,\bV)_\gamma$ is normal and integral. Then it suffices to show that for every nonzero element $f\in\sE_\rJ(\xi,\bV)_\gamma$, the multiplication by $f$ on $\rH^1(\Gamma,\sR_\rJ(\xi,\bV)_\gamma)$ is injective. However, this follows from the fact that $\rH^0(\Gamma,\sR_\rJ(\xi,\bV)_\gamma/(f))$ vanishes due to Assumption \ref{as:galois}(G1).
\end{proof}

Next, we provide an alternative description of the local Selmer groups of $\sR_\rJ(\xi,\bV)_\gamma$ and $\cD_\rJ(\xi,\bV)_\gamma$ at places in $\PP$.

\begin{lem}\label{le:selmer2}
Assume $p>n$. Assume Assumption \ref{as:galois}(G1,G5), and when $\bV$ is indefinite (G2). Let $w$ be a place in $\PP$.
\begin{enumerate}
  \item When $\bV$ is indefinite, we have
      \[
      \rH^1_f(F_w,\cD_\rJ(\xi,\bV)_\gamma)=
      \Ker\(\rH^1(F_w,\cT)\to\rH^1(\rI_{F_w},\cD_\rJ(\xi,\bV)_\gamma/\Fil_w^{-1}\cD_\rJ(\xi,\bV)_\gamma)\).
      \]

  \item The identity
      \[
      \rH^1_f(F_w,\sR_\rJ(\xi,\bV)_\gamma)=
      \Ker\(\rH^1(F_w,\sR_\rJ(\xi,\bV)_\gamma)\to\rH^1(\rI_{F_w},\sR_\rJ(\xi,\bV)_\gamma/\Fil_w^{-1}\sR_\rJ(\xi,\bV)_\gamma)\)
      \]
      of $\sE_\rJ(\xi,\bV)_\gamma$-modules holds away from $\sY_\rJ(\xi,\bV)_\gamma$ (see Remark \ref{re:flat}).
\end{enumerate}
\end{lem}

Recall that we have the filtration $\Fil^\bullet_w\cD_\rJ(\xi,\bV)_\gamma$ and $\Fil^\bullet_w\cR_\rJ(\xi,\bV)_\gamma$ for every $w\in\PP$ in view of Lemma \ref{le:flat}.

\begin{proof}
By Assumption \ref{as:galois}(G5), we may assume $\rI^0/\rJ$ torsion free without loss of generality.

Write $\cT$ for $\cD_\rJ(\xi,\bV)_\gamma$ or $\cR_\rJ(\xi,\bV)_\gamma$, and $-_\natural$ for $-$ or $-\otimes_{\dZ_q}\dQ_q$, according to whether we are in part (1) or (2). For every $\cE_\rI(\xi,\bV)_\gamma$-linear direct summand $\cT'$ of $\cT$, we denote by $\cT'_{\rJ'}$ the image of $\cT'$ in $\cR_{\rJ'}(\xi,\bV)_\gamma$ or $\cD_{\rJ'}(\xi,\bV)_\gamma$ for every intermediate group $\rJ<\rJ'<\rI^0$, and put
\[
\widetilde{\cT'}\coloneqq\varprojlim_{\rJ<\rI\finite\rI^0}\(\cT'_\rI\otimes_{\dZ_q}\dQ_q\).
\]
By definition in part (1) and Proposition \ref{pr:eigen_2}(3) in part (2), the natural map
\[
\cT'\to\varprojlim_{\rJ<\rI\finite\rI^0}\cT'_\rI
\]
is an isomorphism, hence the map $\cT'_\natural\to\widetilde{\cT'}$ is injective.

Write the right-hand side of the both statements by $\rH^1_f(F_w,\cT_\natural)'$. Then by Definition \ref{de:selmer} and Proposition \ref{pr:eigen_2}(2), it suffices to show that
\begin{align}\label{eq:selmer}
\rH^1_f(F_w,\cT_\natural)'=\Ker\(\rH^1(F_w,\cT_\natural)\to
\varprojlim_{\rJ<\rI\finite\rI^0}\frac{\rH^1(F_w,\cT_{\rI,\natural})}{\rH^1_f(F_w,\cT_{\rI,\natural})}\)
\end{align}
holds in part (1) and holds away from $\sY_\rJ(\xi,\bV)_\gamma$ in part (2). In what follows, we say that a map $\cM\to\cN$ of $\cE_\rI(\xi,\bV)_\gamma$-modules is almost injective/surjective in part (1) if it is just injective/surjective, and a map $\sM\to\sN$ of $\sE_\rI(\xi,\bV)_\gamma$-modules is almost injective/surjective in part (2) if it is injective/surjective away from $\sY_\rJ(\xi,\bV)_\gamma$.

Since $\cT_\rI\otimes_{\dZ_q}\dQ_q$ is ordinary at $w$, we have
\[
\rH^1_f(F_w,\cT_{\rI,\natural})=\Ker\(\rH^1(F_w,\cT_{\rI,\natural})
\to\rH^1(\rI_{F_w},(\cT_\rI/\Fil^{-1}_w\cT_\rI)\otimes_{\dZ_q}\dQ_q)\)
\]
by Lemma \ref{le:panchishkin} below. Thus, for \eqref{eq:selmer}, it suffices to show that the natural map
\[
\rH^1(\rI_{F_w},(\cT/\Fil_w^{-1}\cT)_\natural)
\to\varprojlim_{\rJ<\rI<\rI^0}\rH^1(\rI_{F_w},(\cT_\rI/\Fil^{-1}_w\cT_\rI)\otimes_{\dZ_q}\dQ_q)
=\rH^1(\rI_{F_w},\widetilde\cT/\Fil_w^{-1}\widetilde\cT)
\]
is almost injective, which is equivalent to the almost surjectivity of
\[
\rH^0(\rI_{F_w},\widetilde\cT/\Fil_w^{-1}\widetilde\cT)
\to\rH^0(\rI_{F_w},(\widetilde\cT/\Fil_w^{-1}\widetilde\cT)/(\cT/\Fil_w^{-1}\cT)_\natural).
\]
The assumption that $p>n$ and $\rI^0/\rJ$ is torsion free implies that the natural maps
\begin{align*}
\rH^0(\rI_{F_w},\Gr^0_w\widetilde\cT)&\to\rH^0(\rI_{F_w},\widetilde\cT/\Fil_w^{-1}\widetilde\cT),\\
\rH^0(\rI_{F_w},\Gr^0_w\widetilde\cT/\Gr^0_w\cT_\natural)
&\to\rH^0(\rI_{F_w},(\widetilde\cT/\Fil_w^{-1}\widetilde\cT)/(\cT/\Fil_w^{-1}\cT)_\natural)
\end{align*}
are both isomorphisms. Thus, it suffices to show that the natural map
\begin{align}\label{eq:selmer1}
\rH^0(\rI_{F_w},\Gr^0_w\widetilde\cT)\to\rH^0(\rI_{F_w},\Gr^0_w\widetilde\cT/\Gr^0_w\cT_\natural)
\end{align}
is almost surjective.

By Remark \ref{re:galois}, there is a unique decomposition
\[
\Gr^0_w\cT=\bigoplus_{j=0}^n\cT^{(j)}
\]
of $\cE_\rI(\xi,\bV)_\gamma[\Gamma_{F_w}]$-modules such that $\Gamma_{F_w}$ acts on $\cT^{(j)}$ via the homomorphism
\[
(\psi_{n,w,j}^{-1},\psi_{n+1,w,n-j}^{-1})\colon\Gamma_{F_w}\to
\widehat{\sfT_n(F^+_{\PP^+})}\times\widehat{\sfT_{n+1}(F^+_{\PP^+})}.
\]
Denote by $\rJ^{(j)}$ the subgroup of $\rI^0$ generated by the image of $(\psi_{n,w,j}^{-1},\psi_{n+1,w,n-j}^{-1})\res_{\rI_{F_w}}$ (which is contained in $\rI^0$) and $\rJ$. It follows that \eqref{eq:selmer1} can be written as
\[
\bigoplus_{j=0}^n\rH^0(\rJ^{(j)}/\rJ,\widetilde{\cT^{(j)}})
\to\bigoplus_{j=0}^n\rH^0(\rJ^{(j)}/\rJ,\widetilde{\cT^{(j)}}/\cT^{(j)}_\natural).
\]
Therefore, the almost surjectivity of \eqref{eq:selmer1} will be implied by the following more general statement: For every $\cE_\rI(\xi,\bV)_\gamma$-linear direct summand $\cT'$ of $\cT$ and every intermediate group $\rJ<\rJ'<\rI^0$, the map
\[
\rH^0(\rJ'/\rJ,\widetilde{\cT'})\to\rH^0(\rJ'/\rJ,\widetilde{\cT'}/\cT'_\natural)
\]
is almost surjective, or equivalently, the map
\[
\rH^1(\rJ'/\rJ,\cT'_\natural)\to\rH^1(\rJ'/\rJ,\widetilde{\cT'})
\]
is almost injective.

Fix an increasing filtration $\rJ=\rJ_0<\rJ_1<\cdots<\rJ_d=\rJ'$ such that $\rJ_{i+1}/\rJ_i$ is isomorphic to $\dZ_p$. We prove by induction on $i$ that the map
\[
\mu_i\colon\rH^1(\rJ_i/\rJ,\cT'_\natural)\to\rH^1(\rJ_i/\rJ,\widetilde{\cT'})
\]
is almost injective. The case for $i=0$ is trivial. Suppose that $\mu_i$ is almost injective. We have the diagram
\[
\xymatrix{
0 \ar[d] & 0 \ar[d] \\
\rH^1(\rJ_{i+1}/\rJ_i,(\cT^\prime_\natural)^{\rJ_i})
\ar[r]\ar[d] &  \rH^1(\rJ_{i+1}/\rJ_i,\widetilde{\cT'}^{\rJ_i}) \ar[d] \\
\rH^1(\rJ_{i+1}/\rJ,\cT'_\natural)
\ar[r]^-{\mu_{i+1}}\ar[d] & \rH^1(\rJ_{i+1}/\rJ,\widetilde{\cT'}) \ar[d] \\
\rH^1(\rJ_i/\rJ,\cT'_\natural)^{\rJ_{i+1}/\rJ_i}
\ar[r]\ar[d] & \rH^1(\rJ_i/\rJ,\widetilde{\cT'})^{\rJ_{i+1}/\rJ_i} \ar[d] \\
0 & 0
}
\]
in which the bottom map is almost injective due to the induction hypothesis. Thus, it remains to show that the top map is almost injective. By Proposition \ref{pr:eigen_2}(2), $\cT$ hence $\cT'$ are torsion free modules over (the regular local ring) $\Lambda_{\rI^0/\rJ,\dZ_q}$, which implies that $(\cT^\prime_\natural)^{\rJ_i}=0$ unless $i=0$. When $i=0$, the top map is nothing but
\[
\mu_1\colon\rH^1(\rJ_1/\rJ,\cT'_\natural)\to\rH^1(\rJ_1/\rJ,\widetilde{\cT'}).
\]
Since $\rJ_1/\rJ\simeq\dZ_p$, the identify
\[
\rH^1(\rJ_1/\rJ,\cT'_\natural)=
\cT'_\natural\otimes_{\Lambda_{\rI^0/\rJ,\dZ_q}}\Lambda_{\rI^0/\rJ_1,\dZ_q}=(\cT'_{\rJ_1})_\natural
\]
holds in part (1) by Proposition \ref{pr:eigen_2}(1) and holds away from $\sY_\rJ(\xi,\bV)_\gamma$ in part (2); the identity
\[
\rH^1(\rJ_1/\rJ,\widetilde{\cT'})=\varprojlim_{\rJ<\rI\finite\rI^0}\rH^1(\rJ_1/\rJ,\cT'_{\rI}\otimes_{\dZ_q}\dQ_q)
=\varprojlim_{\rJ<\rI\finite\rI^0}(\cT'_{\rI}\otimes_{\dZ_q}\dQ_q)
\otimes_{\Lambda_{\rI^0/\rJ,\dZ_q}}\Lambda_{\rI^0/\rJ_1,\dZ_q}
=\varprojlim_{\rJ_1<\rI_1\finite\rI^0}\(\cT'_{\rI_1}\otimes_{\dZ_q}\dQ_q\)=\widetilde{\cT'_{\rJ_1}},
\]
clear holds; and that $\mu_1$ is the natural map $(\cT'_{\rJ_1})_\natural\to\widetilde{\cT'_{\rJ_1}}$. It follows that $\mu_1$ is almost injective.

The lemma is proved.
\end{proof}

\begin{lem}\label{le:panchishkin}
Let $w$ be a place of $F$ above $p$. Consider a finite dimensional ordinary de Rham representation $V$ of $\Gamma_{F_w}$ (with coefficients in $\dQ_q$) pure of weight $-1$. Let $\Fil^{-1}V$ be the maximal subrepresentation of $V$ of Hodge--Tate weights in $(-\infty,-1]$. Then
\[
\rH^1_f(F_w,V)=\Ker\(\rH^1(F_w,V)\to\rH^1(\rI_{F_w},V/\Fil^{-1}V)\)
\]
\end{lem}

\begin{proof}
Write $F$ for $F_w$ for short. Since the natural map $\rH^1(F,V/\Fil^{-1}V)\to\rH^1(\rI_F,V/\Fil^{-1}V)$ is injective, it suffices to show that
\begin{align}\label{eq:panchishkin}
\rH^1_f(F,V)=\Ker\(\rH^1(F,V)\to\rH^1(F,V/\Fil^{-1}V)\).
\end{align}
Take a finite Galois extension $F'/F$ such that $V$ as a representation of $\Gamma_{F'}$ is semistable. Since $\rH^1_?(F,V)=\rH^1_?(F',V)^{\Gal(F'/F)}$ for $?\in\{f,\;\}$ and the restriction map $\rH^1(F,V/\Fil^{-1}V)\to\rH^1(F',V/\Fil^{-1}V)$ is injective, we may assume $V$ is semistable without loss of generality. Since $V$ is ordinary and pure of weight $-1$, we may apply \cite{Nek93}*{Proposition~1.32} with $X=\Fil^{-1}V$ to conclude that $\rH^1_{\r{st}}(F,V)$ coincides with the image of the natural map $\rH^1(F,\Fil^{-1}V)\to\rH^1(F,V)$. Now again as $V$ is pure of weight $-1$, we have $\rH^1_f(F,V)=\rH^1_{\r{st}}(F,V)$ hence \eqref{eq:panchishkin} follows.
\end{proof}

\section{Bessel periods in the definite case}
\label{ss:bessel_definite}

In this section, we take a subgroup $\rJ<\rI^0$ and a \emph{definite} datum $\bV$ from Definition \ref{de:datum}. We start from a lemma.

\begin{lem}
The $\sE_\rJ(\xi,\bV)_\gamma$-module $\sD_\rJ(\xi,\bV)_\gamma^\sfH$ is finitely generated, and locally torsion free of rank at most one over the normal locus of $\Spec\sE_\rJ(\xi,\bV)_\gamma$.
\end{lem}

\begin{proof}
Except for the rank part, this follows from Proposition \ref{pr:eigen_2}(7). The rank part follows from Lemma \ref{le:ggp} and the fact that classical points are Zariski dense.
\end{proof}

Assume $\xi$ interlacing and fix a $\xi$-distinction (Definition \ref{de:distinction}). Our goal of this section is to construct an element
\[
\blambda_\rJ(\bV)\in\cD_\rJ(\xi,\bV)_\gamma^\sfH
\]
that interpolates Bessel period integrals (or rather sums) at all interlacing classical points, following \cite{LS}.

It suffices to construct elements
\[
\blambda_\rI(\bV)\in\Hom_{\dZ_q}\(\rH(\Sh(\sfG,\rK\rI),\dZ_\xi)_\gamma,\dZ_q\)
\]
for compact subgroups $\rI$ of $\sfT(F^+_{\PP^+})$, satisfying $\blambda_{\rI'}(\bV)=\rho_{\rI',\rI}^\vee(\blambda_\rI(\bV))$ for $\rI\subseteq\rI'$. For $\rk\prec\rk'$ in $\bigcup_\rI\fk_\rI$, we have the dual transfer map
\[
\rho_{\rk',\rk}^\vee\colon\Gamma(\Sh(\sfG,\rK\rk),\dZ_\xi^\vee/p^m)
\to\Gamma(\Sh(\sfG,\rK\rk'),\dZ_\xi^\vee/p^m)
\]
as the composition of the pullback map $\rho_{\rk,\rk'\cap\rk}$ and the pushforward map $\rho_{\rk'\cap\rk,\rk'}$.

Let $\rI$ be an open compact subgroup of $\sfT(F^+_{\PP^+})$. Then taking inner products (with respect to counting measures) induces a map
\[
\varinjlim_{\fk_\rI}\Gamma(\Sh(\sfG,\rK\rk),\dZ_\xi^\vee)_{\gamma^\tc}
\to\Hom_{\dZ_q}\(\rH(\Sh(\sfG,\rK\rI),\dZ_\xi)_\gamma,\dZ_q\).
\]

For every element $\rk\in\fk_\rI^\sfH$, define $\blambda_\rk(\bV)$ to be the image of the constant function of value $1$ under the composite map
\begin{align*}
\Gamma(\Sh(\sfH,\rK_\sfH\rk_\sfH),\dZ_q)
\to\Gamma(\Sh(\sfH,\rK_\sfH\rk_\sfH),\dZ_\xi^\vee)
\xrightarrow{\graph_!}\Gamma(\Sh(\sfG,\rK\rk),\dZ_\xi^\vee)
\to\Gamma(\Sh(\sfG,\rK\rk),\dZ_\xi^\vee)_{\gamma^\tc}
\end{align*}
in which the first map is induced from the dual of the fixed $\xi$-distinction; and the second map is the pushforward map along $\graph$.

\begin{lem}\label{le:lambda1}
For $\rk\prec\rk'\in\fk_\rI^\sfH$, we have $\blambda_{\rk'}(\bV)=\rho_{\rk',\rk}^\vee(\blambda_\rk)(\bV)$.
\end{lem}

\begin{proof}
As $\rk=\rk_\sfH\cdot\rk_\sfB$ and $\rk'=\rk'_\sfH\cdot\rk'_\sfB$, we must have $\rk_\sfB\subseteq\rk'_\sfB$, $\rk'_\sfH\subseteq\rk_\sfH$, hence $\rk'\cap\rk=\rk'_\sfH\cdot\rk_\sfB$. It follows that the diagram
\[
\xymatrix{
\Gamma(\Sh(\sfH,\rK_\sfH(\rk'\cap\rk)_\sfH),\dZ_\xi^\vee)  \ar[r]^-{\graph_!}\ar[d]_-{=} &
\Gamma(\Sh(\sfG,\rK(\rk'\cap\rk)),\dZ_\xi^\vee) \ar[d]^-{\rho_{\rk'\cap\rk,\rk'}} \\
\Gamma(\Sh(\sfH,\rK_\sfH\rk'_\sfH),\dZ_\xi^\vee) \ar[r]^-{\graph_!} &
\rH(\Sh(\sfG,\rK\rk'),\dZ_\xi^\vee)
}
\]
commutes. On the other hand, the natural map $\rk_\sfH/(\rk'\cap\rk)_\sfH\to\rk/(\rk'\cap\rk)$ is bijective, which implies that the diagram
\[
\xymatrix{
\Gamma(\Sh(\sfH,\rK_\sfH\rk_\sfH),\dZ_\xi^\vee)  \ar[r]^-{\graph_!}\ar[d]_-{\rho_{\rk_\sfH,(\rk'\cap\rk)_\sfH}} &
\Gamma(\Sh(\sfG,\rK\rk,\dZ_\xi^\vee) \ar[d]^-{\rho_{\rk,\rk'\cap\rk,}} \\
\Gamma(\Sh(\sfH,\rK_\sfH(\rk'\cap\rk)_\sfH),\dZ_\xi^\vee) \ar[r]^-{\graph_!} &
\Gamma(\Sh(\sfG,\rK(\rk'\cap\rk),\dZ_\xi^\vee)
}
\]
commutes, in which both vertical arrows are pullback maps. The lemma follows by combining the two commutative diagrams.
\end{proof}

By the lemma above and the fact that $\fk_\rI^\sfH$ is cofinal in $\fk_\rI$, there exists a unique element
\[
\blambda_\rI(\bV)\in\varinjlim_{\fk_\rI}\rH(\Sh(\sfG,\rK\rk),\dZ_\xi^\vee)_{\gamma^\tc}
\]
that is the image of $\blambda_\rk(\bV)$ for every $\rk\in\fk_\rI^\sfH$. We also regard $\blambda_\rI(\bV)$ as its image in $\Hom_{\dZ_q}\(\rH(\Sh(\sfG,\rK\rI),\dZ_\xi)_\gamma,\dZ_q\)$.

\begin{lem}\label{le:lambda2}
For open compact subgroups $\rI\subseteq\rI'$ of $\sfT(F^+_{\PP^+})$, we have $\blambda_{\rI'}(\bV)=\rho_{\rI',\rI}^\vee(\blambda_\rI(\bV))$.
\end{lem}

\begin{proof}
It is obvious that we may choose $\rk\in\fk_\rI^\sfH$ and $\rk'\in\fk_{\rI'}^\sfH$ such that $\rk_\sfH=\rk'_\sfH$. Then the lemma follows obviously from the definition of $\blambda_\rI(\bV)$.
\end{proof}

Then we define $\blambda_\rJ(\bV)$ to be the element in
\[
\varprojlim_{\rJ<\rI\finite\rI^0}\Hom_{\dZ_q}\(\rH(\Sh(\sfG,\rK\rI),\dZ_\xi)_\gamma,\dZ_q\)=\cD_\rJ(\xi,\bV)_\gamma
\]
given by the compatible collection of elements $\blambda_\rI(\bV)$. It is clear from the construction that $\blambda_\rJ(\bV)$ belongs to the submodule $\cD_\rJ(\xi,\bV)_\gamma^\sfH$.

\begin{remark}\label{re:liusun}
Assume Assumption \ref{as:galois}(G1). Take an interlacing classical point $x$ of $\Spec\sE_\rJ(\xi,\bV)_\gamma$ (Definition \ref{de:classical2}). We have the induced character $\phi_{x,N}\colon\dT^{\ang{\fm}}_N\to\dQ_x$ for $N=n,n+1$, such that for every embedding $\iota\colon\dQ_x\to\dC$, $\iota\circ\phi_{x,N}$ comes from a hermitian cuspidal automorphic representation $\Pi_{x,N}^\iota$ of $\GL_N(\dA_F)$, unique up to isomorphism by the strong multiplicity one.

Consider the following two statements:
\begin{enumerate}
  \item The image of $\blambda_\rJ(\bV)$ in $\sD_\rJ(\xi,\bV)_\gamma\res_x$ is nonvanishing.

  \item For every embedding $\iota\colon\dQ_x\to\dC$, the central Rankin--Selberg $L$-value $L(\frac{1}{2},\Pi_{x,n}^\iota\times\Pi_{x,n+1}^\iota)\neq0$; and there exists a tempered irreducible admissible representation $\pi^\iota$ of $\sfG(F^+_{\spadesuit^+})$ whose base change is $(\Pi_{x,n}^\iota\boxtimes\Pi_{x,n+1}^\iota)_{\spadesuit}$ such that $\((\pi^\iota)^{\rK_{\spadesuit^+}}\)^\sfH\neq\{0\}$.
\end{enumerate}
Then by Proposition \ref{pr:eigen_2}(6), \cite{LS} and \cite{BPLZZ}*{Theorem~1.8}, (1) implies (2); if moreover every place in $\PP^+$ splits in $F$ and $\sR_\rJ(\xi,\bV)_\gamma\res_x$ is potentially crystalline at places in $\PP$, then (2) also implies (1).
\end{remark}

\section{Bessel periods in the indefinite case}
\label{ss:bessel_indefinite}

In this section, we take a subgroup $\rJ<\rI^0$ and an \emph{indefinite} datum $\bV$ from Definition \ref{de:datum}. We start from a lemma.

\begin{lem}\label{le:free_incoherent}
Under Assumption \ref{as:galois}(G1,G2), the $\sE_\rJ(\xi,\bV)_\gamma$-module $\rH^1_f(F,\sD_\rJ(\xi,\bV)_\gamma)$ is finitely generated, and locally torsion free over the normal locus of $\Spec\sE_\rJ(\xi,\bV)_\gamma$.
\end{lem}

\begin{proof}
The finite generation follows from Lemma \ref{le:selmer1}(2). By Proposition \ref{pr:eigen_2}(7), $\sD_\rJ(\xi,\bV)_\gamma$ is locally torsion free over the normal locus of $\Spec\sE_\rJ(\xi,\bV)$. It follows, using Assumption \ref{as:galois}(G1), that $\rH^1_f(F,\sD_\rJ(\xi,\bV)_\gamma)$ is also locally torsion free over the normal locus of $\Spec\sE_\rJ(\xi,\bV)_\gamma$. The lemma is proved.
\end{proof}

Assume $\xi$ interlacing and fix a $\xi$-distinction (Definition \ref{de:distinction}). Our goal of this section is to construct, under Assumption \ref{as:galois}(G1,G2), an element
\[
\bkappa_\rJ(\bV)\in\rH^1(F,\cD_\rJ(\xi,\bV)_\gamma)^\sfH
\]
that interpolates the Abel--Jacobi image of diagonal cycles at all interlacing classical points.

It suffices to construct elements
\[
\bkappa_\rI(\bV)\in\rH^1\(F,\Hom_{\dZ_q}\(\rH(\Sh(\sfG,\rK\rI),\dZ_\xi)_\gamma,\dZ_q\)\)
\]
for compact subgroups $\rI$ of $\sfT(F^+_{\PP^+})$, satisfying $\bkappa_{\rI'}(\bV)=\rH^1(F,\rho_{\rI',\rI}^\vee)(\bkappa_\rI(\bV))$ for $\rI\subseteq\rI'$. For $\rk\prec\rk'$ in $\bigcup_\rI\fk_\rI$, we have the dual transfer map
\[
\rho_{\rk',\rk}^\vee\colon\rH^i_{\et}(\Sh(\sfG,\rK\rk)_R,\dZ_\xi^\vee/p^m(j))
\to\rH^i_{\et}(\Sh(\sfG,\rK\rk')_R,\dZ_\xi^\vee/p^m(j))
\]
as the composition of the pullback map $\rho_{\rk,\rk'\cap\rk}$ and the pushforward map $\rho_{\rk'\cap\rk,\rk'}$, for every $i,j\in\dZ$ and every $F$-ring $R$.

Let $\rI$ be an open compact subgroup of $\sfT(F^+_{\PP^+})$. Then the Poincar\'{e} duality induces a map
\[
\varinjlim_{\fk_\rI}\rH^{2n-1}_{\et}(\Sh(\sfG,\rK\rk)_{\ol{F}},\dZ_\xi^\vee(n))_{\gamma^\tc}
\to\Hom_{\dZ_q}\(\rH(\Sh(\sfG,\rK\rI),\dZ_\xi)_\gamma,\dZ_q\),
\]
hence a map
\[
\varinjlim_{\fk_\rI}\rH^1\(F,\rH^{2n-1}_{\et}(\Sh(\sfG,\rK\rk)_{\ol{F}},\dZ_\xi^\vee(n))_{\gamma^\tc}\)
\to\rH^1\(F,\Hom_{\dZ_q}\(\rH(\Sh(\sfG,\rK\rI),\dZ_\xi)_\gamma,\dZ_q\)\).
\]

For every element $\rk\in\fk_\rI^\sfH$, define $\bkappa_\rk(\bV)$ to be the image of the fundamental class under the composite map
\begin{align*}
\rH^0_{\et}(\Sh(\sfH,\rK_\sfH\rk_\sfH),\dZ_q)
&\to\rH^0_{\et}(\Sh(\sfH,\rK_\sfH\rk_\sfH),\dZ_\xi^\vee)
\xrightarrow{\graph_!}\rH^{2n}_{\et}(\Sh(\sfG,\rK\rk),\dZ_\xi^\vee(n)) \\
&\to\rH^{2n}_{\et}(\Sh(\sfG,\rK\rk),\dZ_\xi^\vee(n))_{\gamma^\tc}
\to\rH^1\(F,\rH^{2n-1}_{\et}(\Sh(\sfG,\rK\rk)_{\ol{F}},\dZ_\xi^\vee(n))_{\gamma^\tc}\)
\end{align*}
in which the first map is induced from the dual of the fixed $\xi$-distinction; the second map is the pushforward map along $\graph$ \cite{GS23}*{Proposition~A.5}; and the last map comes from the vanishing of $\rH^{2n}_{\et}(\Sh(\sfG,\rK\rk)_{\ol{F}},\dZ_\xi^\vee(n))_{\gamma^\tc}$ which is implied by Assumption \ref{as:galois}(G2).

By the same argument for Lemma \ref{le:lambda1} and the fact that $\fk_\rI^\sfH$ is cofinal in $\fk_\rI$, there exists a unique element
\[
\bkappa_\rI(\bV)\in\varinjlim_{\fk_\rI}
\rH^1\(F,\rH^{2n-1}_{\et}(\Sh(\sfG,\rK\rk)_{\ol{F}},\dZ_\xi^\vee(n))_{\gamma^\tc}\)
\]
that is the image of $\bkappa_\rk(\bV)$ for every $\rk\in\fk_\rI^\sfH$. We also regard $\bkappa_\rI(\bV)$ as its image in
\[
\rH^1\(F,\Hom_{\dZ_q}\(\rH(\Sh(\sfG,\rK\rI),\dZ_\xi)_\gamma,\dZ_q\)\).
\]
By the same argument for Lemma \ref{le:lambda2}, the collection $\{\bkappa_\rI(\bV)\}$ gives rise to an element
\[
\bkappa_\rJ(\bV)\in\varprojlim_{\rJ<\rI\finite\rI^0}
\rH^1\(F,\Hom_{\dZ_q}\(\rH(\Sh(\sfG,\rK\rI),\dZ_\xi)_\gamma,\dZ_q\)\)
=\rH^1(F,\cD_\rJ(\xi,\bV)_\gamma).
\]

\begin{lem}\label{le:crystalline}
The element $\bkappa_\rJ(\bV)$ belongs to $\rH^1_f(F,\cD_\rJ(\xi,\bV)_\gamma)^\sfH$.
\end{lem}

\begin{proof}
By definition, it suffices to show that for every $\rk\in\fk_\rI^\sfH$ and every place $w$ of $F$,
\[
\loc_w(\bkappa_\rk(\bV))\in\rH^1_f\(F_w,\rH^{2n-1}_{\et}(\Sh(\sfG,\rK\rk)_{\ol{F}},\dZ_\xi^\vee(n))_{\gamma^\tc}\).
\]
When $w$ is not above $p$, this is trivial by Remark \ref{re:pure}. Now suppose that $w\mid p$. The fact that $\bkappa_\rk(\bV)$ is a motivic class implies that $\loc_w(\bkappa_\rk(\bV))$ belongs to $\rH^1_g$ (see the proof of \cite{GS23}*{Proposition~9.29} for more details in a similar case), which coincides with $\rH^1_f$ by Remark \ref{re:pure}.
\end{proof}

\section{Level-raising auxiliary primes}
\label{ss:raising}

Recall from \cite{LTXZZ}*{Definition~3.3.4} that a prime $\fl$ of $F^+$ is \emph{special inert} if the following are satisfied:
\begin{enumerate}
  \item $\fl$ is inert in $F$;

  \item the underlying rational prime $\ell$ of $\fl$ is odd and is unramified in $F$;

  \item $\fl$ is of degree one over $\dQ$, that is, $F^+_\fl=\dQ_\ell$.
\end{enumerate}

\begin{definition}\label{de:prime}
A \emph{level-raising prime} of $F^+$ (with respect to $\gamma$) is a special inert prime $\fl$ of $F^+$ (with the underlying rational prime $\ell$) satisfying
\begin{description}
  \item[(P1)] $p$ does not divide $\ell\prod_{i=1}^{n_0}(1-(-\ell)^i)$;

  \item[(P2)] there exists a CM type $\Phi$ as in \cite{LTXZZ}*{\S5.1} satisfying $\dQ_\ell^\Phi=\dQ_{\ell^2}$ (and we choose remaining data in \cite{LTXZZ}*{\S5.1} with $\dQ_\ell^\Phi=\dQ_{\ell^2}$);

  \item[(P3)] the generalized eigenvalues of $\gamma_{n_0}^\natural(\phi_{w(\fl)})$ contain $\ell^{-n_0}$ exactly once; the generalized eigenvalues of $\gamma_{n_1}^\natural(\phi_{w(\fl)})$ contain $\ell^{-(n_1-1)}$ exactly once; the generalized eigenvalues of $(\gamma_n^\natural\otimes\gamma_{n+1}^\natural)(\phi_{w(\fl)})$ contain $\ell^{-2n}$ exactly once;

  \item[(P4)] the eigenvalues of $\gamma_{n_0}^\natural(\phi_{w(\fl)})$ do not contain $-\ell^{-(n_0-1)}$; the eigenvalues of $\gamma_{n_1}^\natural(\phi_{w(\fl)})$ do not contain $-\ell^{-n_1}$.
\end{description}
\end{definition}

\begin{notation}\label{de:cong}
For a level-raising prime $\fl$ of $F^+$, we put
\[
\tC_\fl\coloneqq(\ell+1)\tR_{n_0,\fl}-\tI_{n_0,\fl}\in\dT_{n_0,\fl}
\]
as the Hecke operator appeared in \cite{LTXZZ}*{Proposition~6.3.1(4)}.
\end{notation}

\section{First explicit reciprocity law for ordinary distributions}
\label{ss:first_reciprocity}

In this section, we establish the first explicit reciprocity law for ordinary distributions. We take a \emph{definite} datum $\bV$ from Definition \ref{de:datum} and choose a level-raising prime $\fl$ of $F^+$ (Definition \ref{de:prime}) not in $\ang{\fm}$.

Consider another datum $\bV'=(\fm';\rV'_n,\rV'_{n+1};\Lambda'_n,\Lambda'_{n+1};\rK_n^{\prime},\rK_{n+1}^{\prime};\sfB')$ and a datum $\tj=(\tj_n,\tj_{n+1})$ in which
\begin{itemize}
  \item $\fm'=\fm\cup\{\fl\}$;

  \item $\tj_N\colon\rV_N\otimes_{F^+}\dA_{F^+}^{\infty,\fl}\to\rV'_N\otimes_{F^+}\dA_{F^+}^{\infty,\fl}$ ($N=n,n+1$) are isometries sending $(\Lambda_N)^\fl$ to $(\Lambda'_N)^\fl$ and $(\rK_N)^\fl$ to $(\rK_N^{\prime})^\fl$ for $N\in\{n,n+1\}$, satisfying $\tj_{n+1}=(\tj_n)_\sharp$ and $\tj_p\sfB=\sfB'$.
\end{itemize}
In particular, $\bV'$ is an indefinite datum.

Similar to $\bV$, we have groups $\sfG'$, $\sfH'$, $\rK^{\prime}$, $\rK^{\prime}_{\sfH'}$ for $\bV'$. Moreover, we identify $\sfG'\otimes_\dQ\dQ_p$ with $\sfG\otimes_\dQ\dQ_p$ and $\sfB'$ with $\sfB$ via $\tj_p$.

\begin{assumption}\label{as:law}
The following set of conditions will be used in the explicit reciprocity laws:
\begin{enumerate}
  \item $\gamma$ satisfies Assumption \ref{as:galois}(G1,G2,G5,G6,G7,G8);

  \item for every $v\in\fm$, the generalized eigenvalues of $\gamma_{n_0}^\natural(\phi_{w(v)})$ contain $\|v\|^{-n_0}$ exactly once;

  \item $\xi^{n_0}$ is Fontaine--Laffaille regular (Definition \ref{de:weight0}).
\end{enumerate}
\end{assumption}

Our main goal is to, under Assumption \ref{as:law}, establish a natural isomorphism
\begin{align}\label{eq:reciprocity_first}
\varrho^\sing_\rJ(\bV',\bV)&\colon
\rH^1_\sing(F_{w(\fl)},\cD_\rJ(\xi,\bV')_\gamma)
\xrightarrow{\sim}\cD_\rJ(\xi,\bV)_\gamma/\tC_\fl
\end{align}
of $\dZ[\rK_{\spadesuit^+}\backslash\sfG(F^+_{\spadesuit^+})/\rK_{\spadesuit^+}]\otimes
\dT^{\ang{\fm'}}\otimes\Lambda_{\sfT(F^+_{\PP^+})/\rJ,\dZ_q}$-modules assuming $\cD_\rJ(\xi,\bV)_\gamma\neq 0$.

By definition, the natural maps
\begin{align*}
\rH^1_\sing(F_{w(\fl)},\cD_\rJ(\xi,\bV')_\gamma)&\to\varprojlim_{\rJ<\rI\finite\rI^0}
\rH^1_\sing\(F_{w(\fl)},\Hom_{\dZ_q}\(\rH(\Sh(\sfG',\rK^{\prime}\rI),\dZ_\xi)_\gamma,\dZ_q\)\),\\
\cD_\rJ(\xi,\bV)_\gamma/\tC_\fl&\to\varprojlim_{\rJ<\rI\finite\rI^0}
\left.\Hom_{\dZ_q}\(\rH(\Sh(\sfG,\rK\rI),\dZ_\xi)_\gamma,\dZ_q\)\right/\tC_\fl
\end{align*}
are both isomorphisms.

To establish \eqref{eq:reciprocity_first}, we have several steps:
\begin{enumerate}
  \item upgrade the construction of Tate classes for $N=n_1$ \cite{LTXZZ}*{\S6.2} to the level of ordinary eigenvarieties with general weights $\xi$ away from $\PP^+$,

  \item upgrade the arithmetic level-raising for $N=n_0$ \cite{LTXZZ}*{\S6.3} to the level of ordinary eigenvarieties with general weights $\xi$ away from $\PP^+$,

  \item upgrade the first reciprocity law \cite{LTXZZ}*{\S7.2} to the level of ordinary eigenvarieties with general weights $\xi$ away from $\PP^+$.
\end{enumerate}
To unify the exposition, we will only discuss the (much harder) case where $n_0\geq 4$ (hence $n_1\geq 3$) and leave the (much easier) case where $n_0=2$ to the readers as an exercise. In what follows, $\alpha$ will be an index from $\{0,1\}$. Put
\[
\fr_\alpha\coloneqq\Ker\phi_{\gamma^\tc_{n_\alpha}}\cap\dT_{n_\alpha}^{\ang{\fm'}},
\]
and let $\fr$ be the (maximal) ideal of $\dT_{n_\alpha}^{\ang{\fm'}}$ generated by $\fr_0$ and $\fr_1$.

Let $\rk_\alpha$ be an open compact subgroup of $\sfG_{n_\alpha}(F^+_{\PP^+})$. As in \cite{LTXZZ}*{\S6.1}, we have the moduli scheme $\bM_{n_\alpha}(\rk_\alpha)$ over $O_{F_\fl}\simeq\dZ_{\ell^2}$ that is essentially an integral model of $\Sh(\sfG'_{n_\alpha},\rK_{n_\alpha}^{\prime}\tj_{n_\alpha}\rk_\alpha)$; and we write $\rK_{n_\alpha}^{\bullet}$ as the open compact subgroup obtained by changing the component of $\rK_{n_\alpha}^{\circ}\coloneqq\rK_{n_\alpha}$ at the place $\fl$ to the special (hyperspecial if $\alpha=0$) maximal subgroup that is used in the uniformization of the basic correspondence for the ground stratum. Using Gysin maps with coefficients, we have incidence maps similar to \cite{LTXZZ}*{Construction~5.8.3} when we replace $L$ by the local system $\dZ_{\xi^{n_\alpha}}^\vee/p^m$ for $1\leq m\leq \infty$ hence maps
\begin{align*}
\Inc^*_\circ&\colon\rH^{2r_\alpha}_\fT(\ol\rM_{n_\alpha}^\circ(\rk_\alpha),\dZ_{\xi^{n_\alpha}}^\vee/p^m(r_\alpha))\oplus
\rH^{2r_\alpha}_\fT(\ol\rM_{n_\alpha}^\bullet(\rk_\alpha),\dZ_{\xi^{n_\alpha}}^\vee/p^m(r_\alpha))
\to\Gamma(\Sh(\rV_{n_\alpha},\rK_{n_\alpha}^\circ\rk_\alpha),\dZ_{\xi^{n_\alpha}}^\vee/p^m), \\
\Inc^*_\dag&\colon\rH^{2r_\alpha}_\fT(\ol\rM_{n_\alpha}^\circ(\rk_\alpha),\dZ_{\xi^{n_\alpha}}^\vee/p^m(r_\alpha))\oplus
\rH^{2r_\alpha}_\fT(\ol\rM_{n_\alpha}^\bullet(\rk_\alpha),\dZ_{\xi^{n_\alpha}}^\vee/p^m(r_\alpha))
\to\Gamma(\Sh(\rV_{n_\alpha},\rK_{n_\alpha}^\circ\rk_\alpha),\dZ_{\xi^{n_\alpha}}^\vee/p^m), \\
\Inc^*_\bullet&\colon\rH^{2r_\alpha}_\fT(\ol\rM_{n_\alpha}^\circ(\rk_\alpha),\dZ_{\xi^{n_\alpha}}^\vee/p^m(r_\alpha))\oplus
\rH^{2r_\alpha}_\fT(\ol\rM_{n_\alpha}^\bullet(\rk_\alpha),\dZ_{\xi^{n_\alpha}}^\vee/p^m(r_\alpha))\to\Gamma(\Sh(\rV_{n_\alpha},\rK_{n_\alpha}^\bullet\rk_\alpha),\dZ_{\xi^{n_\alpha}}^\vee/p^m),
\end{align*}
as in \cite{LTXZZ}*{Construction~5.9.1}.

We start from step (1) (so that $\alpha=1$). Let $(\pres{1}\rE^{p,q}_s(\rk_1),\pres{1}\rd^{p,q}_s(\rk_1))$ be the weight spectral sequence abutting to $\rH^{p+q}_\fT(\ol\rM_{n_1}(\rk_1),\rR\Psi\dZ_{\xi^{n_1}}^\vee(r_1))$. Similar to \cite{LTXZZ}*{Construction~5.9.4(1)}, we have the map
\[
\nabla^1\colon\Ker\pres{1}\rE^{0,2r_1}_2(\rk_1)\to \Gamma(\Sh(\sfG_{n_1},\rK_{n_1}\rk_1),\dZ_{\xi^{n_1}}^\vee)
\]
that is the restriction of the map
\[
\tT^{\bullet\circ}_{n_1,\fl}\circ\Inc_\dag^*+(\ell+1)^2\Inc_\bullet^*\colon\pres{1}\rE^{0,2r_1}_1(\rk_1)\to \Gamma(\Sh(\sfG_{n_1},\rK_{n_1}^\bullet\rk_1),\dZ_{\xi^{n_1}}^\vee)
\]
to $\Ker\pres{1}\rd^{0,2r_1}_1(\rk_1)$, which factors through $\pres{1}\rE^{0,2r_1}_2(\rk_1)$ by (an analogue of) \cite{LTXZZ}*{Lemma~5.9.2(3)}, composed with the map
\[
\tT^{\circ\bullet}_{n_1,\fl}\colon\Gamma(\Sh(\sfG_{n_1},\rK_{n_1}^\bullet\rk_1),\dZ_{\xi^{n_1}}^\vee)
\to\Gamma(\Sh(\sfG_{n_1},\rK_{n_1}^\circ\rk_1),\dZ_{\xi^{n_1}}^\vee)
=\Gamma(\Sh(\sfG_{n_1},\rK_{n_1}\rk_1),\dZ_{\xi^{n_1}}^\vee).
\]
The lemma below is an analogue of \cite{LTXZZ}*{Lemma~6.2.2~\&~Theorem~6.2.3}.

\begin{lem}\label{le:first0}
Suppose that the $\Gamma_F$-representation $\gamma_{n_1}^\natural$ is absolutely irreducible and that $\phi_{\gamma_{n_1}}$ is cohomologically generic.
\begin{enumerate}
  \item $\rH^i_\fT(\ol\rM_{n_1}^?(\rk_1),\dZ_{\xi^{n_1}}^\vee)_{\fr_1}$ is $p$-torsion free for every $i\in\dZ$ and $?\in\{\circ,\bullet,\dag\}$, which vanishes when $i$ is odd;

  \item $\pres{1}\rE^{p,q}_2(\rk_1)_{\fr_1}$ is $p$-torsion free, and vanishes if $(p,q)\neq(0,2r_1)$;

  \item The map $\nabla^1_{\fr_1}$ induces an isomorphism
      \[
      \nabla^1_{\fr_1}\colon\(\pres{1}\rE^{0,2r_1}_2(\rk_1)_{\fr_1}\)_{\Gal(\ol\dF_\ell/\dF_{\ell^2})}
      \xrightarrow\sim\Gamma(\Sh(\sfG_{n_1},\rK_{n_1}\rk_1),\dZ_{\xi^{n_1}}^\vee)_{\fr_1}.
      \]
\end{enumerate}
\end{lem}

\begin{proof}
This can be proved in the same way as \cite{LTXZZ}*{Lemma~6.2.2~\&~Theorem~6.2.3}.
\end{proof}

We then move to step (2) (so that $\alpha=0$). Let $(\pres{0}\rE^{p,q}_s(\rk_0),\pres{0}\rd^{p,q}_s(\rk_0))$ be the weight spectral sequence abutting to $\rH^{p+q}_\fT(\ol\rM_{n_0}(\rk_0),\rR\Psi\dZ_{\xi^{n_0}}^\vee(r_0))$. Similar to \cite{LTXZZ}*{Construction~5.9.4(2)}, we have the map
\[
\nabla^0\colon\Ker\pres{0}\rd^{0,2r}_1(\rk_0)\to \Gamma(\Sh(\sfG_{n_0},\rK_{n_0}\rk_0),\dZ_{\xi^{n_0}}^\vee)
\]
that is the restriction of the map
\[
\tT^{\bullet\circ}_{n_0,\fl}\circ\Inc_\circ^*+(\ell+1)\Inc_\bullet^*\colon\pres{0}\rE^{0,2r}_1(\rk_0)\to \Gamma(\Sh(\sfG_{n_0},\rK_{n_0}^\bullet\rk_0),\dZ_{\xi^{n_0}}^\vee)
\]
in (an analogue of) \cite{LTXZZ}*{Lemma~5.9.3(8)} to $\Ker\pres{0}\rd^{0,2r}_1(\rk_0)$, composed with the map
\[
\tT^{\circ\bullet}_{n_0,\fl}\colon \Gamma(\Sh(\sfG_{n_0},\rK_{n_0}^\bullet\rk_0),\dZ_{\xi^{n_0}}^\vee)\to \Gamma(\Sh(\sfG_{n_0},\rK_{n_0}^\circ\rk_0),\dZ_{\xi^{n_0}}^\vee)
=\Gamma(\Sh(\sfG_{n_0},\rK_{n_0}\rk_0),\dZ_{\xi^{n_0}}^\vee).
\]
On the other hand, it is clear that (the natural analogue of) \cite{LTXZZ}*{Lemma~5.9.3} holds for $(\pres{0}\rE^{p,q}_s(\rk_0),\pres{0}\rd^{p,q}_s(\rk_0))$ as well. In particular, we have a canonical short exact sequence
\begin{align}\label{eq:first2}
\resizebox{\hsize}{!}{
\xymatrix{
0 \ar[r] & \rF_{-1}\rH^1(\rI_{\dQ_{\ell^2}},\rH^{2r_0-1}_\fT(\ol\rM_{n_0}(\rk_0),\rR\Psi\dZ_{\xi^{n_0}}^\vee(r_0)))
\ar[r] & \rH^1_\sing(\dQ_{\ell^2},\rH^{2r_0-1}_\fT(\ol\rM_{n_0}(\rk_0),\rR\Psi\dZ_{\xi^{n_0}}^\vee(r_0)))
\ar[r] & \rH^{2r_0-1}_\fT(\ol\rM^\bullet_{n_0}(\rk_0),\dZ_{\xi^{n_0}}^\vee(r_0-1))^{\Gal(\ol\dF_\ell/\dF_{\ell^2})} \to 0
}
}
\end{align}
of $\dZ[(\rK_{n_0})_{\spadesuit^+}\backslash\sfG_{n_0}(F^+_{\spadesuit^+})/(\rK_{n_0})_{\spadesuit^+}]
\otimes\dT_{n_0}^{\ang{\fm'}}$-modules. The lemma below is an analogue of \cite{LTXZZ}*{Proposition~6.3.1}.

\begin{lem}\label{le:first1}
Suppose that the $\Gamma_F$-representation $\gamma_{n_0}^\natural$ is absolutely irreducible and that $\phi_{\gamma_{n_0}}$ is cohomologically generic.
\begin{enumerate}
  \item The map
     \begin{align*}
     (\Inc_\circ^*,\Inc_\dag^*,\Inc_\bullet^*)_{\fr_0}\colon\pres{0}\rE^{0,2r}_1(\rk_0)_{\fr_0}\to
     \Gamma(\Sh(\sfG_{n_0},\rK_{n_0}^\circ\rk_0),\dZ_{\xi^{n_0}}^\vee)_{\fr_0}^{\oplus 2}\oplus \Gamma(\Sh(\sfG_{n_0},\rK_{n_0}^\bullet\rk_0),\dZ_{\xi^{n_0}}^\vee)_{\fr_0}
     \end{align*}
     is surjective with kernel $\rH^{2r_0}_\fT(\ol\rM_{n_0}^\bullet(\rk_0),\dZ_{\xi^{n_0}}^\vee(r_0))_{\fr_0}[p^\infty]$.

  \item The map $\nabla^0_{\fr_0}\colon\Ker\pres{0}\rd^{0,2r}_1(\rk_0)_{\fr_0}\to \Gamma(\Sh(\sfG_{n_0},\rK_{n_0}\rk_0),\dZ_{\xi^{n_0}}^\vee)_{\fr_0}$ is surjective.

  \item The map $\nabla^0_{\fr_0}\circ\pres{0}\rd^{-1,2r}_1(\rk_0)_{\fr_0}$ induces a map
      \[
      \rF_{-1}\rH^1(\rI_{\dQ_{\ell^2}},\rH^{2r_0-1}_\fT(\ol\rM_{n_0}(\rk_0),\rR\Psi\dZ_{\xi^{n_0}}^\vee(r_0))_{\fr_0})\to
      \Gamma(\Sh(\sfG_{n_0},\rK_{n_0}\rk_0),\dZ_{\xi^{n_0}}^\vee)_{\fr_0}/\tC_\fl
      \]
      which is surjective, whose kernel is canonically $\rH^{2r_0}_\fT(\ol\rM_{n_0}^\bullet(\rk_0),\dZ_{\xi^{n_0}}^\vee(r_0))_{\fr_0}[p^\infty]$.
\end{enumerate}
\end{lem}

\begin{proof}
For (1), we first show that the map is surjective. For this, it suffices to show that its mod-$p$ analogue
\begin{align}\label{eq:first1}
\rH^{2r_0}_\fT(\ol\rM_{n_0}^\circ(\rk_0),\dZ_{\xi^{n_0}}^\vee/p(r_0))_{\fr_0}&\oplus
\rH^{2r_0}_\fT(\ol\rM_{n_0}^\bullet(\rk_0),\dZ_{\xi^{n_0}}^\vee/p(r_0))_{\fr_0} \\
\to\Gamma(\Sh(\sfG_{n_0},\rK_{n_0}^\circ\rk_0),\dZ_{\xi^{n_0}}^\vee/p)_{\fr_0}^{\oplus 2}&\oplus \Gamma(\Sh(\sfG_{n_0},\rK_{n_0}^\bullet\rk_0),\dZ_{\xi^{n_0}}^\vee/p)_{\fr_0} \notag
\end{align}
is surjective. Let $\rK_{n_0}^1\subseteq\rK_{n_0}$ be the subgroup of elements acting trivially on $\Lambda_{n_0,v}/p\Lambda_{n_0,v}$ for $v\in\Sigma^+_p\setminus\PP^+$, and denote by $\bM^1_{n_0}(\rk_0)$ the corresponding moduli scheme. By construction, we have the following commutative diagram
\begin{align*}
\resizebox{\hsize}{!}{
\xymatrix{
\(\rH^{2r_0}_\fT(\ol\rM_{n_0}^{1\circ}(\rk_0),\dZ_q(r_0))_{\fr_0}\oplus
\rH^{2r_0}_\fT(\ol\rM_{n_0}^{1\bullet}(\rk_0),\dZ_q(r_0))_{\fr_0}\)\otimes\dZ_{\xi^{n_0}}^\vee/p
\ar[r]\ar@{=}[d]&
\(\Gamma(\Sh(\sfG_{n_0},\rK_{n_0}^{1\circ}\rk_0),\dZ_q)_{\fr_0}^{\oplus 2}\oplus \Gamma(\Sh(\sfG_{n_0},\rK_{n_0}^{1\bullet}\rk_0),\dZ_q)_{\fr_0}\)\otimes\dZ_{\xi^{n_0}}^\vee/p \ar@{=}[d] \\
\rH^{2r_0}_\fT(\ol\rM_{n_0}^{1\circ}(\rk_0),\dZ_{\xi^{n_0}}^\vee/p(r_0))_{\fr_0}\oplus
\rH^{2r_0}_\fT(\ol\rM_{n_0}^{1\bullet}(\rk_0),\dZ_{\xi^{n_0}}^\vee/p(r_0))_{\fr_0} \ar[d]
&
\Gamma(\Sh(\sfG_{n_0},\rK_{n_0}^{1\circ}\rk_0),\dZ_{\xi^{n_0}}^\vee/p)_{\fr_0}^{\oplus 2}\oplus \Gamma(\Sh(\sfG_{n_0},\rK_{n_0}^{1\bullet}\rk_0),\dZ_{\xi^{n_0}}^\vee/p)_{\fr_0}  \ar[d] \\
\rH^{2r_0}_\fT(\ol\rM_{n_0}^{\circ}(\rk_0),\dZ_{\xi^{n_0}}^\vee/p(r_0))_{\fr_0}\oplus
\rH^{2r_0}_\fT(\ol\rM_{n_0}^{\bullet}(\rk_0),\dZ_{\xi^{n_0}}^\vee/p(r_0))_{\fr_0}
\ar[r]^-{\eqref{eq:first1}}&
\Gamma(\Sh(\sfG_{n_0},\rK_{n_0}^{\circ}\rk_0),\dZ_{\xi^{n_0}}^\vee/p)_{\fr_0}^{\oplus 2}\oplus \Gamma(\Sh(\sfG_{n_0},\rK_{n_0}^{\bullet}\rk_0),\dZ_{\xi^{n_0}}^\vee/p)_{\fr_0}
}
}
\end{align*}
in which the upper horizontal map is the base change of the map $(\Inc_\circ^*,\Inc_\dag^*,\Inc_\bullet^*)_{\fr_0}$ for constant coefficients, which is then surjective by \cite{LTXZZ}*{Proposition~6.3.1(2)}. Since the down-right vertical arrow is surjective, it follows that \eqref{eq:first1} is surjective. To find the kernel, note that both sides of the map have the same $\dZ_q$-rank by an argument that is similar to the case in \cite{LTXZZ}*{Proposition~6.3.1}. Since the target is $p$-torsion free, the kernel coincides with $\pres{0}\rE^{0,2r}_1(\rk_0)_{\fr_0}[p^\infty]=\rH^{2r_0}_\fT(\ol\rM_{n_0}^\bullet(\rk_0),\dZ_{\xi^{n_0}}^\vee(r_0))_{\fr_0}[p^\infty]$. Thus, (1) is proved.

The proofs of (2) and (3) follow from the same argument for (3) and (4) of \cite{LTXZZ}*{Proposition~6.3.1}, respectively.
\end{proof}

To continue, we fix, for every place $v\in\PP^+$, a self-dual $O_{F_v}$-lattice $\Lambda_{n_0,v}$ in $\rV_{n_0,v}$ such that $\sfB_{n_0,v}$ induces a Borel subgroup of $\rU(\Lambda_{n_0,v})$. Using this lattice, we regard $\sfG_{n_0,v}$, $\sfB_{n_0,v}$ and $\sfT_{n_0,v}$ as reductive groups defined over $O_{F^+_v}$. We denote by $\sfU_{n_0,v}\subseteq\sfB_{n_0,v}$ the unipotent radical. For integers $0\leq c\leq l$, we denote by $\rI_{n_0,v}^c$ the subgroup of $\sfT_{n_0,v}(O_{F^+_v})$ of elements whose reduction modulo $p^c$ is the identity, and by $\rk_{0,v}^{c,l}$ the subgroup of $\sfG_{n_0,v}(O_{F^+_v})$ of elements whose reduction modulo $p^l$ belong to $\rI_{n_0,v}^c\sfU_{n_0,v}(O_{F^+_v})$.\footnote{The subgroups $\rk_{0,v}^{c,l}$ are ``opposite'' to the ones used in \S\ref{ss:eigenvariety}.} Put
\[
\rI_{n_0}^c\coloneqq\prod_{v\in\PP^+}\rI_{n_0,v}^c,\quad
\rk_0^{c,l}\coloneqq\prod_{v\in\PP^+}\rk_{0,v}^{c,l}\in\fk_{\rI_{n_0}^c}.
\]
The lemma below is an analogue of \cite{LTXZZ}*{Theorem~6.3.4}.

\begin{lem}\label{le:first2}
Assume Assumption \ref{as:law} and $\Gamma(\Sh(\sfG_{n_0},\rK_{n_0}\rk_0^{0,0}),\dZ_{\xi^{n_0}}^\vee/p)_{\fr_0}\neq 0$. Then for every integers $0\leq c\leq l$, we have
\begin{enumerate}
  \item $\rH^i_\fT(\ol\rM_{n_0}^?(\rk_0^{c,l}),\dZ_{\xi^{n_0}}^\vee)_{\fr_0}$ is $p$-torsion free for every $i\in\dZ$ and $?\in\{\circ,\bullet,\dag\}$;

  \item $\pres{0}\rE^{p,q}_2(\rk_0^{c,l})_{\fr_0}$ is $p$-torsion free, and vanishes if $(p,q)\not\in\{(-1,2r_0),(0,2r_0-1),(1,2r_0-2)\}$;

  \item the map
      \[
      \rF_{-1}\rH^1(\rI_{\dQ_{\ell^2}},\rH^{2r_0-1}_\fT(\ol\rM_{n_0}(\rk_0^{c,l}),\rR\Psi\dZ_{\xi^{n_0}}^\vee(r_0))_{\fr_0})\to
      \Gamma(\Sh(\sfG_{n_0},\rK_{n_0}\rk_0^{c,l}),\dZ_{\xi^{n_0}}^\vee)_{\fr_0}/\tC_\fl
      \]
      localized from Lemma \ref{le:first1}(3) is an isomorphism;

  \item the map
      \[
      \rF_{-1}\rH^1(\rI_{\dQ_{\ell^2}},\rH^{2r_0-1}_\fT(\ol\rM_{n_0}(\rk_0^{c,l}),\rR\Psi\dZ_{\xi^{n_0}}^\vee(r_0))_{\fr_0})\to
      \rH^1_\sing(\dQ_{\ell^2},\rH^{2r_0-1}_\fT(\ol\rM_{n_0}(\rk_0^{c,l}),\rR\Psi\dZ_{\xi^{n_0}}^\vee(r_0))_{\fr_0})
      \]
      localized from \eqref{eq:first2} is an isomorphism.
\end{enumerate}
\end{lem}

\begin{proof}
Choose a filtration
\[
\rI_{n_0}^c=\rI^{(s)}<\rI^{(s-1)}<\cdots<\rI^{(1)}<\rI^{(0)}<\rI_{n_0}^0
\]
of groups such that $p\nmid|\rI_{n_0}^0/\rI^{(0)}|$ and $|\rI^{(j-1)}/\rI^{(j)}|=p$ for $1\leq j\leq s$. For $0\leq j\leq s$, let $\rk_0^{(j)}$ be the subgroup of $\rk_0^{(0,0)}$ of elements whose reduction modulo $p^l$ belong to $\rI^{(j)}\prod_{v\in\PP^+}\sfU_{n_0,v}(O_{F^+_v})$. We prove the lemma for $\rk_0^{(j)}$ by induction on $j$.

We first consider the case where $j=0$. Similar to the proof of \cite{LTXZZ}*{Theorem~6.3.4}, using Lemma \ref{le:first1}, it suffices to show that the following two finite abelian groups
\[
\Gamma(\Sh(\sfG_{n_0},\rK_{n_0}\rk_0^{(0)}),\dZ_{\xi^{n_0}}^\vee)_{\fr_0}/\tC_\fl,\qquad
\rH^1_\sing(\dQ_{\ell^2},\rH^{2r_0-1}_\fT(\ol\rM_{n_0}(\rk_0^{(0)}),\rR\Psi\dZ_{\xi^{n_0}}^\vee(r_0))_{\fr_0})
\]
have the same cardinality. If $\rk_0^{(0)}=\rk_0^{0,0}$, then this can be proved using the same argument for the case of constant coefficients, which relies on the key result \cite{LTXZZ1}*{Theorem~3.6.3}. In general, \cite{LTXZZ1}*{Theorem~3.6.3} does not directly apply. However, using Lemma \ref{le:deformation}(1,2) below under Assumption \ref{as:galois}(G7,G8), we have the analogous statement of \cite{LTXZZ1}*{Theorem~3.6.3} when $\rk_0^{(0)}\neq\rk_0^{0,0}$ as well, by the same proof.

For $j>0$, it suffices to prove (1). Indeed, (2) follows from (1) easily; (3) follows from (1) by Lemma \ref{le:first1}(3); and (4) follows from (1) by \eqref{eq:first2} and the fact that the map in (4) is an isomorphism after inverting $p$. In other words, we need to show by induction on $j$ that
\begin{itemize}
  \item[($*$)] $\rH^i_\fT(\ol\rM_{n_0}^?(\rk_0^{(j)}),\dZ_{\xi^{n_0}}^\vee)_{\fr_0}$ is $p$-torsion free for every $i\in\dZ$ and $?\in\{\circ,\bullet,\dag\}$.
\end{itemize}
When $j=0$, it has already been proved. Now take $0<j\leq s$ and assume that ($*$) is known for $j-1$. We need to show that $\rH^i_\fT(\ol\rM_{n_0}^?(\rk_0^{(j)}),\dZ_{\xi^{n_0}}^\vee)_{\fr_0}[p]$ vanishes. Now the natural morphism $\rM_{n_0}^?(\rk_0^{(j)})\to\rM_{n_0}^?(\rk_0^{(j-1)})$ is a Galois cover with Galois group $\rG_j\coloneqq\rk_0^{(j-1)}/\rk_0^{(j)}=\rI^{(j-1)}/\rI^{(j)}\simeq\dZ/p\dZ$. By the Hochschild--Serre spectral sequence, we have a short exact sequence
\[
0 \to \rH^1(\rG_j,\rH^{i-1}_\fT(\ol\rM_{n_0}^?(\rk_0^{(j)}),\dZ_{\xi^{n_0}}^\vee)_{\fr_0})
\to \rH^i_\fT(\ol\rM_{n_0}^?(\rk_0^{(j-1)}),\dZ_{\xi^{n_0}}^\vee)_{\fr_0}
\to \rH^0(\rG_j,\rH^i_\fT(\ol\rM_{n_0}^?(\rk_0^{(j)}),\dZ_{\xi^{n_0}}^\vee)_{\fr_0})
\to 0.
\]
By the induction hypothesis, $\rH^i_\fT(\ol\rM_{n_0}^?(\rk_0^{(j-1)}),\dZ_{\xi^{n_0}}^\vee)_{\fr_0}$ is $p$-torsion free, which implies that
\[
\rH^1(\rG_j,\rH^{i-1}_\fT(\ol\rM_{n_0}^?(\rk_0^{(j)}),\dZ_{\xi^{n_0}}^\vee)_{\fr_0})=0,\quad
\rH^i_\fT(\ol\rM_{n_0}^?(\rk_0^{(j-1)}),\dZ_{\xi^{n_0}}^\vee)_{\fr_0}
=\rH^0(\rG_j,\rH^i_\fT(\ol\rM_{n_0}^?(\rk_0^{(j)}),\dZ_{\xi^{n_0}}^\vee)_{\fr_0}).
\]
It follows that
\begin{align*}
\rH^0(\rG_j,\rH^i_\fT(\ol\rM_{n_0}^?(\rk_0^{(j)}),\dZ_{\xi^{n_0}}^\vee)_{\fr_0}[p])
=\rH^0(\rG_j,\rH^i_\fT(\ol\rM_{n_0}^?(\rk_0^{(j)}),\dZ_{\xi^{n_0}}^\vee)_{\fr_0})[p]
=\rH^i_\fT(\ol\rM_{n_0}^?(\rk_0^{(j-1)}),\dZ_{\xi^{n_0}}^\vee)_{\fr_0}[p],
\end{align*}
which vanishes by the induction hypothesis. Since $\rG_j\simeq\dZ/p\dZ$, we have $\rH^i_\fT(\ol\rM_{n_0}^?(\rk_0^{(j)}),\dZ_{\xi^{n_0}}^\vee)_{\fr_0}[p]=0$.

The lemma is proved.
\end{proof}

We are finally at step (3). We say that an open compact subgroup $\rk$ of $\sfG(F^+_{\PP^+})$ is \emph{Cartesian} if it is of the form $\rk_0\times\rk_1$ with $\rk_\alpha\in\sfG_{n_\alpha}(F^+_{\PP^+})$.

We first take $\rk$ that is Cartesian. We have the corresponding schemes for the product Shimura variety $\Sh(\sfG',\rK'\tj\rk)$ from \cite{LTXZZ}*{\S5.11} such as $\bP(\rk)$, $\sigma(\rk)\colon\bQ(\rk)\to\bP(\rk)$, etc. Moreover, we have similar constructions from \cite{LTXZZ}*{Construction~5.11.7} for the local system $\dZ_{\xi}^\vee/p^m$ for $1\leq m\leq \infty$ instead of the constant coefficients. Let $(\dE^{p,q}_s(\rk),\rd^{p,q}_s(\rk))$ be the spectral sequence abutting to $\rH^{p+q}_\fT(\ol\rQ(\rk),\rR\Psi\dZ_{\xi}^\vee(r))$; and similarly we have natural maps
\begin{align*}
\Delta(\rk) &\colon C_n(\rQ(\rk),\dZ_{\xi}^\vee)\to C^n(\rQ(\rk),\dZ_{\xi}^\vee), \\
\nabla(\rk) &\colon C^n(\rQ(\rk),\dZ_{\xi}^\vee) \to \Gamma(\Sh(\sfG,\rK\rk),\dZ_{\xi}^\vee),
\end{align*}
as in \cite{LTXZZ}*{Construction~5.11.7~\&~Remark~5.11.8}.

\begin{lem}\label{le:first3}
Assume Assumption \ref{as:law} and $\Gamma(\Sh(\sfG_{n_0},\rK_{n_0}\rk_0^{0,0}),\dZ_{\xi^{n_0}}^\vee/p)_{\fr_0}\neq 0$. Suppose that $\rk=\rk_0\times\rk_1$ is Cartesian in which $\rk_0$ is of the form $\rk_0^{c,l}$ as in Lemma \ref{le:first2}.
\begin{enumerate}
  \item We have $\dE^{p,q}_2(\rk)_{\fr}=0$ if $(p,q)\not\in\{(-1,2n),(0,2n-1),(1,2n-2)\}$, and canonical $\Gal(\ol\dF_\ell/\dF_{\ell^2})$-equivariant isomorphisms
      \[
      \begin{dcases}
      \dE^{-1,2n}_2(\rk)_{\fr}\simeq
      \pres{0}\rE^{-1,2r_0}_2(\rk_0)_{\fr_0}\otimes_{\dZ_q}\pres{1}\rE^{0,2r_1}_2(\rk_1)_{\fr_1}, \\
      \dE^{0,2n-1}_2(\rk)_{\fr}\simeq
      \pres{0}\rE^{0,2r_0-1}_2(\rk_0)_{\fr_0}\otimes_{\dZ_q}\pres{1}\rE^{0,2r_1}_2(\rk_1)_{\fr_1}, \\
      \dE^{1,2n-2}_2(\rk)_{\fr}\simeq
      \pres{0}\rE^{1,2r_0-2}_2(\rk_0)_{\fr_0}\otimes_{\dZ_q}\pres{1}\rE^{0,2r_1}_2(\rk_1)_{\fr_1},
      \end{dcases}
      \]
      of finite free $\dZ_q$-modules.

  \item If $\dE^{i,2n-1-i}_2(\rk)_{\fr}(-1)$ has a nontrivial subquotient on which $\Gal(\ol\dF_\ell/\dF_{\ell^2})$ acts trivially, then $i=1$.

  \item The map $\nabla(\rk)$ induces an isomorphism
      \[
      \nabla(\rk)_{\fr}\colon\coker\Delta(\rk)_{\fr}
      \xrightarrow\sim\Gamma(\Sh(\sfG,\rK\rk),\dZ_{\xi}^\vee)_{\fr}/\tC_\fl.
      \]

  \item There is a canonical isomorphism
      \[
      \rH^1_\sing(\dQ_{\ell^2},\rH^{2n-1}_\fT(\ol\rQ(\rk),\rR\Psi\dZ_\xi^\vee(n))_{\fr})
      \simeq\coker\Delta(\rk)_{\fr}.
      \]
\end{enumerate}
\end{lem}

\begin{proof}
These are analogues of \cite{LTXZZ}*{Lemma~7.2.5(2),~Lemma~7.2.5(3),~Theorem~7.2.8(1,2),~\&~Proposition~7.2.7}, respectively, which can be proved by the same argument (now for the local system $\dZ_\xi^\vee$) using Lemma \ref{le:first0} and Lemma \ref{le:first2}.
\end{proof}

As
\begin{align*}
\rH^1(F_{w(\fl)},\rH^{2n-1}_{\et}(\Sh(\sfG',\rK'\tj\rk)_{\ol{F}},\dZ_\xi^\vee(n))_{\gamma^\tc})
&=\rH^1_\sing(\dQ_{\ell^2},\rH^{2n-1}_\fT(\ol\rQ(\rk),\rR\Psi\dZ_\xi^\vee(n))_{\fr}), \\
\Gamma(\Sh(\sfG,\rK\rk),\dZ_{\xi}^\vee)_{\gamma^\tc}
&=\Gamma(\Sh(\sfG,\rK\rk),\dZ_{\xi}^\vee)_{\fr},
\end{align*}
combining Lemma \ref{le:first3} (3,4), we obtain an isomorphism
\[
\varrho(\rk)\colon
\rH^1(F_{w(\fl)},\rH^{2n-1}_{\et}(\Sh(\sfG',\rK'\tj\rk)_{\ol{F}},\dZ_\xi^\vee(n))_{\gamma^\tc})
\xrightarrow\sim\Gamma(\Sh(\sfG,\rK\rk),\dZ_{\xi}^\vee)_{\gamma^\tc}/\tC_\fl
\]
for every $\rk$ of the form $\rk_0^{c,l}\times\rk_1$. The isomorphism $\varrho(\rk)$ enjoys the following functorial property: For another group $\tilde\rk=\rk_0^{\tilde{c},\tilde{l}}\times\tilde\rk_1$ satisfying $\rk_0^{\tilde{c},\tilde{l}}\to\rk_0^{c,l}$\footnote{This is equivalent to that $c\leq\tilde{c}$ and $l\geq\tilde{l}$.} and $\tilde\rk_1\to\rk_1$ so that $\tilde\rk\to\rk$ (Construction \ref{co:liusun}), the diagram
\begin{align}\label{eq:first3}
\xymatrix{
\rH^1(F_{w(\fl)},\rH^{2n-1}_{\et}(\Sh(\sfG',\rK'\tj\tilde\rk)_{\ol{F}},\dZ_\xi^\vee(n))_{\gamma^\tc})
\ar[d]_-{\rho_{\tj\rk,\tj\tilde\rk}^\vee}\ar[r]^-{\varrho(\tilde\rk)} & \Gamma(\Sh(\sfG,\rK\tilde\rk),\dZ_{\xi}^\vee)_{\gamma^\tc}/\tC_\fl \ar[d]^-{\rho_{\rk,\tilde\rk}^\vee} \\
\rH^1(F_{w(\fl)},\rH^{2n-1}_{\et}(\Sh(\sfG',\rK'\tj\rk)_{\ol{F}},\dZ_\xi^\vee(n))_{\gamma^\tc})
\ar[r]^-{\varrho(\rk)} & \Gamma(\Sh(\sfG,\rK\rk),\dZ_{\xi}^\vee)_{\gamma^\tc}/\tC_\fl
}
\end{align}
commutes.

We now construct a preliminary version of \eqref{eq:reciprocity_first}. Take a subgroup $\rI\finite\rI^0$. Choose an element $\rk\in\fk_\rI^\dag$ such that one can find a subgroup $\tilde\rI=\rI_0^c\times\rI_1<\rI$ for some $c\geq 0$ and $\rI_1\finite\rI_{n_1}^0$, and an element $\tilde\rk\in\fk_{\tilde\rI}^\dag$ of the form $\rk_0^{c,l}\times\rk_1$ that is a normal subgroup of $\rk$ and satisfies $\rk=\tilde\rk\cdot\rk_\sfB$ (so that $\tilde\rk\to\rk$). The map
\[
\rho_{\rk,\tilde\rk}^\vee\colon\Gamma(\Sh(\sfG,\rK\tilde\rk),\dZ_{\xi}^\vee)_{\gamma^\tc}
\to\Gamma(\Sh(\sfG,\rK\rk),\dZ_{\xi}^\vee)_{\gamma^\tc},
\]
which is nothing but the pushforward map, induces an isomorphism
\[
\(\Gamma(\Sh(\sfG,\rK\tilde\rk),\dZ_{\xi}^\vee)_{\gamma^\tc}\)_{\rk/\tilde\rk}
\xrightarrow\sim\Gamma(\Sh(\sfG,\rK\rk),\dZ_{\xi}^\vee)_{\gamma^\tc}
\]
hence
\[
\(\Gamma(\Sh(\sfG,\rK\tilde\rk),\dZ_{\xi}^\vee)_{\gamma^\tc}/\tC_\fl\)_{\rk/\tilde\rk}
\xrightarrow\sim\Gamma(\Sh(\sfG,\rK\rk),\dZ_{\xi}^\vee)_{\gamma^\tc}/\tC_\fl.
\]
On the other hand, the map
\[
\rho_{\tj\rk,\tj\tilde\rk}^\vee\colon\rH^{2n-1}_{\et}(\Sh(\sfG',\rK'\tj\tilde\rk)_{\ol{F}},\dZ_\xi^\vee(n))_{\gamma^\tc}
\to\rH^{2n-1}_{\et}(\Sh(\sfG',\rK'\tj\rk)_{\ol{F}},\dZ_\xi^\vee(n))_{\gamma^\tc},
\]
which is nothing but the pushforward map, induces an isomorphism
\[
\(\rH^{2n-1}_{\et}(\Sh(\sfG',\rK'\tj\tilde\rk)_{\ol{F}},\dZ_\xi^\vee(n))_{\gamma^\tc}\)_{\rk/\tilde\rk}
\xrightarrow\sim\rH^{2n-1}_{\et}(\Sh(\sfG',\rK'\tj\rk)_{\ol{F}},\dZ_\xi^\vee(n))_{\gamma^\tc}.
\]
Lemma \ref{le:first3}(2) implies that the natural map
\[
\rH^1(F_{w(\fl)},\rH^{2n-1}_{\et}(\Sh(\sfG',\rK'\tj\tilde\rk)_{\ol{F}},\dZ_\xi^\vee(n))_{\gamma^\tc})_{\rk/\tilde\rk}
\to\rH^1\(F_{w(\fl)},\(\rH^{2n-1}_{\et}(\Sh(\sfG',\rK'\tj\tilde\rk)_{\ol{F}},\dZ_\xi^\vee(n))_{\gamma^\tc}\)_{\rk/\tilde\rk}\)
\]
is an isomorphism. Thus, by taking $\rk/\tilde\rk$-invariant quotients on both sides of $\varrho(\tilde\rk)$, we obtain an isomorphism
\[
\varrho(\rk)\colon\rH^1(F_{w(\fl)},\rH^{2n-1}_{\et}(\Sh(\sfG',\rK'\tj\rk)_{\ol{F}},\dZ_\xi^\vee(n))_{\gamma^\tc})
\xrightarrow\sim\Gamma(\Sh(\sfG,\rK\rk),\dZ_{\xi}^\vee)_{\gamma^\tc}/\tC_\fl,
\]
which does not depend on the choice of $\tilde\rI$ and $\tilde\rk$ by the commutativity of \eqref{eq:first3}.

Now by Remark \ref{re:ordinary}, the two vertical maps in the diagram
\[
\xymatrix{
\rH^1(F_{w(\fl)},\rH^{2n-1}_{\et}(\Sh(\sfG',\rK'\tj\rk)_{\ol{F}},\dZ_\xi^\vee(n))_{\gamma^\tc})
\ar[d]\ar[r]^-{\varrho(\rk)} & \Gamma(\Sh(\sfG,\rK\rk),\dZ_{\xi}^\vee)_{\gamma^\tc}/\tC_\fl \ar[d] \\
\rH^1\(F_{w(\fl)},\Hom_{\dZ_q}\(\rH(\Sh(\sfG',\rK'\tj\rI),\dZ_\xi)_\gamma,\dZ_q\)\)
\ar@{-->}[r]^-{\varrho_\rI} & \Hom_{\dZ_q}\(\rH(\Sh(\sfG,\rK\rI),\dZ_\xi)_\gamma,\dZ_q\)/\tC_\fl
}
\]
induced by duality are both surjective, and there exists a unique isomorphism $\varrho_\rI$ as the dashed arrow rendering this diagram commutative. Again, the commutativity of \eqref{eq:first3} implies that $\varrho_\rI$ depends only on $\rI$ and is functorial in $\rI$ with respect to the dual of $\rho_{\rI',\rI}$. Thus, passing to the limit, we obtain an isomorphism
\[
\varrho_\rJ\coloneqq\varprojlim_{\rJ<\rI\finite\rI^0}\varrho_\rI\colon
\rH^1_\sing(F_{w(\fl)},\cD_\rJ(\xi,\bV')_\gamma)
\xrightarrow{\sim}\cD_\rJ(\xi,\bV)_\gamma/\tC_\fl.
\]
In view of Lemma \ref{le:first4}, we define $\varrho^\sing_\rJ(\bV',\bV)$ \eqref{eq:reciprocity_first} to be
\[
((\ell+1)\cdot\tI^\circ_{n_0,\fl}\otimes\tT^\circ_{n_1,\fl})^{-1}\cdot\varrho_\rJ
\]
in which $(\ell+1)\cdot\tI^\circ_{n_0,\fl}\otimes\tT^\circ_{n_1,\fl}$ is an automorphism of $\cD_\rJ(\xi,\bV)_\gamma/\tC_\fl$ by Definition \ref{de:prime}(P1,P3,P4).

\begin{lem}\label{le:first4}
Assume Assumption \ref{as:law} and $\cD_\rJ(\xi,\bV)_\gamma\neq 0$. For $\xi$ interlacing with a fixed $\xi$-distinction (Definition \ref{de:distinction}), we have
\[
\varrho_\rJ\(\partial_{w(\fl)}\loc_{w(\fl)}\(\bkappa_\rJ(\bV')\)\)=
(\ell+1)\cdot\tI^\circ_{n_0,\fl}\otimes\tT^\circ_{n_1,\fl}\(\blambda_\rJ(\bV)\).
\]
\end{lem}

\begin{proof}
It suffices to prove the lemma for $\rJ=0$. Then it suffices to prove the lemma for $\rJ=\rI$ of the form $\rI_0^c\times\rI_1$ with $c\geq 0$ and $\rI_1\finite\rI_{n_1}^0$. It is possible to choose an element $\tilde\rk\in\fk^\sfH_\rI$, and an element $\rk\in\fk^\dag_\rI$ of the form $\rk_0^{c,l}\times\rk_1$, such that $\tilde\rk\to\rk$. By construction, we need to show that the following two elements
\[
\varrho(\tilde\rk)\(\partial_{w(\fl)}\loc_{w(\fl)}\(\bkappa_{\tilde\rk}(\bV')\)\),\qquad
(\ell+1)\cdot\tI^\circ_{n_0,\fl}\otimes\tT^\circ_{n_1,\fl}\(\blambda_{\tilde\rk}(\bV)\)
\]
have the same image in
\[
\Hom_{\dZ_q}\(\rH(\Sh(\sfG,\rK\rI),\dZ_\xi)_\gamma,\dZ_q\)/\tC_\fl.
\]
For this, it suffices to show that
\[
\rho_{\rk,\tilde\rk}^\vee\(\varrho(\tilde\rk)\(\partial_{w(\fl)}\loc_{w(\fl)}\(\bkappa_{\tilde\rk}(\bV')\)\)\)
=\rho_{\rk,\tilde\rk}^\vee\((\ell+1)\cdot\tI^\circ_{n_0,\fl}\otimes\tT^\circ_{n_1,\fl}\(\blambda_{\tilde\rk}(\bV)\)\),
\]
which is equivalent to that
\[
\varrho(\rk)\(\partial_{w(\fl)}\loc_{w(\fl)}\(\rho_{\rk,\tilde\rk}^\vee\(\bkappa_{\tilde\rk}(\bV')\)\)\)
=(\ell+1)\cdot\tI^\circ_{n_0,\fl}\otimes\tT^\circ_{n_1,\fl}\(\rho_{\rk,\tilde\rk}^\vee\(\blambda_{\tilde\rk}(\bV)\)\).
\]
However, the above identity can be proved by the same argument for \cite{LTXZZ}*{Theorem~7.2.8(3)}.
\end{proof}

\begin{theorem}[First explicit reciprocity law]\label{th:first}
Assume Assumption \ref{as:law} and $\cD_\rJ(\xi,\bV)_\gamma\neq 0$. For $\xi$ interlacing with a fixed $\xi$-distinction (Definition \ref{de:distinction}),
\[
\varrho^\sing_\rJ(\bV',\bV)\(\partial_{w(\fl)}\loc_{w(\fl)}\(\bkappa_\rJ(\bV')\)\)=\blambda_\rJ(\bV)
\]
holds in $\cD_\rJ(\xi,\bV)_\gamma/\tC_\fl$.
\end{theorem}

\begin{proof}
This follows from the definition and Lemma \ref{le:first4}.
\end{proof}

\begin{remark}\label{re:singular}
Lemma \ref{le:first3}(2) implies that under Assumption \ref{as:law}, $\rH^1_\sing(F_{w(\fl)},\cD_\rJ(\xi,\bV')_\gamma)$, which a priori is a quotient $\cE_\rJ(\xi,\bV')_\gamma$-module of $\rH^1(\rI_{F_{w(\fl)}},\cD_\rJ(\xi,\bV')_\gamma)$, is indeed a $\cE_\rJ(\xi,\bV')_\gamma$-linear direct summand.
\end{remark}

\section{Second explicit reciprocity law for ordinary distributions}
\label{ss:second_reciprocity}

In this section, we establish the second explicit reciprocity law for ordinary distributions. We take an \emph{indefinite} datum $\bV$ from Definition \ref{de:datum} and choose a level-raising prime $\fl$ of $F^+$ with respect to $\bV$ (Definition \ref{de:prime}).

Consider another datum $\bV'=(\fm';\rV'_n,\rV'_{n+1};\Lambda'_n,\Lambda'_{n+1};\rK_n^{\prime},\rK_{n+1}^{\prime};\sfB')$ and a datum $\tj=(\tj_n,\tj_{n+1})$ in which
\begin{itemize}
  \item $\fm'=\fm\cup\{\fl\}$;

  \item $\tj_N\colon\rV_N\otimes_{F^+}\dA_{F^+}^{\infty,\fl}\to\rV'_N\otimes_{F^+}\dA_{F^+}^{\infty,\fl}$ ($N=n,n+1$) are isometries sending $(\Lambda_N)^\fl$ to $(\Lambda'_N)^\fl$ and $(\rK_N)^\fl$ to $(\rK_N^{\prime})^\fl$ for $N\in\{n,n+1\}$, satisfying $\tj_{n+1}=(\tj_n)_\sharp$ and $\tj_p\sfB=\sfB'$.
\end{itemize}
In particular, $\bV'$ is a definite datum.

Similar to $\bV$, we have groups $\sfG'$, $\sfH'$, $\rK^{\prime}$, $\rK^{\prime}_{\sfH'}$ for $\bV'$. Moreover, we identify $\sfG'\otimes_\dQ\dQ_p$ with $\sfG\otimes_\dQ\dQ_p$ and $\sfB'$ with $\sfB$ via $\tj_p$.

Our main goal is to, under Assumption \ref{as:galois}(G1,G2), establish a natural map
\begin{align}\label{eq:reciprocity_second}
\varrho^\unr_\rJ(\bV',\bV)&\colon\cD_\rJ(\xi,\bV')_\gamma
\to\rH^1_\unr(F_{w(\fl)},\cD_\rJ(\xi,\bV)_\gamma)
\end{align}
of $\dZ[\rK_{\spadesuit^+}\backslash\sfG(F^+_{\spadesuit^+})/\rK_{\spadesuit^+}]\otimes
\dT^{\ang{\fm'}}\otimes\Lambda_{\sfT(F^+_{\PP^+})/\rJ,\dZ_q}$-modules.

In what follows, $\alpha$ will be an index from $\{0,1\}$. Put
\[
\fr_\alpha\coloneqq\Ker\phi_{\gamma^\tc_{n_\alpha}}\cap\dT_{n_\alpha}^{\ang{\fm'}},
\]
and let $\fr$ be the (maximal) ideal of $\dT_{n_\alpha}^{\ang{\fm'}}$ generated by $\fr_0$ and $\fr_1$. Let $\rk_\alpha$ be an open compact subgroup of $\sfG_{n_\alpha}(F^+_{\PP^+})$. We have the moduli scheme $\bM_{n_\alpha}(\rk_\alpha)$ over $O_{F_\fl}\simeq\dZ_{\ell^2}$ which is essentially an integral model of $\Sh(\sfG_{n_\alpha},\rK_{n_\alpha}\rk_\alpha)$. Using Gysin maps with coefficients, we have the following two maps
\begin{enumerate}
  \item when $\alpha=1$, the Tate cycle class map
      \[
      (\iota_{n_1})_!\circ\pi_{n_1}^*\colon\rH^0_\fT(\ol\rB_{n_1}(\rk_1),\dZ_{\xi^{n_1}}^\vee)
      \to\rH^{2r_1}_\fT(\ol\rM_{n_1}(\rk_1),\dZ_{\xi^{n_1}}^\vee(r_1))_{\Gal(\ol\dF_\ell/\dF_{\ell^2})}
      \]
      as in \cite{LTXZZ}*{\S4.6}.

  \item when $\alpha=0$, the localized Abel--Jacobi map
      \[
      \mho_{0,\fr_0}\colon\rH^0_\fT(\ol\rB_{n_0}(\rk_0),\dZ_{\xi^{n_0}}^\vee)_{\fr_0}
      \to\rH^1(\dF_{\ell^2},\rH^{n_0-1}_\fT(\ol\rM_{n_0}(\rk_0),\dZ_{\xi^{n_0}}^\vee(r_0))_{\fr_0})
      \]
      under Assumption \ref{as:galois}(G2), similar to \cite{LTX}*{Remark~4.3.1}.
\end{enumerate}

We recall the uniformization data that relate $\rB_{n_\alpha}$ to the Shimura set associated with $\bV'$. We will freely use the notation from \cite{LTXZZ}*{\S7.3}. For $\alpha=0,1$, choose hermitian lattices $\{\Lambda^\star_{n_\alpha,v}\}_{v\mid\ell}$ satisfying
\begin{itemize}[label={\ding{118}}]
  \item $\Lambda'_{n_\alpha,\fl}\subseteq\Lambda^\star_{n_\alpha,\fl}\subseteq \ell^{-1}\Lambda'_{n_\alpha,\fl}$,

  \item $\ell\Lambda^\star_{n_\alpha,\fl}\subseteq(\Lambda^\star_{n_\alpha,\fl})^\vee$ such that $(\Lambda^\star_{n_\alpha,\fl})^\vee/\ell\Lambda^\star_{n_\alpha,\fl}$ has length $1-\alpha$,

  \item $(\Lambda^\star_{n,\fl})_\sharp\subseteq\Lambda^\star_{n+1,\fl}\subseteq\ell^{-1}(\Lambda^\star_{n,\fl})_\sharp^\vee$,

  \item $\Lambda^\star_{n_\alpha,v}=\Lambda'_{n_\alpha,v}$ for $\alpha=0,1$ and $v\neq\fl$.
\end{itemize}
Then $(\rV'_n,\tj_n,\{\Lambda^\star_{n,v}\}_{v\mid\ell})$ and $(\rV'_{n+1},\tj_{n+1},\{\Lambda^\star_{n+1,v}\}_{v\mid\ell})$ form definite uniformization data for $\rV_n$ and $\rV_{n+1}$ \cite{LTXZZ}*{Definition~4.4.1}, respectively, satisfying the conditions in \cite{LTXZZ}*{Notation~4.5.7}. For $\alpha=0,1$, let $\rK^\star_{n_\alpha,v}$ be the stabilizer of $\Lambda^\star_{n_\alpha,v}$; denote by
\[
\tT^{\star\prime}_{n_\alpha,\fl}
\in\dZ[\rK^\star_{n_\alpha,\fl}\backslash\sfG'_{n_\alpha}(F^+_\fl)/\rK'_{n_\alpha,\fl}],\quad
\tT^{\prime\star}_{n_\alpha,\fl}
\in\dZ[\rK^\prime_{n_\alpha,\fl}\backslash\sfG'_{n_\alpha}(F^+_\fl)/\rK^\star_{n_\alpha,\fl}]
\]
the characteristic functions of $\rK^\star_{n_\alpha,\fl}\rK'_{n_\alpha,\fl}$ and $\rK'_{n_\alpha,\fl}\rK^\star_{n_\alpha,\fl}$, respectively; and put
\[
\tI'_{n_\alpha,\fl}\coloneqq\tT^{\prime\star}_{n_\alpha,\fl}\circ\tT^{\star\prime}_{n_\alpha,\fl}
\in\dZ[\rK'_{n_\alpha,\fl}\backslash\sfG'_{n_\alpha}(F^+_\fl)/\rK'_{n_\alpha,\fl}].
\]
Finally, put $\rK^\star_{n_\alpha}\coloneqq\tj_{n_\alpha}\rK_{n_\alpha}^\ell\times\prod_{v\mid\ell}\rK^\star_{n_\alpha,v}$. Then we have canonical isomorphisms
\[
\rH^0_\fT(\ol\rB_{n_\alpha}(\rk_\alpha),\dZ_{\xi^{n_\alpha}}^\vee)\simeq
\Gamma(\Sh(\sfG'_{n_\alpha},\rK^\star_{n_\alpha}\tj_{n_\alpha}\rk_\alpha),\dZ_{\xi^{n_\alpha}}^\vee)
\]
for $\alpha=0,1$. By the same proof of \cite{LTX}*{Lemma~4.3.4}, the localized Hecke operators
\begin{align*}
(\tT^{\star\prime}_{n_\alpha,\fl})_{\fr_\alpha}&\colon
\Gamma(\Sh(\sfG'_{n_\alpha},\rK'_{n_\alpha}\tj_{n_\alpha}\rk_\alpha),\dZ_{\xi^{n_\alpha}}^\vee)_{\fr_\alpha}
\to\Gamma(\Sh(\sfG'_{n_\alpha},\rK^\star_{n_\alpha}\tj_{n_\alpha}\rk_\alpha),\dZ_{\xi^{n_\alpha}}^\vee)_{\fr_\alpha} \\
(\tT^{\prime\star}_{n_\alpha,\fl})_{\fr_\alpha}&\colon
\Gamma(\Sh(\sfG'_{n_\alpha},\rK^\star_{n_\alpha}\tj_{n_\alpha}\rk_\alpha),\dZ_{\xi^{n_\alpha}}^\vee)_{\fr_\alpha}
\to\Gamma(\Sh(\sfG'_{n_\alpha},\rK'_{n_\alpha}\tj_{n_\alpha}\rk_\alpha),\dZ_{\xi^{n_\alpha}}^\vee)_{\fr_\alpha}
\end{align*}
are both isomorphisms for $\alpha=0,1$.

For $\rk=\rk_0\times\rk_1$ so that
\begin{align*}
\Gamma(\Sh(\sfG',\rK'\tj\rk),\dZ_\xi^\vee)_{\gamma^\tc}=
\Gamma(\Sh(\sfG'_{n_0},\rK'_{n_0}\tj_{n_0}\rk_0),\dZ_{\xi^{n_0}}^\vee)_{\fr_0}
\otimes_{\dZ_q}\Gamma(\Sh(\sfG'_{n_1},\rK'_{n_1}\tj_{n_1}\rk_1),\dZ_{\xi^{n_1}}^\vee)_{\fr_1},
\end{align*}
we define the map
\[
\varrho(\rk)\colon\Gamma(\Sh(\sfG',\rK'\tj\rk),\dZ_\xi^\vee)_{\gamma^\tc}
\to\rH^1_\unr(F_{w(\fl)},\rH^{2n-1}_{\et}(\Sh(\sfG,\rK\rk)_{\ol{F}},\dZ_\xi^\vee(n))_{\gamma^\tc})
\]
to be the composition of the following four maps:
\begin{itemize}
  \item the isomorphism
      \begin{align*}
      (\tT^{\prime\star}_{n_0,\fl})^{-1}_{\fr_0}\otimes(\tT^{\star\prime}_{n_1,\fl})_{\fr_1}
      \colon \Gamma(\Sh(\sfG',\rK'\tj\rk),\dZ_\xi^\vee)_{\gamma^\tc}
      \xrightarrow{\sim} \Gamma(\Sh(\sfG',\rK^\star\tj\rk),\dZ_\xi^\vee)_{\gamma^\tc},
      \end{align*}

  \item the \emph{inverse} of the endomorphism
      \[
      1\otimes\tT^\star_{n_1,\fl}\colon
      \Gamma(\Sh(\sfG',\rK^\star\tj\rk),\dZ_\xi^\vee)_{\gamma^\tc}\to
      \Gamma(\Sh(\sfG',\rK^\star\tj\rk),\dZ_\xi^\vee)_{\gamma^\tc},
      \]
      which is invertible by the claim (1) in the proof of \cite{LTXZZ}*{Theorem~7.3.4},

  \item the map
      \begin{align*}
      \mho_{0,\fr_0}\otimes((\iota_{n_1})_!\circ\pi_{n_1}^*)_{\fr_1}
      &\colon \Gamma(\Sh(\sfG',\rK^\star\tj\rk),\dZ_\xi^\vee)_{\gamma^\tc} \\
      &\to\rH^1(\dF_{\ell^2},\rH^{2r_0-1}_\fT(\ol\rM_{n_0}(\rk_0),\dZ_\xi^\vee(r_0))_{\fr_0})\otimes_{\dZ_q}
      \(\rH^{2r_1}_\fT(\ol\rM_{n_1}(\rk_1),\dZ_\xi^\vee(r_1))_{\fr_1}\)_{\Gal(\ol\dF_\ell/\dF_{\ell^2})},
      \end{align*}
      and

  \item the natural isomorphism
      \begin{align*}
      &\rH^1(\dF_{\ell^2},\rH^{2r_0-1}_\fT(\ol\rM_{n_0}(\rk_0),\dZ_\xi^\vee(r_0))_{\fr_0})\otimes_{\dZ_q}
      \(\rH^{2r_1}_\fT(\ol\rM_{n_1}(\rk_1),\dZ_\xi^\vee(r_1))_{\fr_1}\)_{\Gal(\ol\dF_\ell/\dF_{\ell^2})} \\
      &\xrightarrow\sim
      \rH^1_\unr(F_{w(\fl)},\rH^{2n-1}_{\et}(\Sh(\sfG,\rK\rk)_{\ol{F}},\dZ_\xi^\vee(n))_{\gamma^\tc})
      \end{align*}
      induced by the K\"{u}nneth decomposition (and Assumption \ref{as:galois}(G2)).
\end{itemize}

The map $\varrho(\rk)$ enjoys the following functorial property: For another group $\tilde\rk=\tilde\rk_0\times\tilde\rk_1$ satisfying $\tilde\rk\to\rk$ (Construction \ref{co:liusun}), the diagram
\begin{align}\label{eq:second1}
\xymatrix{
\Gamma(\Sh(\sfG',\rK'\tj\tilde\rk),\dZ_\xi^\vee)_{\gamma^\tc}
\ar[d]_-{\rho_{\tj\rk,\tj\tilde\rk}^\vee}\ar[r]^-{\varrho(\tilde\rk)} &
\rH^1_\unr(F_{w(\fl)},\rH^{2n-1}_{\et}(\Sh(\sfG,\rK\tilde\rk)_{\ol{F}},\dZ_\xi^\vee(n))_{\gamma^\tc}) \ar[d]^-{\rho_{\rk,\tilde\rk}^\vee} \\
\Gamma(\Sh(\sfG',\rK'\tj\rk),\dZ_\xi^\vee)_{\gamma^\tc}
\ar[r]^-{\varrho(\rk)} & \rH^1_\unr(F_{w(\fl)},\rH^{2n-1}_{\et}(\Sh(\sfG,\rK\rk)_{\ol{F}},\dZ_\xi^\vee(n))_{\gamma^\tc})
}
\end{align}
commutes.

We now construct \eqref{eq:reciprocity_second}. Take a subgroup $\rI\finite\rI^0$. Choose an element $\rk\in\fk_\rI^\dag$ such that one can find a subgroup $\tilde\rI\finite\rI$ and an element $\tilde\rk\in\fk_{\tilde\rI}^\dag$ of the form $\tilde\rk_0\times\tilde\rk_1$ that is a normal subgroup of $\rk$ and satisfies $\rk=\tilde\rk\cdot\rk_\sfB$ (so that $\tilde\rk\to\rk$). The map
\[
\rho_{\tj\rk,\tj\tilde\rk}^\vee\colon\Gamma(\Sh(\sfG',\rK'\tj\tilde\rk),\dZ_{\xi}^\vee)_{\gamma^\tc}
\to\Gamma(\Sh(\sfG',\rK'\tj\rk),\dZ_{\xi}^\vee)_{\gamma^\tc},
\]
which is nothing but the pushforward map, induces an isomorphism
\[
\(\Gamma(\Sh(\sfG',\rK'\tj\tilde\rk),\dZ_{\xi}^\vee)_{\gamma^\tc}\)_{\rk/\tilde\rk}
\xrightarrow\sim\Gamma(\Sh(\sfG',\rK'\tj\rk),\dZ_{\xi}^\vee)_{\gamma^\tc}.
\]
On the other hand, the map
\[
\rho_{\rk,\tilde\rk}^\vee\colon\rH^{2n-1}_{\et}(\Sh(\sfG,\rK\tilde\rk)_{\ol{F}},\dZ_\xi^\vee(n))_{\gamma^\tc}
\to\rH^{2n-1}_{\et}(\Sh(\sfG,\rK\rk)_{\ol{F}},\dZ_\xi^\vee(n))_{\gamma^\tc},
\]
which is nothing but the pushforward map, induces an isomorphism
\[
\(\rho_{\rk,\tilde\rk}^\vee\colon\rH^{2n-1}_{\et}(\Sh(\sfG,\rK\tilde\rk)_{\ol{F}},\dZ_\xi^\vee(n))_{\gamma^\tc}\)_{\rk/\tilde\rk}
\xrightarrow\sim\rH^{2n-1}_{\et}(\Sh(\sfG,\rK\rk)_{\ol{F}},\dZ_\xi^\vee(n))_{\gamma^\tc},
\]
hence an isomorphism
\[
\rH^1_\unr(F_{w(\fl)},\rH^{2n-1}_{\et}(\Sh(\sfG,\rK\tilde\rk)_{\ol{F}},\dZ_\xi^\vee(n))_{\gamma^\tc})_{\rk/\tilde\rk}
\xrightarrow\sim\rH^1_\unr(F_{w(\fl)},\rH^{2n-1}_{\et}(\Sh(\sfG,\rK\rk)_{\ol{F}},\dZ_\xi^\vee(n))_{\gamma^\tc}).
\]
Thus, by taking $\rk/\tilde\rk$-invariant quotients on both sides of $\varrho(\tilde\rk)$, we obtain a map
\[
\varrho(\rk)\colon\Gamma(\Sh(\sfG',\rK'\tj\rk),\dZ_\xi^\vee)_{\gamma^\tc}
\to\rH^1_\unr(F_{w(\fl)},\rH^{2n-1}_{\et}(\Sh(\sfG,\rK\rk)_{\ol{F}},\dZ_\xi^\vee(n))_{\gamma^\tc}),
\]
which does not depend on the choice of $\tilde\rI$ and $\tilde\rk$ by the commutativity of \eqref{eq:second1}.

Now by Remark \ref{re:ordinary}, the two vertical maps in the diagram
\[
\xymatrix{
\Gamma(\Sh(\sfG',\rK'\tj\rk),\dZ_\xi^\vee)_{\gamma^\tc}
\ar[d]\ar[r]^-{\varrho(\rk)} & \rH^1_\unr(F_{w(\fl)},\rH^{2n-1}_{\et}(\Sh(\sfG,\rK\rk)_{\ol{F}},\dZ_\xi^\vee(n))_{\gamma^\tc}) \ar[d] \\
\Hom_{\dZ_q}\(\rH(\Sh(\sfG',\rK'\tj\rI),\dZ_\xi)_\gamma,\dZ_q\)
\ar@{-->}[r]^-{\varrho_\rI} & \rH^1_\unr\(F_{w(\fl)},\Hom_{\dZ_q}\(\rH(\Sh(\sfG,\rK\rI),\dZ_\xi)_\gamma,\dZ_q\)\)
}
\]
induced by duality are both surjective, and there exists a unique isomorphism $\varrho_\rI$ as the dashed arrow rendering this diagram commutative. Again, the commutativity of \eqref{eq:second1} implies that $\varrho_\rI$ depends only on $\rI$ and is functorial in $\rI$ with respect to the dual of $\rho_{\rI',\rI}$. Thus, passing to the limit, we obtain a map
\[
\varrho^\unr_\rJ(\bV',\bV)\coloneqq\varprojlim_{\rJ<\rI\finite\rI^0}\varrho_\rI\colon
\cD_\rJ(\xi,\bV')_\gamma\to\rH^1_\unr(F_{w(\fl)},\cD_\rJ(\xi,\bV)_\gamma)
\]
as claimed in \eqref{eq:reciprocity_second}.

\begin{theorem}[Second explicit reciprocity law]\label{th:second}
Assume Assumption \ref{as:galois}(G1,G2). For $\xi$ interlacing with a fixed $\xi$-distinction (Definition \ref{de:distinction}),
\[
\varrho^\unr_\rJ(\bV',\bV)\(\blambda_\rJ(\bV')\)=\loc_{w(\fl)}\(\bkappa_\rJ(\bV)\)
\]
holds in $\rH^1_\unr(F_{w(\fl)},\cD_\rJ(\xi,\bV)_\gamma)$.
\end{theorem}

\begin{proof}
It suffices to prove the theorem for $\rJ=0$. Then it suffices to prove the theorem for $\rJ=\rI$ of the form $\rI_0\times\rI_1$ with $\rI_\alpha\finite\rI_{n_\alpha}^0$. It is possible to choose an element $\tilde\rk\in\fk^\sfH_\rI$, and an element $\rk\in\fk^\dag_\rI$ of the form $\rk_0\times\rk_1$, such that $\tilde\rk\to\rk$. By construction, we need to show that the following two elements
\[
\varrho(\tilde\rk)\(\blambda_{\tilde\rk}(\bV')\),\qquad
\loc_{w(\fl)}\(\bkappa_{\tilde\rk}(\bV)\)
\]
have the same image in
\[
\rH^1_\unr\(F_{w(\fl)},\Hom_{\dZ_q}\(\rH(\Sh(\sfG,\rK\rI),\dZ_\xi)_\gamma,\dZ_q\)\).
\]
For this, it suffices to show that
\[
\rho_{\rk,\tilde\rk}^\vee\(\varrho(\tilde\rk)\(\blambda_{\tilde\rk}(\bV')\)\)=
\rho_{\rk,\tilde\rk}^\vee\(\loc_{w(\fl)}\(\bkappa_{\tilde\rk}(\bV)\)\),
\]
which is equivalent to that
\[
\varrho(\rk)\(\rho_{\rk,\tilde\rk}^\vee\(\blambda_{\tilde\rk}(\bV')\)\)
=\loc_{w(\fl)}\(\rho_{\rk,\tilde\rk}^\vee\(\bkappa_{\tilde\rk}(\bV)\)\).
\]
However, the above identity can be proved by the same argument for \cite{LTX}*{Theorem~4.3.6} (which is essentially \cite{LTXZZ}*{Theorem~4.6.2}).
\end{proof}

\section{Iwasawa's main conjecture}

In this section, we state Iwasawa's main conjecture over the eigenvariety and state our main theorems.

\begin{definition}
Let $\sT$ be a coherent sheaf on a noetherian scheme $\sE$. For an irreducible normal open subscheme $\sE'$ of $\sE$ on which $\sT$ is torsion, we define the \emph{characteristic divisor} of $\sT$ along $\sE'$ to be
\[
\Char_{\sE'}(\sT)\coloneqq\sum_{z}\length_{\cO_{\sE',z}}(\sT_z)\cdot z\in\Div(\sE'),
\]
where $z$ runs over all closed points of $\sE'$ of codimension one.
\end{definition}

\begin{conjecture}[Iwasawa's main conjecture]\label{co:iwasawa}
Let $\bV$ be a datum from Definition \ref{de:datum} with $\fm=\emptyset$, $\xi=(\xi^n,\xi^{n+1})\in\Xi^n\times\Xi^{n+1}$ a pair of $\Sigma^+_p\setminus\PP^+$-trivial dominant hermitian weights (Definition \ref{de:weight0}), $\gamma$ a homomorphism as in \eqref{eq:gamma}, and $\rJ<\rI^0$ a subgroup. Suppose that
\begin{itemize}
  \item $\gamma$ satisfies Assumption \ref{as:galois}(G1), and (G2) when $\bV$ is indefinite;

  \item $\xi$ is interlacing (and fix a $\xi$-distinction) (Definition \ref{de:weight} and Definition \ref{de:distinction}).
\end{itemize}
Let $\sE'$ be an irreducible component of the normal locus of $\Spec\sE_\rJ(\xi,\bV)_\gamma$.
\begin{enumerate}
  \item Suppose that $\bV$ is definite. If $\blambda_\rJ(\bV)$ is nonzero on $\sE'$, then
      \begin{enumerate}
        \item $\rH^1_f(F,\sR_\rJ(\xi,\bV)_\gamma)$ vanishes over $\sE'$;

        \item $\sX_\rJ(\xi,\bV)_\gamma$ is torsion over $\sE'$;

        \item we have
            \[
            2\Char_{\sE'}\(\sD_\rJ(\xi,\bV)_\gamma^\sfH/\blambda_\rJ(\bV)\)
            =\Char_{\sE'}\(\sX_\rJ(\xi,\bV)_\gamma\).
            \]
      \end{enumerate}

  \item Suppose that $\bV$ is indefinite. If $\bkappa_\rJ(\bV)$ is nonzero on $\sE'$, then
      \begin{enumerate}
        \item both $\rH^1_f(F,\sR_\rJ(\xi,\bV)_\gamma)$ and $\rH^1_f(F,\sD_\rJ(\xi,\bV)_\gamma)^\sfH$ have generic rank one over $\sE'$;

        \item $\sX_\rJ(\xi,\bV)_\gamma$ has generic rank one over $\sE'$;

        \item we have
            \[
            2\Char_{\sE'}\(\rH^1_f(F,\sD_\rJ(\xi,\bV)_\gamma)^\sfH/\bkappa_\rJ(\bV)\)
            =\Char_{\sE'}\(\sX_\rJ(\xi,\bV)_\gamma^\tor\).
            \]
      \end{enumerate}
\end{enumerate}
Here, for an $\sE_\rJ(\xi,\bV)_\gamma$-module $\sM$, we denote by $\sM^\tor$ the maximal torsion submodule of $\sM$ over the \emph{normal} locus of $\Spec\sE_\rJ(\xi,\bV)_\gamma$.
\end{conjecture}

\begin{definition}\label{de:tempered}
Assume Assumption \ref{as:galois}(G1), and when $\bV$ is indefinite (G2). Put
\[
\sW_\rJ(\xi,\bV)_\gamma\coloneqq\bigcup_{w\in\PP}\supp\((\Gr^0_w\sR_\rJ(\xi,\bV)_\gamma)_{\Gamma_{F_w}}\),
\]
which is a Zariski closed subset of $\Spec\sE_\rJ(\xi,\bV)_\gamma$, and denote by $\sW_\rJ(\xi,\bV)_\gamma^1$ the Zariski closure of codimension one points of $\Spec\sE_\rJ(\xi,\bV)_\gamma$ contained in $\sW_\rJ(\xi,\bV)_\gamma$. Define the \emph{tempered locus} of $\Spec\sE_\rJ(\xi,\bV)_\gamma$ to be the complement of $\sW_\rJ(\xi,\bV)_\gamma^1$ in the normal locus of $\Spec\sE_\rJ(\xi,\bV)_\gamma$.
\end{definition}

\begin{remark}\label{re:tempered}
Assume Assumption \ref{as:galois}(G1), and when $\bV$ is indefinite (G2).
\begin{enumerate}
  \item The subset $\sW_\rJ(\xi,\bV)_\gamma$ contains no classical points. Indeed, for a classical point $x$ of $\Spec\sE_\rJ(\xi,\bV)_\gamma$, Remark \ref{re:pure} implies that $(\Gr^0_w\sR_\rJ(\xi,\bV)_\gamma\res_x)_{\Gamma_{F_w}}$ vanishes for every $w\in\PP$, so that $x$ does not belong to $\sW_\rJ(\xi,\bV)_\gamma$.

  \item For $N=\{n,n+1\}$, $v\in\PP^+$, and $0\leq i\leq N-1$, denote by $\rI_N^{[i]}$ the image of the composite homomorphism
      \[
      O_{F_w}^\times\to F_w^\times\xrightarrow{\inc_i}\sfT_N(F_w)\xrightarrow{\Nm_{F_w/F^+_v}}\sfT_N(F^+_v)
      \]
      in which $w$ is a place of $F$ above $v$ and $\inc_i$ is the inclusion of the factor in \eqref{eq:weight} indexed by $i$ (with $\sigma$ the natural embedding $F^+_v\to F_w$), whose image is contained in $\rI_N^0$ and is independent of the choice of $w$. Then $\sW_\rJ(\xi,\bV)_\gamma^1=\emptyset$ if the following holds: for every $v\in\PP^+$ and every $0\leq i\leq n-1$, the image of the natural map $\rI_n^{[i]}\times\rI_{n+1}^{[n-i]}\to\rI^0/\rJ$ is infinite.
\end{enumerate}
\end{remark}

\begin{theorem}\label{th:iwasawa}
Let $\bV,\xi,\gamma,\rJ$ be as in Conjecture \ref{co:iwasawa}. Suppose that
\begin{enumerate}[label=(\roman*)]
  \item $\gamma$ satisfies all of Assumption \ref{as:galois};

  \item $p>2(n_0+1)$;

  \item $\xi$ is interlacing (and fix a $\xi$-distinction) and Fontaine--Laffaille regular (Definition \ref{de:weight});

  \item $\xi^{n_0}$ is Fontaine--Laffaille regular (Definition \ref{de:weight0}).
\end{enumerate}
Let $\sE'$ be an irreducible component of the tempered locus of $\Spec\sE_\rJ(\xi,\bV)_\gamma$.
\begin{enumerate}
  \item Suppose that $\bV$ is definite. If $\blambda_\rJ(\bV)$ is nonzero on $\sE'$, then
      \begin{enumerate}
        \item $\rH^1_f(F,\sR_\rJ(\xi,\bV)_\gamma)$ vanishes over $\sE'$;

        \item $\sX_\rJ(\xi,\bV)_\gamma$ is torsion over $\sE'$;

        \item the divisor
            \[
            2\Char_{\sE'}\(\sD_\rJ(\xi,\bV)_\gamma^\sfH/\blambda_\rJ(\bV)\)
            -\Char_{\sE'}\(\sX_\rJ(\xi,\bV)_\gamma\)
            \]
            of $\sE'$ is effective.
      \end{enumerate}

  \item Suppose that $\bV$ is indefinite. If $\bkappa_\rJ(\bV)$ is nonzero on $\sE'$, then
      \begin{enumerate}
        \item both $\rH^1_f(F,\sR_\rJ(\xi,\bV)_\gamma)$ and $\rH^1_f(F,\sD_\rJ(\xi,\bV)_\gamma)^\sfH$ have generic rank one over $\sE'$;

        \item $\sX_\rJ(\xi,\bV)_\gamma$ has generic rank one over $\sE'$;

        \item the divisor
            \[
            2\Char_{\sE'}\(\rH^1_f(F,\sD_\rJ(\xi,\bV)_\gamma)^\sfH/\bkappa_\rJ(\bV)\)
            -\Char_{\sE'}\(\sX_\rJ(\xi,\bV)_\gamma^\tor\)
            \]
            of $\sE'$ is effective.
      \end{enumerate}
\end{enumerate}
\end{theorem}

Theorem \ref{th:iwasawa} does not really recover \cite{LTX}*{Theorem~1.2.3} due to the extra item (G8) in Assumption \ref{as:galois}. We now introduce a variant of Theorem \ref{th:iwasawa} without (G8) by restricting $\rJ$. Suppose that $\rJ$ has the form $\rI^0_{n_0}\times\rJ_1$ for a subgroup $\rJ_1<\rI^0_{n_1}$. Then
\[
\sE_\rJ(\xi,\bV)_\gamma=\sE_{\rI^0_{n_0}}(\xi^{n_0},\rV_{n_0},\rK_{n_0})_{\gamma_{n_0}}
\otimes_{\dQ_q}\sE_{\rJ_1}(\xi^{n_1},\rV_{n_1},\rK_{n_1})_{\gamma_{n_1}}
\]
in which $\sE_{\rI^0_{n_0}}(\xi^{n_0},\rV_{n_0},\rK_{n_0})_{\gamma_{n_0}}$ is regular of dimension zero. We say that a point $x$ of $\Spec\sE_{\rI^0_{n_0}}(\xi^{n_0},\rV_{n_0},\rK_{n_0})_{\gamma_{n_0}}$ is crystalline if the Galois representation $\sR_{\rI^0_{n_0}}(\xi^{n_0},\rV_{n_0},\rK_{n_0})_{\gamma_{n_0}}\res_x$ is crystalline at every $w\in\PP$; and an open subscheme $\sE'$ of $\Spec\sE_\rJ(\xi,\bV)_\gamma$ is \emph{even-crystalline} if its image in $\Spec\sE_{\rI^0_{n_0}}(\xi^{n_0},\rV_{n_0},\rK_{n_0})_{\gamma_{n_0}}$ consists only crystalline points.

\begin{theorem}\label{th:iwasawa_bis}
Let $\bV,\xi,\gamma,\rJ$ be as in Conjecture \ref{co:iwasawa} in which $\rJ$ is of the form $\rI^0_{n_0}\times\rJ_1$. Suppose that
\begin{itemize}
  \item[(i)] $\gamma$ satisfies all of Assumption \ref{as:galois} without (G8);

  \item[(ii--iv)] same as in Theorem \ref{th:iwasawa}.
\end{itemize}
Then the same conclusion of Theorem \ref{th:iwasawa} holds for every even-crystalline irreducible component of the tempered locus of $\Spec\sE_\rJ(\xi,\bV)_\gamma$.
\end{theorem}

\begin{remark}\label{re:recover}
In the situation of Theorem \ref{th:iwasawa_bis}, suppose that $\rJ_1$ contains $\SU(\rV_{n_1})(F^+_{\PP^+})\cap\rI^0_{n_1}$. In this case, the morphism $\Spec\sE_\rJ(\xi,\bV)_\gamma\to\Spec\Lambda_{\rI^0_{n_1}/\rJ_1,\dQ_q}$ is a proper local isomorphism, so that $\Spec\sE_\rJ(\xi,\bV)_\gamma$ is regular. On the other hand, $\sW_\rJ(\xi,\bV)_\gamma^1$ is empty by Remark \ref{re:tempered}(2). Therefore, Theorem \ref{th:iwasawa_bis} recovers \cite{LTX}*{Theorem~1.2.3} with this special $\rJ_1$ (also with $\xi$ trivial and with every place in $\PP^+$ being split in $F$), in view of \cite{LS}.
\end{remark}

The rest of this article is devoted to the proof of Theorem \ref{th:iwasawa}. Thus, from now on, we will assume (i--v) in Theorem \ref{th:iwasawa} and $\cD_\rJ(\xi,\bV)_\gamma\neq 0$ (since otherwise the conclusion is trivial). We will also explain how to modify the proof toward Theorem \ref{th:iwasawa_bis} without Assumption \ref{as:galois}(G8) at the end of \S\ref{ss:proof}.

\begin{notation}\label{no:congruence}
Let $k\geq 1$ be a positive integer and $\rJ<\rI^0$ a subgroup. We
\begin{itemize}
  \item denote by $\fe_\alpha$ the maximal ideal of $\cE(\xi^{n_\alpha},\rV_{n_\alpha},\rK_{n_\alpha})_{\gamma_{n_\alpha}}$ for $\alpha=0,1$;

  \item put
      \[
      \fp_\alpha^k\coloneqq\Ker\(\dT_{n_\alpha}^{\ang{\emptyset}}\otimes\Lambda_{\rI^0_{n_\alpha},\dZ_q}
      \to\cE(\xi^{n_\alpha},\rV_{n_\alpha},\rK_{n_\alpha})_{\gamma_{n_\alpha}}/\fe_\alpha^k\),
      \]
      for $\alpha=0,1$;

  \item denote by $\fp^k$ the ideal of $\dT^{\ang{\emptyset}}\otimes\Lambda_{\rI^0/\rJ,\dZ_q}$ generated by $(\fp_0^k,\fp_1^k)$.\footnote{Note that $\fp^k$ is not necessarily the $k$-th power of some ideal.}
\end{itemize}
\end{notation}

In the rest of this article, we
\begin{itemize}
  \item put
      \begin{align*}
      \epsilon(\bV)\coloneqq
      \begin{dcases}
      1, &\text{when $\bV$ is definite,}\\
      -1, &\text{when $\bV$ is indefinite;}
      \end{dcases}
      \end{align*}

  \item suppress $\xi$ and $\gamma$ in the notation like $\cD_\rJ(\xi,\bV)_\gamma$ since they will not vary from now on;

  \item denote by $\sE_\rJ^\reg$ the regular locus of $\Spec\sE_\rJ(\bV)$, and by $\sE_\rJ^\temp$ the complement of $\sW_\rJ(\bV)\cup\sY_\rJ(\bV)$ in $\sE_\rJ^\reg$;

  \item for every (set-theoretical) point $z$ of $\sE_\rJ^\temp$, denote by $\sV_z$ the Zariski closure of $z$ in $\sE_\rJ^\temp$, and $\cI_z\subseteq\cE_\rJ(\bV)$ the vanishing ideal of $z$ in $\cE_\rJ(\bV)$;

  \item once again, we suppress $\rJ$ in all subscripts when it is the trivial subgroup.
\end{itemize}
By definition and Remark \ref{re:flat}, $\sE_\rJ^\temp$ is contained in the tempered locus of $\Spec\sE_\rJ(\bV)$ whose complement has codimension at least two. Thus, without loss of generality, we may replace the tempered locus of $\Spec\sE_\rJ(\bV)$ in Theorem \ref{th:iwasawa} by $\sE_\rJ^\temp$.

\section{Specialization of Selmer groups}

In this section, we study the specialization of Selmer groups along closed points of the eigenvariety.

\begin{notation}\label{no:point}
For every closed point $x$ of $\Spec\sE_\rJ(\bV)$, we denote by
\begin{itemize}
  \item $\dQ_x$ its residue field (already), which is a finite extension of $\dQ_q$,

  \item $\dZ_x$ the ring of integers of $\dQ_x$ and $p_x$ the maximal ideal of $\dZ_x$,

  \item $\dZ'_x$ the image of the natural homomorphism $\cE\to\dZ_x$, which is a subring of $\dZ_x$ of finite index,

  \item $\fP_x$ the kernel of $\cE_\rJ(\bV)\to\dZ'_x$,

  \item $\fp_x\subseteq p_x$ the ideal of $\dZ_x$ generated by the image of $\fp^1$ under the natural composite map $\dT^{\ang{\emptyset}}\otimes\Lambda_{\rI^0/\rJ,\dZ_q}\to\cE_\rJ(\bV)\to\dZ_x$.
\end{itemize}
For $\cT=\cR_\rJ(\bV),\cD_\rJ(\bV)$, put $\cT_{/x}\coloneqq\cT\otimes_\cE\dZ_x$ (which contains $\cT/\fP_x$) and $\sT_{/x}\coloneqq\sT\otimes_\sE\dQ_x$, which naturally inherit a filtration $\Fil^\bullet_w$ for every $w\in\PP$.
\end{notation}

\begin{definition}\label{de:selmer1}
Write $\cT$ for $\cR_\rJ(\bV)$, or when $\bV$ is indefinite $\cD_\rJ(\bV)$. Take a closed point $x$ of $\Spec\sE_\rJ(\bV)$.
\begin{enumerate}
  \item We define
      \[
      \rH^1_\ff(F_w,\cT_{/x})\coloneqq
      \begin{dcases}
      \Ker\(\rH^1(F_w,\cT_{/x})\to\rH^1(\rI_{F_w},\sT_{/x})\), & w\in\Sigma\setminus\Sigma_p,\\
      \rH^1_f(F_w,\cT_{/x}), & w\in\Sigma_p\setminus\PP,\\
      \Ker\(\rH^1(F_w,\cT_{/x})\to\rH^1(\rI_{F_w},\sT_{/x}/\Fil_w^{-1}\sT_{/x})\),
      & w\in\PP,
      \end{dcases}
      \]
      and $\rH^1_f(F_w,(\cT_{/x})^*(1))$ to be the annihilator of $\rH^1_f(F_w,\cT_{/x})$ under the local Tate pairing
      \[
      \rH^1(F_w,\cT_{/x})\times\rH^1(F_w,(\cT_{/x})^*(1))\to\dQ_x/\dZ_x.
      \]

  \item Globally, define $\rH^1_\ff(F,\cT_{/x})$ and $\rH^1_\ff(F,(\cT_{/x})^*(1))$ to be the $\dZ_x$-submodules of $\rH^1(F,\cT_{/x})$ and $\rH^1(F,(\cT_{/x})^*(1))$ consisting of classes whose localization at every place $w$ of $F$ belong to $\rH^1_\ff(F_w,\cT_{/x})$ and $\rH^1_\ff(F_w,(\cT_{/x})^*(1))$, respectively.

  \item We define similarly various Selmer groups for $\cT/\fP_x$ instead of $\cT_{/x}$.

  \item Finally, for $\sT\coloneqq\cT\otimes_{\dZ_q}\dQ_q$, put
      \begin{align*}
      \rH^1_\ff(F_w,\sT_{/x})&\coloneqq\rH^1_\ff(F_w,\cT_{/x})\otimes_{\dZ_x}\dQ_x=\rH^1_\ff(F_w,\cT/\fP_x)\otimes_{\dZ_x}\dQ_x,\\ \rH^1_\ff(F,\sT_{/x})&\coloneqq\rH^1_\ff(F,\cT_{/x})\otimes_{\dZ_x}\dQ_x=\rH^1_\ff(F_w,\cT/\fP_x)\otimes_{\dZ_x}\dQ_x.
      \end{align*}
\end{enumerate}
\end{definition}

\begin{remark}\label{re:selmer1}
Let $x$ be a closed point of $\Spec\sE_\rJ(\bV)$. Then $\rH^1_\ff(F_w,\cT_{/x})=\rH^1_f(F_w,\cT_{/x})$ for $w\in\Sigma\setminus\Sigma_p$. For $w\in\PP$, $\rH^1_\ff(F_w,\cT_{/x})$ could be different from $\rH^1_f(F_w,\cT_{/x})$ since the subspace $\Fil_w^{-1}\sT_{/x}$ does not coincide with $\Fil^{-1}(\sT_{/x}\res_{\Gamma_{F_w}})$ -- the one from Lemma \ref{le:panchishkin} in general; they do coincide if $x$ is an interlacing classical point (Definition \ref{de:classical2}). Similarly observations apply to $\cT/\fP_x$.

Nevertheless, as long as $x$ belongs to $\sE_\rJ^\temp$, the local Selmer structure $\rH^1_\ff(F_w,\sT_{/x})$ remains self-dual for every place $w$ of $F$ in the following sense:\footnote{It is well-known that the Bloch--Kato local Selmer structure $\rH^1_f(F_w,\sT_{/x})$ is self-dual.} if we choose an $\Gamma_F$-equivariant isomorphism $(\sT_{/x})^\tc\simeq(\sT_{/x})^*(1)$ (which exists and is unique up to a scalar in $\dQ_x^\times$) and adopt the canonical identification $\rH^1(F_{w^\tc},\sT_{/x})=\rH^1(F_w,(\sT_{/x})^\tc)$, then $\rH^1_\ff(F_w,\sT_{/x})$ and $\rH^1_\ff(F_{w^\tc},\sT_{/x})$ are mutual annihilators under the (induced) local Tate pairing $\rH^1(F_w,\sT_{/x})\times\rH^1(F_{w^\tc},\sT_{/x})\to\dQ_x$.
\end{remark}

\begin{lem}\label{le:specialization}
Write $\cT$ for $\cR_\rJ(\bV)$, or when $\bV$ is indefinite $\cD_\rJ(\bV)$. Let $x$ be a closed point of $\Spec\sE_\rJ(\bV)$.
\begin{enumerate}
  \item For every place $w$ of $F$, the image of $\rH^1_f(F_w,\cT)$ under the natural map $\rH^1(F_w,\cT)\to\rH^1(F_w,\cT/\fP_x)$ is contained in $\rH^1_\ff(F_w,\cT/\fP_x)$.

  \item For every place $w\in\Sigma_p\setminus\PP$, the induced map $\rH^1_f(F_w,\cT)\to\rH^1_\ff(F_w,\cT/\fP_x)$ is surjective.

  \item The image of $\rH^1_f(F,\cT)$ under the natural map $\rH^1(F,\cT)\to\rH^1(F,\cT/\fP_x)$ is contained in $\rH^1_\ff(F,\cT/\fP_x)$.
\end{enumerate}
\end{lem}

\begin{proof}
Part (3) follows from (1). Part (1) for $w\in\PP$ follows from Lemma \ref{le:selmer2}. Part (1) for $w\in\Sigma\setminus\Sigma_p$ follows as $\rH^1(F_w,\sT)=0$. Part (1,2) for $w\in\Sigma_p\setminus\PP$ follows from \cite{LTX}*{Lemma~5.7.1(1)} with $a=\min_{\tau\in\Upsilon_w}\{\zeta^n_{\tau,0}+\zeta^{n+1}_{\tau,0}\}-n$ and $b=\max_{\tau\in\Upsilon_w}\{\zeta^{n}_{\tau,n-1}+\zeta^{n+1}_{\tau,n}\}+(n-1)$.\footnote{When $\cT=\cR$, we need to use a variant of \cite{LTX}*{Lemma~5.7.1(1)} for modules over $\cE$, which holds by the similar proof.}
\end{proof}

Lemma \ref{le:specialization} implies that for every closed point $x$ of $\Spec\sE_\rJ(\bV)$, the image of $\rH^1_\ff(F,(\cR_\rJ(\bV)_{/x})^*(1))$ under the natural map $\rH^1(F,(\cR_\rJ(\bV)_{/x})^*(1))\to\rH^1(F,\cR_\rJ(\bV)^*(1))$ is contained in $\rH^1_f(F,\cR_\rJ(\bV)^*(1))$ (and is annihilated by $\fP_x$).

\begin{proposition}\label{pr:control}
For every closed point $x$ of $\Spec\sE_\rJ(\bV)$, consider the cospecialization map
\[
\rH^1_\ff(F,(\cR_\rJ(\bV)_{/x})^*(1))\to\rH^1_f(F,\cR_\rJ(\bV)^*(1))[\fP_x]
\]
from Lemma \ref{le:specialization}(3).
\begin{enumerate}
  \item The kernel of $\cosp_x$ is finite with order bounded by a constant depending only on $|\dZ_x/\dZ'_x|$.

  \item For every affinoid subdomain $\sU$ of $\sE_\rJ^\temp$, the cokernel of $\cosp_x$ is finite with order bounded by a constant depending only on $\sU$ and $[\dQ_x:\dQ_q]$.
\end{enumerate}
\end{proposition}

\begin{proof}
Denote $F^\ur\subseteq\ol{F}$ the maximal extension of $F$ that is unramified outside $\spadesuit\cup\Sigma_p$ and write $\rH^\bullet(F,-)$ for $\rH^\bullet(F^\ur/F,-)$ for short. In the proof below, we also write $\cE,\cR$ for $\cE_\rJ(\bV),\cR_\rJ(\bV)$ for short.

We first note that, by (the same proof of) \cite{MR04}*{Lemma~3.5.3}, the absolute irreducibility of $\gamma$ implies that the natural map
\[
\rH^1(F,(\cR/\fP_x)^*(1))=\rH^1(F,\cR^*(1)[\fP_x])\to\rH^1(F,\cR^*(1))[\fP_x]
\]
is an isomorphism for every closed point $x$ of $\Spec\sE$.

For (1), $\Ker\cosp_x$ is then contained in the kernel of the map $\rH^1(F,(\cR_{/x})^*(1))\to\rH^1(F,(\cR/\fP_x)^*(1))$, which is dual to a $\dZ_q$-submodule of $\rH^1(F,\cR\otimes_{\cE}(\dZ_x/\dZ'_x))$. Part (1) follows as $\rH^1(F,\cR\otimes_{\cE}(\dZ_x/\dZ'_x))$ is finite with order bounded by a constant depending only on $|\dZ_x/\dZ'_x|$.

For (2), $\coker\cosp_x$ is then contained in
\[
\bigoplus_{w\in\spadesuit\cup\Sigma_p}
\Ker\(\rH^1_\ff(F_w,\cR^*(1)[\fP_x])\to\rH^1_f(F_w,\cR^*(1))[\fP_x]\)
\]
by the snake lemma. By local Tate duality, the order of the above group is the same as that of
\[
\bigoplus_{w\in\spadesuit\cup\Sigma_p}\coker\(\rH^1_f(F_w,\cR)\to\rH^1_\ff(F_w,\cR/\fP_x)\).
\]
There are three cases.

For $w\in\Sigma_p\setminus\PP$, the map $\rH^1_f(F_w,\cR)\to\rH^1_\ff(F_w,\cR/\fP_x)$ is surjective by Lemma \ref{le:specialization}(2).

For $w\in\PP$, one can adopt the same proof (for the same case) of \cite{LTX}*{Proposition~5.2.14(2)} to conclude that the cokernel of $\rH^1_f(F_w,\cR)\to\rH^1_\ff(F_w,\cR/\fP_x)$ is finite with order bounded by a constant depending only on $\sU$ and $[\dQ_x:\dQ_q]$.

For $w\in\spadesuit$, the image of $\rH^1_f(F_w,\cR)\to\rH^1_\ff(F_w,\cR/\fP_x)$ contains the image of the composite map
\[
\rH^1(\kappa_w,\cR^{\rI_{F_w}})\to\rH^1(\kappa_w,(\cR/\fP_x)^{\rI_{F_w}})\to\rH^1_\ff(F_w,\cR/\fP_x)
\]
in which the second one is an inclusion. We now show that for every closed point $x$ of $\Spec\sE$, both
\begin{align}\label{eq:control1}
\coker\(\rH^1(\kappa_w,\cR^{\rI_{F_w}})\to\rH^1(\kappa_w,(\cR/\fP_x)^{\rI_{F_w}})\)
\end{align}
and
\begin{align}\label{eq:control3}
\frac{\rH^1_\ff(F_w,\cR/\fP_x)}{\rH^1(\kappa_w,(\cR/\fP_x)^{\rI_{F_w}})}=\rH^1(\rI_{F_w},\cR/\fP_x)^{\Gamma_{\kappa_w}}[p^\infty]
\end{align}
are trivial. Let $\rP_{F_w}$ be the maximal subgroup of $\rI_{F_w}$ whose pro-order is coprime to $p$ so that $\rI_{F_w}/\rP_{F_w}\simeq\dZ_p$. By \cite{CHT08}*{Lemma~2.4.11}, $\cR_0\coloneqq\cR^{\rP_{F_w}}$ is an $\cE[\Gamma_{F_w}]$-linear direct summand such that the natural map $\rH^i(\rI_{F_w},\cR_0/\fP_x)\to\rH^i(\rI_{F_w},\cR/\fP_x)$ is an isomorphism for $i\geq 0$ and every closed point $x$ of $\sE_\rJ^\temp$. Thus, without loss of generality, we may assume that $\rP_{F_w}$ acts trivially on $\cR$. Assumption \ref{as:galois}(G5) implies that there exists a free $\dZ_q$-module $R$ contained in $\cR$ satisfying $\cR=R\otimes_{\dZ_q}\cE$, such that the action of $\rI_{F_w}$ on $\cR$ factors through $R$ and that $\rH^1(\rI_{F_w},R)$ is a finite free $\dZ_q$-module. It follows that $\cR^{\rI_{F_w}}\to(\cR/\fP_x)^{\rI_{F_w}}$ is surjective, so that \eqref{eq:control1} is trivial due to the vanishing of $\rH^2(\kappa_w,-)$. It also follows that
\[
\eqref{eq:control3}=\rH^1(\rI_{F_w},\cR/\fP_x)^{\Gamma_{\kappa_w}}[p^\infty]
\subseteq\rH^1(\rI_{F_w},\cR/\fP_x)[p^\infty]=\rH^1(\rI_{F_w},R)[p^\infty]\otimes_{\dZ_q}\cR/\fP_x=0.
\]

The proposition is proved.
\end{proof}

\begin{proposition}\label{pr:specialize}
For every point $z$ of $\sE_\rJ^\temp$ of codimension at most one. There exists a Zariski open subset $\sU_z$ of $\sV_z$ such that the natural maps
\begin{align*}
\rH^1_f(F,\sR_\rJ(\bV))/\fP_x\rH^1_f(F,\sR_\rJ(\bV))&\to\rH^1_\ff(F,\sR_\rJ(\bV)_{/x}), \\
\sD_\rJ(\bV)^\sfH/\fP_x\sD_\rJ(\bV)^\sfH&\to(\sD_\rJ(\bV)_{/x})^\sfH, \quad\text{when $\epsilon(\bV)=1$},\\
\rH^1_f(F,\sD_\rJ(\bV))^\sfH/\fP_x\rH^1_f(F,\sD_\rJ(\bV))^\sfH&\to\rH^1_\ff(F,\sD_\rJ(\bV)_{/x})^\sfH,\quad\text{when $\epsilon(\bV)=-1$},
\end{align*}
are all injective for every closed point $x$ of $\sU_z$.
\end{proposition}

\begin{proof}
In the proof below, we write $\cE,\cR,\sD,\sD^\flat,\sR$ for $\cE_\rJ(\bV),\cR_\rJ(\bV),\sD_\rJ(\bV),\sD^\flat_\rJ(\bV),\sR_\rJ(\bV)$ for short.

We first consider the second map. Since $\dZ[\rK_{\spadesuit^+}\backslash\sfG(F^+_{\spadesuit^+})/\rK_{\spadesuit^+}]$ is a finitely generated algebra, the locus where $\sD^\sfH/\fP_x\sD^\sfH\to(\sD_{/x})^\sfH$ is injective is Zariski open. It suffices to show that for every point $z$ of $\sE_\rJ^\temp$ of codimension one, the map $\sD^\sfH/\fm_z\sD^\sfH\to\sD/\fm_z\sD$ is injective, where $\fm_z$ denotes the maximal ideal of $\cO_{\sE_\rJ^\temp,z}$. As $\cO_{\sE_\rJ^\temp,z}$ is a discrete valuation ring, $\fm_z\sD_z^\sfH=(\fm_z\sD_z)^\sfH$, so that the composition
\[
\sD^\sfH/\fm_z\sD^\sfH=\sD_z^\sfH/\fm_z\sD_z^\sfH=\sD_z^\sfH/(\fm_z\sD_z)^\sfH\to\sD_z/\fm_z\sD_z
=\sD/\fm_z\sD
\]
is injective.

We then consider the third map assuming the statement for the first map. Lemma \ref{le:flat} tell us that $\sD=\sD^\flat\otimes_\sE\sR$. There exists a Zariski closed subset $\sZ$ of $\sE_\rJ^\temp$ of codimension at least two such that both $(\sD^\flat)^\sfH$ and $\sD^\flat$ are locally free over $\sE_\rJ^\temp\setminus\sZ$. In particular, over $\sE_\rJ^\temp\setminus\sZ$, the third map is identified with
\[
((\sD^\flat)^\sfH\otimes_{\sE}\rH^1_f(F,\sR))/\fP_x((\sD^\flat)^\sfH\otimes_{\sE}\rH^1_f(F,\sR))
\to(\sD^\flat_{/x})^\sfH\otimes_{\dQ_x}\rH^1_\ff(F,\sR_{/x}).
\]
Thus, the statement for the third map follows from that of the first and the (same proof for the) second.

Finally for the first map, we claim that the natura map
\begin{align}\label{eq:control2}
\rH^1_f(F,\sR)/\fP_x\rH^1_f(F,\sR)\to\rH^1(F,\sR)/\fP_x\rH^1(F,\sR)
\end{align}
is injective for every closed point $x$ of $\sE_\rJ^\temp$ away from a Zariski closed subset of codimension at least two. Assuming the claim, the statement follows from the same argument as in the proof of \cite{LTX}*{Proposition~5.2.11}.

We now show the claim by showing that $\rH^1(F,\sR)/\rH^1_f(F,\sR)$ is locally torsion free over $\sE_\rJ^\temp$. Indeed, if this is the case, then $\rH^1(F,\sR)/\rH^1_f(F,\sR)$ is locally free away from a Zariski closed subset of codimension at least two, so that \eqref{eq:control2} is injective away from the same locus. For $\rH^1(F,\sR)/\rH^1_f(F,\sR)$, it is an $\sE$-submodule of
\[
\bigoplus_{w\in\Sigma_p}\rH^1(F_w,\sR)/\rH^1_f(F_w,\sR)
\]
by Definition \ref{de:selmer}. Thus, it suffices to show that for every $w\in\Sigma_p$, $\rH^1(F_w,\sR)/\rH^1_f(F_w,\sR)$ is locally torsion free over $\sE_\rJ^\temp$. We first consider an element $w\in\PP$. By Lemma \ref{le:selmer2}(2), $\rH^1(F_w,\sR)/\rH^1_f(F_w,\sR)$ is an $\sE$-submodule of $\rH^1(\rI_{F_w},\sR/\Fil_w^{-1}\sR)^{\Gamma_{\kappa_w}}$. The support of the maximal torsion submodule of $\rH^1(\rI_{F_w},\sR/\Fil_w^{-1}\sR)^{\Gamma_{\kappa_w}}$ is contained in the support of $(\Gr^0_w\sR)_{\Gamma_{F_w}}$, which is disjoint from $\sE_\rJ^\temp$ by Definition \ref{de:tempered}. It follows that $\rH^1(F_w,\sR)/\rH^1_f(F_w,\sR)$ is locally torsion free over $\sE_\rJ^\temp$. We then consider an element $w\in\Sigma_p\setminus\PP$. Applying \cite{LTX}*{Lemma~5.7.1} to
\[
\Lambda=\cE,\quad T=\cR,\quad a=-\min_{\tau\in\Upsilon_w}\{\zeta^n_{\tau,0}+\zeta^{n+1}_{\tau,0}\}-n,\quad
b=\max_{\tau\in\Upsilon_w}\{\zeta^{n}_{\tau,n-1}+\zeta^{n+1}_{\tau,n}\}+n-1,
\]
we know that the $\cE$-module $\rH^1(F_w,\cR)/\rH^1_f(F_w,\cR)$ can be embedded into a finite free $\cE$-module, which implies that $\rH^1(F_w,\sR)/\rH^1_f(F_w,\sR)$ is locally torsion free over $\sE_\rJ^\temp$.

The proposition is proved.
\end{proof}

\section{A bipartite Euler system over eigenvariety}

In this section, we construct a bipartite Euler system over the eigenvariety.

\begin{notation}
We introduce the following notation.
\begin{enumerate}
  \item Denote by $\fL_0$ the set of primes of $F^+$ not in $\ang{\emptyset}$ that are \emph{inert} in $F$.

  \item Denote by $\fN_0$ the set of (possibly empty) finite sets consisting of elements in $\fL_0$ with \emph{distinct} underlying rational primes. For $\fn\in\fN_0$, put $\fn\fl\coloneqq\fn\cup\{\fl\}$ for $\fl\in\fL_0$ whose underlying rational prime is coprime to $\fn$, and $\fn/\fl\coloneqq\fn\setminus\{\fl\}$ for $\fl\in\fn$.

  \item We define a \emph{directed} graph $\fX_0$ with vertices $v(\fn)$ indexed by elements in $\fN_0$, and arrows $a(\fn,\fn\fl)$ from $v(\fn)$ to $v(\fn\fl)$ for $\fl\in\fL_0\setminus\fn$.

  \item Finally, put
      \begin{align*}
      \fN_0^\defin\coloneqq\{\fn\in\fN_0\res(-1)^{|\fn|}=\epsilon(\bV)\},\quad
      \fN_0^\indef\coloneqq\{\fn\in\fN_0\res(-1)^{|\fn|}\neq\epsilon(\bV)\}.
      \end{align*}
\end{enumerate}
\end{notation}

\begin{notation}\label{no:initial}
We start from the \emph{initial datum} $\bV=(\emptyset;\rV_n,\rV_{n+1};\Lambda_n,\Lambda_{n+1};\rK_n,\rK_{n+1};\sfB)$.
\begin{enumerate}
  \item Fix a finite set $[\bV]$ of rational primes different from $p$ and not underlying $\spadesuit^+$ such that the image of the homomorphism
      \[
      \bigotimes_{v\in\Sigma^+_{[\bV]}}\dT_v\to\cE(\bV)
      \]
      already generates $\cE(\bV)$ over $\Lambda_{\rI^0,\dZ_q}$.

  \item For every vertex $v(\fn)$ of $\fX_0$, we choose a datum $\bV^\fn=(\fn;\rV^\fn_n,\rV^\fn_{n+1};\Lambda^\fn_n,\Lambda^\fn_{n+1};\rK^{\fn}_n,\rK^{\fn}_{n+1};\sfB^\fn)$, together with a datum $\tj^\fn=(\tj^\fn_n,\tj^\fn_{n+1})$ in which $\tj^\fn_N\colon\rV_N\otimes_{F^+}\dA_{F^+}^{\infty,\fn}\to\rV^\fn_N\otimes_{F^+}\dA_{F^+}^{\infty,\fn}$ ($N=n,n+1$) are isometries sending $(\Lambda_N)^\fn$ to $(\Lambda^\fn_N)^\fn$ and $(\rK_N)^\fn$ to $(\rK^{\fn}_N)^\fn$ for $N\in\{n,n+1\}$, satisfying $\tj^\fn_{n+1}=(\tj^\fn_n)_\sharp$ and $\tj^\fn_p\sfB=\sfB^\fn$. When $\fn=\emptyset$, we just take $\bV^\emptyset$ to be $\bV$ and $\tj^\emptyset$ to be the identity.

      Then $\bV^\fn$ is definite/indefinite when $\fn\in\fN_1^\defin/\fN_1^\indef$. Similar to $\bV$, we have groups $\sfG^\fn$, $\sfH^\fn$, $\rK^{\fn}$, $\rK^{\fn}_{\sfH^\fn}$ for $\bV^\fn$. Moreover, we identify $\sfG^\fn\otimes_\dQ\dQ_p$ with $\sfG\otimes_\dQ\dQ_p$ and $\sfB^\fn$ with $\sfB$ via $\tj^\fn_p$.

  \item For every arrow $a=a(\fn,\fn\fl)$ of $\fX_0$, we choose a datum $\tj^a=(\tj^a_n,\tj^a_{n+1})$ in which
      \[
      \tj^a_N\colon\rV^\fn_N\otimes_{F^+}\dA_{F^+}^{\infty,\fl}\to\rV^{\fn\fl}_N\otimes_{F^+}\dA_{F^+}^{\infty,\fl}
      \]
      are isometries sending $(\Lambda^\fn_N)^\fl$ to $(\Lambda^{\fn\fl}_N)^\fl$ and $(\rK^{\fn}_N)^\fl$ to $(\rK^{\fn\fl}_N)^\fl$ and such that $\tj^{\fn\fl}_N=(\tj^a_N)^{\fn\fl}\circ(\tj^\fn_N)^{\fn\fl}$ for $N\in\{n,n+1\}$, satisfying $\tj^a_{n+1}=(\tj^a_n)_\sharp$.
\end{enumerate}
\end{notation}

\begin{definition}\label{de:congruence}
Let $k\geq 1$ be a positive integer.
\begin{enumerate}
  \item We denote by $\fL_k$ the subset of $\fL_0\setminus\Sigma^+_{[\bV]}$ consisting of $\fl$ that is a level-raising prime of $F^+$ (Definition \ref{de:prime}) satisfying $\tC_\fl\in\fp^k_0$ (Notation \ref{de:cong} and Notation \ref{no:congruence}).

  \item We denote by $\fN_k$ the set of (possibly empty) finite sets consisting of elements in $\fL_k$ with \emph{distinct} underlying rational primes.

  \item We define the directed graph $\fX_k$ to be the full subgraph of $\fX_0$ spanned by $\fN_k$.

  \item We say that an element $\fl$ of $\fL_k$ is \emph{effective} if $\cD(\bV^{\{\fl\}})$ is nontrivial, and that an element $\fn$ of $\fN_k$ is \emph{effective} if it either empty or contains an effective element of $\fL_k$. We denote by $\fN_k^\eff$ the subset of $\fN_k$ of effective elements.
\end{enumerate}
\end{definition}

\begin{remark}\label{re:location}
We have the following remarks concerning Definition \ref{de:congruence}.
\begin{enumerate}
  \item It is clear from the definition that $\fL_0\supseteq\fL_1\supseteq\fL_2\supseteq\cdots$.

  \item We have $\bigcap_k\fL_k=\emptyset$.

  \item For every $\fl\in\fL_k$ with the underlying rational prime $\ell$ and every subgroup $\rJ<\rI^0$, there is a unique decomposition
      \[
      \cR_\rJ(\bV)/\fp^k=
      (\cR_\rJ(\bV)/\fp^k)^{\phi_{w(\fl)}=1}\oplus(\cR_\rJ(\bV)/\fp^k)^{\phi_{w(\fl)}=\ell^2}\oplus(\cR_\rJ(\bV)/\fp^k)^\dag
      \]
      of $\cE_\rJ(\bV)[\Gamma_{F_{w(\fl)}}]$-modules satisfying $\bigoplus_{i\in\dZ}\rH^i(F_{w(\fl)},(\cR_\rJ(\bV)/\fp^k)^\dag)=0$ and that both $(\cR_\rJ(\bV)/\fp^k)^{\phi_{w(\fl)}=1}$ and $(\cR_\rJ(\bV)/\fp^k)^{\phi_{w(\fl)}=\ell^2}$ are free $\cE_\rJ(\bV)/\fp^k$-modules of rank one.

  \item For every $\fl\in\fL_1$, the set $\{k\geq 1\res \fl\in\fL_k\}$ has a maximum, which we denote by $k(\fl)$.

  \item By Assumption \ref{as:galois}(G4) and the Chebotarev density theorem, $\fL_k$ is an infinite set for every $k\geq 1$.
\end{enumerate}
\end{remark}

Put $\fN^\defin_k\coloneqq\fN_k\cap\fN^\defin_0$ and $\fN^\indef_k\coloneqq\fN_k\cap\fN^\indef_0$. For every $\fl\in\fL_1$, we fix isomorphisms
\begin{align}\label{eq:rigidify}
\rH^1_\unr(F_{w(\fl)},\cR(\bV)/\fp^{k(\fl)})\simeq \cE(\bV)/\fp^{k(\fl)},\quad
\rH^1_\sing(F_{w(\fl)},\cR(\bV)/\fp^{k(\fl)})\simeq \cE(\bV)/\fp^{k(\fl)}
\end{align}
once and for all, which is possible by Remark \ref{re:location}(3,4).

We now introduce congruence modules.

\begin{definition}\label{de:congruence1}
Let $k\geq 1$ be a positive integer and $\rJ<\rI^0$ a subgroup.
\begin{enumerate}
  \item For $\fn\in\fN_k$, put
      \[
      \fp^{\fn,k}_\alpha\coloneqq\fp^k_\alpha\cap(\dT^{\ang{\fn}}_{n_\alpha}\otimes\Lambda_{\rI^0_{n_\alpha},\dZ_q}),\qquad
      \fp^{\fn,k}\coloneqq\fp^k\cap(\dT^{\ang{\fn}}\otimes\Lambda_{\rI^0,\dZ_q}).
      \]

  \item For $\fn\in\fN_k$, we define the ($\rJ$-invariant) \emph{congruence module} (of depth $k$) at $\fn$ to be
     \[
     \cC_\rJ^{\fn,k}\coloneqq
     \begin{dcases}
     \Hom_{\dT^{\ang{\fn}}\otimes\Lambda_{\rI^0,\dZ_q}}
     \(\cD_\rJ(\bV^\fn)/\fp^{\fn,k},\cE_\rJ(\bV)/\fp^k\), &\fn\in\fN_k^\defin, \\
     \Hom_{\dT^{\ang{\fn}}\otimes\Lambda_{\rI^0,\dZ_q}[\Gamma_F]}
     \(\cD_\rJ(\bV^\fn)/\fp^{\fn,k},\cR_\rJ(\bV)/\fp^k\), &\fn\in\fN_k^\indef,
     \end{dcases}
     \]
     as a module over $\dZ[\rK_{\spadesuit^+}\backslash\sfG(F^+_{\spadesuit^+})/\rK_{\spadesuit^+}]\otimes\Lambda_{\rI^0,\dZ_q}$. When $\fn=\emptyset$, we also allow $k=\infty$ and regard $\fp^{\emptyset,\infty}$ as the zero ideal.

  \item For every arrow $a=a(\fn,\fn\fl)$ of $\fX_k$ with $\fn\in\fN_k^\defin$, we define the \emph{reciprocity map}
     \[
     \varrho_\rJ^{a,k}\colon\cC_\rJ^{\fn,k}\to\cC_\rJ^{\fn\fl,k}
     \]
     of $\dZ[\rK_{\spadesuit^+}\backslash\sfG(F^+_{\spadesuit^+})/\rK_{\spadesuit^+}]\otimes\Lambda_{\rI^0,\dZ_q}$-modules as follows. When $\fn\not\in\fN_k^\eff$, we define $\varrho_\rJ^{a,k}$ to be the zero map. When $\fn\in\fN_k^\eff$, we define $\varrho_\rJ^{a,k}$ to be the composition of the map
     \begin{align*}
     \varrho^\sing_\rJ(\bV^{\fn\fl},\bV^\fn)^\vee\colon\cC_\rJ^{\fn,k}
     &=\Hom_{\dT^{\ang{\fn\fl}}\otimes\Lambda_{\rI^0,\dZ_q}}
     \(\cD_\rJ(\bV^\fn)/\fp^{\fn,k},\cE_\rJ(\bV)/\fp^k\) \\
     &\to\Hom_{\dT^{\ang{\fn\fl}}\otimes\Lambda_{\rI^0,\dZ_q}}
     \(\rH^1_\sing(F_{w(\fl)},\cD_\rJ(\bV^{\fn\fl})/\fp^{\fn\fl,k}),\cE_\rJ(\bV)/\fp^k\)
     \end{align*}
     and the inverse of the isomorphism
     \begin{align*}
     \cC_\rJ^{\fn\fl,k}&=\Hom_{\dT^{\ang{\fn\fl}}\otimes\Lambda_{\rI^0,\dZ_q}[\Gamma_F]}
      \(\cD_\rJ(\bV^{\fn\fl})/\fp^{\fn\fl,k},\cR_\rJ(\bV)/\fp^k\) \\
      &\xrightarrow{\sim}\Hom_{\dT^{\ang{\fn\fl}}\otimes\Lambda_{\rI^0,\dZ_q}}
      \(\rH^1_\sing(F_{w(\fl)},\cD_\rJ(\bV^{\fn\fl})/\fp^{\fn\fl,k}),\cE_\rJ(\bV)/\fp^k\)
     \end{align*}
     from Lemma \ref{le:congruence}(1b) below.

  \item For every arrow $a=a(\fn,\fn\fl)$ of $\fX_k$ with $\fn\in\fN_k^\indef$, we define the \emph{reciprocity map}
     \[
     \varrho_\rJ^{a,k}\colon\cC_\rJ^{\fn,k}\to\cC_\rJ^{\fn\fl,k}
     \]
     of $\dZ[\rK_{\spadesuit^+}\backslash\sfG(F^+_{\spadesuit^+})/\rK_{\spadesuit^+}]\otimes\Lambda_{\rI^0,\dZ_q}$-modules as follows. When $\fn\not\in\fN_k^\eff$, we define $\varrho_\rJ^{a,k}$ to be the zero map. When $\fn\in\fN_k^\eff$, we define $\varrho_\rJ^{a,k}$ to be the composition of the isomorphism
     \begin{align*}
     \cC_\rJ^{\fn,k}&=\Hom_{\dT^{\ang{\fn}}\otimes\Lambda_{\rI^0,\dZ_q}[\Gamma_F]}
      \(\cD_\rJ(\bV^\fn)/\fp^{\fn,k},\cR_\rJ(\bV)/\fp^k\) \\
      &\xrightarrow\sim\Hom_{\dT^{\ang{\fn}}\otimes\Lambda_{\rI^0,\dZ_q}}
      \(\rH^1_\unr(F_{w(\fl)},\cD_\rJ(\bV^\fn)/\fp^{\fn,k}),\cE_\rJ(\bV)/\fp^k\)
     \end{align*}
     from Lemma \ref{le:congruence}(2a) below, and the map
     \begin{align*}
     \varrho^\unr_\rJ(\bV^{\fn\fl},\bV^\fn)^\vee&\colon\Hom_{\dT^{\ang{\fn}}\otimes\Lambda_{\rI^0,\dZ_q}}
      \(\rH^1_\unr(F_{w(\fl)},\cD_\rJ(\bV^\fn)/\fp^{\fn,k}),\cE_\rJ(\bV)/\fp^k\)\\
      &\to\Hom_{\dT^{\ang{\fn\fl}}\otimes\Lambda_{\rI^0,\dZ_q}}
      \(\cD_\rJ(\bV^{\fn\fl})/\fp^{\fn\fl,k},\cE_\rJ(\bV)/\fp^k\)=\cC_\rJ^{\fn\fl,k}.
     \end{align*}
\end{enumerate}
\end{definition}

In what follows, for a nonarchimedean place $w$ of $F$ not above $p$, a finitely generated topological $\dZ_q$-module $D$ with a continuous unramified action by $\Gamma_{F_w}$, and an element $c\in\dZ_q^\times$, we denote by $D_{(\phi_w=c)}$ the localization of $D$ at the unique maximal ideal of $\dZ_q[\Gamma_{F_w}/\rI_{F_w}]$ containing $\phi_w-c$. In particular, we have natural surjective maps
\[
D_{(\phi_w=1)}\to\rH^1_\unr(F_w,D),\qquad
D_{(\phi_w=\|w\|)}\to\rH^1_\sing(F_w,D),
\]
where $\|w\|$ denotes the residual cardinality of $F_w$.

\begin{lem}\label{le:congruence}
Let $k\geq 1$ be a positive integer and $\rJ<\rI^0$ a subgroup. Consider elements $\fn\in\fN_k^\indef$ and $\fm\in\fN_k$ satisfying $\fn\subseteq\fm$.
\begin{enumerate}
  \item For every arrow $a=a(\fn/\fl,\fn)$ of $\fX_k$ (with $\fl\in\fn$), we have
      \begin{enumerate}
        \item $\cD_\rJ(\bV^\fn)/\fp^{\fm,k}$ is unramified at $w(\fl)$;

        \item the natural map
            \[
            \Hom_{\dT^{\ang{\fm}}\otimes\Lambda_{\rI^0,\dZ_q}[\Gamma_F]}\(\cD_\rJ(\bV^\fn)/\fp^{\fm,k},\cR_\rJ(\bV)/\fp^k\)\to
            \Hom_{\dT^{\ang{\fm}}\otimes\Lambda_{\rI^0,\dZ_q}}
            \(\rH^1_\sing(F_{w(\fl)},\cD_\rJ(\bV^\fn)/\fp^{\fm,k}),\cE_\rJ(\bV)/\fp^k\),
            \]
            where we have used \eqref{eq:rigidify}, is an isomorphism;

        \item for every nontrivial $\dT^{\ang{\fm}}\otimes\Lambda_{\rI^0,\dZ_q}[\Gamma_F]$-linear subquotient $D$ of $\cD_\rJ(\bV^\fn)/\fp^{\fm,k}$, $D_{(\phi_{w(\fl)}=\ell^2)}$ is nontrivial and the natural map $D_{(\phi_{w(\fl)}=\ell^2)}\to\rH^1_\sing(F_{w(\fl)},D)$ is an isomorphism.
      \end{enumerate}

  \item For every arrow $a=a(\fn,\fn\fl)$ of $\fX_k$ with $\fl\not\in\fm$, we have
      \begin{enumerate}
        \item the natural map
            \[
            \Hom_{\dT^{\ang{\fm}}\otimes\Lambda_{\rI^0,\dZ_q}[\Gamma_F]}\(\cD_\rJ(\bV^\fn)/\fp^{\fm,k},\cR_\rJ(\bV)/\fp^k\)\to
            \Hom_{\dT^{\ang{\fm}}\otimes\Lambda_{\rI^0,\dZ_q}}
            \(\rH^1_\unr(F_{w(\fl)},\cD_\rJ(\bV^\fn)/\fp^{\fm,k}),\cE_\rJ(\bV)/\fp^k\),
            \]
            where we have used \eqref{eq:rigidify}, is an isomorphism;

        \item for every nontrivial $\dT^{\ang{\fm}}\otimes\Lambda_{\rI^0,\dZ_q}[\Gamma_F]$-linear subquotient $D$ of $\cD_\rJ(\bV^\fn)/\fp^{\fm,k}$, $D_{(\phi_{w(\fl)}=1)}$ is nontrivial and the natural map $D_{(\phi_{w(\fl)}=1)}\to\rH^1_\unr(F_{w(\fl)},D)$ is an isomorphism.
      \end{enumerate}
\end{enumerate}
\end{lem}

\begin{proof}
Let $\cE_\rJ(\bV^\fn)^\fm$ be the $\Lambda_{\rI^0/\rJ,\dZ_q}$-subalgebra of $\End_{\Lambda_{\rI^0/\rJ,\dZ_q}}\(\cD_\rJ(\bV^\fn)\)$ generated by the image of $\dT^{\ang{\fm}}$, which is a subalgebra of $\cE_\rJ(\bV^\fn)$. We have the $\cE_\rJ(\bV^\fn)^\fm[\Gamma_F]$-module $\cR_\rJ(\bV^\fn)^\fm$ similar to $\cE_\rJ(\bV^\fn)$, and the natural isomorphism
\[
\cD_\rJ^\flat(\bV^\fn)^\fm\otimes_{\cE_\rJ(\bV^\fn)^\fm}\cR_\rJ(\bV^\fn)^\fm\to\cD_\rJ(\bV^\fn),
\]
where $\cD_\rJ^\flat(\bV^\fn)^\fm\coloneqq\Hom_{\cE_\rJ(\bV^\fn)^\fm[\Gamma_F]}\(\cR_\rJ(\bV^\fn)^\fm,\cD_\rJ(\bV^\fn)\)$. By definition, $\cE_\rJ(\bV^\fn)^\fm/\fp^{\fm,k}$ is naturally a quotient $\dT^{\ang{\fm}}\otimes\Lambda_{\rI^0,\dZ_q}$-ring of $\cE_\rJ(\bV)/\fp^k$. Since the following two $\cE_\rJ(\bV^\fn)^\fm/\fp^{\fm,k}[\Gamma_F]$-modules
\[
\cR_\rJ(\bV^\fn)^\fm/\fp^{\fm,k},\qquad
\cR_\rJ(\bV)/\fp^k\otimes_{\cE_\rJ(\bV)/\fp^k}\cE_\rJ(\bV^\fn)^\fm/\fp^{\fm,k}
\]
are both liftings of $(\gamma_n^\natural\otimes\gamma_{n+1}^\natural)^\tc(n)$ to the ring $\cE_\rJ(\bV^\fn)^\fm/\fp^{\fm,k}$ with the same pseudo-character, they must be isomorphic by Assumption \ref{as:galois}(G1). In particular, for every element $\fl\in\fL_k\setminus\fm$, $\cR_\rJ(\bV^\fn)^\fm/\fp^{\fm,k}$ is unramified at $w(\fl)$; and we have a decomposition
\begin{align}\label{eq:congruence0}
\cR_\rJ(\bV^\fn)^\fm/\fp^{\fm,k}=(\cR_\rJ(\bV^\fn)^\fm/\fp^{\fm,k})^{\phi_{w(\fl)}=1}\oplus
(\cR_\rJ(\bV^\fn)^\fm/\fp^{\fm,k})^{\phi_{w(\fl)}=\ell^2}\oplus(\cR_\rJ(\bV^\fn)^\fm/\fp^{\fm,k})^\dag
\end{align}
of $\cE_\rJ(\bV^\fn)^\fm[\Gamma_{F_{w(\fl)}}]$-modules satisfying $\bigoplus_{i\in\dZ}\rH^i(F_{w(\fl)},(\cR_\rJ(\bV^\fn)^\fm/\fp^{\fm,k})^\dag)=0$ and that both $(\cR_\rJ(\bV^\fn)^\fm/\fp^{\fm,k})^{\phi_{w(\fl)}=1}$ and $(\cR_\rJ(\bV^\fn)^\fm/\fp^{\fm,k})^{\phi_{w(\fl)}=\ell^2}$ are free $\cE_\rJ(\bV^\fn)^\fm/\fp^{\fm,k}$-modules of rank one, similar to Remark \ref{re:location}(3).

We now show (1) and leave the similar case (2) to the readers. We already know (1a). Since $D$ is residually absolutely irreducible, it follows that $D_{(\phi_{w(\fl)}=\ell^2)}$ is nontrivial. The decomposition \eqref{eq:congruence0} implies that the natural map $D_{(\phi_{w(\fl)}=\ell^2)}\to\rH^1_\sing(F_{w(\fl)},D)$ is an isomorphism. Thus, (1c) follows. It remains to verify (1b), or equivalently, that the natural map
\begin{align}\label{eq:congruence3}
\Hom_{\dT^{\ang{\fm}}\otimes\Lambda_{\rI^0,\dZ_q}[\Gamma_F]}\(\cD_\rJ(\bV^\fn)/\fp^{\fm,k},\cR_\rJ(\bV)/\fp^k\)\to
\Hom_{\dT^{\ang{\fm}}\otimes\Lambda_{\rI^0,\dZ_q}}\((\cD_\rJ(\bV^\fn)/\fp^{\fm,k})^{\phi_{w(\fl)}=\ell^2},(\cR_\rJ(\bV)/\fp^k)^{\phi_{w(\fl)}=\ell^2}\)
\end{align}
is an isomorphism. We already show that as a $\dT^{\ang{\fm}}\otimes\Lambda_{\rI^0,\dZ_q}[\Gamma_F]$-module, $\cD_\rJ(\bV^\fn)/\fp^{\fm,k}$ is generated by $(\cD_\rJ(\bV^\fn)/\fp^{\fm,k})^{\phi_{w(\fl)}=\ell^2}$, the injectivity of \eqref{eq:congruence3} is clear. For the surjectivity, note that the discussion in the first paragraph gives an isomorphism
\[
\cD_\rJ(\bV^\fn)/\fp^{\fm,k}\simeq
(\cD_\rJ^\flat(\bV^\fn)^\fm/\fp^{\fm,k})\otimes_{\cE_\rJ(\bV)/\fp^k}(\cR_\rJ(\bV)/\fp^k)
\]
of $\dT^{\ang{\fm}}\otimes\Lambda_{\rI^0,\dZ_q}[\Gamma_F]$-modules. In particular, any element $\phi$ in the target of \eqref{eq:congruence3} is of the form $\phi^\flat\otimes 1_{(\cR_\rJ(\bV)/\fp^k)^{\phi_{w(\fl)}=\ell^2}}$ for a unique element
\[
\phi^\flat\in\Hom_{\dT^{\ang{\fm}}\otimes\Lambda_{\rI^0,\dZ_q}}\(\cD_\rJ^\flat(\bV^\fn)^\fm/\fp^{\fm,k},\cE_\rJ(\bV)/\fp^k\).
\]
Then the element
\[
\phi^\flat\otimes1_{\cR_\rJ(\bV)/\fp^k}\in
\Hom_{\dT^{\ang{\fm}}\otimes\Lambda_{\rI^0,\dZ_q}[\Gamma_F]}\(\cD_\rJ(\bV^\fn)/\fp^{\fm,k},\cR_\rJ(\bV)/\fp^k\)
\]
maps to $\phi$ under \eqref{eq:congruence3}. In other words, \eqref{eq:congruence3} is surjective.

The lemma is proved.
\end{proof}

\begin{proposition}\label{pr:congruence}
Let $k\geq1$ be a positive integer and $\rJ<\rI^0$ a subgroup. The following statements hold for every $\fn\in\fN_k^\eff$.
\begin{enumerate}
  \item For every $\fm\in\fN_k$ containing $\fn$, the natural map
      \[
      \cD_\rJ(\bV^\fn)/\fp^{\fm,k}\to\cD_\rJ(\bV^\fn)/\fp^{\fn,k}
      \]
      is an isomorphism.

  \item The finite sets $\cC_\rJ^{\fn,k}$ and $\cC_\rJ^{\emptyset,k}$ have the same cardinality; in particular, $\cD_\rJ(\bV^\fn)\neq 0$.

  \item If $\fn\in\fN_k^\defin(\cap\fN_k^\eff)$, then for every arrow $a=a(\fn,\fn\fl)$ of $\fX_k$ and every $\fm\in\fN_k$ containing $\fn\fl$, the map $\varrho^\sing_\rJ(\bV^{\fn\fl},\bV^\fn)$ \eqref{eq:reciprocity_first} induces an isomorphism
      \[
      \rH^1_\sing(F_{w(\fl)},\cD_\rJ(\bV^{\fn\fl})/\fp^{\fm,k})
      \xrightarrow\sim\cD_\rJ(\bV^\fn)/\fp^{\fm,k}.
      \]

  \item If $\fn\in\fN_k^\indef\setminus\{\emptyset\}(\cap\fN_k^\eff)$, then for every arrow $a=a(\fn,\fn\fl)$ of $\fX_k$ and every $\fm\in\fN_k$ containing $\fn\fl$, the map $\varrho^\unr_\rJ(\bV^{\fn\fl},\bV^\fn)$ \eqref{eq:reciprocity_second} induces an isomorphism
      \[
      \cD_\rJ(\bV^{\fn\fl})/\fp^{\fm,k}
      \xrightarrow\sim\rH^1_\unr(F_{w(\fl)},\cD_\rJ(\bV^\fn)/\fp^{\fm,k}).
      \]
\end{enumerate}
\end{proposition}

\begin{proof}
We first explain that for every fixed element $\fn\in\fN_k^\defin\cap\fN_k^\eff$, (1) and (2) together imply (3). Indeed, by definition, $\tC_\fl$ acts trivially on $\cD_\rJ(\bV^\fn)/\fp^{\fn,k}$, hence also acts trivially on $\cD_\rJ(\bV^\fn)/\fp^{\fm,k}$ by (1). Then the isomorphism \eqref{eq:reciprocity_first} induces an isomorphism
\[
\rH^1_\sing(F_{w(\fl)},\cD_\rJ(\bV^{\fn\fl}))/\fp^{\fm,k}\xrightarrow\sim\cD_\rJ(\bV^\fn)/\fp^{\fm,k}.
\]
By Remark \ref{re:singular}, the left-hand side coincides $\rH^1_\sing(F_{w(\fl)},\cD_\rJ(\bV^{\fn\fl})/\fp^{\fm,k})$. Thus, (3) follows.

We prove the four statements simultaneously by the induction on $|\fn|$.

When $|\fn|=0$, that is, $\fn=\emptyset$, (1) follows from the fact that $\fL_k$ is disjoint from $\Sigma^+_{[\bV]}$; (2) is obvious; hence (3) follows by the above explanation; and (4) is vacuous.

Now we consider an element $\fn\in\fN_k^\eff$ with $|\fn|>0$ and assume that (1--4) are known for every $\fn'\in\fN_k^\eff$ with $|\fn'|<|\fn|$.

We first consider the case where $\fn\in\fN_k^\defin$. As (4) is vacuous, by the above explanation, it suffices to consider (1) and (2). Choose an element $\fl'$ of $\fn$ such that $\fn'\coloneqq\fn/\fl'$ remains effective. We claim that
\begin{itemize}
  \item[($*$)] $\cD(\bV^\fn)/\fp^{\fm,k}$ and $\rH^1_\unr(F_{w(\fl')},\cD(\bV^{\fn'})/\fp^{\fm,k})$ are isomorphic as $\dT^{\ang{\fm}}\otimes\Lambda_{\rI^0,\dZ_q}$-modules.
\end{itemize}
Assuming this, we know that $\cD(\bV^\fn)/\fp^{\fm,k}$ and $\rH^1_\unr(F_{w(\fl')},\cD(\bV)/\fp^{\fm,k})$ are also isomorphic as $\dT^{\ang{\fm}}\otimes\Lambda_{\rI^0,\dZ_q}$-modules. Thus, (1) for $\fn'$ implies (1) for $\fn$. Taking the dual of the previous isomorphism (with $\fm=\fn$), we obtain an isomorphism
\[
\Hom_{\dT^{\ang{\fn}}\otimes\Lambda_{\rI^0,\dZ_q}}
\(\rH^1_\unr(F_{w(\fl')},\cD_\rJ(\bV^{\fn'})/\fp^{\fn,k}),\cE_\rJ(\bV)/\fp^k\)
\simeq\Hom_{\dT^{\ang{\fn}}\otimes\Lambda_{\rI^0,\dZ_q}}\(\cD_\rJ(\bV^\fn)/\fp^{\fn,k},\cE_\rJ(\bV)/\fp^k\)
\]
of $\dZ_q$-modules. By Lemma \ref{le:congruence}(2a), the left-hand side is canonically isomorphic to
\[
\Hom_{\dT^{\ang{\fn}}\otimes\Lambda_{\rI^0,\dZ_q}[\Gamma_F]}
\(\cD_\rJ(\bV^{\fn'})/\fp^{\fn,k},\cR_\rJ(\bV)/\fp^k\),
\]
which, by (1) for $\fn'$, is isomorphic to
\[
\Hom_{\dT^{\ang{\fn}}\otimes\Lambda_{\rI^0,\dZ_q}[\Gamma_F]}
\(\cD_\rJ(\bV^{\fn'})/\fp^{\fn',k},\cR_\rJ(\bV)/\fp^k\),
\]
which is further isomorphic to $\cC_\rJ^{\fn',k}$ by the Chebotarev density theorem. Thus, (2) for $\fn$ follows from (2) for $\fn'$.

It remains to prove ($*$). There are two cases. First, if $|\fn|>1$, then $\fn'$ is effective so that ($*$) follows by (4) for $\fn'$. Second, if $|\fn|=1$, then $\fn=\{\fl'\}$ for some \emph{effective} element $\fl'\in\fL_k$. The K\"{u}nneth decomposition induces isomorphisms:
\begin{align*}
\cD(\bV^\fn)/\fp^{\fm,k}&\simeq\cD(\xi^{n_0},\rV^\fn_{n_0},\rK^\fn_{n_0})_{\gamma_{n_0}}/\fp_0^{\fm,k}
\otimes_{\dZ_q}\cD(\xi^{n_1},\rV^\fn_{n_1},\rK^\fn_{n_1})_{\gamma_{n_1}}/\fp_1^{\fm,k},\\
\rH^1_\unr(F_{w(\fl')},\cD(\bV)/\fp^{\fm,k})&\simeq
\rH^1_\unr(F_{w(\fl')},\cD(\xi^{n_0},\rV_{n_0},\rK_{n_0})_{\gamma_{n_0}}/\fp_0^{\fm,k})
\otimes_{\dZ_q}\(\cD(\xi^{n_1},\rV_{n_1},\rK_{n_1})_{\gamma_{n_1}}/\fp_1^{\fm,k}\)_{\Gamma_{F_{w(\fl')}}}.
\end{align*}
By (an analogue of) \cite{LTXZZ}*{Lemma~7.3.3}, the cycle class map gives rise to an isomorphism
\[
\cD(\xi^{n_1},\rV^\fn_{n_1},\rK^\fn_{n_1})_{\gamma_{n_1}}/\fp_1^{\fm,k}
\simeq\(\cD(\xi^{n_1},\rV_{n_1},\rK_{n_1})_{\gamma_{n_1}}/\fp_1^{\fm,k}\)_{\Gamma_{F_{w(\fl')}}}
\]
of $\dT^{\ang{\fm}}_{n_1}\otimes\Lambda_{\rI^0_{n_1},\dZ_q}$-modules. Applying Lemma \ref{le:congruence2} below (with $\fn=\emptyset$ and $\fl=\fl'$), we know that the $\dT^{\ang{\fm}}_{n_0}\otimes\Lambda_{\rI^0_{n_0},\dZ_q}$-modules $\cD(\xi^{n_0},\rV^\fn_{n_0},\rK^\fn_{n_0})_{\gamma_{n_0}}/\fp_0^{\fm,k}$ and $\rH^1_\unr(F_{w(\fl')},\cD(\xi^{n_0},\rV_{n_0},\rK_{n_0})_{\gamma_{n_0}}/\fp_0^{\fm,k})$, both nontrivial, are isomorphic. Thus, ($*$) follows.

We then consider the case where $\fn\in\fN_k^\indef$. As (3) is vacuous, it suffices to prove (1,2,4). Choose an element $\fl'$ of $\fn$ so that $\fn'\coloneqq\fn/\fl'$ remains effective. We then have an isomorphism
\[
\rH^1_\sing(F_{w(\fl)},\cD_\rJ(\bV^{\fn})/\fp^{\fm,k})
\xrightarrow\sim\cD_\rJ(\bV^{\fn'})/\fp^{\fm,k}
\]
by (3) for $\fn'$. By (1) for $\fn'$, we know that the natural map
\[
\rH^1_\sing(F_{w(\fl)},\cD_\rJ(\bV^{\fn})/\fp^{\fm,k})\to\rH^1_\sing(F_{w(\fl)},\cD_\rJ(\bV^{\fn})/\fp^{\fn,k})
\]
is an isomorphism, which implies (1) for $\fn$ by Lemma \ref{le:congruence}(1c). On the other hand, we have
\begin{align*}
\cC_\rJ^{\fn',k}&=\Hom_{\dT^{\ang{\fn'}}\otimes\Lambda_{\rI^0,\dZ_q}}\(\cD_\rJ(\bV^{\fn'})/\fp^{\fn',k},\cE_\rJ(\bV)/\fp^k\)\\
&=\Hom_{\dT^{\ang{\fn'}}\otimes\Lambda_{\rI^0,\dZ_q}}\(\cD_\rJ(\bV^{\fn'})/\fp^{\fn,k},\cE_\rJ(\bV)/\fp^k\)\\
&=\Hom_{\dT^{\ang{\fn}}\otimes\Lambda_{\rI^0,\dZ_q}}\(\cD_\rJ(\bV^{\fn'})/\fp^{\fn,k},\cE_\rJ(\bV)/\fp^k\)\\
&\simeq\Hom_{\dT^{\ang{\fn}}\otimes\Lambda_{\rI^0,\dZ_q}}
\(\rH^1_\sing(F_{w(\fl)},\cD_\rJ(\bV^{\fn})/\fp^{\fn,k}),\cE_\rJ(\bV)/\fp^k\) \\
&\simeq\Hom_{\dT^{\ang{\fn}}\otimes\Lambda_{\rI^0,\dZ_q}[\Gamma_F]}
\(\cD_\rJ(\bV^{\fn})/\fp^{\fn,k},\cR_\rJ(\bV)/\fp^k\)=\cC_\rJ^{\fn,k},
\end{align*}
where we have used (1) for $\fn'$, the Chebotarev density theorem, and Lemma \ref{le:congruence}(1b). Thus, (2) for $\fn$ follows from (2) for $\fn'$.

It remains to verify (4) for $\fn$, for which we may assume $\rJ$ trivial. By the K\"{u}nneth decomposition, the map in question is the tensor product of the Abel--Jacobi map
\begin{align}\label{eq:congruence1}
\cD(\xi^{n_0},\rV^{\fn\fl}_{n_0},\rK^{\fn\fl}_{n_0})_{\gamma_{n_0}}/\fp_0^{\fm,k}
\to\rH^1_\unr(F_{w(\fl)},\cD(\xi^{n_0},\rV^\fn_{n_0},\rK^\fn_{n_0})_{\gamma_{n_0}}/\fp_0^{\fm,k})
\end{align}
and the cycle class map
\begin{align}\label{eq:congruence2}
\cD(\xi^{n_1},\rV^{\fn\fl}_{n_1},\rK^{\fn\fl}_{n_1})_{\gamma_{n_1}}/\fp_1^{\fm,k}
\to\(\cD(\xi^{n_1},\rV^\fn_{n_1},\rK^\fn_{n_1})_{\gamma_{n_1}}/\fp_1^{\fm,k}\)_{\Gamma_{F_{w(\fl)}}}.
\end{align}
By (an analogue of) \cite{LTXZZ}*{Lemma~7.3.3}, \eqref{eq:congruence2} is an isomorphism. Since the target of \eqref{eq:congruence1} is nontrivial by (2) for $\fn$, it suffices to show that \eqref{eq:congruence1} is surjective in view of Lemma \ref{le:congruence2} below. For the surjectivity of \eqref{eq:congruence1}, it suffices to show that the Abel--Jacobi map
\[
\Gamma(\Sh(\sfG^{\fn\fl}_{n_0},\rK^{\fn\fl}_{n_0}\tj_{n_0}^{\fn\fl}\rk_0^{c,l}),\dZ_{\xi^{n_0}}^\vee)/\fp_0^{\fm,k}\to
\rH^1_\unr(F_{w(\fl)},
\rH^{2r_0-1}_{\et}(\Sh(\sfG^\fn_{n_0},\rK^\fn_{n_0}\tj_{n_0}^\fn\rk_0^{c,l})_{\ol{F}},\dZ_{\xi^{n_0}}^\vee(r_0))/\fp_0^{\fm,k})
\]
is surjective for every integers $0\leq c\leq l$, where $\rk_0^{c,l}$ are open compact subgroups of $\sfG_{n_0}(F^+_{\PP^+})$ introduced right above Lemma \ref{le:first2}. Take an arbitrary pair of integers $0\leq c\leq l$. By \cite{LTX}*{Theorem~3.10.6~\&~Remark~3.10.7}, it suffices to show that the $\Gamma_F$-equivariant map
\[
\vartheta^\eta\colon
\rH^{2r_0-1}_{\et}(\Sh(\sfG^\fn_{n_0},\pres{\varpi}\rK^\fn_{n_0}\tj_{n_0}^\fn\rk_0^{c,l})_{\ol{F}},
\Omega^\eta_{n_0}\otimes_{\dZ_p}\dZ_{\xi^{n_0}}^\vee(r_0))
\to
\rH^{2r_0-1}_{\et}(\Sh(\sfG^\fn_{n_0},\rK^\fn_{n_0}\tj_{n_0}^\fn\rk_0^{c,l})_{\ol{F}},\dZ_{\xi^{n_0}}^\vee(r_0))
\]
is surjective after taking quotient by $\fp_0^{\fm,k}$. By Lemma \ref{le:congruence}(1c), it suffices to show that $\rH^1_\sing(F_{w(\fl')},\vartheta^\eta/\fp_0^{\fm,k})$ is surjective. Now we have a commutative diagram
\[
\resizebox{\hsize}{!}{
\xymatrix{
\rH^1_\sing(F_{w(\fl')},\rH^{2r_0-1}_{\et}(\Sh(\sfG^\fn_{n_0},\pres{\varpi}\rK^\fn_{n_0}\tj_{n_0}^\fn\rk_0^{c,l})_{\ol{F}},
\Omega^\eta_{n_0}\otimes_{\dZ_p}\dZ_{\xi^{n_0}}^\vee(r_0))/\fp_0^{\fm,k}) \ar[r]^-{\text{\ding{192}}} &
\rH^1_\sing(F_{w(\fl')},\rH^{2r_0-1}_{\et}(\Sh(\sfG^\fn_{n_0},\rK^\fn_{n_0}\tj_{n_0}^\fn\rk_0^{c,l})_{\ol{F}},\dZ_{\xi^{n_0}}^\vee(r_0))/\fp_0^{\fm,k}) \\
\rF_{-1}\rH^1(\rI_{F_{w(\fl')}},\rH^{2r_0-1}_{\et}(\Sh(\sfG^\fn_{n_0},\pres{\varpi}\rK^\fn_{n_0}\tj_{n_0}^\fn\rk_0^{c,l})_{\ol{F}},
\Omega^\eta_{n_0}\otimes_{\dZ_p}\dZ_{\xi^{n_0}}^\vee(r_0))/\fp_0^{\fm,k}) \ar[r]^-{\text{\ding{193}}} \ar[u]^-{\text{\ding{195}}} \ar@{->>}[d]_-{\text{\ding{197}}} &
\rF_{-1}\rH^1(\rI_{F_{w(\fl')}},\rH^{2r_0-1}_{\et}(\Sh(\sfG^\fn_{n_0},\rK^\fn_{n_0}\tj_{n_0}^\fn\rk_0^{c,l})_{\ol{F}},\dZ_{\xi^{n_0}}^\vee(r_0))/\fp_0^{\fm,k}) \ar[u]_-{\text{\ding{196}}}^\simeq \ar[d]^-{\text{\ding{198}}}_\simeq \\
\Gamma(\Sh(\sfG^{\fn'}_{n_0},\rK^{\fn'}_{n_0}\tj_{n_0}^{\fn'}\rk_0^{c,l}),\Omega^\eta_{n_0}\otimes_{\dZ_p}\dZ_{\xi^{n_0}}^\vee)/\fp_0^{\fm/\fl',k} \ar[r]^-{\text{\ding{194}}} & \Gamma(\Sh(\sfG^{\fn'}_{n_0},\rK^{\fn'}_{n_0}\tj_{n_0}^{\fn'}\rk_0^{c,l}),\dZ_{\xi^{n_0}}^\vee)/\fp_0^{\fm/\fl',k}
}
}
\]
in which the map \ding{192} is $\rH^1_\sing(F_{w(\fl')},\vartheta^\eta/\fp_0^{\fm,k})$; the map \ding{193} is $\rF_{-1}\rH^1(\rI_{F_{w(\fl')}},\vartheta^\eta/\fp_0^{\fm,k})$; the map \ding{194} is dual to the map $\bi$ defined above \cite{LTXZZ}*{Corollary~6.3.5} (for $\rV_N=\pres\varpi\rV^{\fn'}_{n_0}$ and $\fp=\fl$); the maps \ding{195} and \ding{196} are the maps from \eqref{eq:first2}; the maps \ding{197} and \ding{198} are the surjective maps from Lemma \ref{le:first1}(3). Since $\Gamma(\Sh(\sfG^{\fn'}_{n_0},\rK^{\fn'}_{n_0}\tj_{n_0}^{\fn'}\rk_0^{c,l}),\dZ_{\xi^{n_0}}^\vee)/\fp_0^{\fm/\fl',k}$ is nontrivial by (2) for $\fn'$, we may apply Lemma \ref{le:first2} to conclude that both \ding{196} and \ding{198} are isomorphisms. Now by a similar argument for \cite{LTXZZ}*{Corollary~6.3.5} using Lemma \ref{le:first2}, we know that \text{\ding{194}} is surjective. It follows that $\rH^1_\sing(F_{w(\fl')},\vartheta^\eta/\fp_0^{\fm,k})$ is surjective.

The proposition is all proved.
\end{proof}

\begin{lem}\label{le:congruence2}
Let $k$ be a positive integer. For every arrow $a(\fn,\fn\fl)$ of $\fX_k$ with $\fn\in\fN_k^\indef$ and every $\fm\in\fN_k$ containing $\fn\fl$, if both of the following two $\dT^{\ang{\fm}}_{n_0}\otimes\Lambda_{\rI^0_{n_0},\dZ_q}$-modules
\[
\cD(\xi^{n_0},\rV^{\fn\fl}_{n_0},\rK^{\fn\fl}_{n_0})_{\gamma_{n_0}}/\fp_0^{\fm,k},\quad
\rH^1_\unr(F_{w(\fl)},\cD(\xi^{n_0},\rV^{\fn}_{n_0},\rK^{\fn}_{n_0})_{\gamma_{n_0}}/\fp_0^{\fm,k})
\]
are nontrivial, then they are isomorphic.
\end{lem}

\begin{proof}
This can be proved by the same argument for \cite{LTX}*{Lemma~5.3.9}, using Theorem \ref{th:deformation} below instead of \cite{LTXZZ1}*{Theorem~3.38}, and Lemma \ref{le:first2} instead of \cite{LTXZZ}*{Theorem~6.3.4}.
\end{proof}

\begin{corollary}\label{co:congruence}
Let $k\geq 1$ be a positive integer and $\rJ<\rI^0$ a subgroup. For every arrow $a=a(\fn,\fn\fl)$ of $\fX_k$ with $\fn\in\fN_k^\eff\setminus(\{\emptyset\}\cap\fN_k^\indef)$, the map $\varrho_\rJ^{a,k}$ is an isomorphism.
\end{corollary}

\begin{proof}
This follows from Definition \ref{de:congruence1} and Proposition \ref{pr:congruence}(1,3,4).
\end{proof}

\begin{definition}\label{de:selmer2}
Let $k\geq 1$ be a positive integer and $\rJ<\rI^0$ a subgroup.
\begin{enumerate}
  \item For every place $w$ of $F$, we denote by $\rH^1_f(F_w,\cR_\rJ(\bV)/\fp^k)$ the propagation of $\rH^1_f(F_w,\cR_\rJ(\bV))$, that is, its image under the natural map $\rH^1(F_w,\cR_\rJ(\bV))\to\rH^1(F_w,\cR_\rJ(\bV)/\fp^k)$.

  \item For $\fl\in\fL_k$ (with the underlying rational prime $\ell$), we denote by $\rH^1_\ordi(F_{w(\fl)},\cR_\rJ(\bV)/\fp^k)$ the image of the natural map
      \[
      \rH^1(F_{w(\fl)},\cR_\rJ(\bV)/\fp^k)^{\phi_{w(\fl)}=\ell^2}\to\rH^1(F_{w(\fl)},\cR_\rJ(\bV)/\fp^k).
      \]

  \item For $\fn\in\fN_k$, put
      \begin{align*}
      \rH^1_{(\fn)}(F,\cR_\rJ(\bV)/\fp^k)&\coloneqq
      \Ker\(\rH^1(F,\cR_\rJ(\bV)/\fp^k)\to
      \prod_{\substack{w\in\Sigma\setminus\PP \\ w\nmid\fn}}
      \frac{\rH^1(F_w,\cR_\rJ(\bV)/\fp^k)}{\rH^1_f(F_w,\cR_\rJ(\bV)/\fp^k)}\times
      \prod_{\fl\in\fn}\frac{\rH^1(F_{w(\fl)},\cR_\rJ(\bV)/\fp^k)}{\rH^1_\ordi(F_{w(\fl)},\cR_\rJ(\bV)/\fp^k)}\),\\
      \rH^1_{f(\fn)}(F,\cR_\rJ(\bV)/\fp^k)&\coloneqq
      \Ker\(\rH^1(F,\cR_\rJ(\bV)/\fp^k)\to
      \prod_{\substack{w\in\Sigma \\ w\nmid\fn}}
      \frac{\rH^1(F_w,\cR_\rJ(\bV)/\fp^k)}{\rH^1_f(F_w,\cR_\rJ(\bV)/\fp^k)}\times
      \prod_{\fl\in\fn}\frac{\rH^1(F_{w(\fl)},\cR_\rJ(\bV)/\fp^k)}{\rH^1_\ordi(F_{w(\fl)},\cR_\rJ(\bV)/\fp^k)}\).
      \end{align*}
\end{enumerate}
We have similar definitions when one replaces $\cR_\rJ(\bV)$ by $\cR_\rJ(\bV)_{/x}$, $\fp$ by $\fp_x$, and $\rH^1_f$ by $\rH^1_\ff$ for every closed point $x$ of $\Spec\sE_\rJ(\bV)$.
\end{definition}

By Remark \ref{re:location}(3), for every $\fl\in\fL_k$, we have
\[
\rH^1_f(F_{w(\fl)},\cR_\rJ(\bV)/\fp^k)=\rH^1_\unr(F_{w(\fl)},\cR_\rJ(\bV)/\fp^k)
\]
and that the natural map
\[
\rH^1_\ordi(F_{w(\fl)},\cR_\rJ(\bV)/\fp^k)\to\rH^1_\sing(F_{w(\fl)},\cR_\rJ(\bV)/\fp^k)
\]
is an isomorphism. In particular, for every $\fl\in\fL_k$, we have the localization map
\[
\loc_{w(\fl)}\colon\rH^1_{(\fn)}(F,\cR_\rJ(\bV)/\fp^k)\to
\begin{dcases}
\rH^1_\ordi(F_{w(\fl)},\cR_\rJ(\bV)/\fp^k)=\cE_\rJ(\bV)/\fp^k, & \fl\in\fn, \\
\rH^1_f(F_{w(\fl)},\cR_\rJ(\bV)/\fp^k)=\cE_\rJ(\bV)/\fp^k, & \fl\not\in\fn.
\end{dcases}
\]

\begin{definition}
Let $k\geq 1$ be a positive integer and $\rJ<\rI^0$ a subgroup. Take $\fn\in\fN_k$.
\begin{enumerate}
  \item When $\fn\in\fN_k^\defin\setminus\fN_k^\eff$, denote $\blambda_\rJ^{\fn,k}=0$. When $\fn\in\fN_k^\defin\cap\fN_k^\eff$, denote by $\blambda_\rJ^{\fn,k}$ the image of $\blambda_\rJ(\bV^\fn)$ under the natural composite map
      \[
      \cD_\rJ(\bV^\fn)^\sfH\to\(\cD_\rJ(\bV^\fn)/\fp^k\)^\sfH\to
      \Hom_{\dT^{\ang{\fn}}\otimes\Lambda_{\rI^0,\dZ_q}}\(\cC_\rJ^{\fn,k},\cE_\rJ(\bV)/\fp^k\)^\sfH.
      \]
      When $\fn=\emptyset$, we also have the limit version
      \[
      \blambda_\rJ^{\emptyset,\infty}\in
      \Hom_{\dT^{\ang{\emptyset}}\otimes\Lambda_{\rI^0,\dZ_q}}\(\cC_\rJ^{\emptyset,\infty},\cE_\rJ(\bV)\)^\sfH.
      \]

  \item When $\fn\in\fN_k^\indef\setminus\fN_k^\eff$, denote $\bkappa_\rJ^{\fn,k}=0$. When $\fn\in\fN_k^\indef\cap\fN_k^\eff$, denote by $\bkappa_\rJ^{\fn,k}$ the image of $\bkappa_\rJ(\bV^\fn)$ under the natural composite map
      \[
      \rH^1(F,\cD_\rJ(\bV^\fn))^\sfH\to\rH^1(F,\cD_\rJ(\bV^\fn)/\fp^k)^\sfH
      \to\Hom_{\dT^{\ang{\fn}}\otimes\Lambda_{\rI^0,\dZ_q}}\(\cC_\rJ^{\fn,k},\rH^1(F,\cR_\rJ(\bV)/\fp^k)\)^\sfH.
      \]
      When $\fn=\emptyset$, we also have the limit version
      \[
      \bkappa_\rJ^{\emptyset,\infty}\in
      \Hom_{\dT^{\ang{\emptyset}}\otimes\Lambda_{\rI^0,\dZ_q}}\(\cC_\rJ^{\emptyset,\infty},\rH^1(F,\cR_\rJ(\bV))\)^\sfH,
      \]
      which is actually an element of $\Hom_{\dT^{\ang{\emptyset}}\otimes\Lambda_{\rI^0,\dZ_q}}\(\cC_\rJ^{\emptyset,\infty},\rH^1_f(F,\cR_\rJ(\bV))\)^\sfH$ by Lemma \ref{le:crystalline}.

  \item For every closed point $x$ of $\Spec\sE_\rJ(\bV)$, denote by
      \[
      \blambda^{\fn,k}_x\in\Hom_{\dT^{\ang{\fn}}\otimes\Lambda_{\rI^0,\dZ_q}}\(\cC_\rJ^{\fn,k},\dZ_x/\fp_x^k\)^\sfH
      \]
      and
      \[
      \bkappa^{\fn,k}_x\in
      \Hom_{\dT^{\ang{\fn}}\otimes\Lambda_{\rI^0,\dZ_q}}\(\cC_\rJ^{\fn,k},\rH^1(F,\cR_\rJ(\bV)_{/x}/\fp_x^k)\)^\sfH
      \]
      the images of $\blambda^{\fn,k}_\rJ$ and $\bkappa^{\fn,k}_\rJ$, respectively, under the specialization map $\cE_\rJ(\bV)\to\dZ_x$.
\end{enumerate}
\end{definition}

\begin{proposition}\label{pr:specialization}
Let $\rJ<\rI^0$ be a subgroup.
\begin{enumerate}
  \item For every $k\geq 1$ and every $\fn\in\fN_k^\indef$, we have
      \[
      \bkappa_\rJ^{\fn,k}\in\Hom_{\dT^{\ang{\fn}}\otimes\Lambda_{\rI^0,\dZ_q}}
      \(\cC_\rJ^{\fn,k},\rH^1_{(\fn)}(F,\cR_\rJ(\bV)/\fp^k)\)^\sfH.
      \]

  \item For every affinoid subdomain $\sU$ of $\sE_\rJ^\temp$, there exists an integer $k(\sU)\geq 0$ depending only on $\sU$ such that for every $k\geq 1$ and every $\fn\in\fN_k^\indef$, we have
      \[
      p^{k(\sU)}\cdot\bkappa^{\fn,k}_x\in\Hom_{\dT^{\ang{\fn}}\otimes\Lambda_{\rI^0,\dZ_q}}
      \(\cC_\rJ^{\fn,k},\rH^1_{\ff(\fn)}(F,\cR_\rJ(\bV)_{/x}/\fp_x^k)\)^\sfH.
      \]
\end{enumerate}
\end{proposition}

\begin{proof}
The proof of (1) goes in the same way as that of \cite{LTX}*{Proposition~5.5.6(1,2)} except that we need in addition to show that for every place $w\in\spadesuit$, we have
\[
\loc_w\circ\bkappa^{\fn,k}\in\Hom_{\dT^{\ang{\fn}}\otimes\Lambda}\(\cC^{\fn,k},\rH^1_\ff(F_w,\cR_\rJ(\bV)/\fp^k)\)
\]
for every $k\geq 1$ and every $\fn\in\fN_k^\indef$. For this, it suffices to show that $\rH^1(\rI_{F_w},\cD_\rJ(\bV^\fn))^{\Gamma_{\kappa_w}}$ vanishes for every $\fn\in\fN^\indef$. Let $\rP_{F_w}$ be the maximal subgroup of $\rI_{F_w}$ whose pro-order is coprime to $p$ so that $\rI_{F_w}/\rP_{F_w}\simeq\dZ_p$. By \cite{CHT08}*{Lemma~2.4.11}, $\cD^\fn\coloneqq\cD_\rJ(\bV^\fn)^{\rP_{F_w}}$ is an $\cE_\rJ(\bV)[\Gamma_{F_w}]$-linear direct summand such that the natural map $\rH^1(\rI_{F_w},\cD^\fn)\to\rH^1(\rI_{F_w},\cD_\rJ(\bV^\fn))$ is an isomorphism. Assumption \ref{as:galois}(G5) implies that there exists a free $\dZ_q$-module $D^\fn$ contained in $\cD^\fn$ satisfying $\cD^\fn=D^\fn\otimes_{\dZ_q}\Lambda_{\rI^0/\rJ,\dZ_q}$, such that the action of $\rI_{F_w}$ on $\cD^\fn$ factors through $D^\fn$ and that $\rH^1(\rI_{F_w},D^\fn)$ is a finite free $\dZ_q$-module. It follows that
\[
\rH^1(\rI_{F_w},\cD_\rJ(\bV^\fn))^{\Gamma_{\kappa_w}}=\(\rH^1(\rI_{F_w},D^\fn)\otimes_{\dZ_q}\Lambda\)^{\Gamma_{\kappa_w}},
\]
which is a finitely generated torsion free $\Lambda_{\rI^0/\rJ,\dZ_q}$-module. If it is nonzero, then its support must be the whole $\Spec\Lambda_{\rI^0/\rJ,\dZ_q}$. However, this is not possible since the specialization of $\rH^1(\rI_{F_w},\cD_\rJ(\bV^\fn))^{\Gamma_{\kappa_w}}$ at every classical point of $\Spec\Lambda_{\rI^0/\rJ,\dQ_q}$ vanishes. Thus, $\rH^1(\rI_{F_w},\cD_\rJ(\bV^\fn))^{\Gamma_{\kappa_w}}$ vanishes and (1) follows.

For (2), in view of Lemma \ref{le:specialization}, it suffices to show that for every $w\in\PP$, there exists an integer $k(\sU)\geq 0$ depending only on $\sU$ such that for every $k\geq 1$ and every $\fn\in\fN_k^\indef$, we have
\[
p^{k(\sU)}\cdot\loc_w\circ\bkappa^{\fn,k}_x
\in\Hom_{\dT^{\ang{\fn}}\otimes\Lambda_{\rI^0,\dZ_q}}\(\cC_\rJ^{\fn,k},\rH^1_\ff(F_w,\cR_\rJ(\bV)_{/x}/\fp_x^k)\).
\]

We first take an element $w\in\PP$. By Lemma \ref{le:crystalline} and Lemma \ref{le:selmer2}(1), we have
\[
\loc_w(\bkappa_\rJ(\bV^\fn))\in\Ker\(\rH^1(F_w,\cD_\rJ(\bV^\fn))\to
\rH^1(\rI_{F_w},\cD_\rJ(\bV^\fn)/\Fil_w^{-1}\cD_\rJ(\bV^\fn))\).
\]
By condition (ii) in Theorem \ref{th:iwasawa}, the tautological map
\[
\cD_\rJ(\bV^\fn)\to\Hom_{\dT^{\ang{\fn}}\otimes\Lambda_{\rI^0,\dZ_q}}\(\cC_\rJ^{\fn,k},\cR_\rJ(\bV)/\fp^k\)
\]
must preserve the filtration $\Fil^\bullet_w$. It follows that
\[
\loc_w\circ\bkappa^{\fn,k}_x\in\Hom_{\dT^{\ang{\fn}}\otimes\Lambda_{\rI^0,\dZ_q}}\(\cC_\rJ^{\fn,k},
\Ker\(\rH^1(F_w,\cR_\rJ(\bV)_{/x}/\fp_x^k)\to\rH^1(\rI_{F_w},(\cR_\rJ(\bV)_{/x}/\fp_x^k)/\Fil^{-1}_w(\cR_\rJ(\bV)_{/x}/\fp_x^k))\)\).
\]
The rest of the proof follows from the same way as \cite{LTX}*{Proposition~5.6.4}.
\end{proof}

For every integer $k\geq 0$, denote by $\fX_k^{\r{arrow}}$ the set of arrows of $\fX_k$. For every integer $k\geq 1$ and every subgroup $\rJ<\rI^0$, we call the data
\begin{align*}
\left(\{\bV^\fn,\tj^\fn,\cC_\rJ^{\fn,k}\res\fn\in\fN_k\},\{\tj^a,\varrho_\rJ^{a,k}\res a\in\fX_k^{\r{arrow}}\},
\{\blambda_\rJ^{\fn,k}\res\fn\in\fN_k^\defin\},\{\bkappa_\rJ^{\fn,k}\res\fn\in\fN_k^\indef\}\right)
\end{align*}
defined above the \emph{bipartite Euler system} for the Galois module $\cR_\rJ(\bV)$ over the (integral) eigenvariety $\Spec\cE_\rJ(\bV)$. It enjoys the following remarkable relations.

\begin{theorem}\label{th:euler}
Let $a=a(\fn,\fn\fl)$ be an arrow of $\fX_k$ for some integer $k\geq 1$ and $\rJ<\rI^0$ a subgroup.
\begin{enumerate}
  \item When $\fn\in\fN_k^\defin$, the diagram
      \[
      \xymatrix{
      \cC_\rJ^{\fn,k} \ar[rr]^-{\varrho_\rJ^{a,k}} \ar[d]_-{\blambda_\rJ^{\fn,k}} && \cC_\rJ^{\fn\fl,k} \ar[d]^-{\bkappa_\rJ^{\fn\fl,k}} \\
      \cE_\rJ(\bV)/\fp^k && \rH^1_{(\fn\fl)}(F,\cR_\rJ(\bV)/\fp^k) \ar[ll]_-{\loc_{w(\fl)}}
      }
      \]
      commutes.

  \item When $\fn\in\fN_k^\indef$, the diagram
      \[
      \xymatrix{
      \cC_\rJ^{\fn,k} \ar[rr]^-{\varrho_\rJ^{a,k}} \ar[d]_-{\bkappa_\rJ^{\fn,k}} && \cC_\rJ^{\fn\fl,k} \ar[d]^-{\blambda_\rJ^{\fn\fl,k}} \\
      \rH^1_{(\fn)}(F,\cR_\rJ(\bV)/\fp^k) \ar[rr]^-{\loc_{w(\fl)}}  && \cE_\rJ(\bV)/\fp^k
      }
      \]
      commutes.
\end{enumerate}
\end{theorem}

\begin{proof}
Parts (1) and (2) follow from Theorem \ref{th:first} and Theorem \ref{th:second}, respectively.
\end{proof}

\section{Proof of Theorem \ref{th:iwasawa}}
\label{ss:proof}

In this section, we complete the proof of Theorem \ref{th:iwasawa}. In this section, the subgroup $\rJ<\rI^0$ will be fixed and the data $\bV^\fn$ will not be explicitly used for $\fn\neq\emptyset$. Thus, we will further suppress $\rJ$ in all subscripts (but this does not mean that $\rJ$ is trivial) and also suppress the part $(\bV)$; for example, $\cE_\rJ(\bV),\cR_\rJ(\bV)$ in this section are just $\cE,\cR$. We also write $\Lambda,\Lambda_\eta$ for $\Lambda_{\rI^0/\rJ,\dZ_q},\Lambda_{\rI^0/\rJ,\dQ_q}$ for short.

\begin{lem}\label{le:iwasawa2}
For every closed point $x$ of $\Spec\sE$ and every integer $k\geq 1$,
\begin{enumerate}
  \item the natural map $\rH^1_\ff(F,\cR_{/x})/p_x^k\rH^1_\ff(F,\cR_{/x})\to\rH^1_\ff(F,\cR_{/x}/p_x^k)$ is injective;

  \item the natural map $\rH^1_\ff(F,(\cR_{/x})^*(1)[p_x^k])\to\rH^1_\ff(F,(\cR_{/x})^*(1))[p_x^k]$ is an isomorphism.
\end{enumerate}
\end{lem}

\begin{proof}
Since the quotient $\dZ_x$-module $\rH^1(F_w,\cR_{/x})/\rH^1_\ff(F_w,\cR_{/x})$ is torsion free for every place $w$ of $F$ by Definition \ref{de:selmer1}, (1) follows from \cite{MR04}*{Lemma~3.7.1}. Since $\cR_{/x}$ is residually absolutely irreducible of dimension at least two, (2) follows by \cite{MR04}*{Lemma~3.5.4}.
\end{proof}

Take an affinoid subdomain $\sU$ as in Proposition \ref{pr:specialization}(2) and a closed point $x$ in $\sU$. Put
\begin{align*}
\lambda_{\sU,x}^{\fn,k}&\coloneqq p^{k(\sU)}\cdot\blambda^{\fn,k}_x
\colon\cC^{\fn,k}\to\dZ_x/\fp_x^k,\\
\kappa_{\sU,x}^{\fn,k}&\coloneqq p^{k(\sU)}\cdot\bkappa^{\fn,k}_x
\colon\cC^{\fn,k}\to \rH^1_{\ff(\fn)}(F,(F,\cR_{/x}/\fp_x^k)),
\end{align*}
(including $k=\infty$) whenever applicable.

For every pair of integers $1\leq k\leq j$, put
\[
\delta_{\sU,x}(k,j)\coloneqq
\min\left\{\left.\min_{\phi\in\cC^{\fn,k}}\ind\(\lambda_{\sU,x}^{\fn,k}(\phi),\dZ_x/\fp_x^k\)\,\right| \fn\in\fN^\defin_j\right\},
\]
where
\[
\ind\(\lambda_{\sU,x}^{\fn,k}(\phi),\dZ_x/\fp_x^k\)\coloneqq
\sup\left\{i\geq0\left|\lambda_{\sU,x}^{\fn,k}(\phi)\in (p_x^i+\fp_x^k)/\fp_x^k\right.\right\}.
\]
As $\delta_{\sU,x}(k,j)\leq\delta_{\sU,x}(k,j+1)$, we may put
\[
\delta_{\sU,x}(k)\coloneqq\lim_{j\to\infty}\delta_{\sU,x}(k,j)\leq\infty.
\]
We omit $\sU$ in all the subscripts when it is irrelevant or clear from the context.

\begin{proposition}\label{pr:iwasawa}
Let $x$ be a closed point of $\sE^\temp$ and $k\geq 1$ an integer.
\begin{enumerate}
  \item Suppose that $\epsilon(\bV)=1$. If $\lambda_x^{\emptyset,k}\neq 0$, then
      \begin{enumerate}
        \item $\rH^1_\ff(F,\cR_{/x})$ vanishes;

        \item $\rH^1_\ff(F,(\cR_{/x})^*(1))$ has $\dZ_x$-corank zero;

        \item we have
            \[
            \length_{\dZ_x}\rH^1_\ff(F,(\cR_{/x})^*(1))^*+2\delta_x(k)
            =2\cdot\min_{\phi\in\cC^{\emptyset,\infty}}\length_{\dZ_x}\(\dZ_x/\lambda_x^{\emptyset,\infty}(\phi)\).
            \]
      \end{enumerate}

  \item Suppose that $\epsilon(\bV)=-1$. If $\kappa_x^{\emptyset,k}\neq 0$, then
      \begin{enumerate}
        \item $\rH^1_\ff(F,\cR_{/x})$ is free of rank one over $\dZ_x$;

        \item $\rH^1_\ff(F,(\cR_{/x})^*(1))$ has $\dZ_x$-corank one;

        \item we have
            \[
            \length_{\dZ_x}\rH^1_\ff(F,(\cR_{/x})^*(1))^*[p^\infty]+2\delta_x(k)
            =2\cdot\min_{\phi\in\cC^{\emptyset,\infty}}
            \length_{\dZ_x}\(\rH^1_\ff(F,\cR_{/x})/\dZ_x\kappa_x^{\emptyset,\infty}(\phi)\).
            \]
      \end{enumerate}
\end{enumerate}
\end{proposition}

\begin{proof}
This follows from the same proof of \cite{LTX}*{Proposition~5.6.5} in view of Lemma \ref{le:iwasawa2}, Corollary \ref{co:congruence}, and Remark \ref{re:selmer1}.
\end{proof}

We need a slight variant of \cite{LTX}*{Lemma~5.6.6}. Let $\cM$ be a finitely generated $\cE$-module, $\sE'$ an irreducible component of $\sE^\temp$ on which $\cM$ is torsion, and $z$ a point of $\sE'$ of codimension one. We say that a closed point $x$ of $\sE'$ is a \emph{$z$-clean point for $\cM$} if $x$ belongs to $\sV_z$ and satisfies that there exists a Zariski open neighbourhood $U$ of $x$ in $\sE'$ such that
\[
\cM\res_U\simeq\bigoplus_i(\cE/\cI_z^{n_i})\res_U.
\]
By \cite{Bour}*{Chap.~VII~\S4.4,~Theorem~5}, $z$-clean points for $\cM$ are Zariski dense in $\sV_z$.

\begin{lem}\label{le:iwasawa4}
Let $x$ be a closed point of $\cE'$ and $\sZ$ a proper Zariski closed subset of $\sE'$.
\begin{enumerate}
  \item There exists a sequence $(x_m)_{m\geq 1}$ of closed points of $\sE'\setminus\sZ$ such that
     \begin{enumerate}
       \item both sequences $\{|\dZ_{x_m}/\dZ'_{x_m}|\}_{m\geq 1}$ and $\{[\dQ_{x_m}:\dQ_q]\}_{m\geq 1}$ are bounded;

       \item $\length_{\dZ_q}\(\dZ'_{x_m}\otimes_\cE\dZ'_x\)$ tends to infinity when $m\to\infty$.
     \end{enumerate}

  \item Let $\cM$ be a finitely generated $\cE$-module and $z$ a point of $\sE'$ of codimension one, such that $x$ is a $z$-clean point for $\cM$. Then for every sequence $(x_m)_{m\geq 1}$ in (1), we have
      \[
      \length_{\cO_{\sE',z}}(\cM_z)=\lim_{m\to\infty}l_m^{-1}\cdot\length_{\dZ_q}\(\cM/\fP_{x_m}\cM\),
      \]
      where $l_m\coloneqq\length_{\dZ_q}\(\dZ'_{x_m}\otimes_\cE\cE/\cI_z\)$ is a sequence of positive integers that tends to infinity by (1b).
\end{enumerate}
\end{lem}

\begin{proof}
For every closed point $y$ of $\Spec\Lambda_\eta$, we similarly denote by $\dQ_y$ the residue field of $y$, $\dZ_y$ the ring of integers of $\dQ_y$, and $\dZ'_y\subseteq\dZ_y$ the image of the induced homomorphism $\Lambda\to\dQ_y$. Denote by $\pi\colon\sE'\to\Spec\Lambda_\eta$ the natural morphism. Put $y=\pi(x)$, which is a closed point of $\Spec\Lambda_\eta$.

For (1), by \cite{LTX}*{Lemma~5.6.6}, we may find a sequence $(y_m)_{m\geq 1}$ of closed points of $(\Spec\Lambda_\eta)\setminus\pi(\sZ)$ satisfying that $\dZ'_{y_m}\simeq\dZ'_y$ (as rings) for every $m\geq 1$ and
\[
\lim_{m\to\infty}\length_{\dZ_q}\(\dZ'_{y_m}\otimes_\Lambda\dZ'_y\)=\infty.
\]
As
\[
\dZ'_x\otimes_\cE\(\dZ'_{y_m}\otimes_\Lambda\cE\)=\dZ'_x\otimes_{\dZ'_y}\(\dZ'_{y_m}\otimes_\Lambda\dZ'_y\),
\]
we have
\[
\lim_{m\to\infty}\length_{\dZ_q}\(\dZ'_x\otimes_\cE\(\dZ'_{y_m}\otimes_\Lambda\cE\)\)=\infty.
\]
It follows that there exists a sequence $(x_m)_{m\geq 1}$ of closed points of $(\Spec\sE)\setminus\sZ$ with $y_m=\pi(x_m)$ that satisfies (1b). By throwing away finitely many terms, we may assume that $(x_m)_{m\geq 1}$ is contained in $\sE'$. We now show that this sequence satisfies (1a) as well.

It is clear that $[\dQ_{x_m}:\dQ_x]\leq[\dQ_{y_m}:\dQ_y]$, in which the latter is independent of $m$, so that the sequence $\{[\dQ_{x_m}:\dQ_q]\}_{m\geq 1}$ is bounded. For the sequence $\{|\dZ_{x_m}/\dZ'_{x_m}|\}_{m\geq 1}$, it suffices to show that $\dZ_{x_m}/\dZ'_{x_m}$ is annihilated by a power of $p$ independent of $m$. By \cite{SP}*{0335}, $\cE$ is a (Noetherian) Nagata ring, so that its normalization $\widetilde\cE$ is module-finite over $\cE$ \cite{SP}*{035S}. In particular, $\widetilde\cE/\cE$ is a finitely generated $\cE$-module, whose support is disjoint from $\sE'$ as the latter is regular. Fix an affinoid subdomain $\sU$ of $\sE'$ containing $x$. Then there exists an integer $k(\sU)\geq 0$ such that $p^{k(\sU)}$ annihilates $\Ann_\cE(\widetilde\cE/\cE)\otimes_{\cE}\dZ'_{x'}$ for every closed point $x'$ in $\sU$. Without loss of generality, we may assume that $x_m$ belongs to $\sU$ for every $m\geq 1$. Let $k\geq 0$ be an integer such that $p^k$ annihilates $\dZ_{y_m}/\dZ'_{y_m}$ for every $m\geq 1$. For every $m\geq 1$, denote by $\dZ''_{x_m}$ the image of the natural map $\dZ_{y_m}\otimes_\Lambda\cE\to\dQ_{x_m}$, so that we have the commutative
\[
\xymatrix{
\dZ'_{y_m}\otimes_\Lambda\cE \ar[r]\ar[d] & \dZ_{y_m}\otimes_\Lambda\cE \ar[r]\ar[d] & \dZ_{y_m}\otimes_\Lambda\widetilde\cE \ar[d] \\
\dZ'_{x_m} \ar[r] & \dZ''_{x_m} \ar[r] & \dZ_{x_m}
}
\]
in which all vertical maps are surjective. Then $p^k$ and $p^{k(\sU)}$ annihilate $\dZ''_{x_m}/\dZ'_{x_m}$ and $\dZ_{x_m}/\dZ''_{x_m}$, respectively, so that $p^{k+k(\sU)}$ annihilates $\dZ_{x_m}/\dZ'_{x_m}$ for every $m\geq 1$. Thus, (1) follows.

For (2), by definition ,we may fix a Zariski open neighbourhood $U$ of $x$ in $\sE'$ and an isomorphism
\[
\cM\res_U\simeq\bigoplus_i(\cE/\cI_z^{n_i})\res_U.
\]
Then we obtain maps $\varphi\colon\bigoplus_i(\cE/\cI_z^{n_i})\to i_*(\cM\res_U)$ and $\varphi'\colon\cM\to i_*(\cM\res_U)$ of quasi-coherent sheaves on $\Spec\cE$, where $i\colon U\to\Spec\cE$ is the open immersion. Denote by $\cM'$ the (finitely generated) $\cE$-submodule of $i_*(\cM\res_U)$ generated by the images of $\varphi$ and $\varphi'$. Thus, we obtain maps $\varphi\colon\bigoplus_i(\cE/\cI_z^{n_i})\to\cM'$ and $\varphi'\colon\cM\to\cM'$ of finitely generated $\cE$-modules that are isomorphisms after restriction to $U$. We claim that for every map $\varphi\colon\cM_1\to\cM_2$ of finitely generated $\cE$-modules that is an isomorphism after restriction to $U$,
\[
\lim_{m\to\infty}l_m^{-1}\cdot\length_{\dZ_q}\(\cM_1/\fP_{x_m}\cM_1\)
=\lim_{m\to\infty}l_m^{-1}\cdot\length_{\dZ_q}\(\cM_2/\fP_{x_m}\cM_2\)
\]
holds. Then (2) follows as
\begin{align*}
\length_{\cO_{\sE',z}}(\cM_z)=\sum_in_i
&=\lim_{m\to\infty}l_m^{-1}\cdot\length_{\dZ_q}\(\bigoplus_i(\cE/\cI_z^{n_i})/\fP_{x_m}\bigoplus_i(\cE/\cI_z^{n_i})\) \\
&=\lim_{m\to\infty}l_m^{-1}\cdot\length_{\dZ_q}\(\cM'/\fP_{x_m}\cM'\) \\
&=\lim_{m\to\infty}l_m^{-1}\cdot\length_{\dZ_q}\(\cM/\fP_{x_m}\cM\).
\end{align*}

For the claim, we may assume that $\varphi$ is either injective or surjective, and that the sequence $(x_m)_{m\geq 1}$ is contained in $U$. Suppose that $\varphi$ is surjective, and put $\cM_0\coloneqq\Ker\varphi$, which satisfies $(\cM_0)_x=0$. Then it is clear that $\cM_0\otimes_\cE\dZ'_{x_m}$ is annihilated a power of $p$ independent of $m$, which implies the claim by (1a). Suppose that $\varphi$ is injective, and put $\cM_0\coloneqq\coker\varphi$. Then in addition we need to show that $\Tor_1^\cE(\cM_0,\Theta_{x_m})$ is annihilated by a power of $p$ independent of $m$. Choose a short exact sequence
\[
0\to \cM'_0 \to \cM''_0 \to \cM_0\to 0
\]
of $\cE$-modules in which $\cM''_0$ is a finitely generated flat $\cE$-module. Then $\Tor_1^\cE(\cM_0,\Theta_{x_m})=(\cM'_0\otimes_\cE\Theta_{x_m})[p^\infty]$ for every $m\geq 1$. Write $\cM'_0\otimes_\cE\Theta_{x_m}=\Theta_{x_m}^{r_m}\oplus M_m$ for some integer $r_m\geq 0$ and some finitely generated torsion $\Theta_{x_m}$-module $M_m$. Write $\cM'_0\otimes_\cE\Theta_x=\Theta_x^r\oplus M$ in the similar way. Since $\cM'_0$ is also flat over $U$, we have $r=r_m$ for $m\geq 1$. Let $m_0$ be a positive integer such that $p^{m_0}$ annihilates $M$. By (1b), for $m$ sufficiently large, we have $\cM'_0\otimes_\cE(\Theta_{x_m}/p^{m_0+1}\Theta_{x_m})=\cM'_0\otimes_\cE(\Theta_x/p^{m_0+1}\Theta_x)$, so that $M_m$ is annihilated by $p^{m_0}$. Thus, the claim follows.

The lemma is proved.
\end{proof}

\begin{proof}[Proof of Theorem \ref{th:iwasawa}]
The proof is completely parallel to that of \cite{LTX}*{Theorem~5.1.3}. For readers' convenience, we include all details.

Denote by $\sZ$ the common vanishing locus in $\sE'$ of the image of the map $\blambda^{\emptyset,\infty}$ or $\bkappa^{\emptyset,\infty}$, depending on whether $\epsilon(\bV)=1$ or $\epsilon(\bV)=-1$, which a proper Zariski closed subset of $\sE'$ by our assumptions in (1) and (2), respectively.

We first consider (1). For every closed point $x$ of $\sE'\setminus\sZ$, there exists an integer $k=k(x)\geq 1$ such that $\lambda^{\emptyset,k}_x\neq 0$. By Proposition \ref{pr:iwasawa}(1a,1b), $\rH^1_\ff(F,\cR_{/x})=0$ and $\rH^1_\ff(F,(\cR_{/x})^*(1))$ has $\dZ_x$-corank zero. Thus, (1a) follows by Lemma \ref{le:free_incoherent} and Proposition \ref{pr:specialize}; and (1b) follows by Proposition \ref{pr:control}.

For (1c), it amounts to showing that for every closed point $z$ of $\sE'$ of height one,
\begin{align}\label{eq:iwasawa1}
\length_{\cO_{\sE',z}}\(\rH^1_f(F,\cR^*(1))^*\)_z\leq2\length_{\cO_{\sE',z}}\(\cE/\blambda^{\emptyset,\infty}(\phi)\)_z
\end{align}
holds for every $\phi\in\cC^{\emptyset,\infty}$. Take an element $\phi\in\cC^{\emptyset,\infty}$ such that $\blambda^{\emptyset,\infty}(\phi)$ does not vanish at $z$ (otherwise, \eqref{eq:iwasawa1} is trivial) and choose a $z$-clean point $x$ for both $\rH^1_f(F,\cR^*(1))^*$ and $\cE/\blambda^{\emptyset,\infty}(\phi)$. Applying Lemma \ref{le:iwasawa4}, we obtain a sequence $(x_m)_{m\geq 1}$ of closed points (in an affinoid subdomain containing $x$) of $\sE'$ satisfying Lemma \ref{le:iwasawa4}(1), such that
\begin{align*}
\length_{\cO_{\sE',z}}\(\rH^1_f(F,\cR^*(1))^*\)_z
&=\lim_{m\to\infty}l_m^{-1}\cdot\length_{\dZ_q}\(\rH^1_f(F,\cR^*(1))^*/\fP_{x_m}\rH^1_f(F,\cR^*(1))^*\),\\
\length_{\cO_{\sE',z}}\(\cE/\blambda^{\emptyset,\infty}(\phi)\)_z
&=\lim_{m\to\infty}l_m^{-1}\cdot\length_{\dZ_q}\(\cE/(\blambda^{\emptyset,\infty}(\phi),\fP_{x_m})\).
\end{align*}
Since $l_m$ tends to infinity, by Proposition \ref{pr:control} and Lemma \ref{le:iwasawa4}(1a), we have
\[
\lim_{m\to\infty}l_m^{-1}\cdot\length_{\dZ_q}\(\rH^1_f(F,\cR^*(1))^*/\fP_{x_m}\rH^1_f(F,\cR^*(1))^*\)
=\lim_{m\to\infty}l_m^{-1}\cdot\length_{\dZ_q}\(\rH^1_\ff(F,(\cR_{/x})^*(1))^*\),
\]
and
\[
\lim_{m\to\infty}l_m^{-1}\cdot\length_{\dZ_q}\(\cE/(\blambda^{\emptyset,\infty}(\phi),\fP_{x_m})\)
=\lim_{m\to\infty}l_m^{-1}\cdot\length_{\dZ_q}\(\dZ_{x_m}/\lambda^{\emptyset,\infty}_{x_m}(\phi)\).
\]
Now \eqref{eq:iwasawa1} follows from Proposition \ref{pr:iwasawa}(1c) by taking, for every $m\geq 1$, an integer $k_m\geq 1$ satisfying $\lambda^{\emptyset,k_m}_{x_m}\neq 0$.

We then consider (2). By Proposition \ref{pr:specialize} and Lemma \ref{le:iwasawa2}(1), there exists a Zariski dense open subset $U$ of $\sE'\setminus\sZ$ such that for every closed point $x$ of $U$, there exists an integer $k=k(x)\geq 1$ such that $\kappa^{\emptyset,k}_x\neq 0$. By Proposition \ref{pr:iwasawa}(2a,2b), $\rH^1_\ff(F,\cR_{/x})$ is free of rank one over $\dZ_x$ and $\rH^1_\ff(F,(\cR_{/x})^*(1))$ has $\dZ_x$-corank one. Thus, (2a) follows by Lemma \ref{le:free_incoherent} and Proposition \ref{pr:specialize}; and (2b) follows by Proposition \ref{pr:control}.

For (2c), it amounts to showing that for every closed point $z$ of $\sE'$ of height one,
\begin{align}\label{eq:iwasawa2}
\length_{\cO_{\sE',z}}\(\rH^1_f(F,\cR^*(1))^*_0\)_z\leq2
\length_{\cO_{\sE',z}}\(\rH^1_f(F,\cR)/\bkappa^{\emptyset,\infty}(\phi)\)_z
\end{align}
holds for every $\phi\in\cC^{\emptyset,\infty}$, where $\rH^1_f(F,\cR^*(1))^*_0$ denotes the maximal $\cE$-submodule of $\rH^1_f(F,\cR^*(1))^*$ that is torsion over $\sE'$. Take an element $\phi\in\cC^{\emptyset,\infty}$ such that $\bkappa^{\emptyset,\infty}(\phi)$ does not vanish at $z$ (otherwise, \eqref{eq:iwasawa1} is trivial) and choose a $z$-clean point $x$ for both $\rH^1_f(F,\cR^*(1))^*_0$ and $\rH^1_f(F,\cR)/\bkappa^{\emptyset,\infty}(\phi)$, such that $\rH^1_f(F,\cR^*(1))^*/\rH^1_f(F,\cR^*(1))^*_0$ is locally free at $x$ and that the natural map
\begin{align}\label{eq:iwasawa4}
\rH^1_f(F,\sR)/\fP_x\rH^1_f(F,\sR)\to\rH^1_\ff(F,\sR_{/x})
\end{align}
is injective -- this is possible by Proposition \ref{pr:specialize}. Applying Lemma \ref{le:iwasawa4}, we obtain a sequence $(x_m)_{m\geq 1}$ of closed points (in an affinoid subdomain containing $x$) of $\sE'$ satisfying Lemma \ref{le:iwasawa4}(1), such that
\begin{align*}
\length_{\cO_{\sE',z}}\(\rH^1_f(F,\cR^*(1))^*_0\)_z
&=\lim_{m\to\infty}l_m^{-1}\cdot\length_{\dZ_q}\(\rH^1_f(F,\cR^*(1))^*_0/\fP_{x_m}\rH^1_f(F,\cR^*(1))^*_0\),\\
\length_{\cO_{\sE',z}}\(\cE/\blambda^{\emptyset,\infty}(\phi)\)_z
&=\lim_{m\to\infty}l_m^{-1}\cdot\length_{\dZ_q}\(\rH^1_f(F,\cR)/(\blambda^{\emptyset,\infty}(\phi)+\fP_{x_m}\rH^1_f(F,\cR))\).
\end{align*}
Since $l_m$ tends to infinity, by Proposition \ref{pr:control} and Lemma \ref{le:iwasawa4}(1a), we have
\[
\lim_{m\to\infty}l_m^{-1}\cdot\length_{\dZ_q}\(\rH^1_f(F,\cR^*(1))^*_0/\fP_{x_m}\rH^1_f(F,\cR^*(1))^*_0\)
=\lim_{m\to\infty}l_m^{-1}\cdot\length_{\dZ_q}\(\rH^1_\ff(F,(\cR_{/x})^*(1))^*[p^\infty]\).
\]
We claim that
\begin{align}\label{eq:iwasawa3}
\lim_{m\to\infty}l_m^{-1}\cdot\length_{\dZ_q}\(\rH^1_f(F,\cR)/(\blambda^{\emptyset,\infty}(\phi)+\fP_{x_m}\rH^1_f(F,\cR))\)
\geq\lim_{m\to\infty}l_m^{-1}\cdot\length_{\dZ_q}\(\rH^1_\ff(F,\cR_{/x})/\kappa^{\emptyset,\infty}_{x_m}(\phi)\).
\end{align}
Assuming this, \eqref{eq:iwasawa2} follows from Proposition \ref{pr:iwasawa}(2c) by taking, for every $m\geq 1$, an integer $k_m\geq 1$ satisfying $\kappa_{x_m}^{\emptyset,k_m}\neq 0$.

To check \eqref{eq:iwasawa3}, it suffices to show that
\begin{align}\label{eq:iwasawa5}
\limsup_{m\to\infty}\length_{\dZ_q}\(\coker\(\rH^1_f(F,\cR)/\fP_{x_m}\rH^1_f(F,\cR)\to\rH^1_\ff(F,\cR_{/x})\)\)
\end{align}
is finite. Since the map \eqref{eq:iwasawa4} is injective and $\rH^1_\ff(F,\cR)$ has rank one, there exists a positive integer $k$ such that the natural map
\[
\rH^1_f(F,\cR)/(\fP_x,p^k)\rH^1_f(F,\cR)\to\rH^1_\ff(F,\cR_{/x}/p^k\cR_{/x})
\]
is nonzero. By Lemma \ref{le:iwasawa4}(1b), there exists an integer $m_k\geq 1$ such that for every $m\geq k$, $\Theta_x/p^k$ is a quotient of $\Theta_x\otimes_\cE\Theta_{x_m}$, which implies that $\fP_{x_m}\subseteq(\fP_x,p^k)$ hence $\dZ'_x/p^k\dZ'_x$ is a quotient of $\dZ'_{x_m}/p^k\dZ'_{x_m}$. It follows that the natural map
\[
\rH^1_f(F,\cR)/(\fP_{x_m},p^k)\rH^1_f(F,\cR)\to\rH^1_\ff(F,\cR_{/x_m}/p^k\cR_{/x_m})
\]
is also nonzero for $m\geq m_k$. Since this map factors through $\rH^1_\ff(F,\cR_{x_m})/p^k\rH^1_\ff(F,\cR_{/x_m})$ and $\rH^1_\ff(F,\cR_{/x_m})$ is a free $\dZ_{x_m}$-module of rank one, we have
\[
\length_{\dZ_q}\(\coker\(\rH^1_f(F,\cR)/\fP_{x_m}\rH^1_f(F,\cR)\to\rH^1_\ff(F,\cR_{/x})\)\)
<\length_{\dZ_q}(\dZ_{x_m}/p^k\dZ_{x_m})+\length_{\dZ_q}(\dZ_{x_m}/\dZ'_{x_m})
\]
for every $m\geq m_k$. It follows by Lemma \ref{le:iwasawa4}(1a) that \eqref{eq:iwasawa5} is finite. Thus, \eqref{eq:iwasawa3} hence (2c) are confirmed.

The theorem is all proved.
\end{proof}

To end this section, we explain how to adapt the proof of Theorem \ref{th:iwasawa} to the variant Theorem \ref{th:iwasawa_bis}. Note that the only problem of removing Assumption \ref{as:galois}(G8) is that we lose Theorem \ref{th:deformation}. From now on, suppose that $\rJ$ has the form $\rI^0_{n_0}\times\rJ_1$ for a subgroup $\rJ_1<\rI^0_{n_1}$.

For every place $v\in\PP^+$, we fix a self-dual $O_{F_v}$-lattice $\Lambda_{n_0,v}$ in $\rV_{n_0,v}$ such that $\sfB_{n_0,v}$ induces a Borel subgroup of $\rU(\Lambda_{n_0,v})$ as in \S\ref{ss:first_reciprocity}. Using this lattice, we regard $\sfG_{n_0,v}$, $\sfB_{n_0,v}$ and $\sfT_{n_0,v}$ as reductive groups defined over $O_{F^+_v}$. Put
\[
\rk_0\coloneqq\prod_{v\in\PP^+}\sfG_{n_0,v}(F^+_v).
\]
We also fix an element $\rk'_0\in\fk_{\rI^0_{n_0}}^\dag$ that is contained in $\rk_0$.

Consider a definite/indefinite datum $\bV=(\fm;\rV_n,\rV_{n+1};\Lambda_n,\Lambda_{n+1};\rK_n,\rK_{n+1};\sfB)$ as in Definition \ref{de:datum}. We define the \emph{ordinary stabilization map}
\[
\fs\colon\rH(\Sh(\sfG_{n_0},\rK_{n_0}\rk_0),\dZ_{\xi^{n_0}})\to\rH(\Sh(\sfG_{n_0},\rK_{n_0}\rI^0_{n_0}),\dZ_{\xi^{n_0}}),
\]
which depends on the choice of $\rk'_0$, to be the composition of the pullback map
\[
\rH(\Sh(\sfG_{n_0},\rK_{n_0}\rk_0),\dZ_{\xi^{n_0}})\to\rH(\Sh(\sfG_{n_0},\rK_{n_0}\rk'_0),\dZ_{\xi^{n_0}}).
\]
the projection map
\[
\rH(\Sh(\sfG_{n_0},\rK_{n_0}\rk'_0),\dZ_{\xi^{n_0}})\to\rH^\ordi(\Sh(\sfG_{n_0},\rK_{n_0}\rk'_0),\dZ_{\xi^{n_0}})
\]
from Remark \ref{re:supersingular}, and the inverse of the isomorphism
\[
\rH(\Sh(\sfG_{n_0},\rK_{n_0}\rI^0_{n_0}),\dZ_{\xi^{n_0}})\to\rH^\ordi(\Sh(\sfG_{n_0},\rK_{n_0}\rk'_0),\dZ_{\xi^{n_0}})
\]
from Lemma \ref{le:ordinary}.

For a subgroup $\rI\finite\rI^0$ of the form $\rI^0_{n_0}\times\rI_1$, put
\[
\rH'(\Sh(\sfG,\rK\rI),\dZ_\xi)\coloneqq
\rH(\Sh(\sfG_{n_0},\rK_{n_0}\rk_0),\dZ_{\xi^{n_0}})\otimes_{\dZ_q}
\rH(\Sh(\sfG_{n_1},\rK_{n_1}\rI_1),\dZ_{\xi^{n_1}})
\]
so that the map $\fs$ induces a map
\[
\fs_*\colon\rH'(\Sh(\sfG,\rK\rI),\dZ_\xi)\to\rH(\Sh(\sfG,\rK\rI),\dZ_\xi).
\]
Furthermore, put
\[
\cD'_\rJ(\xi,\bV)\coloneqq\varprojlim_{\rJ_1<\rI_1\finite\rI^0}
\Hom_{\dZ_q}\(\rH'(\Sh(\sfG,\rK(\rI^0_{n_0}\times\rI_1)),\dZ_{\xi^N}),\dZ_q\)
\]
so that have the natura map
\[
\fs^*\colon\cD_\rJ(\xi,\bV)\to\cD'_\rJ(\xi,\bV)
\]
of $\dZ[\rK_{\spadesuit^+}\backslash\sfG(F^+_{\spadesuit^+})/\rK_{\spadesuit^+}]
\otimes\dT^{\ang{\fm}}\otimes\Lambda_{\sfT(F^+_{\PP^+})/\rJ,\dZ_q}$-modules. Similarly, we have the localized version
\[
\fs^*_\gamma\colon\cD_\rJ(\xi,\bV)_\gamma\to\cD'_\rJ(\xi,\bV)_\gamma,
\]
which is surjective under Assumption \ref{as:galois}(G2,G7). Similarly, let $\cE'_\rJ(\xi,\bV)_\gamma$ be the $\Lambda_{\sfT(F^+_{\PP^+})/\rJ,\dZ_q}$-subalgebra of
\[
\End_{\Lambda_{\sfT(F^+_{\PP^+})/\rJ,\dZ_q}}\(\cD'_\rJ(\xi,\bV)_\gamma\)
\]
generated by the image of $\dT^{\ang{\fm}}$. Put
\begin{align*}
\sD'_\rJ(\xi,\bV)_\gamma\coloneqq\cD'_\rJ(\xi,\bV)_\gamma\otimes_{\dZ_q}\dQ_q,\qquad
\sE'_\rJ(\xi,\bV)_\gamma\coloneqq\cE'_\rJ(\xi,\bV)_\gamma\otimes_{\dZ_q}\dQ_q.
\end{align*}
Then $\fs^*_\gamma\otimes_{\dZ_q}\dQ_q$ has a unique $\dZ[\rK_{\spadesuit^+}\backslash\sfG(F^+_{\spadesuit^+})/\rK_{\spadesuit^+}]
\otimes\dT^{\ang{\fm}}\otimes\Lambda_{\sfT(F^+_{\PP^+})/\rJ,\dQ_q}$-linear section, so that $\sE'_\rJ(\xi,\bV)_\gamma$ is a quotient of $\sE_\rJ(\xi,\bV)_\gamma$. By definition, $\Spec\sE'_\rJ(\xi,\bV)_\gamma$ is precisely the maximal even-crystalline (open) subscheme of $\Spec\sE_\rJ(\xi,\bV)_\gamma$. To prove Theorem \ref{th:iwasawa_bis}, we adopt the same argument in the proof of Theorem \ref{th:iwasawa} but for $\cD'_\rJ(\xi,\bV^\fn)_\gamma$. This way, we avoid the use of Theorem \ref{th:deformation} while only requiring \cite{LTXZZ1}*{Theorem~3.38}.

Indeed, the proof of the main theorems in \cite{LTX} actually follows from the above strategy.

\section{Appendix: An R=T theorem for ordinary distributions}

In this section, we prove an R=T theorem for ordinary distributions over the Iwasawa algebra, similar to \cite{LTXZZ1}*{Theorem~3.38}.

Let $N$ be a positive integer satisfying $p>2(N+1)$. Following \cite{CHT08}, we denote by $\sC_{\dZ_q}^f$ the category of Artinian local rings over $\dZ_q$ with residue field $\dF_q$.

Consider a quintuple $(\xi,\fm,\rV,\Lambda,\sfB)$ in which
\begin{itemize}
  \item $\xi$ is a $\PP^+$-trivial Fontaine--Laffaille regular hermitian weight of rank $N$ (Definition \ref{de:weight0});

  \item $\fm$ is a (possibly empty) finite set of nonarchimedean places of $F^+$ disjoint from $\spadesuit^+\cup\Sigma_p^+$ such that every place $v\in\fm$ is inert in $F$ and satisfies $p\nmid(\|v\|^2-1)$;

  \item $\rV$ is a hermitian space over $F$ (of general signature) of rank $N$ that is split at every place in $\PP^+$, with $\sfG\coloneqq\rU(\rV)$;

  \item $\Lambda$ is a $\prod_{v\not\in\Sigma^+_\infty\cup\spadesuit^+\cup\PP^+}O_{F_v}$-lattice in $\rV\otimes_F\dA_F^{\Sigma_\infty\cup\spadesuit\cup\PP}$ satisfying $\Lambda\subseteq\Lambda^\vee$ and that $\Lambda_v^\vee/\Lambda_v$ has length one (resp.\ zero) when $v\in\fm$ (resp.\ $v\not\in\fm$);

  \item $\rK$ is a neat open compact subgroup of $\sfG(\dA_{F^+}^{\infty,\PP^+})$ of the form
      \begin{align*}
      \rK=\prod_{v\in\spadesuit^+}\rK_v\times
      \prod_{v\not\in\Sigma^+_\infty\cup\spadesuit^+\cup\PP^+}\rU(\Lambda)(O_{F^+_v});
      \end{align*}

  \item $\sfB$ is a Borel subgroup of $\sfG\otimes_\dQ\dQ_p$, with $\sfT$ its Levi quotient.
\end{itemize}
Denote by $\rI^0$ the maximal open compact subgroup of $\sfT(F^+_{\PP^+})$. We also put
\[
d(\rV)\coloneqq\sum_{v\in\Sigma^+_\infty}p_vq_v,
\]
where $(p_v,q_v)$ denotes the signature of $\rV$ at $v$.

For every subgroup $\rI\finite\rI^0$, put
\[
\rH^{d(V)}_{\et}(\Sh(\sfG,\rK\rI),\dZ_\xi)\coloneqq
\varprojlim_{\fk_\rI^{\r{op}}}\rH^{d(V)}_{\et}(\Sh(\sfG,\rK\rk)_{\ol{F}},\dZ_\xi).
\]
Put
\[
\cD(\xi,\rV,\rK)\coloneqq\varprojlim_{\rI\finite\rI^0}\Hom_{\dZ_q}\(\rH^{d(V)}_{\et}(\Sh(\sfG,\rK\rI),\dZ_\xi),\dZ_q\),
\]
as a module over $\dT_N^{\ang{\fm}}\otimes\Lambda_{\sfT(F^+_{\PP^+}),\dZ_q}$, admitting a continuous action by $\Gamma_F$.

Consider a hermitian homomorphism (Definition \ref{de:hermitian})
\begin{align*}
\gamma\colon\Gamma_{F^+}\to\sG_N(\dF_q)
\end{align*}
that is unramified away from $\spadesuit^+\cup\Sigma^+_p$. Associate with $\gamma$, we have a homomorphism
\[
\phi_\gamma\colon\dT^{\ang{\emptyset}}_N\to\dF_q.
\]

\begin{definition}\label{de:filtered}
Let $v$ be a place in $\PP^+$.
\begin{enumerate}
  \item We say that a homomorphism $\rho\colon\Gamma_{F_{w(v)}}\to\GL_N(R)$, in which $R$ is an arbitrary local $\dZ_q$-ring, is \emph{filtered} if there exists a (necessarily unique) complete filtration
      \[
      0 = \Fil^{-1}(\rho) \subseteq \Fil^0(\rho)
      \subseteq\cdots\subseteq \Fil^{N-1}(\rho)=R^N
      \]
      of free $R$-modules that is stable under $\rho$, such that for every $0\leq i\leq N-1$, if we denote by $\rho_i\colon\Gamma_{F_{w(v)}}\to R^\times$ the character by which $\Gamma_{F_{w(v)}}$ acts on
      \[
      \Gr^i(\rho)\coloneqq\frac{\Fil^i(\rho)}{\Fil^{i-1}(\rho)},
      \]
      then the restriction of $\rho_i$ to the maximal torsion subgroup of $\Gamma_{F_{w(v)}}^\ab$ coincides with $\epsilon_p^{-i}$.

  \item We say that a local lifting
      \[
      \gamma_v^\sharp\colon\Gamma_{F^+_v}\to\sG_N(R)
      \]
      of $\gamma_v$ to an object $R\in\sC_{\dZ_q}^f$ is \emph{filtered} if $(\gamma_v^\sharp)^\natural\colon\Gamma_{F_{w(v)}}\to\GL_N(R)$ is filtered in the above sense.
\end{enumerate}
\end{definition}

We assume that the homomorphism $\gamma$ satisfies the following conditions.
\begin{description}
  \item[(D1)] The restriction $\gamma^\natural\res_{\Gal(\ol{F}/F(\zeta_p))}$ is absolutely irreducible.

  \item[(D2)] The homomorphism $\phi_\gamma$ is cohomologically generic in the sense of \cite{LTXZZ}*{Definition~D.1.1}.

  \item[(D3)] For every $v\in\spadesuit^+$, every lifting of $\gamma_v$ is minimally ramified \cite{LTXZZ}*{Definition~3.4.8}.

  \item[(D4)] For every $v\in\fm$, the generalized eigenvalues of $\gamma^\natural(\phi_{w(v)})$ in $\ol\dF_p$ contain the pair $\{\|v\|^{-N},\|v\|^{-N+2}\}$ exactly once.

  \item[(D5)] For every $w\in\Sigma_p$, $(\gamma^\natural)_w$ is regular Fontaine--Laffaille crystalline of weights $(\xi_\tau)_{\tau\mid w}$.

  \item[(D6)] For every $v\in\PP^+$, $\gamma_v$ is filtered such that the set
      \[
      \left\{\left.(\gamma_v)^\sharp_i\cdot\epsilon_p^i\right| 0\leq i\leq N-1\right\}
      \]
      of (unramified) characters of $\Gamma_{F_{w(v)}}$ are distinct.
\end{description}

Consider a global deformation problem \cite{LTXZZ1}*{Definition~3.1.6}
\[
\sS\coloneqq(\gamma,\eta_{F/F^+}^N\epsilon_p^{1-N},\spadesuit^+\cup\fm\cup\Sigma^+_p,\{\sD_v\}_{v\in\spadesuit^+\cup\fm\cup\Sigma^+_p})
\]
where
\begin{itemize}
  \item for $v\in\spadesuit^+$, $\sD_v$ is the local deformation problem classifying all liftings of $\gamma_v$;

  \item for $v\in\fm$, $\sD_v$ is the local deformation problem $\sD_v^\ram$ of $\gamma_v$ from \cite{LTXZZ1}*{Definition~3.5.1};

  \item for $v\in\Sigma^+_p\setminus\PP^+$, $\sD_v$ is the local deformation problem $\sD_v^\FL$ of $\gamma_v$ from \cite{LTXZZ1}*{Definition~3.2.5};

  \item for $v\in\PP^+$, $\sD_v$ is the local deformation problem $\sD_v^{\r{fil}}$ of $\gamma_v$ classifying filtered liftings (Definition \ref{de:filtered}(2)).
\end{itemize}
Then we have the \emph{global universal deformation ring} $\sfR^\univ_\sS$ from \cite{LTXZZ1}*{Proposition~3.1.7}. The ring $\sfR^\univ_\sS$ is naturally a topological ring over $\Lambda_{\sfT(F^+_{\PP^+}),\dZ_q}$ in the following way: For every object $R\in\sC_{\dZ_q}^f$, the map
\[
\Hom_{\dZ_q}(\sfR^\univ_\sS,R)\to\Hom_{\dZ_q}(\Lambda_{\sfT(F^+_{\PP^+}),\dZ_q},R)
\]
sends a deformation $\gamma^\sharp$ to the homomorphism given by the characters $\{\chi_v\colon\sfT(F^+_v)\to R^\times\res v\in\PP^+\}$ satisfying that for every $v\in\PP^+$ and every $0\leq i\leq N-1$, the character $(\gamma_v^\sharp)^\natural_i$ is given by the composition
\[
\Gamma_{F_{w(v)}}\xrightarrow{\rec_{w(v)}}\widehat{F_{w(v)}^\times}\xrightarrow{\inc_i}\widehat{\sfT(F_{w(v)})}
\xrightarrow{\Nm_{F_{w(v)}/F^+_v}}\widehat{\sfT(F^+_v)}\xrightarrow{\chi_v^{-1}\cdot\epsilon_p^{-i}} R^\times
\]
in which $\inc_i$ is the inclusion of the factor in \eqref{eq:weight} indexed by $i$ (with $\sigma$ the natural embedding $F^+_v\to F_{w(v)}$).

\begin{theorem}\label{th:deformation}
Suppose that $\fm$ is empty if $N$ is odd, that $\gamma$ satisfies (D1--6) above, and that $\cD(\xi,\rV,\rK)_\gamma\neq0$. Let $\Box$ be a finite set of places of $F^+$ containing $\ang{\fm}$, and denote by $\sfT^\Box$ the subring of $\End_{\Lambda_{\sfT(F^+_{\PP^+}),\dZ_q}}\(\cD(\xi,\rV,\rK)_\gamma\)$ generated by the image of $\dT^\Box_N$. Then
\begin{enumerate}
  \item There is a canonical isomorphism $\sfR^\univ_\sS\xrightarrow{\sim}\sfT^\Box$ of flat local complete intersection rings over $\Lambda_{\rI^0,\dZ_q}$.

  \item The $\sfT^\Box$-module $\cD(\xi,\rV,\rK)_\gamma$ is finite and free.
\end{enumerate}
In particular, the natural inclusion $\sfT^\Box\hookrightarrow\sfT^{\ang{\fm}}$ is an isomorphism.
\end{theorem}

Take an element $v\in\PP^+$. Let $\sD_v^{\r{fil}}$ be the local deformation problem of $\gamma$ classifying filtered liftings, which is naturally a formal scheme over $\Spf\Lambda_{\sfT(F^+_v),\dZ_q}$ in the way that is similar to $\sfR^\univ_\sS$. Denote by $\rI^0_v$ the maximal open compact subgroup of $\sfT(F^+_v)$ and $\rI^{00}_v$ the maximal torsion free quotient of $\rI^0_v$, so that $\Spf\Lambda_{\rI^{00}_v,\dZ_q}$ is a connected component of $\Spf\Lambda_{\rI^0_v,\dZ_q}$.

\begin{lem}\label{le:deformation}
Assume (D5) and (D6).
\begin{enumerate}
  \item The image of the composite morphism $\sD_v^{\r{fil}}\to\Spf\Lambda_{\sfT(F^+_v),\dZ_q}\to\Spf\Lambda_{\rI^0_v,\dZ_q}$ is contained in $\Spf\Lambda_{\rI^{00}_v,\dZ_q}$.

  \item If we denote by $\fa$ the augmentation ideal of $\Lambda_{\rI^{00}_v,\dZ_q}$ (so that $\Lambda_{\rI^{00}_v,\dZ_q}/\fa=\dZ_q$), then $\sD_v^{\r{fil}}\otimes_{\Lambda_{\rI^{00},\dZ_q}}\Lambda_{\rI^{00},\dZ_q}/\fa=\sD_v^\FL$ -- the deformation ring of regular Fontaine--Laffaille liftings of $\gamma_v$ (with regular Fontaine--Laffaille weights $\{0,1,\dots,N-1\}$).

  \item The induced morphism $\sD_v^{\r{fil}}\to\Spf\Lambda_{\rI^{00}_v,\dZ_q}$ is formally smooth of pure relative dimension $N^2+[F_v^+:\dQ_p]\cdot\frac{N(N-1)}{2}$.
\end{enumerate}
\end{lem}

\begin{proof}
To ease notation, we suppress $v$ in the proof. In particular, $F\coloneqq F_{w(v)}$ is an extension of $F^+\coloneqq F^+_v$ of degree either $1$ or $2$. Put $\rho\coloneqq\gamma^\natural\colon\Gamma_F\to\GL_N(\dF_q)$. Fix a Frobenius element $\phi\in\Gamma_{F^+}$ satisfying $\epsilon_p(\phi)=1$.

Write $\sD^+$ for $\sD^{\r{fil}}$ and define $\sD$ to be the formal scheme over $\Spf\dZ_q$ classifying filtered liftings of $\rho$. Denote by $\sT$ the formal scheme over $\Spf\dZ_q$ classifying liftings of $(\rho_0,\dots,\rho_{N-1})$, and define $\sT^+$ to be $\sT$ when $[F:F^+]=1$ and to be the closed subscheme classifying liftings $(\rho^\sharp_0,\dots,\rho^\sharp_{N-1})$ satisfying $\rho^\sharp_i\cdot(\rho^\sharp_{N-1-i})^\phi=1$ for $0\leq i\leq N-1$ when $[F:F^+]=2$. Then we have a natural commutative diagram
\[
\xymatrix{
\sD^+ \ar[r]^-\nu\ar[d]_-{\tau^+} & \sD \ar[d]^-{\tau} \\
\sT^+ \ar[r] & \sT
}
\]
in which $\nu$ sends a lifting $\gamma^\sharp$ to $(\gamma^\sharp)^\natural$ and $\tau$ sends a lifting $\rho^\sharp$ to $(\rho^\sharp_0,\dots,\rho^\sharp_{N-1})$.

The morphism $\sD^+\to\Spf\Lambda_{\rI^0,\dZ_q}$ we want to study is the composition of $\tau^+$ with the morphism $\sigma\colon\sT^+\to\Spf\Lambda_{\rI^0,\dZ_q}$ that sends an $R$-point $(\rho^\sharp_0,\dots,\rho^\sharp_{N-1})$ to the unique character $\chi^\sharp\colon\rI^0\to R^\times$ such that for every $0\leq i\leq N-1$,
\[
\rI_F\xrightarrow{\rec}O_F^\times\xrightarrow{\Nm_{F/F^+}\circ\inc_i}\rI^0\xrightarrow{(\chi^\sharp)^{-1}\cdot\epsilon_p^{-i}} R^\times
\]
coincides with $\rho^\sharp_i\res_{\rI_F}$. By (D5), the image $\sigma$ is contained in $\Spf\Lambda_{\rI^{00},\dZ_q}$, that is, (1) follows.

Put
\[
\lambda\coloneqq\sigma\circ\tau^+\colon \sD^+\to\Spf\Lambda_{\rI^{00},\dZ_q},
\]
with the rigid analytic (generic) fiber $\lambda_{\dQ_q}\colon\sD^+_{\dQ_q}\to\Sp\Lambda_{\rI^{00},\dQ_q}$. We claim that every geometric fiber of $\lambda_{\dQ_q}$ has dimension at least $N^2+[F^+:\dQ_p]\cdot\frac{N(N-1)}{2}$.

Assume the claim. By (D6), $\sD^+\otimes_{\Lambda_{\rI^{00},\dZ_q}}\Lambda_{\rI^{00},\dZ_q}/\fa$ is a (necessarily closed) subscheme of $\sD^{\r{FL}}$. By \cite{LTXZZ1}*{Proposition~3.14}, $\sD^{\r{FL}}$ is formally smooth over $\Spf\dZ_q$ of pure relative dimension $N^2+[F^+:\dQ_p]\cdot\frac{N(N-1)}{2}$. It follows that by the claim that $\sD^+\otimes_{\Lambda_{\rI^{00},\dZ_q}}\Lambda_{\rI^{00},\dZ_q}/\fa=\sD^{\r{FL}}$, that is, (2) follows. In particular, $\lambda$ is formally smooth over the (unique) closed point of $\Spf\Lambda_{\rI^{00},\dZ_q}$. By the claim again and Nakayama's lemma, $\lambda$ is formally smooth, that is, (3) follows.

It remains to show the claim. Consider indices $0\leq i<j\leq N-1$. Denote by $\sE_{i,j}$ the formal scheme over $\sT$ such that the fiber over a point $(\rho^\sharp_0,\dots,\rho^\sharp_{N-1})$ of $\sT$ classifies extensions of $\rho^\sharp_j$ by $\rho^\sharp_i$. By (D6), $\sE_{i,j}$ is a vector bundle over $\sT$ of rank $[F:\dQ_p]$. Let $\varepsilon_{i,j}\colon\sD\to\sE_{i,j}$ be the morphism that sends a lifting $\rho^\sharp$ to the extension class given by $\rho^\sharp$, and denote by $\hat\sE_{i,j}$ the formal completion of $\sE_{i,j}$ along the image of the unique closed point of $\sD$ under $\varepsilon_{i,j}$ so that $\varepsilon_{i,j}$ factors through $\hat\sE_{i,j}$. Put
\[
\hat\sE\coloneqq\prod_{0\leq i<j\leq N-1}\hat\sE_{i,j},
\]
where the fiber product is taken over $\sT$, and put $\varepsilon\coloneqq(\varepsilon_{i,j})\colon\sD\to\hat\sE$. The morphism $\tau$ then factors as
\[
\sD\xrightarrow{\varepsilon}\hat\sE\xrightarrow{\pi}\sT,
\]
where $\pi$ is the natural projection.

On the other hand, if we denote by $\sF$ the formal scheme over $\Spf\dZ_q$ that classifies liftings of the filtration $\Fil^\bullet(\rho)$, which is formally smooth over $\Spf\dZ_q$ of pure relative dimension $\frac{N(N-1)}{2}$, then we have a natural morphism $\sD\to\sF$ sending a lifting $\rho^\sharp$ to $\Fil^\bullet(\rho^\sharp)$. Denote by $\sM$ the formal completion of $\GL_{N,\dZ_q}$ along the identity on the special fiber, which acts on $\sD$ by conjugation and on $\sF$ by the natural action. Denote by $\sD_0\subseteq\sD$ the fiber over the standard filtration, namely, the one whose stabilizer is the upper-triangular Borel (without loss of generality, we may assume that the original filtration $\Fil^\bullet(\rho)$ is the standard one), and put
\[
\tau_0\coloneqq\tau\res_{\sD_0},\qquad
\sD^+_0\coloneqq\nu^{-1}\sD^+,\qquad
\tau^+_0\coloneqq\tau^+\res_{\sD^+_0},\qquad
\lambda_0\coloneqq\sigma\circ\tau^+_0.
\]
Since $\sG$ acts transitively on $\sF$ and the morphism $\sD\to\sF$ is $\sG$-equivariant, it suffices to show that each geometric fiber of $\lambda_{0,\dQ_q}$ has dimension at least
\[
N^2+[F^+:\dQ_p]\cdot\frac{N(N-1)}{2}-\frac{N(N-1)}{2}=(1+[F^+:\dQ_p])\cdot\frac{N(N-1)}{2}+N.
\]
We now have two cases.

First, suppose that $[F:F^+]=1$. Then $\sD^+=\sD$ (and $\sT^+=\sT$). The morphism $\sigma\colon\sT\to\Spf\Lambda_{\rI^{00},\dZ_q}$ is formally smooth of relative dimension $N$. Thus, it suffices to show that each geometric fiber of the morphism $\tau_{0,\dQ_q}$ has dimension $(1+[F^+:\dQ_p])\cdot\frac{N(N-1)}{2}$. Take a geometric point $x$ of $\sT_{\dQ_q}$ and denote by $\sD_{0,x}$ its fiber. Let $\sN\subseteq\sM$ be the upper-triangular unipotent subgroup, which is formally smooth over $\Spf\dZ_q$ of relative dimension $\frac{N(N-1)}{2}$. By (D6), the conjugation action of $\sN_{\dQ_q}$ on $\sD_{0,x}$ is free, and the morphism $\tau_{0,\dQ_q}$ induces an isomorphism $\sD_{0,x}/\sN_{\dQ_q}\simeq\hat\sE_x$, where $\hat\sE_x$ denotes the fiber of $x$ under $\pi$. It follows that $\sD_{0,x}$ has dimension $(1+[F^+:\dQ_p])\cdot\frac{N(N-1)}{2}$. The claim is proved.

Next, suppose that $[F:F^+]=2$. Write $r\coloneqq\lfloor\frac{N}{2}\rfloor$. For $0\leq i<j\leq N-1$, denote by $\sE_{i,j}^\circ$ the Zariski open subset of $(\sE_{i,j})_{\dQ_q}$ that is the complement of the zero section (of the bundle $\sE_{i,j}\to\sT$), and put $\hat\sE_{i,j}^\circ\coloneqq(\hat\sE_{i,j})_{\dQ_q}\cap\sE_{i,j}^\circ$. Put $\hat\sE^\circ\coloneqq\prod_{0\leq i<j\leq N-1}\hat\sE_{i,j}^\circ$. Denote by $\sD_0^\circ\subseteq\sD_0$ the open subset whose image under $\varepsilon$ is contained $\hat\sE^\circ$, and by $\sD_0^{+\circ}$ the inverse image of $\sD_0^\circ$ under $\nu$, as rigid analytic spaces over $\dQ_q$. Note that by (D6), for every filtered homomorphism represented by $\sD_0^\circ$, its centralizer must be the scalar matrices. It follows by \cite{LTXZZ1}*{Lemma~2.3} that
\begin{enumerate}
  \item the induced morphism $\sD_0^{+\circ}\to\sD_0^\circ$ is a relative principal homogeneous space of $\sZ_{\dQ_q}$, where $\sZ$ denotes the center of $\sM$;

  \item if we denote by $\sD_0^{\circ+}$ the image of $\sD_0^{+\circ}$ under $\nu$, then it is the (Zariski closed) locus of liftings $\rho^\sharp$ such that $(\rho^\sharp)^\phi$ and $\pres\rt(\rho^\sharp)^{-1}$ are conjugate.
\end{enumerate}
On the other hand, now the morphism $\sigma\colon\sT^+\to\Spf\Lambda_{\rI^{00},\dZ_q}$ is formally smooth of relative dimension $r$. It suffices to show that each geometric fiber of the morphism $\tau_0\res_{\sD_0^{\circ+}}\colon\sD_0^{\circ+}\to\sT^+_{\dQ_q}$ has dimension $(1+[F^+:\dQ_p])\cdot\frac{N(N-1)}{2}+(N-r-1)$. Denote by $\hat\sE^{\circ+}$ the image of $\sD_0^{\circ+}$ under $\varepsilon$, which is a closed subspace of $\hat\sE^\circ\times_{\sT_{\dQ_q}}\sT^+_{\dQ_q}$, so that $\tau_0\res_{\sD_0^{\circ+}}$ factorises as
\[
\sD_0^{\circ+}\xrightarrow{\varepsilon^{\circ+}}\hat\sE^{\circ+}\xrightarrow{\pi^{\circ+}}\sT^+_{\dQ_q}.
\]
Once again, the morphism $\varepsilon^{\circ+}$ is a relative principal homogeneous space of $\sN_{\dQ_q}$. Thus, it suffices to show that the morphism $\pi^{\circ+}$ is of pure relative dimension $[F^+:\dQ_p]\cdot\frac{N(N-1)}{2}+(N-r-1)$. Fix a geometric point $x$ of $\sT^+_{\dQ_q}$ represented by characters $\rho^x_0,\dots,\rho^x_{N-1}\colon\Gamma_F\to\dK^\times$. For every $0\leq i<j\leq N-1$, we have an isomorphism
\[
\iota\colon\Ext^1(\rho^x_j,\rho^x_i)\to\Ext^1(\rho^x_{N-1-i},\rho^x_{N-1-j})
\]
of $\dK$-linear spaces of dimension $[F:\dQ_p]$ that is the composition of the isomorphism
\[
\Ext^1(\rho^x_j,\rho^x_i)\to\Ext^1((\rho^x_j)^\phi,(\rho^x_i)^\phi)
=\Ext^1((\rho^x_{N-i-j})^{-1},(\rho^x_{N-1-i})^{-1})
\]
induced by taking $\phi$-conjugation, and the natural isomorphism
\[
\Ext^1((\rho^x_{N-i-j})^{-1},(\rho^x_{N-1-i})^{-1})
\simeq\Ext^1(\rho^x_{N-i-i},\rho^x_{N-1-j}).
\]
In particular, when $j=N-1-i$, $\iota$ is an involution of $\Ext^1(\rho^x_{N-1-i},\rho^x_i)$, satisfying that
\[
\Ext^1(\rho^x_{N-1-i},\rho^x_i)=\Ext^1(\rho^x_{N-1-i},\rho^x_i)^{\iota=1}\oplus
\Ext^1(\rho^x_{N-1-i},\rho^x_i)^{\iota=-1}
\]
in which both eigenspaces have dimension $[F^+:\dQ_p]$. By (2) and easy linear algebra, a $\dK$-point
\[
\(x,\(0\neq e_{i,j}\in\Ext^1(\rho^x_j,\rho^x_i)\)_{0\leq i<j\leq N-1}\)
\]
of $\hat\sE^\circ$ locates in $\hat\sE^{\circ+}$ if and only if there exist elements $\alpha_0,\dots,\alpha_{N-1}\in\dK^\times$ such that
\[
e_{N-i-i,N-1-j}=(\alpha_i/\alpha_j)\cdot\iota(e_{i,j})
\]
holds for every $0\leq i<j\leq N-1$ (which forces $\alpha_i/\alpha_{N-1-i}\in\{\pm1\}$ for $0\leq i\leq N-1$). It follows that the dimension of the fiber of $x$ in $\hat\sE^{\circ+}$ equals
\begin{align*}
\sum_{0\leq i<j< N-1-i}\dim\Ext^1(\rho^x_j,\rho^x_i) +\frac{1}{2}\sum_{i=0}^{r-1}\dim\Ext^1(\rho^x_{N-1-i},\rho^x_i)+(N-1-r)
=[F^+:\dQ_p]\cdot\frac{N(N-1)}{2}+(N-r-1).
\end{align*}
Thus, the claim follows.

The lemma is proved.
\end{proof}

\begin{proof}[Proof of Theorem \ref{th:deformation}]
Put $\rI^{00}\coloneqq\prod_{v\in\PP^+}\rI^{00}_v$. Similarly, $\sfR^\univ_\sS$ and $\sfT^\Box$ are indeed rings over $\Lambda_{\rI^{00},\dZ_q}$ -- a regular local ring of dimension $1+N\sum_{v\in\PP^+}[F^+_v:\dQ_p]$. We run the same argument in the proof of \cite{LTXZZ1}*{Theorem~3.38} but replace $\sO$ by $\Lambda_{\rI^{00},\dZ_q}$, and $1$ (when it stands for the dimension of $\sO$) by $1+N\sum_{v\in\PP^+}[F^+_v:\dQ_p]$ (the dimension of $\Lambda_{\rI^{00},\dZ_q}$). We conclude that
\begin{enumerate}
  \item there is a canonical isomorphism $\sfR^\univ_\sS\xrightarrow{\sim}\sfT^\Box$ of local complete intersection rings over $\Lambda_{\rI^0,\dZ_q}$;

  \item the $\sfT^\Box$-module $\cD(\xi,\rV,\rK)_\gamma$ is finite and free.
\end{enumerate}
The flatness of $\sfT^\Box$ over $\Lambda_{\rI^0,\dZ_q}$ follows from (2) and the fact that $\cD(\xi,\rV,\rK)_\gamma$ is finite and free over $\Lambda_{\rI^0,\dZ_q}$ by (the same argument for) Proposition \ref{pr:eigen_2}.

Theorem \ref{th:deformation} is proved.
\end{proof}

\end{document}